\documentclass{amsart}

\usepackage{amsmath, amsfonts, amssymb}
\usepackage{enumerate}
\usepackage{enumitem}
\usepackage{hyperref}
\hypersetup{
  colorlinks   = true,
  urlcolor     = blue, 
  linkcolor    = blue, 
  citecolor   = red 
}
\usepackage{graphicx}
\usepackage[all]{xy}
\usepackage{wrapfig}
\usepackage{fancyvrb}
\usepackage[mathscr]{euscript}
\usepackage[T1]{fontenc}
\usepackage{listings}
\usepackage{tikz}
\usepackage{amsthm}
\usepackage{dsfont}
\usepackage{tikz-cd}
\usepackage{adjustbox}
\usepackage{centernot}
\usepackage{mathtools}
\usepackage[noadjust]{cite}
\usepackage{bbm}
\usepackage{slashed}
\usepackage{stmaryrd}

\makeatletter
\DeclareRobustCommand{\cev}[1]{
  {\mathpalette\do@cev{#1}}
}
\newcommand{\do@cev}[2]{
  \vbox{\offinterlineskip
    \sbox\z@{$\m@th#1 x$}
    \ialign{##\cr
      \hidewidth\reflectbox{$\m@th#1\vec{}\mkern4mu$}\hidewidth\cr
      \noalign{\kern-\ht\z@}
      $\m@th#1#2$\cr
    }
  }
}
\makeatother

\let\oldtocsection=\tocsection

\let\oldtocsubsection=\tocsubsection

\let\oldtocsubsubsection=\tocsubsubsection

\numberwithin{equation}{subsection}

\renewcommand{\tocsection}[2]{\hspace{0em}\oldtocsection{#1}{#2}}
\renewcommand{\tocsubsection}[2]{\hspace{1em}\oldtocsubsection{#1}{#2}}
\renewcommand{\tocsubsubsection}[2]{\hspace{2em}\oldtocsubsubsection{#1}{#2}}

\title{Endoscopy for Modular Hecke Categories}
\author{Colton Sandvik}
\date{}

\newcommand{\C}{\mathbb{C}}
\newcommand{\Q}{\mathbb{Q}}

\newcommand{\Z}{\mathbb{Z}}
\newcommand{\F}{\mathbb{F}}

\renewcommand{\O}{\mathbb{O}}
\renewcommand{\P}{\mathbb{P}}

\newcommand{\K}{\mathbb{K}}
\newcommand{\A}{\mathbb{A}}
\newcommand{\G}{\mathbb{G}}
\renewcommand{\k}{\mathbbm{k}}
\newcommand{\scrO}{\mathcal{O}}
\newcommand{\scrQ}{\mathscr{Q}}
\newcommand{\scrZ}{\mathscr{Z}}

\DeclareMathOperator{\Sch}{Sch}
\DeclareMathOperator{\Stk}{Stk}
\DeclareMathOperator{\IndStk}{IndStk}
\DeclareMathOperator{\an}{an}

\DeclareMathOperator{\pt}{pt}

\DeclareMathOperator{\pr}{pr}

\DeclareMathOperator{\TwStk}{TwStk}

\DeclareMathOperator{\Hom}{Hom}

\DeclareMathOperator{\Ho}{Ho}

\DeclareMathOperator{\End}{End}

\DeclareMathOperator{\Vect}{Vect}
\DeclareMathOperator{\Ob}{Ob}

\DeclareMathOperator{\Ch}{Ch}

\renewcommand{\mod}[1]{#1\textnormal{-mod}}
\newcommand{\rgrmod}[1]{\textnormal{mod}^{\Z}\textnormal{-}#1}

\newcommand{\grbim}[1]{#1\textnormal{-bim}^{\Z}}

\DeclareMathOperator{\Cat}{Cat}
\DeclareMathOperator{\Alg}{Alg}
\DeclareMathOperator{\BiAlg}{BiAlg}

\DeclareMathOperator{\id}{id}

\DeclareMathOperator{\Mod}{Mod}
\DeclareMathOperator{\LMod}{LMod}

\DeclareMathOperator{\Coinv}{Coinv}

\newcommand{\infCat}{\infty\textnormal{-}\Cat}
\DeclareMathOperator{\Corr}{Corr}
\DeclareMathOperator{\Sh}{Sh}

\newcommand{\uw}{\underline{w}}
\newcommand{\ux}{\underline{x}}
\newcommand{\uy}{\underline{y}}

\newcommand{\BSBim}{\textbf{BS}\textnormal{Bim}}
\newcommand{\SBim}{\mathbb{S}\textnormal{Bim}}
\newcommand{\us}{\underline{s}}
\newcommand{\Realize}{\mathfrak{R}}
\newcommand{\NiceRealize}{\widetilde{\mathfrak{R}}}

\newcommand{\uH}{\underline{H}}
\newcommand{\W}[2]{{}_{#1}W_{#2}}
\newcommand{\uW}[2]{{}_{#1}\underline{W}_{#2}}

\newcommand{\scrF}{\mathcal{F}}
\newcommand{\scrG}{\mathcal{G}}
\newcommand{\scrH}{\mathcal{H}}
\newcommand{\scrC}{\mathcal{C}}

\newcommand{\scrD}{\mathcal{D}}
\newcommand{\scrE}{\mathcal{E}}

\newcommand{\scrA}{\mathcal{A}}
\newcommand{\scrK}{\mathcal{K}}

\newcommand{\scrB}{\mathcal{B}}
\newcommand{\scrL}{\mathcal{L}}
\newcommand{\scrP}{\mathcal{P}}

\DeclareMathOperator{\Perv}{Perv}

\DeclareMathOperator{\Par}{Par}

\DeclareMathOperator{\Loc}{Loc}

\DeclareMathOperator{\BS}{BS}
\DeclareMathOperator{\cons}{\textnormal{c}}

\DeclareMathOperator{\ic}{ic}

\DeclareMathOperator{\lisse}{l}
\DeclareMathOperator{\indlisse}{il}

\DeclareMathOperator{\supp}{supp}
\newcommand{\DD}{\mathbb{D}}
\DeclareMathOperator{\IC}{IC}
\DeclareMathOperator{\Av}{Av}
\DeclareMathOperator{\For}{For}
\newcommand{\RHom}{\textbf{R}\mathcal{H}om}
\newcommand{\uk}{\underline{\k}}

\newcommand{\bfX}{\mathbf{X}}
\newcommand{\bfY}{\mathbf{Y}}

\DeclareMathOperator{\SL}{SL}

\DeclareMathOperator{\GL}{GL}

\newcommand{\D}[2]{{}_{#1 \backslash} \scrD_{/ #2}}
\newcommand{\paraD}[3]{{}_{#1 \backslash} \scrD_{#3 / #2}}
\newcommand{\DE}[1]{\scrD_{/ #1}}
\newcommand{\DME}[2]{\:_{#1 \fatbslash} D_{/ #2} }
\newcommand{\DEE}[2]{\:_{#1 \backslash} D_{/ #2} }
\newcommand{\PME}[2]{\:_{#1 \fatbslash} \Par_{/ #2} }
\newcommand{\PEE}[2]{\:_{#1 \backslash} \Par_{/ #2} }
\newcommand{\Parity}[2]{{}_{#1 \backslash} \textnormal{Par}_{/ #2}}
\newcommand{\paraParity}[3]{{}_{#1 \backslash} \textnormal{Par}_{#3 / #2}}
\newcommand{\OneParity}[2]{{}_{#1 \backslash}^1 \textnormal{Par}_{/ #2}}

\newcommand{\PE}[1]{ \textnormal{Par}_{/ #1}}
\newcommand{\Whit}[1]{ \mathcal{W}_{/ #1}}

\newcommand{\ForME}[1]{{}_{#1 \backslash} \textnormal{For}}
\newcommand{\uDelta}{\underline{\Delta}}
\newcommand{\unabla}{\underline{\nabla}}
\newcommand{\ParH}{\textnormal{Par}^H}

\newcommand{\Fl}{\mathcal{F}\ell}
\newcommand{\eFl}{\widetilde{\mathcal{F}\ell}}
\newcommand{\BGB}{B \backslash G/B}
\newcommand{\UGU}{U \backslash G/U}

\renewcommand{\H}{\mathbb{H}}
\DeclareMathOperator{\Fun}{Fun}
\newcommand{\fr}[1]{\mathfrak{#1}}
\DeclareMathOperator{\op}{op}

\renewcommand{\Pr}{\textnormal{Pr}}

\DeclareMathOperator{\ev}{ev}

\DeclareMathOperator{\odd}{odd}

\DeclareMathOperator{\fg}{fg}
\DeclareMathOperator{\pty}{pty}

\newcommand{\et}{\textnormal{ét}}
\DeclareMathOperator{\ch}{ch}
\newcommand{\abs}[1]{\lvert #1 \rvert}
\DeclareMathOperator{\re}{re}
\DeclareMathOperator{\mon}{mon}
\DeclareMathOperator{\ind}{Ind}
\DeclareMathOperator{\AS}{AS}
\renewcommand{\emptyset}{\varnothing}

\newtheorem{corollary}[subsubsection]{Corollary}
\newtheorem{lemma}[subsubsection]{Lemma}
\newtheorem{proposition}[subsubsection]{Proposition}
\newtheorem{theorem}[subsubsection]{Theorem}
\newtheorem{notation}[subsubsection]{Notation}
\newtheorem{claim}[subsubsection]{Claim}

\theoremstyle{definition}
\newtheorem{definition}[subsubsection]{Definition}
\newtheorem{example}[subsubsection]{Example}
\newtheorem{remark}[subsubsection]{Remark}

\newenvironment{midsecproof}[1]{\vspace{\topsep} \noindent \textit{Proof of #1.}}{\hfill$\square$}

\begin{document}
\begin{abstract}
        Generalizing the theory of parity sheaves on complex algebraic stacks due to Juteau--Mautner--Williamson, we develop a theory of twisted equivariant parity sheaves. 
        We use this formalism to construct a modular incarnation of Lusztig and Yun's monodromic Hecke category. 
        We then give two applications: (1) a modular categorification of the monodromic Hecke algebra, and (2) a monoidal equivalence between the monodromic Hecke category of parity sheaves and the ordinary Hecke category of parity sheaves on the endoscopic group.  
  \end{abstract}

	\maketitle

        \begingroup
        \hypersetup{hidelinks}
	\tableofcontents
        \endgroup

	\section{Introduction}\label{sec:endo_intro}

        \subsection{Motivation}

Endoscopic patterns in representation theory largely date back to Lusztig's classification of irreducible complex representations of finite reductive groups \cite{Lu84}.
Informally, Lusztig showed one can reduce the study of all irreducible complex representations to just unipotent representations of the endoscopic group.
In recent years, related results have shown up throughout geometric representation theory. 
These results have largely taken the form of an equivalence between a monodromic Hecke category and a Hecke category corresponding to the endoscopic group. 

\subsubsection{Monodromic Hecke Categories}

Let $G$ be a connected reductive group. Denote by $B$ a Borel subgroup of $G$ with maximal torus $T$ and unipotent radical $U$. Let $\k$ be a field of characteristic $p \geq 0$.
The \emph{Hecke category}, $D_{\cons} (\BGB, \k)$, is defined to be the derived category of constructible sheaves on $\BGB$.
It is a monoidal category under the convolution product. The Hecke category is well-studied in geometric representation theory. 
The semisimple complexes in $D_{\cons} (\BGB, \C)$ categorify the Hecke algebra and the perverse sheaves in $D_{\cons} (\BGB, \C)$ can be identified with category $\scrO$ for the Lie algebra $\fr{g}$ of $G$.

We will study twisted equivariant versions of $D_{\cons} (\BGB, \k)$ where we twist the left and right equivariance constraints by rank 1 multiplicative local systems. The set of all such local systems is denoted by $\Ch (T, \k)$.
For $\scrL, \scrL' \in \Ch (T, \k)$, one can consider $\D{\scrL'}{\scrL} (\k) \coloneq D_{\cons} (T \backslash_{\scrL'} \UGU /_{\scrL} T, \k)$, the $(T \times T, \scrL' \boxtimes \scrL^{-1})$-equivariant derived category of constructible sheaves on $\UGU$ with coefficients in $\k$.
Following Lusztig and Yun, we call this category the \emph{monodromic Hecke category} \cite{LY}.
The monodromic Hecke categories also admit a convolution bifunctor; although, it does not endow $\D{\scrL'}{\scrL} (\k)$ with a monoidal structure unless $\scrL = \scrL'$.
In the case where $\scrL, \scrL' = \uk_T$, we recover the non-monodromic Hecke category $\D{\uk}{\uk} (\k) = D_{\cons} (\BGB, \k)$.

\subsubsection{Endoscopic Equivalences}

If $G$ is defined over $\overline{\F}_p$ and $\k = \overline{\Q}_{\ell}$ where $p \neq \ell$, one can consider a mixed incarnation of the Hecke category $D_{\cons}^m (\BGB, \overline{\Q}_\ell)$.
Similarly, there is a mixed incarnation of the monodromic Hecke category $\D{\scrL'}{\scrL}^m (\k)$.
The (mixed) monodromic Hecke category decomposes into a direct sum of full subcategories called \emph{blocks}. 
We denote $\D{\scrL}{\scrL}^{m, \circ} (\k)$ the block containing the monoidal unit. This block is called the \emph{neutral block}.

Associated to $\scrL$, we can define a root system $\Phi_{\scrL}$. This gives rise to the \emph{endoscopic group} $H_{\scrL}^{\circ}$ which is a reductive group over $\overline{\F}_p$
with maximal torus $T$ and root system $\Phi_{\scrL}$.
The choice of Borel $B$ gives a Borel $B_{\scrL}$ of $H_{\scrL}^{\circ}$. 
The main result of \cite{LY} gives an equivalence of monoidal categories
\begin{equation}\label{eq:LY_endo}
  D_{\cons}^m (B_{\scrL} \backslash H_{\scrL}^{\circ} /B_{\scrL}, \overline{\Q}_\ell) \cong \D{\scrL}{\scrL}^{m, \circ} (\overline{\Q}_\ell).
\end{equation}
This equivalence is referred to as the endoscopic-monodromic equivalence.
The equivalence of \cite{LY} can also be extended to the non-neutral blocks using 2-categories.
In \cite{Li}, the equivalence of \cite{LY} was extended to the case of loop groups; albeit, without the extension to non-neutral blocks.

The main obstacle in defining a version of (\ref{eq:LY_endo}) for positive characteristic coefficients is that the categories of mixed sheaves are no longer defined.
Nonetheless, in recent years there have been a few modular incarnations of the endoscopic-monodromic equivalence obtained by replacing mixed biequivariant sheaves with non-mixed free-monodromic sheaves \cite{Gou, Eteve24a}.
In this paper, we will prove another modular version of the endoscopic-monodromic equivalence, but on the biequivariant side.

Motivated by work of Achar and Riche, we will replace the mixed categories appearing in (\ref{eq:LY_endo}) with categories of parity sheaves \cite{AR2}.
The category of parity sheaves $\Par (\BGB, \k)$ on $\BGB$ has a rich history of its own. It provides a categorification of the Hecke algebra giving rise to the $p$-canonical bases and $p$-Kazhdan--Lusztig polynomials.
While there is no significant obstruction in defining parity sheaves contained in $\D{\scrL'}{\scrL} (\k)$, the original formulation of \cite{JMW} is not sufficiently general to allow for them.

\subsubsection{Parity Sheaves}

Following \cite{JMW}, we extend the theory of parity sheaves to incorporate twisted equivariance.
More precisely, let $H$ be a complex algebraic group, $X$ be an algebraic $H$-stack over $\C$, and $\scrL$ a multiplicative rank one local system on $H$.
Under suitable conditions on $D_{\cons} (H \backslash_{\scrL} X, \k)$, we define an additive category of parity sheaves $\Par (H \backslash_{\scrL}
X, \k)$.
In the case of $\scrL = \uk_G$, we recover the usual theory of parity sheaves on $H \backslash X$.

The theory of twisted equivariant parity sheaves shares many similarities with the usual theory of parity sheaves. In particular, parity extensions are unique and there
are variants of the existence criterion of \cite{JMW}.
We expect that twisted equivariant parity sheaves can find themselves useful in other areas of geometric representation theory.
For example, one can discuss twisted equivariant parity sheaves on toric varieties and affine Grassmannians.

We will apply these constructions to the monodromic Hecke category to construct categories of parity sheaves $\Parity{\scrL'}{\scrL} (\k) \coloneq \Par (T \backslash_{\scrL'} \UGU /_{\scrL} T, \k)$.
We will show that both all parity extensions exist and that convolution preserves parity sheaves.
The categories $\Parity{\scrL'}{\scrL} (\k)$ also admit a block decomposition inherited from the derived category of sheaves.
We will write $\Parity{\scrL}{\scrL}^{\circ} (\k)$ for the block containing the monoidal unit.

\subsection{Main Results}

We can now state a preliminary version of the endoscopic-monodromic equivalence which serves as a modular analogue of \cite{LY}.

\begin{theorem}[See Theorem \ref{thm:endoscopy_neutral_block_Kac_moody} for more precise version]
  There is a monoidal equivalence of categories
  \[\Par (B_{\scrL} \backslash H_{\scrL}^{\circ} / B_{\scrL}, \k) \cong \Parity{\scrL}{\scrL}^\circ (\k).\]
\end{theorem}

In the precise version of the above theorem, we include numerous generalizations. We will briefly mention them here along with some related history.

\begin{enumerate}
  \item We will replace $G$ by a Kac--Moody group of either finite or affine type. In the characteristic 0 coefficient setting, the endoscopic-monodromic equivalence of \cite{LY} had already been extended
    to loop groups in Li's thesis \cite{Li}.
  \item The equivalence will also extend to the non-neutral blocks when $G$ is of finite type. This was already done in \cite[\S10]{LY} for characteristic 0 coefficients.
  \item The allowable coefficient rings will be enlarged to noetherian domains of finite global dimension.
  \item The equivalence admits a variation where left (resp. right) equivariance is replaced by left (resp. right) topological monodromy.
\end{enumerate}

\subsection{Connections and Future Work}

\subsubsection{Diagrammatic Hecke Categories}

For unipotent monodromy, there is a diagrammatic incarnation of the Hecke category defined in \cite{EW}.
Riche and Williamson proved that diagrammatic Hecke category coincides with the geometric Hecke category of parity sheaves when the category of parity sheaves is well-defined \cite{RW}.
In a subsequent paper, we will give a diagrammatic description of the monodromic Hecke categories of parity sheaves.
As in \cite{EW}, the diagrammatic category is defined in a much larger generality than the geometric one. 
In particular, we can make sense of the monodromic Hecke category for arbitrary Coxeter groups and realizations.

\subsubsection{Free-Monodromic Hecke Categories}

Gouttard has proven a variant of the endoscopic-monodromic equivalence for free-monodromic tilting sheaves on $\UGU$ rather than biequivariant sheaves on $\UGU$ \cite{Gou}.
Similar equivalences have also been produced in more recent work of Eteve \cite{Eteve24a} and Dhillon--Li--Yun--Zhu \cite{DLYZ}.
One can ask whether this equivalence can be lifted to mixed free-monodromic tilting sheaves. Some work is needed to define what free-monodromic mixed sheaves should mean on $\UGU$. 
In particular, the parity constraints are not satisfied, so the construction of \cite{AR2} does not work directly.
Nonetheless, Achar, Makisumi, Riche, and Williamson have defined a category of mixed free-monodromic tilting sheaves with unipotent monodromy using almost purely categorical constructions \cite{AMRW1}.
The author expects that the Achar--Makisumi--Riche--Williamson category can be generalized to arbitrary monodromy.
A benefit of working with the free-monodromic setting rather than the biequivariant setting is that we will be able to produce a degrading functor from the mixed category to the non-mixed category.
We expect that this will allow for some explicit computations of free-monodromic character sheaves using $p$-Kazhdan--Lusztig combinatorics (cf., \cite{Eteve24b}).

\subsection{Notations and Conventions}

Throughout the entirety of the paper, our derived categories of sheaves will be triangulated.
Despite this, Appendix \ref{apdx:A}, which provides the foundations of the sheaf theory for twisted equivariant sheaves, is written using $\infty$-categories.
This is somewhat crucial in order to make sense of twisted equivariant sheaves with respect to any multiplicative local system.
For example, in Gouttard's thesis \cite{Gou}, a necessary restriction is made to work with only multiplicative local systems arising from finite central isogenies.
By constructing twisted equivariant sheaves via categorical coinvariants, we are able to side-step this issue. 

By default, $\k$ will always denote a noetherian domain of finite global dimension.

\subsubsection{Geometry and Sheaves}

The six-functor formalism of constructible sheaves on complex algebraic stacks is not well-detailed in literature, and even less so is the six-functor formalism for
twisted equivariant sheaves on stacks.
As a result, we have included Appendix \ref{apdx:A} which develops these ideas. We will review its contents now.

Let $H$ be a complex algebraic group, and let $X$ be an algebraic $H$-stack over $\C$ of finite type. The groupoid of all multiplicative locally free rank one local
systems on $H$ is denoted by $\Ch (H, \k)$.\footnote{Some references call these rank one character sheaves on $H$.}
Let $\scrL \in \Ch (H, \k)$.
We will write $D_{\cons} (H \backslash_{\scrL} X, \k)$ for the $(H, \scrL)$-equivariant bounded derived category of constructible sheaves on $X$ with coefficients in $\k$-modules.
As a matter of convention, we set $D_{\cons} (X /_{\scrL} H, \k) \coloneq D_{\cons} (H \backslash_{\scrL^{-1}} X, \k)$. 

If $\varphi : H' \to H$ is a morphism of algebraic groups and $f : X' \to X$ is a morphism of algebraic stacks such that $f$ is equivariant with respect to $\varphi$, we
have a pair of adjoint functors
\[f^* : D_{\cons} (H \backslash_{\scrL} X, \k) \rightleftarrows D_{\cons} (H' \backslash_{\varphi^* \scrL} X', \k) : f_*.\]
There are also $!$-versions of the above functors.

If $\scrL, \scrL' \in \Ch (H, \k)$, the tensor product of twisted equivariant sheaves gives a bifunctor
\[(-) \otimes^L (-) : D_{\cons} (H \backslash_{\scrL} X, \k) \otimes D_{\cons} (H \backslash_{\scrL'} X, \k) \to D_{\cons} ( H \backslash_{\scrL \otimes \scrL'} X, \k).\]
Similarly, the inner Hom between sheaves gives a bifunctor
\[\RHom (-,-) : D_{\cons} (H \backslash_{\scrL} X, \k)^\textnormal{op} \otimes D_{\cons} (H \backslash_{\scrL'} X, \k) \to D_{\cons} (H \backslash_{\scrL^{-1} \otimes \scrL'} X, \k).\]

The six-functors for sheaves are related in the usual sense by the Verdier duality functor
\[\DD : D_{\cons} (H \backslash_{\scrL} X, \k)^\textnormal{op} \to D_{\cons} (H \backslash_{\scrL^{-1}} X, \k).\]

If $\k \to \k'$ is a ring homomorphism, there is an extension-of-scalars functor, which we denote by
\[\k' (-) : D_{\cons} (H \backslash_{\scrL} X, \k) \to D_{\cons} (H \backslash_{\k' (\scrL)} X, \k').\]

\subsubsection{Fixed Stratifications}

If $\{X_\lambda\}_{\lambda \in \Lambda}$ is an $H$-equivariant stratification of $X$, we will write $D_{\Lambda} (H \backslash_{\scrL} X, \k)$ for the full subcategory of
$D_{\cons} (H \backslash_{\scrL} X, \k)$ of sheaves constructible with respect to $\Lambda$.
For each $\lambda \in \Lambda$, there is a subcategory of ``local systems'', denoted $\Loc_{\textnormal{f}} (X_\lambda /_{\scrL} H, \k)$ consisting of sheaves $\scrK \in D_{(X_{\lambda})} (H \backslash_{\scrL} X_\lambda, \k)$ such that $\For_H \scrK$ is a local system on $X_{\lambda}$ of finite type.

\subsubsection{Kac--Moody Groups}\label{sec:review_of_Kac_Moody}

We will first review some preliminaries on Kac--Moody groups and their associated flag varieties.
Let $A$ be a generalized Cartan matrix. From $A$, we can associate a Kac--Moody root datum $(S, \bfX, \{\alpha_s\}_{s \in S}, \{\alpha_s^\vee\}_{s \in S})$.
Let $W$ denote the Weyl group of this root datum. The simple reflections in $W$ are in bijection with $S$, and we abuse notation and use elements of $S$ to denote the simple reflections.
The root datum then gives rise to a realization $\fr{h}_\k = (V, \{\alpha_s\}_{s \in S}, \{\alpha_s^\vee\}_{s \in S})$ of $W$ as follows. Let $\bfY = \Hom_{\Z} (\bfX, \Z)$.
Set $V = \k \otimes_{\Z} \bfY$. For $s \in S$, we set $\alpha_s$ (resp. $\alpha_s^\vee$) to be the image of the corresponding simple root (resp. simple
coroot) in $V^*$ (resp. $V$).
All geometric realizations of a generalized Cartan matrix are automatically balanced, but are not necessarily Demazure surjective.
We can fix this as follows. Suppose $\alpha_s : \bfY \to \Z$ and $\alpha_s^\vee : \bfX \to \Z$ are surjective, we can then set $\Z' = \Z$.
Otherwise, we set $\Z' = \Z[\frac{1}{2}]$. It is a standard fact that $\fr{h}_\k$ satisfies Demazure surjectivity provided there exists a ring homomorphism $\Z' \to \k$.
Note that if $\Z = \Z'$, then this is automatically satisfied.

One can associate to $A$ an integral Kac--Moody group $G_{\Z}$ with a Borel subgroup $B_{\Z}$.
Let $U_{\Z}$ denote the pro-unipotent radical of $B_{\Z}$ and let $T_{\Z}$ denote the canonical maximal torus whose group of characters can be canonically identified with $\bfX$.
Let $G, B, U$, and $T$ denote the base change to $\C$ of $G_{\Z}$, $B_{\Z}$, $U_{\Z}$, and $T_{\Z}$, respectively. Let $\eFl = U \backslash G$ denote the enhanced flag ind-variety.
The enhanced flag variety admits a Bruhat decomposition
\[\eFl = \bigsqcup_{w \in W} \eFl_w,\]
where $\eFl_w$ is the $B$-orbit of a lift of $w \in W$. Note that each Bruhat orbit is isomorphic to a product of a torus with an affine space $\eFl_w \cong T \times \A^{\ell (w)}.$
The $B$-orbits give a stratification of $\eFl$. The closures of the strata are given by the Bruhat order,
\[\eFl_{\leq w} = \overline{\eFl_w} = \bigsqcup_{x \leq w} \eFl_x.\]

Let $J \subset S$ of finite type. We denote by  $W_J$ the finite subgroup of $W$ generated by $s \in J$ and by $W^J$ the subset of $W$ consisting of elements $w$ which are minimal in the coset $wW_J$.
For each $w \in W$, we write $\overline{w}$ for the element of $W^J$ such that $w \in \overline{w} W_J$.
Since $W_J$ is finite, it has the longest element $w_0^J \in W_J$. To $J$, we can also define a parabolic subgroup $P_{J, \Z}$ of $G_{\Z}$ with pro-unipotent radical
$U^J_{\Z}$ and Levi subgroup $L_{J,\Z}$.
Let $P_J, U^J, L_J$ denote the base change to $\C$ of $P_{J, \Z}$, $U^J_{\Z}$, and $L_{J,\Z}$, respectively.
Let $\eFl^J = U^J \backslash G$ denote the partial enhanced flag ind-variety.
The partial flag ind-variety has a parabolic Bruhat decomposition
\[\eFl^J = \bigsqcup_{\overline{w} \in W^J} \eFl^J_{\overline{w}},\]
where $\eFl_{\overline{w}}^J \cong \A^{\ell (\overline{w})} \times L_J$.
The closures of $\eFl_{\overline{w}}^J$ are denoted by $\Fl_{\leq w}^J$ and $\eFl_{\leq w}^J$, respectively. As in the non-parabolic case, they are a union of
Bruhat strata indexed by elements in $W^J$ below $\overline{w}$ in the restricted Bruhat order.
We will write $j_{\overline{w}}^J$ for the embedding of $\eFl_{\overline{w}}^J$ into $\eFl^J$.

The most prominent example of a parabolic subgroup we will encounter is the special case of $J = \{s\}$ for some $s \in S$.
In this case, we will replace the superscript $J$ by $s$. In particular, we will write $\eFl^s = \eFl^{\{s\}}$ and likewise $\eFl^s = \eFl^{\{s\}}$.

\subsection{Acknowledgements}

The author thanks Pramod Achar for his continuous support and careful reading of earlier drafts of this paper. 
The author also thanks Ana Bălibanu, Gurbir Dhillon, Alberto San Miguel Malaney, and Simon Riche for useful comments and discussions that influenced this project.

The author was partially supported by NSF Grant DMS-2231492.
	
	\section{Twisted Equivariant Parity Sheaves}\label{sec:parity}

	\subsection{Definitions and Uniqueness}

In this section, we require that $\k$ be a noetherian complete local ring or a field.
Let $X$ be an (ind)-algebraic stack of finite type.
Let $H$ be a connected algebraic group acting on $X$ and $\scrL \in \Ch (H, \k)$.
Let $\{X_\lambda \}_{\lambda \in \Lambda}$ be an $H$-stable stratification of $X$. We write $j_{\lambda} : X_{\lambda} \hookrightarrow X$ for the inclusion maps.
For each $\lambda$, denote by $\Loc_{\textnormal{ff}} (H \backslash X_{\lambda}, \k)$ the full subcategory of $D_{(X_{\lambda})} (H \backslash_{\scrL} X_{\lambda}, \k)$ consisting of sheaves $\scrF$ such that $\For_H \scrF$ is a locally free local system of finite type on $X_{\lambda}$.

In order to simplify notation, for a $\k$-linear triangulated category $\scrC$, objects $c,d \in \scrC$, and $n \in \Z$, we will write
\[\Hom_{\scrC}^n (c,d) \coloneq \Hom_{\scrC} (c, d[n]).\]
Similarly, we will write
\[\Hom_{\scrC}^\bullet (c,d) \coloneq \bigoplus_{n \in \Z} \Hom_{\scrC}^n (c,d),\]
viewed as a graded $\k$-module.

There is a series of conditions on $D_{\Lambda} (H\backslash_{\scrL} X, \k)$ which ensures that parity sheaves are sensible to define.
These conditions are called the \emph{parity conditions}. For all $\scrK, \scrK' \in \Loc_{\textnormal{ff}} (H \backslash_{\scrL} X, \k)$, we require that
\begin{equation}\label{eq:parity_conditions}
    \Hom_{D_{\Lambda} (H \backslash_{\scrL} X, \k)}^n (\scrK, \scrK') = \begin{cases} 0 & \text{for } n \text{ odd, and} \\ \text{a free } \k\text{-module} & \text{for } n \text{ even}.  \end{cases}
\end{equation}

Throughout, we will assume that $D_{\Lambda} (H\backslash_{\scrL} X, \k)$ satisfies the parity conditions.

\begin{definition}\label{def:parity}
    Let $? \in \{*, !\}$. Let $\scrF \in D_{\Lambda} (H \backslash_{\scrL} X, \k)$.
    \begin{enumerate}
      \item  $\scrF$ is $?$-\emph{even} (resp. $?$-\emph{odd}) if, for all $\lambda \in \Lambda$, and $n \in \Z$,
        $H^n (j_\lambda^? \scrF) = 0$ when $n$ is odd, and $H^n (j_\lambda^? \scrF)$ is a local system (i.e., viewed as a non-equivariant sheaf on $X_{\lambda}$) whose stalks are finite rank free $\k$-modules when $n$ is even.
        Let $D_{\Lambda}^{?-\ev} (H \backslash_{\scrL} X, \k)$ (resp.
        $D_{\Lambda}^{?-\odd} (H \backslash_{\scrL} X, \k)$) denote the full subcategory of $D_{\Lambda} (H \backslash_{\scrL} X, \k)$ consisting of $?$-even (resp.
        $?$-odd) objects.
      \item  $\scrF$ is $?$-\emph{parity} if it is either $?$-even or $?$-odd. Let $D_{\Lambda}^{?-\pty} (H \backslash_{\scrL} X, \k)$ denote the full subcategory of
        $D_{\Lambda} (H \backslash_{\scrL} X, \k)$ consisting of $?$-parity objects.
      \item $\scrF$ is \emph{even} (resp. \emph{odd}) if it is both $*$-even and $!$-even (resp. odd). Let $D_{\Lambda}^{\ev} (H \backslash_{\scrL} X, \k)$ (resp.
        $D_{\Lambda}^{\odd} (H \backslash_{\scrL} X, \k)$) denote the full subcategory of \newline $D_{\Lambda} (H \backslash_{\scrL} X, \k)$ consisting of even (resp. odd) objects.
      \item  $\scrF$ is \emph{parity} if there is a decomposition $\scrF = \scrF_{\ev} \oplus \scrF_{\odd}$ where $\scrF_{\ev}$ is even and $\scrF_{\odd}$ is odd. Let
        $\Par_{\Lambda} (H \backslash_{\scrL} X, \k)$ denote the full subcategory of $D_{\Lambda} (H \backslash_{\scrL} X, \k)$ consisting of parity objects.
    \end{enumerate}
  \end{definition}

  The remaining results in this section are all twisted equivariant variations of results from \cite{JMW, MauR}.
  We have omitted their proofs since they share nearly identical arguments. 

  \begin{proposition}[{\cite[Proposition 2.6]{JMW}}]\label{prop:hom_star_and_shriek}
    If $\scrF$ is $*$-parity and $\scrG$ is $!$-parity, then we have an isomorphism of $\k$-modules,
    \[\Hom_{D_\Lambda (H \backslash_{\scrL} X, \k)}^n (\scrF, \scrG) \cong \bigoplus_{\lambda \in \Lambda} \Hom_{D_\Lambda (H \backslash_{\scrL} X_\lambda, \k)}^n
    (j_\lambda^* \scrF, j_\lambda^! \scrG).\]
    Moreover, both sides are free $\k$-modules.
  \end{proposition}
  
  \begin{proposition}[{\cite[Theorem 2.12]{JMW}}]\label{prop:uniqueness_of_parity_sh}
    Let $\scrF$ be an indecomposable object in $\Par_{\Lambda} (H \backslash_{\scrL} X, \k)$. Then
    \begin{enumerate}
      \item the support of $\scrF$ is of the form $\overline{X_\lambda}$, for some $\lambda \in \Lambda$;
      \item $j_\lambda^* \scrF \cong \scrK_{\lambda} [m]$, for some $m \in \Z$, and $\scrK_{\lambda} \in \Loc (H \backslash_{\scrL} X_\lambda, \k)$;
      \item any indecomposable parity object in $\Par_{\Lambda} (H \backslash_{\scrL} X, \k)$ supported on $\overline{X_\lambda}$ and extending $\scrK_\lambda [m]$ is
        isomorphic to $\scrF$.
    \end{enumerate}
  \end{proposition}

  We conclude with a remark on the behavior of parity sheaves under extension of scalars. Let $\k \to \k'$ be a ring morphism of complete local rings. Consider the extension of scalars functor
  \[\k' (-) \coloneq \k' \otimes_\k^L (-) : D_\Lambda (H \backslash_{\scrL} X, \k) \to D_{\Lambda} (H \backslash_{\k' (\scrL)} X, \k').\]
  
  \begin{lemma}[{\cite[Lemma 2.36]{JMW}}]\label{lem:par_ext_of_scalars}
    Suppose that $\scrF \in D_{\Lambda} (H \backslash_{\scrL} X, \k)$ is $?$-even (resp. odd), then $\k' (\scrF)$ is $?$-even (resp. odd).
    In particular, $\k' (-)$ restricts to a functor
    \[\Par_\Lambda (H \backslash_{\scrL} X, \k) \to \Par_{\Lambda} (H \backslash_{\k' (\scrL)} X, \k').\]
  \end{lemma}
  
  \begin{lemma}[{\cite[Lemma 2.2]{MauR}}]\label{lem:ll_exists}
    Let $\scrE, \scrE' \in \Par_\Lambda (H \backslash_{\scrL} X, \k)$. The Hom space $\Hom_{\Par_\Lambda (H \backslash_{\scrL} X, \k)} (\scrE, \scrE')$
    is a free $\k$-module. Moreover, extension of scalars induces an isomorphism
    \[ \k' \otimes_\k \Hom_{\Par_\Lambda (H \backslash_{\scrL} X, \k)} (\scrE, \scrE') \stackrel{\sim}{\to} \Hom_{\Par_\Lambda (H \backslash_{\k'(\scrL)} X, \k')} (\k' (\scrE), \k'(\scrE')).\]
  \end{lemma}

\subsection{Baby Decomposition Theorem}

Let $\pi : H \backslash_{\scrL} X \to H' \backslash_{\scrL'} Y$ be a (bounded) morphism of twisted (ind)-algebraic stacks (see \ref{def:twisted_stacks}).
Note that $\pi$ consists of the following data:
\begin{itemize}
  \item an underlying (bounded) morphism of stacks $\pi : X \to Y$,
  \item a morphism of algebraic groups $\varphi : H \to H'$, and
  \item an isomorphism of local systems $\scrL \cong \varphi^* \scrL'$.
\end{itemize}

Suppose $X$ (resp. $Y$) is stratified with an $H$-stable (resp. $H'$-stable) stratification $X = \bigsqcup_{\lambda \in \Lambda_X} X_\lambda$  (resp. $Y = \bigsqcup_{\mu
\in \Lambda_Y} Y_\mu$).
We insist that these stratifications make $D_{\Lambda_X} (H \backslash_{\scrL} X, \k)$ and $D_{\Lambda_Y} (H' \backslash_{\scrL'} Y, \k)$ satisfy the parity conditions.
We will further impose that $\pi$ be stratified (see \ref{def:stratified_mor}).
In particular, $F_\mu \coloneq \pi^{-1} (Y_\mu)$ is a union of strata in $X$. We will write $\Lambda_{X, \mu}$ for the subset of $\Lambda_X$ such that
\[F_\mu = \bigcup_{\lambda \in \Lambda_{X, \mu}} X_{\lambda}.\]
For any $\lambda \in \Lambda_{X, \mu}$, we can write
\[\pi_{\lambda, \mu} : H \backslash_{\scrL} X_\lambda \to H' \backslash_{\scrL'} Y_\mu \]
for the restriction of $\pi$.

\begin{definition}\label{def:even_mor}
  We say that a morphism $\pi : H \backslash_{\scrL} X \to H' \backslash_{\scrL'} Y$ is \emph{even} if for all $\mu \in \Lambda_Y$ and $\lambda \in \Lambda_{X, \mu}$,
  the $*$-pushforward along the restricted morphism $\pi_{\lambda, \mu *}$ takes even (resp. odd) complexes to even (resp. odd) complexes.
\end{definition}

\begin{remark}
  Our definition of even differs significantly from that given in \cite{JMW}. In \emph{loc. cit.}, if $\scrF$ is an $H$-equivariant parity sheaf on $X$, then $\For_H \scrF$ is also a parity sheaf on $X$.
  In particular, the parity conditions are satisfied on not just $H \backslash X$ but also $X$.
  A stratified morphism $\pi : X \to Y$ is said to be \emph{JMW-even} if for all $\lambda \in \Lambda_X , \mu \in \Lambda_Y$, and $\scrL \in \Loc_{\textnormal{ff}} (X_{\lambda}, \k)$, the cohomology of the fibers $F_{\lambda, \mu}$ of the induced morphism $\pi_{\lambda, \mu} : X_\lambda \to Y_{\mu}$ with coefficients in $\scrL\vert_{F_{\lambda, \mu}}$ is torsion free and concentrated in even degrees.
  Similarly, a stratified morphism $\pi : H \backslash X \to H \backslash Y$ is said to be \emph{JMW-even} if the underlying map $\pi : X \to Y$ is JMW-even.
  If one tries to define JMW-even directly on $\pi : H \backslash X \to H \backslash Y$ without assuming parity conditions on $X$ and $Y$, then in order for the cohomology of the fiber to be well-defined, one needs that an $H$-stable base point of $Y_{\mu}$ exists.  
  Nonetheless, if $H, H' = 1$, then the proof of \cite[Proposition 2.34]{JMW} states JMW-even implies even.
\end{remark}

A key ingredient in the proof of the existence of parity sheaves is given by the following proposition.

\begin{proposition}[Baby Decomposition Theorem]\label{prop:pushforward_of_even}
  Let $\pi : H \backslash_{\scrL} X \to H' \backslash_{\scrL'} Y$ be a proper morphism.
  Then direct image $\pi_*$ of a $?$-even (resp. $?$-odd) object is again $?$-even (resp. $?$-odd).
  As a result, the direct image $\pi_*$ takes parity sheaves to parity sheaves.
\end{proposition}
\begin{proof}
  If the result holds for $?$-even sheaves, then it will hold for $?$-odd sheaves by shifting. 
  Additionally, since $\pi$ is proper, there are isomorphisms.
  \[\pi_* \DD \scrF \cong \DD \pi_! \scrF \cong \DD \pi_* \scrF.\]
  As a result it suffices to prove that $\pi_*$ takes $!$-even sheaves to $!$-even sheaves.

  Let $\scrF \in D (H \backslash_{\scrL} X, \k)$ be $!$-even. Consider the cartesian square
  \[
    \begin{tikzcd}
      F_\mu \arrow[r, "i_\mu"] \arrow[d, "\pi"] & X \arrow[d, "\pi"] \\
      Y_\mu \arrow[r, "{j_{Y, \mu}}"]           & Y.
    \end{tikzcd}
  \]
  By base change, we have an isomorphism
  \begin{equation}\label{eq:pushforward_of_even_1} j_{Y, \mu}^! \pi_* \scrF \cong \pi_* i_\mu^! \scrF.
  \end{equation}
  Choose a filtration
  \[F_\mu = F_r \supset F_{r-1} \supset \ldots \supset F_0 = \emptyset,\]
  where $F_p$ is closed in $F_{p+1}$ and $F_p \setminus F_{p-1} = X_{\lambda_p}$ for some $\lambda_p \in \Lambda_{X, \mu}$.
  We will write $i_p$ for the inclusion $i_p : F_p \hookrightarrow F_\mu$, and $i_{p, \mu}$ for the composition $i_{p, \mu} = i_p \circ i_\mu$.
  Let $\pi_p : F_p \to F_\mu$ denote the restriction of $\pi$ to $F_p$.

  We will prove by induction on $p$ that $\pi_{p*} i_{p, \mu}^! \scrF$ is an even object on $H' \backslash_{\scrL'} Y_\mu$.
  If $p = 0$, then $\scrF =0$ and the claim holds trivially.

  Now suppose that $p > 0$. Consider the commutative diagram
  \[
    \begin{tikzcd}
      & F_{p+1} \arrow[d, "i_{p+1}"] \arrow[rd, "{i_{p+1, \mu}}"] & X_{\lambda_{p+1}} \arrow[l, "i"'] \arrow[d, "{j_{X, \lambda_{p+1}}}"] \\
      F_p \arrow[ru, "j"] \arrow[r, "i_p"] \arrow[rr, "{i_{p, \mu}}"', bend right] & F_\mu \arrow[r, "i_\mu"]                                  & X.
    \end{tikzcd}
  \]
  Consider the open-closed distinguished triangle for the pair $(F_p, X_{\lambda_{p+1}})$ applied to $i_{p+1, \mu}^! \scrF$,
  \[i_* j_{X, \lambda_{p+1}}^! \scrF \to i_{p+1, \mu}^! \scrF \to j_* i_{p, \mu}^! \scrF \to,\]
 and the commutative diagram
  \[
    \begin{tikzcd}
      X_{\lambda_{p+1}} \arrow[r, "j"] \arrow[rd, "{\pi_{\lambda_{p+1}, \mu}}"'] & F_{p+1} \arrow[d, "\pi_{p+1}"] & F_p \arrow[l, "i"'] \arrow[ld, "\pi_p"] \\
      & Y_\mu.                          &
    \end{tikzcd}
  \]
  After applying $\pi_{p+1 *}$, we obtain a distinguished triangle
  \[\pi_{\lambda_{p+1}, \mu*} j_{X, \lambda_{p+1}}^! \scrF \to \pi_{p+1 *} i_{p+1, \mu}^! \scrF \to \pi_{p*}  i_{p, \mu}^! \scrF \to.\]
  By definition of an even morphism, $\pi_{\lambda_{p+1}, \mu*} j_{X, \lambda_{p+1}}^! \scrF$ is even, and by induction, $\pi_{p*}  i_{p, \mu}^! \scrF $ is even.
  Since the subcategory of even objects is closed under extensions, we must have that $\pi_{p+1 *} i_{p+1, \mu}^! \scrF$ is even.
  Therefore, by induction $\pi_* i_\mu^! \scrF$ is even. Since this holds for all $\mu$, by equation (\ref{eq:pushforward_of_even_1}), we must have that $\pi_* \scrF$ is
  $!$-even as desired.
\end{proof}

Since our definition for an even morphism is rather opaque, we will conclude this section with some examples.

\begin{example}\quad
  \begin{enumerate}
    \item Let $X$ be an (ind)-algebraic stack with an $H$-action. Let $\scrL \in \Ch (H, \k)$. Let $X$ be stratified by itself.
      Suppose that $D_{(X)} (H \backslash_{\scrL} X, \k)$ is stratified and satisfies the parity conditions.
      Let $F$ be a simply connected algebraic variety such that $H^n (F; \k)$ is 0 when $n$ is odd and a free $\k$-module when $n$ is even.
      Define $Y = X \times F$ which inherits an $H$-action from $H$-acting on $X$ and acting trivially on $F$.
      It is clear that $D_{(Y)} (H \backslash_{\scrL} Y, \k)$ is also stratified and satisfies the parity conditions.
      Let $\pi : H \backslash_{\scrL} Y \to H \backslash_{\scrL} X$ be the projection map. Let $\scrF \in D_{(X)} (H \backslash_{\scrL} X, \k)$ be an even complex.
      Then
      \[ H^\bullet (\pi_* (\scrF \boxtimes \uk_F)) \cong H^\bullet (\scrF) \otimes^L H^\bullet (F; \k). \]
      The conditions on the cohomology of $F$ and $\scrF$ being even ensures that the above cohomology is concentrated in even degrees and free in odd degrees.
      The case of $\scrF$ being odd is similar. We can then conclude that $\pi_*$ is even.
    \item We can generalize the above example to multiple strata as follows. Let $\pi : H \backslash_{\scrL} X \to H \backslash_{\scrL} Y$ be a stratified morphism of
      twisted (ind)-algebraic stacks.
      Suppose that for each $\lambda \in \Lambda_{X, \mu}$, $X_\lambda$ can be identified $H$-equivariantly with $Y_\mu \times F_{\lambda, \mu}$ such that $F_{\lambda,
      \mu}$ is a simply connected algebraic variety such that $H^n (F_{\lambda, \mu}; \k)$ is 0 when $n$ is odd and a free $\k$-module when $n$ is even.
      Moreover, under this identification suppose that $\pi_{\lambda, \mu}$ is identified with the projection $Y_\mu \times F_{\lambda, \mu} \to Y_\mu$.
      Then the argument for (1) shows that $\pi_*$ is even.
  \end{enumerate}
\end{example}

\subsection{Mixed Categories}\label{sec:mixed_cats}

In \cite{AR2}, Achar and Riche introduced the mixed modular derived category. It has served as a useful replacement for mixed sheaves in positive characteristic.
We will provide a brief overview of these categories and their key features in the twisted equivariant setting.
All the proofs in this section can be easily adapted from \cite{AR2}, and as such will be omitted.
We only include the components of this theory necessary for the present work. In upcoming work, we will study these categories more systematically. 

Throughout, we will fix an (ind-algebraic) $H$-stack $X$ with an $H$-stable stratification $X = \bigsqcup_{\lambda \in \Lambda} X_{\lambda}$ and a multiplicative local system $\scrL \in \Ch (H, \k)$.
We impose the following two constraints on $D_{\Lambda} (H \backslash_{\scrL} X, \k)$:
\begin{enumerate}
  \item The parity conditions should be satisfied for $D_{\Lambda} (H \backslash_{\scrL} X, \k)$.
  \item For each $\lambda \in \Lambda$ and $\scrK \in \Loc_{\textnormal{ff}} (H \backslash_{\scrL} X_{\lambda}, \k)$, there exists an indecomposable parity extension $\scrE_{\lambda} (\scrK)$ of $\scrK$.
\end{enumerate}

\begin{definition}
  The \emph{mixed derived category of $(H, \scrL)$-equivariant constructible sheaves} is the following triangulated category 
  \[ D^m_{\Lambda} (H \backslash_{\scrL} X, \k) \coloneq K^b \Par_{\Lambda} (H \backslash_{\scrL} X, \k) .\]
\end{definition}

The mixed derived category admits two shift functors:
\begin{enumerate}
  \item There is the internal shift inherited from $\Par_{\Lambda} (H \backslash_{\scrL} X, \k) $. We denote this shift by $(1) : D^m_{\Lambda} (H \backslash_{\scrL} X, \k) \to D^m_{\Lambda} (H \backslash_{\scrL} X, \k)$.
  \item There is the cohomological shift $[1] :  D^m_{\Lambda} (H \backslash_{\scrL} X, \k) \to D^m_{\Lambda} (H \backslash_{\scrL} X, \k)$.
\end{enumerate}
We also introduce the notation $\langle n \rangle \coloneq (-n) [n]$. The functor $\langle 1 \rangle$ is called the \emph{Tate twist}.

\begin{proposition}[{\cite[Proposition 2.3]{AR2}}]\label{prop:recollement_mixed}
  Let $U$ be an open union of strata in $X$. Let $Z$ denote its complement. Let $j : U \to X$ and $i : Z \to X$ denote the inclusions.
  There is a recollement diagram
  \[\begin{tikzcd}
    {D^m_{\Lambda} (H \backslash_{\scrL} Z, \k)} \arrow[rr, "i_*"] &  & {D^m_{\Lambda} (H \backslash_{\scrL} X, \k)} \arrow[rr, "j^*"] \arrow[ll, "i^*"', bend right] \arrow[ll, "i^!", bend left] &  & {D^m_{\Lambda} (H \backslash_{\scrL} U, \k).} \arrow[ll, "j_*", bend left] \arrow[ll, "j_!"', bend right]
    \end{tikzcd}\]
  where $i_*$ and $j^*$ are the functors induced from the corresponding functors of parity sheaves.\footnote{$i^*, i^!, j_*$, and $j_!$ are not the similarly denoted functors in the non-mixed category, since they do not send parity sheaves to parity sheaves.}
\end{proposition}

Let $Y$ be a locally closed union of strata in $X$. Let $h : Y \to X$ denote the inclusion.
Proposition \ref{prop:recollement_mixed} allows one to define adjoint pairs of sheaf functors
\[\begin{tikzcd}
{D^m_{\Lambda} (H \backslash_{\scrL} Y, \k)} \arrow[rr, "h_*"', bend right] & \bot & {D^m_{\Lambda} (H \backslash_{\scrL} X, \k)} \arrow[rr, "h^!"', bend right] \arrow[ll, "h^*"', bend right] & \bot & {D^m_{\Lambda} (H \backslash_{\scrL} Y, \k).} \arrow[ll, "h_!"', bend right]
\end{tikzcd}\]
For example, we can write $h : Y \to X$ as the composition $Y \stackrel{j}{\to} \overline{Y} \stackrel{i}{\to} X$. We can then define $h_! = i_! \circ j_!$.
The fact that $h$ is independent of the factorization follows from \cite[Lemma 2.6]{AR2}.

The following lemma follows from a standard argument using recollement.
\begin{lemma}\label{lem:generation_of_mixed_cat}
  The category $D_{\Lambda}^m (H \backslash_{\scrL} X, \k)$ is generated under cohomological shifts and extensions by the standard (resp. costandard) sheaves $j_{\lambda !} \scrK (n)$ (resp. $j_{\lambda *} (\scrK) (n)$) for all $\lambda \in \Lambda$,  $\scrK \in \Loc_{\textnormal{ff}} (H \backslash_{\scrL} X_{\lambda}, \k)$, and $n \in \Z$.
\end{lemma}

If $f : H' \backslash_{\scrL'} X' \to H \backslash_{\scrL} X$ is a proper, smooth, and even (bounded) morphism of twisted (ind)-algebraic stacks, then $f_*$ and $f^*$ both take parity sheaves to parity sheaves.
As a result, there are induced functors
\[f^* : D_{\Lambda}^m ( H \backslash_{\scrL} X, \k) \rightleftarrows D_{\Lambda'}^m (H' \backslash_{\scrL'} X', \k) : f_*.\] 
Let $Y$ be a locally closed union of strata in $X$. Let $h : Y \to X$ denote the inclusion. 
Consider the cartesian square
\[ \begin{tikzcd}
  H' \backslash_{\scrL'} f^{-1} (Y) \arrow[r, "h'"] \arrow[d, "f'"] & H' \backslash_{\scrL'} X' \arrow[d, "f"] \\
  H \backslash_{\scrL} Y \arrow[r, "h"]                             & H \backslash_{\scrL} X.                  
  \end{tikzcd}
\]
Note that $f'$ is proper, smooth, and even. By the same argument given in \cite[Proposition 2.8]{AR2}, there are natural isomorphisms of functors
\[f_* \circ h_*' \cong h_* \circ f_*', \hspace{2cm} f_* \circ h_!' \cong h_! \circ f_*',\]
\[(h')^* \circ f^* \cong (f')^* \circ h^*, \hspace{2cm} (h')^! \circ f^* \cong (f')^* \circ h^!,\]
\[f^* \circ h_* \cong h_*' \circ (f')^*, \hspace{2cm} f^* \circ h_! \cong h_!' \circ (f')^*,\]
\[h^* \circ f_* \cong f_*' \circ (h')^*, \hspace{2cm} h^! \circ f_* \cong f_*' \circ (h')^!.\]

Finally, we conclude with a useful Hom vanishing constraint.

\begin{lemma}[{\cite[Lemma 3.2]{AR2}}]\label{lem:hom_vanishing_for_ME}
  Suppose that for all $\lambda \in \Lambda$, there is an $H$-equivariant isomorphism $X_{\lambda} \cong H \times Y_{\lambda}$ where $Y_{\lambda}$ is an affine space.
  In this case, $\Loc_{\textnormal{ff}} (H \backslash_{\scrL} H \times Y_{\lambda}, \k)$ is generated by the sheaf $\scrL \boxtimes \uk_{Y_{\lambda}}$.
  We will write $\scrK_{\lambda} \in \Loc_{\textnormal{ff}} (H \backslash_{\scrL} X_{\lambda}, \k)$ for the sheaf corresponding to $\scrL \boxtimes \uk_{Y_{\lambda}}$.
  Moreover, for all $\lambda, \mu \in \Lambda$, we have
  \[\Hom_{D_{\Lambda}^m (H \backslash_{\scrL} X, \k)} (j_{\lambda !} \scrK_{\lambda}, j_{\mu *} \scrK_{\mu} \langle n \rangle [i]) \cong \begin{cases} \k & \lambda = \mu, n = \dim X_{\lambda} - \dim X_{\mu}, i=-n, \\ 0 & \text{otherwise.}\end{cases}\]
\end{lemma}

	\section{Monodromic Hecke Categories}\label{sec:mhc}

Fix a Kac--Moody group $G$ with Borel $B$ containing a maximal torus $T$ as in \ref{sec:review_of_Kac_Moody}.
Let $W$ denote the Weyl group of $G$ with respect to $T$. We will also denote $S$ the set of simple reflections in $W$.
We return to the setting where $\k$ is any noetherian domain of finite global dimension. 

\subsection{Endoscopic Weyl Groups}

We review the theory of endoscopic Weyl groups following \cite{LY}. 
We encounter additional problems when working with Kac--Moody groups since only real roots have associated coroots.
The definition of the endoscopic Weyl group only needs minor modifications; however, modified arguments need to be given for some foundations (cf., \cite{H22, H23}).

Fix a multiplicative local system $\scrL \in \Ch (T,\k)$. We define a $W$-action on $\Ch (T)$ by $w \cdot \scrL = (w^{-1})^* \scrL$. 
For $\scrL, \scrL' \in \Ch (T, \k)$ in the same $W$-orbit, we set
\[\W{\scrL'}{\scrL} \coloneq \{w \in W \mid w(\scrL) \cong \scrL' \}.\]
If $\scrL = \scrL'$, we will write $W_{\scrL} \coloneq \W{\scrL}{\scrL}$, i.e., the stabilizer in $W$ of $\scrL$.

\subsubsection{Definition}

Let $\Phi_{\re}$ denote the set of real roots in $\Phi$ (i.e., roots of the form $w \cdot \alpha_s$ for $w \in W$ and $s \in S$).
To each real root $\alpha = w \cdot \alpha_s$, there is a corresponding real coroot $\alpha^{\vee} \coloneq w \cdot \alpha_s^{\vee}$.
We define a collection of real coroots
\[\Phi_{\re, \scrL}^\vee \coloneq \{ \alpha^\vee \in \Phi_{\re}^\vee \mid (\alpha^\vee)^* \scrL \cong \uk_{\G_m} \},\]
where we view $\alpha^{\vee}$ as a morphism $\G_m \to T$.
We can then define a collection of real roots $\Phi_{\re, \scrL}$ consisting of roots $\alpha \in \Phi_{\re}$ such that $\alpha^\vee \in \Phi_{\re, \scrL}^\vee$.
Denote $\Phi_{\re, \scrL}^+ = \Phi_{\re, \scrL} \cap \Phi^+$.
We will write $W_{\scrL}^\circ$ for the subgroup of $W$ generated by reflections of roots in $\Phi_{\re, \scrL}$.

Let $R$ denote the set of reflections in $W$, i.e., elements of the form $wsw^{-1}$ where $w \in W$ and $s \in S$. 
Let $R_{\scrL}^{\circ}$ denote the set of reflections in $W_{\scrL}^{\circ}$.
For each $w \in W$, we can define $N(w) \coloneq \{ r \in R_{\scrL}^{\circ} \mid rw^{-1} < w^{-1} \}$.
By \cite[Theorem 3.3]{Dyer}, $W_{\scrL}^{\circ}$ is a Coxeter group with simple reflections
\[S_{\scrL}^{\circ} \coloneq \{ t \in R \mid N(t) \cap W_{\scrL}^{\circ} = \{t\} \}.\]
We call the reflections in $S_{\scrL}^{\circ}$ the \emph{endosimple} reflections.

Let $\leq_{\scrL}$ and $\ell_{\scrL}$ denote the Bruhat order and length on $(W_{\scrL}^{\circ}, S_{\scrL}^{\circ})$.
Since $S_{\scrL}^{\circ} \subset R$, for all $x,y \in W_{\scrL}^{\circ}$, 
\begin{equation}\label{eq:endo_vs_bruhat_orders}
   x \leq_{\scrL} y \implies x \leq y. 
\end{equation}

\begin{lemma}[{\cite[Lemma 5.13]{H22}}]\label{lem:endosimple_and_dyer}
  \begin{enumerate}
    \item One has $R \cap W_{\scrL}^{\circ} = R_{\scrL}^{\circ}$.
    \item Let $\alpha \in \Phi_{\re, \scrL}^+$. Then $r_{\alpha} \in S_{\scrL}^{\circ}$ if and only if $r_{\alpha} \Phi_{\re, \scrL}^+ \cap \Phi_{\re, \scrL}^{-} = \{-\alpha\}$.
    \item Let $w \in W_{\scrL}^{\circ}$. Let $\uw = (r_1, \ldots, r_k)$ be a reduced expression in $(W_{\scrL}^{\circ}, S_{\scrL}^{\circ})$. 
          Then 
          \[ \{ \alpha \in \Phi_{\re, \scrL}^+ \mid w \alpha < 0 \} = \{ \alpha_{r_k}, r_k \cdot \alpha_{r_{k-1}}, \ldots, r_{k} \ldots r_2 \cdot \alpha_{r_1} \}. \]
          In particular,
          \[\ell_{\scrL} (w) \coloneq \# \{ \alpha \in \Phi_{\re, \scrL}^+ \mid w \alpha < 0 \}.\]
  \end{enumerate}
\end{lemma}

When $G$ is finite type, Lemma \ref{lem:endosimple_and_dyer} (2) shows that $S_{\scrL}^{\circ}$ agrees with the Coxeter generators of $W_{\scrL}^{\circ}$ given in \cite{LY}.

\subsubsection{Blocks}

\begin{lemma}\label{lem:endo_subgp_is_normal}
$W_{\scrL}^{\circ}$ is a normal subgroup of $W_{\scrL}$.
\end{lemma}
\begin{proof}
    Let $w \in W_{\scrL}$ and $\alpha \in \Phi_{\re, \scrL}^{\circ}$.
    By definition, $w r_{\alpha} w^{-1} = r_{w\alpha}$.
    We can then compute
    \[(w\alpha^{\vee})^* \scrL \cong (\alpha^{\vee})^* w \scrL \cong (\alpha^{\vee})^* \scrL \cong \uk_{\G_m}.\]  
\end{proof}

Let $\scrL, \scrL' \in \Ch^{\circ} (T,\k)$ be in the same $W$-orbit. We define the set of \emph{blocks} as the set of cosets
\[\uW{\scrL'}{\scrL} \coloneq \W{\scrL'}{\scrL} / W_{\scrL}^{\circ}.\]
By Lemma \ref{lem:endo_subgp_is_normal}, the set of blocks can also be defined by $\uW{\scrL'}{\scrL} = W_{\scrL'}^{\circ} \backslash \W{\scrL'}{\scrL}$.
Each block $\beta \in \uW{\scrL'}{\scrL}$ inherits a partial order $\leq$ from the Bruhat order in $W$. 

Let $\scrL, \scrL', \scrL'' \in \Ch^{\circ} (T,\k)$ be in the same $W$-orbit. 
Let $\beta \in \uW{\scrL'}{\scrL}$ and $\gamma \in \uW{\scrL''}{\scrL'}$.
The set $\gamma \cdot \beta \coloneq \{xy \mid x \in \gamma, y \in \beta\}$ is equal to $W_{\scrL''}^{\circ} xy = x W_{\scrL'}^{\circ} y = xy W_{\scrL}^{\circ}$ (for any $x \in \beta$ and $y \in \gamma$).
This makes $\gamma \cdot \beta$ an element of $\W{\scrL''}{\scrL}$. Moreover, there is a map
\[(-) \cdot (-) :  \uW{\scrL''}{\scrL'} \times \uW{\scrL'}{\scrL} \to  \uW{\scrL''}{\scrL}.\]
This map is associative in the obvious sense.

The following is an analogue of \cite[Lemma 4.2]{LY} and \cite[Corollary 4.3]{LY}. However, our proofs differ slightly. 
\begin{lemma}\label{lem:min_elts_in_blocks}
Let $\scrL, \scrL', \scrL'' \in \Ch^{\circ} (T,\k)$ be in the same $W$-orbit.
Let $\beta \in \uW{\scrL'}{\scrL}$, $\gamma \in \uW{\scrL''}{\scrL'}$. 
\begin{enumerate}
    \item There exists a unique minimal length (in $W$) element $w^{\beta} \in \beta$ such that $w^{\beta} \Phi_{\re, \scrL}^+ \subseteq \Phi^+$.
    \item The minimal elements in each block satisfy $w^{\gamma} w^{\beta} = w^{\gamma \beta}$.
\end{enumerate}
\end{lemma} 
\begin{proof}
    We will just prove the first statement. The second statement follows the criterion for $w^{\beta}$ given in the first statement (cf., \cite[Corollary 4.3]{LY}).
    Let $w^{\beta}$ be of minimal length in $\beta$.
    Let $r = r_{\alpha}$ for $\alpha \in \Phi_{\re, \scrL}^+$.
    By assumption, $\ell (wr) > \ell (w)$. As a result, $w \cdot \alpha \in \Phi^+$ for all $\alpha \in \Phi_{\re, \scrL}^+$.

    Suppose $v \in W_{\scrL}^{\circ}$ with $v \neq e$.
    Since $\ell_{\scrL} (v) > 0$, there must exist some $\alpha \in \Phi_{\re, \scrL}^+$ such that $v \cdot \alpha \in \Phi^{-}$.
    Therefore, $w^{\beta} v \cdot \alpha \in \Phi^{-}$, so the only minimal length element of $\beta$ sending $\Phi_{\re, \scrL}^+$ to $\Phi_{\re, \scrL'}^+$ is $w^{\beta}$. 
\end{proof}

\begin{lemma}\label{lem:beta_conj_of_endo_weyl_gps}
Let $\beta \in \uW{\scrL'}{\scrL}$. Then the map $w \mapsto w^{\beta} w w^{\beta, -1}$ gives an isomorphism of Coxeter groups $W_{\scrL}^{\circ} \stackrel{\sim}{\to} W_{\scrL'}^{\circ}$.
\end{lemma}
\begin{proof}
    It follows from Lemma \ref{lem:min_elts_in_blocks} that $\ell_{\scrL} (w) = \ell_{\scrL'} (w^{\beta} w w^{\beta, -1})$ for all $w \in W_{\scrL}^{\circ}$.
    Moreover, if $\alpha \in \Phi_{\re, \scrL}^+$, then $w^{\beta} r_\alpha w^{\beta, -1} = r_{w^{\beta} \alpha}$ and $w^{\beta} \alpha \in \Phi_{\re, \scrL'}^+$.
    As a result, conjugation by $w^{\beta}$ induces a bijection between $S_{\scrL}^{\circ}$ and $S_{\scrL'}^{\circ}$.
\end{proof}

Let $\beta \in \uW{\scrL'}{\scrL}$. 
We can define a length function $\ell_{\beta} : \beta \to \Z_{\geq 0}$ as follows.
For each $w \in \beta$, there exists a unique $v \in W_{\scrL}^{\circ}$ such that $w = w^{\beta} v$. We then define $\ell_{\beta} (w) = \ell_{\scrL} (v)$.
By Lemma \ref{lem:min_elts_in_blocks}, the only element of $\beta$ with length 0 is $w^{\beta}$.
Similarly, we can define a partial order on $\beta$, denoted $\leq_{\beta}$. For $w.w' \in \beta$, let $v,v' \in W_{\scrL}^{\circ}$ such that $w = w^{\beta} v$ and $w' = w^{\beta} v'$.
We then define $w \leq_{\beta} w'$ if and only if $v \leq_{\scrL} v'$.

The following lemmas are a summary of results found in \cite[\S4]{LY}. Their proofs can follow from those in \emph{loc. cit.} after obvious modifications are made. 

\begin{lemma}\label{lem:simple_reflns_are_minimal}
Let $s \in S$ be a simple reflection such that $s \scrL \neq \scrL$. Then $s$ is the minimal element in the block $s W_{\scrL}^{\circ}$.
\end{lemma}

\begin{lemma}\label{lem:block_translation_and_order}
    Let $\beta \in \uW{\scrL'}{\scrL}$.
    \begin{enumerate}
        \item If $\gamma \in \uW{\scrL''}{\scrL'}$, the map $(\beta, \leq_{\beta}) \to (\gamma \beta, \leq_{\gamma \beta})$ given by $w \mapsto w^{\gamma} w$ is an isomorphism of posets.
        \item If $\delta \in \uW{\scrL}{\scrL''}$, then the map $(\beta, \leq_{\beta}) \to ( \beta \delta, \leq_{\beta \delta})$ given by $w \mapsto w w^{\delta}$ is an isomorphism of posets.
        \item For $w,w' \in \beta$, if $w \leq_{\beta} w'$, then $w \leq w'$.
    \end{enumerate}
\end{lemma}

\begin{lemma}\label{lem:length_for_blocks}
  Let $\beta \in \uW{\scrL'}{\scrL}$, $w \in \beta$, and $\uw = (s_1, \ldots, s_k)$ be a reduced expression of $w$. 
  Let $\scrL_{k+1} \coloneq \scrL$ and $\scrL_i = s_i \ldots s_k \scrL$ for $1 \leq i \leq k$.
  Then
  \begin{enumerate}
    \item $\ell_{\beta} (w) = \# \{ \alpha \in \Phi_{\re, \scrL}^+ \mid w \alpha < 0 \}$.
    \item For $\gamma \in \uW{\scrL''}{\scrL'}$, we have
        \[\ell_{\gamma \beta} (w^{\gamma} w) = \ell_{\beta} (w).\]
    \item We have $\ell_{\beta} (w) = \# \{ 1 \leq i \leq k \mid \scrL_i = \scrL_{i+1} \}$.
  \end{enumerate}
\end{lemma}

\subsubsection{Multiplicative Local Systems on Reductive Groups}

We need to impose some constraints on the allowable local systems in $\Ch (T, \k)$ to ensure that the collection of roots $\Phi_{\re, \scrL}$ is invariant under extension of scalars.

Each $\scrL \in \Ch (T, \k)$ defines a group homomorphism $\rho_{\scrL} : \bfY = \pi_1 (T) \to \k^{\times}$. 
Given a real root $\alpha \in \Phi_{\re}$, we define $\mon_{\alpha} (\scrL) \coloneq \rho_{\scrL} (\alpha^{\vee}) \in \k^{\times}$.

\begin{definition}
  Let $\scrL \in \Ch (T, \k)$. 
  \begin{enumerate}
    \item We say that $\scrL$ is \emph{torsion} if there exists some $n \in \Z_{> 0}$ which is invertible in $\k$ such that $\scrL^{\otimes n} \cong \uk_T$.
    \item We say that $\scrL$ is \emph{good} if for all real roots $\alpha \in \Phi_{\re}$ one has that $\mon_{\alpha} (\scrL) \in (\k^{\times} + 1) \cup \{1\}$. 
  \end{enumerate}
\end{definition}
Denote the subsets of $\Ch (T, \k)$ consisting of torsion (resp. good) local systems by $\Ch^{\mu} (T, \k)$ (resp. $\Ch^{\circ} (T, \k)$).  

It is easy to check that the class of torsion (resp. good) local systems in $\Ch (T,\k)$ is closed under the $W$-action and duals. 
Most of the results in this paper hold provided we restrict to good multiplicative local systems.
We only will require these local systems to be torsion in \S\S\ref{subsec:whit}-\ref{subsec:all_block_endo}.

\begin{lemma}\label{lem:torsion_implies_good}
  Let $\scrL \in \Ch (T, \k)$. If $\scrL$ is torsion, then $\scrL$ is good.
\end{lemma}
\begin{proof}
  Let $\alpha \in \Phi_{\re, \scrL}$ and write $\zeta = \mon_{\alpha} (\scrL)$. 
  Take $n \in \Z_{> 0}$ such that $\zeta$ is an $n$-th root of unity.
  Assume that $\zeta \neq 1$.
  Consider the polynomial $f(x) = \frac{x^n - 1}{x-1}$ in $k [x]$. This polynomial can be expressed by
  \[f (x) = x^{n-1} + x^{n-2} + \ldots + x + 1.\]
  Alternatively, we can factor $f(x)$ as 
  \[f (x) = \prod_{i=1}^{n-1} (x - \zeta^i).\]
  By evaluating $f(x)$ at $x=1$, we see that
  \[n = \prod_{i=1}^{n-1} (1 - \zeta^i).\]
  As a result, we must have that $1 - \zeta$ is a unit in $\k$.
\end{proof}

\begin{remark}\label{rem:good}\quad
    \begin{enumerate}
        \item When $\mon_{\alpha} (\scrL) = 1$, we have that $(\alpha^{\vee})^* (\scrL) \cong \uk_{\G_m}$.
        \item Suppose that $(\alpha^{\vee})^* (\scrL) \not\cong \uk_{\G_m}$. Let $j : \G_m \hookrightarrow \A^1$.
            We can compute $(j_* (\alpha^{\vee})^* \scrL)_e$ as the two term chain complex
            \[\begin{tikzcd}
        \k \arrow[rr, "\mon_{\alpha} (\scrL) -1"] && \k
        \end{tikzcd}\]
        concentrated in degrees 0 and 1. In particular, we see that $(\alpha^{\vee})^* (\scrL)$ is clean on $\A^1$ if and only if $\mon_{\alpha} (\scrL) -1 \in \k^{\times}$.
        \item If $\k$ is a field, then every character sheaf is good. If $\k$ is an algebraic field extension of $\F_p$ or a finite extension of $\Z_p$, then every character sheaf is torsion.
    \end{enumerate}
\end{remark}

\begin{lemma}\label{lem:eos_and_root_systems}
  Let $\scrL, \scrL' \in \Ch^{\circ} (T, \k)$ and let $\varphi : \k \to \k'$ be a ring homomorphism between noetherian domains.
  \begin{enumerate}
    \item $\k' (\scrL) \in \Ch (T, \k')$ is good;
    \item $\Phi_{\re, \scrL} = \Phi_{\re, \k' (\scrL)}$;
    \item $W_{\scrL}^{\circ} = W_{\k (\scrL)}^{\circ}$.
  \end{enumerate}
\end{lemma}
\begin{proof}
  (1): Let $\alpha^{\vee} \in \Phi_{\re}^{\vee}$. Then it follows from definitions that 
  \begin{equation}\label{eq:eos_and_root_systems_1}
     \mon_{\alpha} (\k' (\scrL)) = \varphi (\mon_{\alpha} (\scrL)).
  \end{equation}
  The claim then follows from $\varphi$ taking $\k^{\times}$ to $(\k')^{\times}$.

  (2): Let $\alpha \in \Phi_{\re}$. If $\alpha \in \Phi_{\re, \scrL}$, then $(\alpha^{\vee})^* \k' (\scrL) \cong \k' \left( (\alpha^{\vee})^* \scrL \right) \cong \uk_{\G_m}'$. In other words, $\Phi_{\re, \scrL} \subseteq \Phi_{\re, \k' (\scrL)}$.
  Now suppose that $\alpha \notin \Phi_{\re, \scrL}$. We then must have that $\mon_{\alpha} (\scrL) - 1 \in \k^{\times}$. In particular, from (\ref{eq:eos_and_root_systems_1}), we have that $\mon_{\alpha} (\k' (\scrL)) - 1 \in \k^{\times}$.
  As a consequence $\mon_{\alpha} (\k' (\scrL)) \neq 1$ which forces $(\alpha^{\vee})^* \k' (\scrL)$ to be non-constant.

  (3): This is immediate from definitions.
\end{proof}

\subsection{Hecke Categories}\label{subsec:hecke}

\subsubsection{Definition and Basic Structure}

In this section, we recall the construction of the monodromic Hecke category and discuss some of its basic properties.
We regard $\eFl = \bigcup_{w \in W} \eFl_{\leq w}$ as an ind-scheme.
The algebraic group of pro-finite type $T \times B$ acts on $\eFl$ via $(t,b) \cdot x = txb^{-1}$.
This action is compatible in the sense of \S\ref{subsec:twisted_eq_sh_on_ind_stacks}. In more detail, for each $w \in W$, there exists a normal, finite codimension subgroup $J_w$ of $B$ such that $J_w \subseteq U$ and $J_w$ acts trivially on $\eFl_{\leq w}$. Without loss of generality, by taking common intersections, we may assume that $J_x \subset J_w$ for $x \leq w$.

Let $\scrL, \scrL' \in \Ch (T, \k)$. We can abuse notation and regard $\scrL$ as a multiplicative local system on $B$ via the pullback along the projection map $B \to T$.
 We can then define the \emph{monodromic Hecke category} as 
 \[\D{\scrL'}{\scrL} (\k) \coloneq D_{\cons} (T \backslash_{\scrL'} \eFl /_{\scrL} B, \k).\]
 It will be useful to unpack this definition. We can define
\[\D{\scrL'}{\scrL} (\leq w, \k) = D_{\cons} (T \backslash_{\scrL'} (\eFl_{\leq w} / (U / J_w) ) /_{\scrL} T, \k)\]
and
\[\D{\scrL'}{\scrL} (w, \k) =  D_{\cons} (T \backslash_{\scrL'} (\eFl_{w} / (U / J_w) ) /_{\scrL} T, \k).\]
Then the monodromic Hecke category is given as the limit of categories
\[\D{\scrL'}{\scrL} (\k) = \lim_{\stackrel{\longrightarrow}{w \in W}} \D{\scrL'}{\scrL} (\leq w, \k),\] 
where the transition maps are given by pushforwards of the closed embeddings $i_{x,y} : \eFl_{\leq x} \to \eFl_{\leq y}$.

When $G$ is a reductive group, then the monodromic Hecke category's definition can be simplified as follows:
\[\D{\scrL'}{\scrL} (\k) = D_{\cons} (T \backslash_{\scrL'} (\UGU) /_{\scrL} T, \k).\]

\begin{proposition}\label{prop:mhc_strat_and_parity}
  Fix $w \in W$ and define a subset
  \[\Gamma (w) = \{ (wt, t) \in T \times T \mid t \in T\}.\]
  We regard $\Gamma (w)$ as an algebraic group isomorphic to $T$.
  \begin{enumerate}
    \item If $\scrL' \neq w (\scrL)$, then $\D{\scrL'}{\scrL} (w, \k) = 0$. If $\scrL' = w (\scrL)$, then there is an equivalence of categories
      \[\D{w(\scrL)}{\scrL} (w, \k)\cong D_{\cons}( \Gamma (w) \backslash \{\dot{w}\}, \k)\]
      obtained by taking the stalk at $\dot{w}$ for a lift $\dot{w} \in G$ of $w$.
    \item The category $\D{\scrL'}{\scrL} (\k)$ satisfies the parity conditions.
  \end{enumerate}
\end{proposition}
\begin{proof}
  \emph{(1):} Note that since $\eFl_w \cong \A^{\ell (w)} \times \dot{w} T$, requiring that a sheaf on $\eFl_w$ be $(T\times T, \scrL' \boxtimes
  \scrL^{-1})$-equivariant forces that $\scrL' = w\scrL$.
  If $\scrL' = w\scrL$, we can take the stalk at a lift $\dot{w}$ of $w$ to obtain a functor,
  \[i_{\dot{w}}^* : \D{w(\scrL)}{\scrL} (w, \k) \to D_{\cons} ((T \times T) {}_{(w \scrL \boxtimes \scrL^{-1})} \backslash \dot{w} T, \k) \cong D_{\cons} (\Gamma (w)
  \backslash \{ \dot{w} \}, \k)  .\]
    So $w \scrL \boxtimes \scrL^{-1}$ restricted to $\Gamma (w)$ is a
  constant sheaf. It is easy to check that since the cohomology of twisted equivariant constructible sheaves on $\eFl_w$ are local systems, that $i_{\dot{w}}^*$ gives an
  equivalence of triangulated categories.

  \emph{(2):} The parity conditions follow immediately from (1) along with the observation that $D_{\cons} ( T \backslash \pt, \k)$ satisfies the parity
  conditions since the equivariant cohomology $H_T^\bullet (\pt; \k) \cong H_{\Gamma (w)} (\{ \dot{w} \}; \k)$ is a free $\k$-module in even degrees and vanishes
  in odd degrees.
\end{proof}

The previous proposition gives us the freedom to use standard, costandards, IC-extensions, parity sheaves, and simple perverse sheaves in $\D{\scrL'}{\scrL} (\k)$.
The general theory of these is given in Appendix \ref{apdx:A}, and we will review it here.
For each $w \in \W{\scrL'}{\scrL}$, we define $\scrK_{\dot{w}}^{\scrL} \in \D{\scrL'}{\scrL} (w, \k)$ as the sheaf corresponding to the constant sheaf on $\{\dot{w}\}/\Gamma (w)$ under the
equivalence given in Proposition \ref{prop:mhc_strat_and_parity}.
Alternatively, under the isomorphism of schemes $\eFl_w \cong \A^{\ell (w)} \times \dot{w} T$, $\scrK_{\dot{w}}^{\scrL}$ is characterized as the unique sheaf (up to isomorphism) which satisfies $\For_{T \times T, \scrL' \boxtimes \scrL^{-1}} \scrK_{\dot{w}}^{\scrL} \cong \uk_{\A^{\ell (w)}} \boxtimes \scrL$.
A priori, $\scrK_{\dot{w}}^{\scrL}$ depends on the choice on the choice of lift $\dot{w}$. 
It is easy to show that the isomorphism class of $\scrK_{\dot{w}}^{\scrL}$ does not depend on the choice $\dot{w}$. 
As a result, we will simply write $\scrK_w^{\scrL}$ to refer to any object whose isomorphism class is $\scrK_{\dot{w}}^{\scrL}$. We warn that $\scrK_{w}^{\scrL}$ is only defined up to non-unique isomorphism-- there is no canonical choice of $\scrK_w^{\scrL}$ except when $w = e$.

We can then define the following sheaves in $\D{\scrL'}{\scrL} (\k)$
\[\Delta_w^{\scrL} \coloneq j_{w!} \scrK_w^{\scrL} [\ell (w)] \hspace{0.5cm}\text{and}\hspace{0.5cm} \nabla_w^{\scrL} \coloneq j_{w!} \scrK_w^{\scrL} [\ell (w)],\]
which are called the \emph{standard} and \emph{costandard} sheaves respectively.

Recall from Appendix \ref{apdx:A} that $\D{\scrL'}{\scrL} (\k)$ has a perverse $t$-structure. Explicitly a sheaf $\scrF \in \D{\scrL'}{\scrL} (\k)$ is perverse if $\For_{T \times T, \scrL' \boxtimes \scrL^{-1}} \scrF [2\dim T]$ is a perverse on $\UGU$.
It is then easy to see that $\Delta_w^{\scrL}$ and $\nabla_w^{\scrL}$ are both perverse sheaves. We can define the IC-sheaves,
\[\IC_w^{\scrL} = \text{Im} (\Delta_w^{\scrL} \to \nabla_w^{\scrL}).\]

The tensor product for twisted equivariant sheaves defines a bifunctor
\[(-) \otimes^L (-) : \D{\scrL_4}{\scrL_3} (\k) \times \D{\scrL_2}{\scrL_1} (\k) \to \D{\scrL_4 \scrL_2}{\scrL_3 \scrL_1} (\k), \]
for $\scrL_1, \scrL_2, \scrL_3, \scrL_4 \in \Ch^{\circ} (T,\k)$ in the same $W$-orbit.
Similarly, the tensor product admits a right adjoint taking
\[\RHom (-, -): \D{\scrL_4}{\scrL_3}^{\op} (\k) \times \D{\scrL_2}{\scrL_1} \to \D{\scrL_4^{-1} \scrL_2}{ \scrL_3^{-1} \scrL_1} (\k),\]
for $\scrL_1, \scrL_2, \scrL_3, \scrL_4 \in \Ch^{\circ} (T, \k)$ in the same $W$-orbit.

\subsubsection{Parabolic Variants}

Let $J \subseteq I$ of finite type. Associated to such a $J$ is a standard parabolic subgroup $P_J$ containing $B$ with Levi decomposition $P_J = U^J L_J$ where $L_J$ is
a connected reductive group containing $T$.
Let $W_J$ be the subgroup of $W$ generated by $J$, which can be identified with the Weyl group of $L_J$. Let $w_J$ denote the element of $W_J$ of maximal length.
Let $W^J \subseteq W$ denote the minimal length representatives of the cosets in $W/W_J$.
The parabolic Bruhat decomposition states that the orbits of the right $P_J$-action on $\eFl$ are indexed by $W/W_J$.
For each $\overline{w} \in W/W_J$, we will write $\eFl_{\overline{w}}^J = U \backslash B \dot{w} P_J$ where $\dot{w} \in N_G(T)$ is a lift of a representative of $\overline{w}$. We will also write $j_{\overline{w}} : \eFl_{\overline{w}}^J \hookrightarrow \eFl$ for the inclusion map.
The orbit closures are given by unions of strata from the Bruhat order
\[ \overline{\eFl_{ \overline{w}}^J} = \bigcup_{\stackrel{\overline{x} \in W/W_J}{\overline{x} \leq \overline{w}}} \eFl_{\overline{x}}^J, \]
where the order on $W / W_J$ is given by the restricted Bruhat order from $W$. More precisely, the partial order on $W / W_J$ is given by $\overline{x} \leq \overline{y}$ if $x_{-} < y_{-}$ where $x_{-}$ (resp. $y_{-}$) is of minimal length in $\overline{x}$ (resp. $\overline{y}$).  

Let $\Ch_J^{\circ} (T, \k)$ denote the subset of $\Ch^{\circ} (T, \k)$ consisting of local systems $\scrL \in \Ch^{\circ} (T, \k)$ such that $s_j (\scrL) \cong \scrL$ for all $j \in J$.
As a result, the action of $W/W_J$ on $\Ch_J^{\circ} (T, \k)$ is well-defined. Precisely, if $\overline{w} \in W/W_J$ and $w \in W$ is a representative for the coset $\overline{w}$,
we can define $\overline{w} \cdot \scrL = w (\scrL)$ which is independent of the choice of representative.
If $\scrL, \scrL' \in \Ch_{J}^{\circ} (T, \k)$ are in the same $W$-orbit, we can define ${}_{\scrL'} (W/W_J)_{\scrL} = \{ \overline{w} \in W/W_J \mid \overline{w} (\scrL) = \scrL' \}$.

\begin{lemma}\label{lem:extending_character_sheaves_to_levis}
  Let $\scrL \in \Ch^{\circ} (T, \k)$ and $J \subset I$ of finite type.
  Then $\scrL$ can be extended to a multiplicative rank one local system $\scrL^J$ on $L_J$ if and only if $\scrL \in \Ch_J^{\circ} (T, \k)$.
\end{lemma}
\begin{proof}
  First, we observe that $\scrL \in \Ch^{\circ} (T, \k)$ is in $\Ch_{J}^{\circ} (T, \k)$ if and only if $(\alpha_j^\vee)^* (\scrL) \cong \uk_{\G_m}$ for all $j \in J$.

  Note that $\pi_1 (L_J) = \bfY / \Z\Phi_J^\vee$ where $\Phi_J^\vee = \{\alpha_j^\vee\}_{j \in J}$.
  The inclusion of the maximal torus $T \hookrightarrow L_J$ induces a quotient map
  \begin{equation}\label{eq:extending_character_sheaves_to_levis_1}
    \pi_1 (T) = \bfY \twoheadrightarrow \bfY / \Z\Phi_J = \pi_1 (L_J).
  \end{equation}
  In particular, the irreducible representations of $\pi_1 (L_J)$ are in 1-1 correspondence with $\scrL \in \Ch (T, \k)$ such that $(\alpha_{j}^\vee)^* (\scrL) \cong \uk_{\G_m}$ for all $j \in J$.
  For such a $\scrL$, let $\scrL^J$ be the associated local system on $L_J$.
  Computing $(\scrL^J)\mid_T$ is given algebraically by the composition of the irreducible representation for $\pi (L_J)$ with the quotient map of
  (\ref{eq:extending_character_sheaves_to_levis_1}).
  It is easy to see that this corresponds with $\scrL \in \Ch^{\circ} (T, \k)$.
\end{proof}

Fix $J \subset I$ of finite type.
For each $\overline{w} \in W/ W_J$, there exists a normal subgroup $J_{\overline{w}}$ of $P_J$ of finite codimension, contained in $U^J$,
such that the right action of $J_{\overline{w}}$ on $\eFl_{\leq \overline{w}}^J$ is trivial.
For a fixed $\overline{w} \in W / W_J$, we can pick $J_{\overline{x}}$ such that if $\overline{w} \leq \overline{x}$, then $J_{\overline{w}} \subseteq J_{\overline{x}}$.
In particular, the action of $T \times P_J$ on $\eFl$ is compatible.

Let $\scrL \in \Ch_J^{\circ} (T, \k)$ and $\scrL' \in \Ch^{\circ} (T, \k)$. We can view $\scrL^J$ as a multiplicative local system on $P_J$ via the projection map $P_J \to L_J$.
We define the \emph{parabolic monodromic Hecke category}
\[
   \paraD{\scrL'}{\scrL}{J} (\k) =  D_{\cons} (T \backslash_{\scrL'} \eFl /_{\scrL^J} P_J, \k) . 
  \]
It is also useful to consider the restriction of the parabolic monodromic Hecke category to the natural stratification. Let $\overline{w} \in W/W_J$. We then define
\[\paraD{\scrL'}{\scrL}{J} (\overline{w}, \k) \coloneq D_{\cons} (T \backslash_{\scrL'} (\eFl_{\overline{w}}^J / (U^J/J_{\overline{w}})) /_{\scrL^J} L_J, \k).\]

\begin{lemma}\label{lem:partial_mhc_is_stratified_and_satisfies_parity_cond}
  Let $\fr{o}$ be a $W$-orbit in $\Ch^{\circ} (T, \k)$. Let $\scrL,\scrL' \in \Ch_{J}^{\circ} (T, \k) \cap \fr{o}$.
  \begin{enumerate}
    \item If $\scrL' \neq \overline{w} (\scrL)$, then $\paraD{\scrL'}{\scrL}{J} (\overline{w}, \k) = 0$. If $\scrL' = \overline{w} (\scrL)$, then there is an equivalence of categories
      \[\paraD{\overline{w}\scrL}{\scrL}{J} (\overline{w}, \k) \cong D_{\cons}(T \backslash \{\dot{w}\}, \k)\]
      obtained by taking the stalk at $\dot{w}$ for a lift $\dot{w} \in N_G (T)$ of $\overline{w}$.
    \item The category $\paraD{\scrL'}{\scrL}{J} (\k)$ satisfies the parity conditions.
  \end{enumerate}
\end{lemma}
\begin{proof}
  Statement (2) follows from the same argument given for Proposition \ref{prop:mhc_strat_and_parity} (2). 
  We closely follow the argument given in \cite[\S3.10]{LY}.
  Pick a lift $\dot{w} \in N_G (T)$ of $\overline{w} \in W/W_J$. There is an isomorphism
  \begin{equation}\label{eq:partial_mhc_is_stratified_and_satisfies_parity_cond_1}
    \eFl_{\overline{w}}^J / (U^J/J_{\overline{w}}) \cong \dot{w} \cdot \left( (\dot{w}^{-1} U \dot{w} \cap L_J J_{\overline{w}}) \backslash L_J J_{\overline{w}} \right).
  \end{equation}
  Under this isomorphism, the left action by $t \in T$ on the left-hand side of (\ref{eq:partial_mhc_is_stratified_and_satisfies_parity_cond_1}) becomes the left translation of $\dot{w}^{-1} t \dot{w}$ on $\dot{w} \cdot \left( (\dot{w}^{-1} U \dot{w} \cap L_J J_{\overline{w}}) \backslash L_J J_{\overline{w}} \right)$.
  By taking the stalk at $\dot{w}$ we get an equivalence of categories
  \[\paraD{\scrL'}{\scrL}{J} (\overline{w}, \k) \cong D_{\cons} (T \backslash_{(\scrL' \otimes \overline{w} \scrL^{J, -1})\mid_T} \{ \dot{w} \}, \k) \cong D_{\cons} (T \backslash_{\scrL' \otimes \scrL^{-1}} \{\dot{w}\}),\]
  where the second isomorphism follows from $\scrL^J \mid_T \cong \scrL$.
  As a result, $\paraD{\scrL'}{\scrL}{J} (\overline{w}, \k) = 0$ unless $\scrL' = \overline{w} \scrL$, in which case it $\paraD{\scrL'}{\scrL}{J} (\overline{w}, \k) \cong D_{\cons} (T \backslash \{\dot{w} \}, \k)$.
\end{proof}

For each $\overline{w} \in {}_{\scrL'} (W /W_J)_{\scrL}$, we can define $\scrK_{\overline{w}}^{J, \scrL} \in \paraD{\overline{w}\scrL}{\scrL}{J} (\overline{w}, \k)$ as the sheaf corresponding to the constant sheaf on $T \backslash \{\dot{w}\}$.
Alternatively, $\scrK_{\overline{w}}^{J, \scrL}$ can be identified with $\scrL^J$ viewed as a $(\dot{w}^{-1} U \dot{w} \cap L_J)$-equivariant sheaf on $L_J$. 
In particular, since $\scrL^J$ is a multiplicative local system, $\scrK_{\overline{w}}^{J, \scrL}$ does not depend on the choice of $\dot{w}$ up to a (non-unique) isomorphism.

We can then define variations of \emph{standard} and \emph{costandard} sheaves in $\paraD{\scrL'}{\scrL}{J} (\k)$.
Let $\overline{w} \in {}_{\scrL'} (W /W_J)_{\scrL}$ and $w \in W^J$ be a representative in the coset $\overline{w}$.
\[\Delta_{\overline{w}}^{J, \scrL} \coloneq j_{\overline{w} !} \scrK_{\overline{w}}^{J, \scrL} [\ell(w)] \hspace{0.5cm} \text{and} \hspace{0.5cm} \nabla_{\overline{w}}^{J, \scrL} \coloneq j_{\overline{w} *} \scrK_{\overline{w}}^{J, \scrL} [\ell(w)].\]
It can be easily checked that $\Delta_{\overline{w}}^{J, \scrL}$ and $\nabla_{\overline{w}}^{J, \scrL}$ are perverse sheaves, so we can consider the IC-extension
\[\IC_{\overline{w}}^{J, \scrL} \coloneq \textnormal{Im} \left(\Delta_{\overline{w}}^{J, \scrL} \to \nabla_{\overline{w}}^{J, \scrL} \right).\]

\subsubsection{Averaging along Parabolics}

\begin{remark}\label{rem:usefulness_of_finite_parabolics}
  Let $J \subseteq I$ be of finite type. A Borel subgroup $B_J$ of $L_J$ containing $T$ is given by $B \cap L_J$. Similarly, its unipotent radical $U_J$ of $B_J$ is
  given by $U \cap L_J$.
  Alternatively, $B_J$ can be identified with the quotient $B/U^J$.

  Despite $\eFl_G$ being infinite dimensional, the quotient $U \backslash P_J$ is finite dimensional.
  There are isomorphisms of varieties (cf., \cite[Example 7.1.7]{Ku}),
  \begin{equation}\label{eq:usefulness_of_finite_parabolics_2}
    U \backslash P_J \cong U_J \backslash L_J = \eFl_{L_J}.
  \end{equation}

  Consider the special case when $J = \{s\}$ for some $s \in S$. Since $\eFl_{\leq s} = U \backslash P_s$, we obtain an equivalence of categories
  \begin{equation}\label{eq:usefulness_of_finite_parabolics_3}
    \D{\scrL'}{\scrL} (\leq s, \k) \cong D_{\cons} (T \backslash_{\scrL'}( \eFl_{L_s} /U_s) /_{\scrL} T, \k),
  \end{equation}
  where the right-hand side of (\ref{eq:usefulness_of_finite_parabolics_3}) is simply the monodromic Hecke category for $L_s$.
\end{remark}

Consider the homomorphism $\varphi : B \hookrightarrow P_s$ of group schemes (of pro-finite type). 
They both act on $\eFl$ on the right.
If $\scrL, \scrL' \in \Ch^{\circ} (T,\k)$ are in the same $W$-orbit and $s \in W_{\scrL}^\circ$, by Lemma \ref{lem:extending_character_sheaves_to_levis}, there is an extension $\scrL^s$ of $\scrL$ to $L_s$.
There is then a morphism of twisted ind-algebraic stacks 
\[\pi_s : T \backslash_{\scrL'} \eFl /_{\scrL} B \to T \backslash_{\scrL'} \eFl /_{\scrL^s} P_s\]
induced by the identity map $\eFl \to \eFl$ and the morphism $\varphi : B \hookrightarrow P_s$.
We can then apply the general formalism of averaging and forgetting twisted equivariance to obtain adjoint functors
\[\pi_s^* : \paraD{\scrL'}{\scrL}{s} (\k) \rightleftarrows \D{\scrL'}{\scrL} (\k) : \pi_{s*}.\]
Likewise, there are variants for $!$-pushforward/pullback. 
Since $P_s/B \cong \P^1$, there are natural isomorphisms of functors $\pi_{s*} \cong \pi_{s!}$ and $\pi_s^! \cong \pi_s^* [2]$.

\begin{lemma}\label{lem:pi_s_descent}
  Let $s \in W_{\scrL}^\circ$ be a simple reflection in $W$. Suppose that $\ell (w) > \ell (ws)$, then there is an isomorphism
  \[\pi_{s}^* \IC_{\overline{w}}^{s, \scrL} [1] \cong \IC_w^{\scrL}. \]
\end{lemma}
\begin{proof}
    Since $\pi_s^! [-1] \cong \pi_s^* [1]$ and $P_s/B \cong \P^1$, it is clear that $\pi_{s}^* \IC_{\overline{w}}^{s, \scrL} [1]$ is a simple perverse sheaf supported on $\eFl_{\leq w}$. 
    The lemma then follows after noting that the restriction $\pi_{s}^* \IC_{\overline{w}}^{s, \scrL} [1]$ along $\eFl_w$ is nonzero since $\ell (w) > \ell (ws)$.
\end{proof}

\begin{lemma}\label{lem:pi_s_pushforward_of_stds}
  Let $s \in W_{\scrL}^\circ$ be a simple reflection in $W$. Suppose that $\ell (w) > \ell (ws)$, then there are isomorphisms
  \[\pi_{s*} \Delta_{ws}^{\scrL} \cong \Delta_{\overline{w}}^{s, \scrL} \hspace{0.5cm} \text{and} \hspace{0.5cm} \pi_{s*} \Delta_{w}^{\scrL} \cong \Delta_{\overline{w}}^{s, \scrL} [-1].\]
\end{lemma}
\begin{proof}
  Let $x \in W$. Let $R(x) = \{ \alpha \in \Phi^+ \mid x\alpha < 0 \}$. 
  We can consider $U_x = \prod_{\alpha \in R (x)} U_{\alpha}$ where $U_{\alpha}$ is the root subgroup for $\alpha$ of $G$.
  Let $\dot{x} \in N_G (T)$ be a lift of $x$.
  There is a $B \times B$-equivariant isomorphism of schemes $U\dot{x}B \stackrel{\sim}{\to} U_x \times \dot{x} B$ where $B \times B$ acts by $(b_1, b_2) \cdot (u, \dot{x}g) = (b_1 u b_1^{-1}, b_1 \dot{x} g b_2^{-1})$.
  Similarly, there is a $B \times P_s$-equivariant isomorphism of schemes $U\dot{ws}P_s \stackrel{\sim}{\to} U_{ws} \times \dot{ws} P_s$ where $B \times P_s$ acts in an analogous way.
  By taking $x = ws$, it follows that the map
  \[a_{ws} : \eFl_{ws} \times^B P_s \to \eFl_{\overline{w}}^s \]
  induced by the right action of $B$ on $\eFl_{ws}$ is an isomorphism of schemes.
  Similarly, by taking $x = w$, the map
  \[a_w : \eFl_{w} \times^B P_s \to \eFl_{\overline{w}}^s\] 
  is a trivial $\A^1$-fibration.

  Note that the descent of $\scrL \boxtimes \scrL^s$ regarded as a sheaf on $T \times L_s$ to a sheaf on $T \times^T L_s \cong L_s$ is given by $\scrL^s$.
  This observation along with the properties of $a_w$ and $a_{ws}$ discussed earlier yield isomorphisms 
  \[\pi_{s!} \scrK_{ws}^{\scrL} \cong a_{ws!} (\scrK_{ws}^{\scrL} \widetilde{\boxtimes} \scrL^s) \cong \scrK_{\overline{w}}^{s, \scrL} \hspace{0.5cm} \text{and} \hspace{0.5cm}  \pi_{s!} \scrK_{w}^{\scrL} \cong a_{w!} (\scrK_{w}^{\scrL} \widetilde{\boxtimes} \scrL^s) \cong \scrK_{\overline{w}}^{s, \scrL} [-2].\]
  The result readily follows.
\end{proof}

\subsection{Convolution}\label{sec:equiv_conv}

\subsubsection{Definition and Properties}

Just as in the non-monodromic Hecke category, one can define a convolution bifunctor, 
\[(-) \star (-) : \D{\scrL''}{\scrL'} (\k) \times \D{\scrL'}{\scrL} (\k) \to \D{\scrL''}{\scrL} (\k).\]

We will begin with its construction. Consider the diagram
\[\begin{tikzcd}
  & U \backslash G \times^U G/U \arrow[r] \arrow[ld] \arrow[rd] & U \backslash G \times^B G/U \arrow[r, "m"] & U \backslash G/U. \\
\UGU &                                                             & \UGU                                       &                 
\end{tikzcd}\]
Let $T \times T \times T$ act on $U \backslash G \times^U G /U$ via $(t_1, t_2, t_3) \cdot (x,y) = (t_1xt_2^{-1} , t_2yt_3^{-1})$.
We can define a morphism of twisted ind-algebraic stacks
\[p : \frac{T \backslash_{\scrL''} (U \backslash G \times^U G /U) /_{\scrL} T}{T} {}_{\uk_T} \to T \backslash_{\scrL''} (\UGU) /_{\scrL'} T \times T \backslash_{\scrL'} (\UGU) /_{\scrL} T,\]
which is given by the projection maps $U \backslash G \times^U G /U \to \UGU \times \UGU$ and on $T^3 \to T^4$ is given by $(t_1, t_2, t_3) \mapsto (t_1, t_2, t_2, t_3)$.
Let $\scrF \in \D{\scrL''}{\scrL'} (\k)$ and $\scrG \in \D{\scrL'}{\scrL} (\k)$.
The $(T \times T \times T, \scrL'' \boxtimes \uk_T \boxtimes \scrL^{-1})$-equivariant sheaf $p^* (\scrF \boxtimes \scrG)$ descends to a $(T \times T, \scrL'' \boxtimes \scrL^{-1})$-equivariant sheaf $\scrF \widetilde{\boxtimes} \scrG$ on $U \backslash G \times^B G/U$.
We then define
\[\scrF \star \scrG \coloneq m_! (\scrF \widetilde{\boxtimes} \scrG),\]
which naturally is also $(T \times T, \scrL'' \boxtimes \scrL^{-1})$-equivariant, and thus can be regarded as an object on $\D{\scrL''}{\scrL} (\k)$.

Convolution for monodromic Hecke categories is surveyed well in other literature (cf., \cite{LY, Gou, Li}). We review some basic facts about convolution whose proofs can be found in \emph{loc. cit.}

\begin{lemma}\label{lem:skyscraper_is_unit}
  Consider the object
  \[\Delta_e^{\scrL} \cong \IC_e^{\scrL} \cong \nabla_e^{\scrL} \in \D{\scrL}{\scrL} (\k).\]
  The functors
      \[\IC_e^{\scrL'} \star (-) : \D{\scrL'}{\scrL} (\k) \to \D{\scrL'}{\scrL} (\k),\]
      \[(-) \star \IC_e^{\scrL} : \D{\scrL'}{\scrL} (\k) \to \D{\scrL'}{\scrL} (\k)\]
      are naturally isomorphic to the identity functor.
\end{lemma}

\begin{lemma}\label{lem:assoc_of_conv}
  The convolution bifunctor is suitably associative. Namely, the following diagram commutes up to natural isomorphism.
  \begin{equation}
    \begin{tikzcd}
      \D{\scrL'''}{\scrL''} (\k) \times \D{\scrL''}{\scrL'} (\k) \times \D{\scrL'}{\scrL} (\k) \arrow[r, "\id \times \star"] \arrow[d, "\star \times \id"] & \D{\scrL'''}{\scrL''} (\k) \times \D{\scrL''}{\scrL} (\k) \arrow[d, "\star"]               \\
      \D{\scrL'''}{\scrL'} (\k) \times \D{\scrL'}{\scrL} (\k) \arrow[r, "\star"]                                          & \D{\scrL'''}{\scrL} (\k).
      \end{tikzcd}
    \end{equation}
\end{lemma}

\begin{lemma}\label{lem:conv_rules_1}
  We have natural isomorphisms\footnote{Naturality here should be interpreted as after picking lifts $\dot{x}$ and $\dot{y}$ of $x$ and $y$. Then $\dot{x}\dot{y}$ is a lift of $xy$, and all objects are fixed and not just isomorphism classes.}
  \begin{enumerate}
    \item $\Delta_{xy}^{\scrL} \cong \Delta_{x}^{y (\scrL)} \star \Delta_y^{\scrL}$ if $\ell (xy) = \ell (x) + \ell (y)$;
    \item $\nabla_{xy}^{\scrL} \cong \nabla_x^{y (\scrL)} \star \nabla_y^{\scrL}$ if $\ell (xy) = \ell (x) + \ell (y)$;
  \end{enumerate}
\end{lemma}

The foundations of many inductive arguments on properties of convolution rely on understanding how convolving by an $\IC_s^{\scrL}$ behaves for $s \in S$. 
In the non-monodromic case, this is described by the so-called push-pull lemma. However, in the monodromic case, the description is more subtle.
When $s \in W_{\scrL}^\circ$, then convolving by $\IC_s^{\scrL}$ behaves like the non-monodromic case and is governed by a version of the push-pull lemma.
If $s \notin W_{\scrL}^\circ$, then convolving by $\IC_s^{\scrL}$ can be thought of ``translating'' between monodromic Hecke categories. 

\begin{lemma}\label{lem:conv_with_ICs_rewritting}
  Let $\scrL, \scrL' \in \Ch^{\circ} (T,\k)$ be in the same $W$-orbit. 
  Let $s \in W$ be a simple reflection such that $s \in W_{\scrL}^\circ$, then there is a canonical isomorphism of endofunctors,
  \[(-) \star \IC_s^{\scrL} \cong \pi_s^* \pi_{s*} (-) [1].\]
\end{lemma}
\begin{proof}
  Let $\scrF \in \D{\scrL'}{\scrL} (\k)$. By Remark \ref{rem:usefulness_of_finite_parabolics}, we can regard $\IC_s^{\scrL}$ as $\scrL^s [1]$ viewed as a $U_s$-equivariant sheaf on $L_s$. 
  Let $a : \eFl \times^B P_s \to \eFl$ be the map given by the right action of $B$ on $\eFl$.
  We then have an isomorphism 
  \[ \scrF \star \IC_s^{\scrL} \cong a_! (\scrF \widetilde{\boxtimes} \scrL^s [1]),\] 
  where $\scrF \widetilde{\boxtimes} \scrL^s
  [1]$ is the descent of $\scrF \boxtimes \scrL^s [1]$ to a sheaf on $\eFl \times^B P_s$.
  From the discussion in \S A.4.11, we see that $a_! (\scrF \widetilde{\boxtimes} \scrL^s [1])$ has the same underlying complex as $\pi_{s*} \scrF [1]$ in
  $\paraD{\scrL'}{\scrL}{s} (\k)$.
  After only remembering $(T, \scrL)$-equivariance on the right, we obtain an isomorphism of sheaves $\scrF \star \IC_s \cong \pi_s^* \pi_{s*} \scrF [1]$ in $\D{\scrL'}{\scrL}
  (\k)$ as desired.
\end{proof}

\begin{lemma}\label{lem:conv_ICs_with_stds}
  Let $s \in S$ be a simple reflection such that $s \in W_{\scrL}^\circ$, and let $w \in \W{\scrL'}{\scrL}$.
  Then
  \begin{equation}
    \textnormal{rank} H^i (j_x^* (  j_{w!} \scrK_w^{\scrL} \star \IC_s^{\scrL}) ) \cong
    \begin{cases} 1 & \text{if } ws > w, i = -1, \text{ and } x \in \{w, ws\}, \\ 1 & \text{if } ws < w, i = 1, \text{ and } x \in \{w, ws\}, \\ 0 & \text{otherwise.}
    \end{cases}
  \end{equation}
\end{lemma}
\begin{proof}
  We will assume that $ws < w$, but we will compute both $j_{w!} \scrK_w^{\scrL} \star \IC_s^{\scrL}$ and $j_{ws!} \scrK_w^{\scrL} \star \IC_s^{\scrL}$.
  Lemma \ref{lem:conv_with_ICs_rewritting} and Lemma \ref{lem:pi_s_pushforward_of_stds} gives isomorphisms
  \[j_{ws!} \scrK_{ws}^{\scrL} \star \IC_s^{\scrL} \cong j_{\overline{w} !} \scrK_{w, ws}^{\scrL} [1] \hspace{0.5cm} \text{and} \hspace{0.5cm} j_{w!} \scrK_{w}^{\scrL} \star \IC_s^{\scrL} \cong j_{\overline{w} !} \scrK_{w,ws}^{\scrL} [-1],\]
  where $\scrK_{w,ws}^{\scrL} = \pi_s^* \scrK_{\overline{w}}^{s, \scrL}$. 
  
  We will first prove that $j_{ws}^* \scrK_{w,ws}^{\scrL} \cong \scrK_{ws}^{\scrL}$.
  Under the identification 
  \[\eFl_{\overline{w}}^s / U \cong U \backslash (U_{ws} \times \dot{ws} L_s/U_s),\] 
  the sheaf $\scrK_{w,ws}^{\scrL}$ corresponds to the $U \times U_s$-equivariant local system $\uk_{U_{ws}} \boxtimes \scrL^s$.\footnote{Note that $\scrK_{w,ws}^{\scrL}$ is not literally a sheaf $\eFl_{\overline{w}}^s / U$ since this stack is not locally of finite type. Instead, $\scrK_{w,ws}^{\scrL}$ is a sheaf on $\eFl_{\overline{w}}^s / (U/J_w)$. We suppress this subtlety here and throughout similar isomorphisms in this proof to make statements simpler. Since $U$ is pro-unipotent, asking for $U$-equivariance is simply a constructibility constraint which coincides with that imposed by $U/J_w$. As a result, this simplification does not lead to any ambiguous sheaf theory.}
  Similarly, under the isomorphism $\eFl_{ws} / U \cong U \backslash (U_{ws} \times \dot{ws} T)$, the sheaf $\scrK_{ws}^{\scrL}$ corresponds to the $U$-equivariant local system $\uk_{U_{ws}} \boxtimes \scrL$. 
  Since $\scrL^s\vert_T \cong \scrL$, we have that $j_{ws}^* \scrK_{w,ws}^{\scrL} \cong \scrK_{ws}^{\scrL}$.
  
  We will now prove that $j_{w}^* \scrK_{w,ws}^{\scrL} \cong \scrK_{w}^{\scrL}$.
  Under the isomorphism $\eFl_{w}/U \cong U \backslash (U_w \times \dot{w} T)$, the sheaf $\scrK_{w}^{\scrL}$ corresponds to the $U$-equivariant local system $\uk_{U_{w}} \boxtimes \scrL$. 
  We claim that via the isomorphism $B\dot{s}B/U \cong U_s \times \dot{s} T$, there is an isomorphism of local systems $\scrL^s \vert_{U_s \times \dot{s} T} \cong \uk_{U_s} \boxtimes \scrL$.
  To this consider the sequence of maps
  \begin{equation}\label{eq:conv_ICs_with_stds_1}
    \begin{tikzcd}
    T \arrow[r, "i", hook] & U_s \times \dot{s} T \arrow[r, "i_s", hook] & L_s/U_s \arrow[r, "\mu", "\sim"'] & L_s/U_s
    \end{tikzcd}
  \end{equation}
    where $i : T \to U_s \times \dot{s} T$ is the map $i(t) = (1,\dot{s}t)$, $i_s$ is the inclusion of $B\dot{s}B/U$ into $L_s/U_s \cong P_s/U$, and $\mu$ is the isomorphism of varieties given by $\mu (x) = \dot{s}^{-1} x$.
    It is clear from its definition that the composition given in \ref{eq:conv_ICs_with_stds_1} coincides with the inclusion map $B \hookrightarrow P_s$.
    As a result, we obtain an isomorphism of local systems $\scrL \cong i^* i_s^* \scrL^s$. Since $i_s^* \scrL^s$ is a rank-1 local system, we conclude that $i_s^* \scrL^s \cong \uk_{U_s} \boxtimes \scrL$.
    Via the isomorphism $\eFl_w/U \cong U \backslash U_{ws} \dot{ws} \times U_s \times \dot{s} T$, we can compute the restriction 
    \[ j_w^* \scrK_{w,ws}^{\scrL} \cong \uk_{U_{ws}} \boxtimes \scrL^s\vert_{U_s \times \dot{s} T} \cong \uk_{U_{ws}} \boxtimes \uk_{U_s} \boxtimes \scrL \cong \scrK_w^{\scrL},\]
    which completes the stalk computation for $\scrK_{w,ws}^{\scrL}$.
\end{proof}

\begin{lemma}\label{lem:conv_canceling_ICs}
  Let $s \in W$ be a simple reflection. Then there are natural isomorphisms
  \[\nabla_s^{s \scrL} \star \Delta_s^{\scrL} \cong \Delta_e^{\scrL} \cong \Delta_s^{s \scrL} \star \nabla_s^{\scrL}.\]
\end{lemma}
\begin{proof}
  By Remark \ref{rem:usefulness_of_finite_parabolics}, it suffices to consider the sheaves $\nabla_s^{s \scrL}$ and $\Delta_s^{\scrL}$ as objects of $\D{\scrL}{s \scrL} (L_s, \k)$ and $\D{s \scrL}{\scrL}
  (L_s, \k)$, respectively.
  Under this identification, the convolution $\nabla_s^{s \scrL} \star \Delta_s^{\scrL}$ corresponds to the convolution on $U_s \backslash L_s / U_s$.
  The result then follows from \cite[Lemma 8.3.4]{Gou} when $s \notin W_{\scrL}^{\circ}$ and Lemma 8.3.7 of \emph{loc. cit.} when $s \in W_{\scrL}^{\circ}$.
\end{proof}

\begin{corollary}\label{cor:conv_with_ICs_bad_s}
  Let $\scrL \in \Ch^{\circ} (T, \k)$.
  Let $s \in W$ be a simple reflection with $s \notin W_{\scrL}^\circ$.
  \begin{enumerate}
    \item The natural maps $\Delta_s^{\scrL} \to \IC_s^{\scrL} \to \nabla_s^{\scrL}$ are isomorphisms.
    \item The functor
      \[ (-) \star \IC_s^{\scrL} : \D{\scrL'}{s \scrL} \to \D{\scrL'}{\scrL}\]
      is an equivalence of triangulated categories which is $t$-exact for the perverse $t$-structure.
    \item Let $w \in \W{\scrL'}{s \scrL}$, then
      \[ \Delta_w^{s \scrL} \star \IC_s^{\scrL} \cong \Delta_{ws}^{\scrL}, \hspace{1cm} \nabla_w^{s \scrL} \star \IC_s^{\scrL} \cong \nabla_{ws}^{\scrL}, \hspace{1cm} \IC_w^{s \scrL} \star \IC_s^{\scrL}
      \cong \IC_{ws}^{\scrL}.\]
  \end{enumerate}
  Moreover, the obvious analogues of (2) and (3) obtained by convolving on the left by $\IC_s^{\scrL}$ also hold.
\end{corollary}
\begin{proof}
  \emph{(1): } By Remark \ref{rem:usefulness_of_finite_parabolics}, it suffices to prove the statement for $G$ semisimple of rank 1. There is an open-closed distinguished triangle
  \[ \Delta_s^{\scrL} \to \nabla_s^{\scrL} \to j_{e*} j_e^* \nabla_s^{\scrL} \to.\]
  We will show that $j_e^* \nabla_s^{\scrL} \cong 0$, which will show the desired isomorphism.

  Since the forgetful functor $\For_{T \times T, \scrL' \boxtimes \scrL^{-1}}$ is conservative, it also suffices to simply forget the twisted equivariance on
  $\nabla_s^{\scrL}$ so that $\nabla_s^{\scrL}$ is a constructible sheaf in $\eFl_{L_s}$.
  Explicitly, under the isomorphism $U_s \backslash B_s \dot{s} B_s \cong \A^1 \times \dot{s} T$, $\nabla_s^{\scrL}$ is the $*$-extension of $\scrL^{L_s} \coloneq \uk_{\A^1} \boxtimes \scrL$ to $\eFl_{L_s}$.

  Write $Z^0 = Z^0 (L_s)$ for the identity component of the center of $L_s$. By \cite[\S 1.18]{Jan}, there is a central isogeny
  \[\nu : Z^0 \times \SL_2 \to L_s.\]
  This induces a map
  \[\nu' : Z^0 \times \eFl_{\SL_2} \to \eFl_{L_s},\]
  which satisfies $f (Z^0 \times \eFl_{\SL_2, w} ) \subseteq \eFl_{L_s, w}$ for $w=e,s$. As a result, there is a cartesian diagram
  \[
    \begin{tikzcd}
      {Z^0 \times \eFl_{\SL_2, e}} \arrow[d, "\nu_e'"] \arrow[r, "\id_{Z^0} \times j_e", hook] & Z^0 \times \eFl_{\SL_2} \arrow[d, "\nu'"] & {Z^0 \times \eFl_{\SL_2, s}} \arrow[d, "\nu_s'"] \arrow[l, "\id_{Z^0} \times j_s"', hook'] \\
      {\eFl_{L_s, e}} \arrow[r, "j_e"', hook]                                                        & \eFl_{L_s}                                & {\eFl_{L_s, s},} \arrow[l, "j_s", hook']                                                         
      \end{tikzcd}
  \]
  where the vertical arrows are smooth, surjective, and with connected fibers.
  As a result, in order to show that $j_e^* \nabla_s^{\scrL} \cong j_e^* j_{s*} \scrL^{L_s} [2] = 0$, we can instead show that $(\nu_e')^* j_e^* j_{s*} \scrL^{L_s} = 0$.
  From functorality and smooth base change, we have an isomorphism
  \[(\nu_e')^* j_e^* j_{s*} \scrL^{L_s} \cong (\id_{Z^0} \times j_e)^* (\id_{Z^0} \times j_s)_* (\nu_s')^* \scrL^{L_s}.\]

  Recall that $\eFl_{\SL_2} \cong \A^2 \setminus \{0\}$, $\eFl_{\SL_2, s} \cong \A^1 \times \G_m$, and $\eFl_{\SL_2, e} \cong \G_m \times \{0\}$.
  Under the above identifications, there are projection maps $p_s : \eFl_{L_s, s} \to T$ and $p : \eFl_{\SL_2, s} \to \G_m$.
  These projection maps are compatible in the sense that the following diagram commutes
  \[
    \begin{tikzcd}
      {Z^0 \times \eFl_{\SL_2, s}} \arrow[r, "p"] \arrow[d, "\nu_s'"] & \G_m \arrow[d, "\alpha_s^\vee"] \\
      {\eFl_{L_s,s}} \arrow[r, "p_s"]                                 & T.                              
      \end{tikzcd}
  \]
  Since $\scrL^{L_s} \cong p_s^* \scrL$, functorality gives an isomorphism
  \[ (\nu_s')^* \scrL^{L_s} \cong p^* (\alpha_s^\vee)^* \scrL.\]
  From the definition of $s \notin W_{\scrL}^\circ$, we have that $(\alpha_s^\vee)^* \scrL$ must be a non-trivial rank-one local system $\scrL_{\G_m}$ on $\G_m$.
  A simple sheaf functor computation yields isomorphisms
  \[
    (\id_{Z^0} \times j_e)^* (\id_{Z^0} \times j_s)^* (\nu_s')^* \scrL^{L_s} \cong \uk_{Z^0} \boxtimes j_e^* j_{s*} (\uk_{\A^1} \boxtimes \scrL_{\G_m}).
  \]
  As in Remark \ref{rem:good}, the local system $\uk_{\A^1} \boxtimes \scrL_{\G_m}$ on $\A^1 \times \G_m$ is clean with respect to $\A^2 \setminus \{0\}$; therefore, $j_e^* j_{s*} (\uk_{\A^1} \boxtimes \scrL_{\G_m}) = 0$.

  Since the natural map $\Delta_s^{\scrL} \to \nabla_s^{\scrL}$ is an isomorphism, clearly we will also have natural isomorphisms $\Delta_s^{\scrL} \stackrel{\sim}{\to} \IC_s^{\scrL}
  \stackrel{\sim}{\to} \nabla_s^{\scrL}$.

  \emph{(2), (3): } By Lemma \ref{lem:conv_canceling_ICs}, it is clear that $(-) \star \IC_s^{\scrL}$ is an equivalence of triangulated categories.
  Moreover, a combination of Lemma \ref{lem:conv_canceling_ICs} and Lemma \ref{lem:conv_rules_1} shows that $\Delta_w^{s \scrL} \star \IC_s^{\scrL} \cong \Delta_{ws}^{\scrL}$
  and $\nabla_w^{s \scrL} \star \IC_s^{\scrL} \cong \nabla_{ws}^{\scrL}$ for all $w \in \W{\scrL'}{s \scrL}$.
  By naturality of these isomorphisms, we have that $\IC_w^{s \scrL} \star \IC_s^{\scrL} \cong  \IC_{ws}^{\scrL}$.
  As a result, it is clear that $(-) \star \IC_s^{\scrL}$ is a perverse $t$-exact.
\end{proof}

\begin{corollary}\label{cor:biadjoint_of_conv_with_ICs}
  Let $s \in W$ be a simple reflection. Then the functor
  \[(-) \star \IC_s^{\scrL} : \D{\scrL'}{s \scrL} (\k) \to \D{\scrL'}{\scrL} (\k) \]
  has right adjoint also given by $(-) \star \IC_s^{s \scrL}$.
\end{corollary}
\begin{proof}
  If $s \in W_{\scrL'}^{\circ}$, then the statement follows from the proof of \cite[Corollary 3.9]{LY} after replacing \cite[Lemma 3.8]{LY} with Lemma \ref{lem:conv_with_ICs_rewritting}.
  If $s \notin W_{\scrL'}^{\circ}$, then the statement follows from Corollary \ref{cor:conv_with_ICs_bad_s}.
\end{proof}

We summarize all the various convolution rules for monodromic standards and costandards from this section in the following proposition.

\begin{proposition}\label{prop:conv_rules}
  We have natural isomorphisms
  \begin{enumerate}
    \item $\Delta_{xy}^{\scrL} \cong \Delta_{x}^{y (\scrL)} \star \Delta_y^{\scrL}$ if $\ell (xy) = \ell (x) + \ell (y)$;
    \item $\nabla_{xy}^{\scrL} \cong \nabla_x^{y (\scrL)} \star \nabla_y^{\scrL}$ if $\ell (xy) = \ell (x) + \ell (y)$;
    \item $\nabla_{x^{-1}}^{x (\scrL)} \star \Delta_x^{\scrL} \cong \Delta_e^{\scrL} \cong \Delta_{x^{-1}}^{x (\scrL)} \star \nabla_x^{\scrL}$.
  \end{enumerate}
\end{proposition}
\begin{proof}
  Statements (1) and (2) are the same verbatim as in Lemma \ref{lem:conv_rules_1}. Statement (3) follows from a combination of (1), (2), and Lemma
  \ref{lem:conv_canceling_ICs}.
\end{proof}

\subsubsection{Blocks}

For the remainder of the section, we fix a $W$-orbit $\fr{o}$ in $\Ch^{\circ} (T, \k)$.
Let $\scrL,\scrL' \in \fr{o}$.
In general, since $W_{\scrL}^\circ$ need not be equal to $W_{\scrL}$, one has to be careful when extending results to all of $W_{\scrL}$. The general strategy will be to analyze the
blocks and describe how they interact.
The block decomposition of $\W{\scrL'}{\scrL}$ gives rise to a block decomposition of $\D{\scrL'}{\scrL} (\k)$.

\begin{definition}\label{def:blocks_of_mhc}
  Let $\beta \in \uW{\scrL'}{\scrL}$ be a block. Let $\D{\scrL'}{\scrL}^\beta (\k)$ be the full triangulated subcategory of $\D{\scrL'}{\scrL} (\k)$ generated by $\{ \Delta_w^{\scrL} \}_{w \in \beta}$.
  We call $\D{\scrL'}{\scrL}^\beta (\k)$ the \emph{$\beta$-block} of $\D{\scrL'}{\scrL} (\k)$.

  When $\beta$ is the coset $W_{\scrL}^\circ$ in $W_{\scrL}$, we will write $\D{\scrL}{\scrL}^\circ (\k)$. This is called the \emph{neutral block} of $\D{\scrL}{\scrL} (\k)$.
\end{definition}

\begin{proposition}[{\cite[Proposition 4.11]{LY}}]\label{prop:block_decomp_of_mhc}
  Let $\scrL,\scrL' \in \fr{o}$. We have a direct sum decomposition of triangulated categories
  \begin{equation}\label{eq:block_decomp_of_mhc_1}
    \D{\scrL'}{\scrL} (\k) = \bigoplus_{\beta \in \uW{\scrL'}{\scrL}} \D{\scrL'}{\scrL}^\beta (\k).
  \end{equation}
\end{proposition}

\begin{proposition}[{\cite[Proposition 4.13]{LY}}]\label{prop:conv_preserves_blocks}
  Let $\scrL, \scrL', \scrL'' \in \fr{o}$. Let $\beta \in \uW{\scrL'}{\scrL}$ and $\gamma \in \uW{\scrL''}{\scrL'}$. If $\scrG \in \D{\scrL'}{\scrL}^\beta (\k)$ and $\scrF \in \D{\scrL''}{\scrL'}^\gamma (\k)$,
  then $\scrF \star \scrG \in \D{\scrL''}{\scrL}^{\gamma \beta } (\k)$.
  In particular, convolution restricts to a bifunctor
  \[(-) \star (-) : \D{\scrL''}{\scrL'}^\gamma (\k) \times \D{\scrL'}{\scrL}^\beta (\k) \to \D{\scrL''}{\scrL}^{\gamma \beta} (\k).\]
\end{proposition}

\subsubsection{Minimal IC Sheaves}

As we saw, the skyscraper sheaf $\IC_e^{\scrL}$ acts on the right of $\D{\scrL'}{\scrL} (\k)$ as a unit.
We want to construct a version of these sheaves but on each block, rather than just the neutral block.
The definition of the minimal elements in a block $\beta \in \uW{\scrL'}{\scrL}$ suggests that such a replacement should allow us to translate problems in non-neutral blocks to
neutral blocks.
The correct replacement for $\IC_e^{\scrL}$ for non-neutral blocks will be the IC-sheaves corresponding to these minimal elements.

\begin{definition}\label{def:min_IC}
  For $\beta \in \uW{\scrL'}{\scrL}$, the sheaf $\IC_{w^\beta}^{\scrL}$ is called the \emph{minimal IC sheaf} in $\D{\scrL'}{\scrL}^\beta (\k)$.
\end{definition}

We have already seen examples of minimal IC sheaves in Corollary \ref{cor:conv_with_ICs_bad_s}. Namely, if $s \in \W{\scrL'}{\scrL}$ is a simple reflection in $W$ such that $s
\notin W_{\scrL}^\circ$, then $s$ is minimal in the block containing it by Lemma \ref{lem:simple_reflns_are_minimal}.
In particular, $\Delta_s^{\scrL} \cong \IC_s^{\scrL} \cong \nabla_s^{\scrL}$ is a minimal IC sheaf. The following propositions provide a generalization of Corollary
\ref{cor:conv_with_ICs_bad_s} to all blocks. An analogous result is provided by \cite[Proposition 5.2]{LY}, but we provide a slightly different argument.

\begin{proposition}\label{prop:structure_of_min_IC}
  Let $\scrL,\scrL' \in \fr{o}$, and let $\beta \in \uW{\scrL'}{\scrL}$.
  \begin{enumerate}
    \item The natural maps $\Delta_{w^\beta}^{\scrL} \to \IC_{w^\beta}^{\scrL} \to \nabla_{w^\beta}^{\scrL}$ are isomorphisms.
    \item Let $\scrL'' \in \fr{o}$ and $\gamma \in \uW{\scrL''}{\scrL'}$. The functors
      \[\IC_{w^\gamma}^{\scrL'} \star (-) : \D{\scrL'}{\scrL}^{\beta} (\k) \to \D{\scrL''}{\scrL}^{ \gamma \beta } (\k)\]
      and
      \[(-) \star \IC_{w^\beta}^{\scrL} : \D{\scrL''}{\scrL'}^{\gamma} (\k) \to \D{\scrL''}{\scrL}^{\gamma \beta} (\k)\]
      are perverse $t$-exact equivalences of categories.
    \item The equivalence $\IC_{w^\gamma}^{\scrL'} \star (-)$ satisfy
      \[\IC_{w^\gamma}^{\scrL'} \star \Delta_x^{\scrL} \cong \Delta_{w^\gamma x}^{\scrL}, \hspace{0.3cm} \IC_{w^\gamma}^{\scrL'} \star \nabla_x^{\scrL} \cong \nabla_{w^\gamma x}^{\scrL}, \hspace{0.3cm}
      \IC_{w^\gamma}^{\scrL'} \star \IC_x^{\scrL} \cong \IC_{w^\gamma x}^{\scrL}, \]
      for all $x \in \beta$. Analogous statements for the right convolution by $\IC_{w^\beta}^{\scrL}$ also hold.
  \end{enumerate}
\end{proposition}
\begin{proof}
  We will prove all 3 statements by induction on $\ell (w^\beta)$. If $\ell (w^\beta) = 0$, then $w^\beta = e$ and the statements follow from Lemma \ref{lem:skyscraper_is_unit}.
  Suppose the statements hold for $\ell (w^\beta) < n$. Let $\beta \in \uW{\scrL'}{\scrL}$ such that $\ell (w^\beta) = n$.
  We can pick a simple reflection $s \in W$ such that $\ell (w^\beta s) < \ell (w^\beta)$. Note that $\ell (w^\beta s) < \ell (w^\beta)$ forces $s \notin W_{\scrL}^{\circ}$.
  Let $\beta' \in \uW{\scrL'}{s \scrL}$ be the block containing $w^\beta s$.
  If $w^\beta s \neq w^{\beta'}$, then $\ell_{\beta'} (w^{\beta'}) < \ell_{\beta'} (w^\beta s)$. By \ref{eq:endo_vs_bruhat_orders}, we must have that $\ell
  (w^{\beta'}) < \ell (w^\beta s)$.
  As a consequence, we have inequalities
  \[\ell (w^{\beta'} s) \leq \ell (w^{\beta'}) + 1 \leq \ell (w^\beta s) < \ell (w^\beta).\]
  This gives a contradiction since $w^{\beta'} s \in \beta$ and $w^\beta$ is minimal in $\beta$. Therefore, $w^\beta s = w^{\beta'}$.

  By Corollary \ref{cor:conv_with_ICs_bad_s} (3),
  \[\Delta_{w^{\beta'}}^{s \scrL} \star \IC_s^{\scrL} \cong \Delta_{w^{\beta'} s}^{\scrL} \cong \Delta_{w^\beta}^{\scrL}.\]
  Similarly, $\nabla_{w^{\beta'}}^{s \scrL} \star \IC_s^{\scrL} \cong \nabla_{w^\beta}^{\scrL}$.
  Statement (1) then follows from the observation that the natural map $\Delta_{w^{\beta}}^{\scrL} \to \nabla_{w^\beta}^{\scrL}$ factors through isomorphisms
  \[\Delta_{w^{\beta}}^{\scrL} \cong \Delta_{w^{\beta'}}^{s \scrL} \star \IC_s^{\scrL} \cong \nabla_{w^{\beta'}}^{s \scrL} \star \IC_s^{\scrL} \cong \nabla_{w^\beta}^{\scrL},\]
  where the second isomorphism follows from induction.

  Statement (2) follows from (1) and Proposition \ref{prop:conv_rules}. We will just prove (3) for standard sheaves. The proofs for the costandard and IC statements are the same, mutatis mutandis.
  By the inductive hypothesis, $\Delta_w^{\scrL'} \star \IC_{w^{\beta'} }^{s \scrL} \cong \Delta_{ww^{\beta'}}^{s \scrL}$. A short computation yields
  \begin{align*}
    \Delta_w^{\scrL'} \star \IC_{w^\beta}^{\scrL} &\cong \Delta_w^{\scrL'} \star \IC_{w^{\beta'}}^{s \scrL} \star \IC_s^{\scrL} \\
    &\cong \Delta_{ww^{\beta'}}^{s \scrL} \star \IC_s^{\scrL} \\
    &\cong \Delta_{ww^\beta}
  \end{align*}
  where the third isomorphism follows from Corollary \ref{cor:conv_with_ICs_bad_s} (3).
\end{proof}

\subsection{Parity Complexes}

Throughout this section, we assume that $\k$ is a field or a complete local ring.
The goal of this section is to show that parity sheaves exist in the monodromic Hecke category and that convolution preserves parity sheaves.
We have already seen that $\D{\scrL'}{\scrL} (\k)$ satisfies the parity conditions in Proposition \ref{prop:mhc_strat_and_parity}, so we already have that if
parity extensions exist than they are unique up to isomorphism.

Let $\scrL, \scrL' \in \fr{o}$. We will write $\Parity{\scrL'}{\scrL} (\k)$ for the full additive subcategory of $\D{\scrL'}{\scrL} (\k)$ consisting of parity sheaves.
  The category of parity sheaves inherits a block decomposition
  \[\Parity{\scrL'}{\scrL} (\k) = \bigoplus_{\beta \in \uW{\scrL'}{\scrL}} \Parity{\scrL'}{\scrL}^\beta (\k). \]

We can now state the main two results of this section.

\begin{theorem}\label{thm:existence_of_parity}
  Let $\scrL,\scrL' \in \fr{o}$. Let $\beta \in \uW{\scrL'}{\scrL}$.
  For every $w \in \beta$, there is an indecomposable sheaf $\scrE_w^{\scrL} \in \Parity{\scrL'}{\scrL}^\beta (\k)$ satisfying
  \begin{enumerate}
    \item $\scrE_w^{\scrL}$ is supported on $\eFl_{\leq w}$;
    \item $j_w^* \scrE_w^{\scrL} \cong \scrK_w^{\scrL} [\ell (w)]$.
  \end{enumerate}
\end{theorem}

\begin{theorem}\label{thm:conv_preserves_parity}
  Convolution sends parity sheaves to parity sheaves, in particular, convolution restricts to a bifunctor between additive categories,
  \[(-) \star (-) : \Parity{\scrL''}{\scrL'} (\k) \times \Parity{\scrL'}{\scrL} (\k) \to \Parity{\scrL''}{\scrL} (\k).\]
\end{theorem}

These theorems will be proved more-or-less simultaneously. We briefly outline the strategy below.
\begin{enumerate}
  \item Understand how parity changes under convolution by $\IC_s^{\scrL}$ via Proposition \ref{prop:structure_of_min_IC} and Lemma \ref{lem:conv_with_ICs_rewritting}. 
  \item Use convolution to construct parity extensions.
  \item Use Theorem \ref{thm:existence_of_parity} and Proposition \ref{prop:uniqueness_of_parity_sh} to conclude that convolution preserves parity sheaves.
\end{enumerate}

\begin{lemma}\label{lem:conv_with_ICs_bad_s_parity}
  Let $s \in W$ be a simple reflection and $s \notin W_{\scrL}^\circ$.
  The functor $(-) \star \IC_s^{\scrL} : \D{\scrL'}{s \scrL} (\k) \to \D{\scrL'}{\scrL} (\k)$ flips parity, i.e., sends even objects to odd objects (and vice versa).
  Therefore, $(-) \star \IC_s^{s \scrL}$ restricts to a functor,
  \[(-) \star \IC_s^{\scrL} : \Parity{\scrL'}{s \scrL} (\k) \to \Parity{\scrL'}{\scrL} (\k).\]
\end{lemma}
\begin{proof}
  This follows immediately from Corollary \ref{cor:conv_with_ICs_bad_s}.
\end{proof}

Our goal is to establish a similar parity preserving result when $s \in W_{\scrL}^\circ$. To do so, we will use the push-pull lemma (Lemma
\ref{lem:conv_with_ICs_rewritting}) combined with Proposition \ref{prop:pushforward_of_even}.
By Lemma \ref{lem:partial_mhc_is_stratified_and_satisfies_parity_cond}, the partial monodromic Hecke category $\paraD{\scrL'}{\scrL}{s} (\k)$ satisfies the parity conditions.

\begin{lemma}\label{lem:pi_s_and_parity}
  The functors $\pi_{s*}$ and $\pi_s^*$ take parity sheaves to parity sheaves.
\end{lemma}
\begin{proof}
  We will first prove that $\pi_s^*$ preserves parity sheaves. Let $\scrF \in \paraD{\scrL}{\scrL}{s}$ be $*$-parity.
  By definition, $\pi_s^* \scrF$ is clearly $*$-parity. 
  On the other hand, if $\scrG \in \paraD{\scrL'}{\scrL}{s}$ is $!$-parity, then $\DD (\scrG) \in \paraD{(\scrL')^{-1}}{\scrL^{-1}}{s}$ is $*$-parity.
  Since $\pi_s^! \cong \pi_s^* [2]$, it follows that  $\pi_s^! \DD (\scrG) \cong \pi_s^* \DD (\scrG) [2]$ is also $*$-parity.

  We will now prove that $\pi_{s*}$ preserves parity complexes. By Proposition \ref{prop:pushforward_of_even}, it suffices to check that $\pi_s$ is an even morphism.
  This follows from recollement and Lemma \ref{lem:pi_s_pushforward_of_stds}.
\end{proof}

\begin{corollary}\label{cor:conv_with_ICs_good_s}
  Let $s \in W$ be a simple reflection such that $s \in W_{\scrL}^\circ$, the functor $(-) \star \IC_s^{\scrL} : \D{\scrL'}{\scrL} (\k) \to \D{\scrL'}{\scrL} (\k)$ preserves parity sheaves. As a result, there is a functor
  \[ (-) \star \IC_s^{\scrL} : \Parity{\scrL'}{\scrL} (\k) \to \Parity{\scrL'}{\scrL} (\k). \]
\end{corollary}
\begin{proof}
  This follows immediately from Lemma \ref{lem:conv_with_ICs_rewritting} and Lemma \ref{lem:pi_s_and_parity}.
\end{proof}

\begin{midsecproof}{Theorem \ref{thm:existence_of_parity}}
  Let $w \in \beta$. Let $\uw = (s_1, \ldots, s_k)$ be a reduced expression for $w$.
  Define
  \[\scrE_{\uw}^{\scrL} = \IC_{s_1}^{s_2 \ldots s_k (\scrL)} \star \IC_{s_2}^{s_3 \ldots s_k (\scrL)} \star \ldots \star \IC_{s_k}^{\scrL}.\]
  In particular, $\scrE_{\uw}^{\scrL} \in \D{\scrL'}{\scrL}^\beta$.
  By Lemma \ref{lem:conv_with_ICs_bad_s_parity} and Corollary \ref{cor:conv_with_ICs_good_s}, we have that $\scrE_{\uw}^{\scrL}$ is a
  parity sheaf.
  Moreover, $\scrE_{\uw}^{\scrL}$ is supported on $\eFl_{\leq w}$.

  We claim that $j_w^* \scrE_{\uw}^{\scrL} \cong \scrK_w^{\scrL} [\ell (w)]$. We will prove this by induction on $k$. The case of $k=1$ is obvious.
  Let $\ux = (s_1, \ldots, s_{k-1})$ and write $x = x s_k$. Assume that $j_x^* \scrE_{\ux}^{s_k \scrL} \cong \scrK_x^{s_k\scrL} [\ell (x)]$.
  If $s_k \notin W_{\scrL}^{\circ}$, we are done by Proposition \ref{prop:structure_of_min_IC}. As a result, we may assume that $s_k \in W_{\scrL}^{\circ}$.
  By induction, we have a distinguished triangle
  \begin{equation}\label{eq:existence_of_parity_1}
    \Delta_x^{\scrL} \to \scrE_{\ux}^{\scrL} \to j_{< x *} j_{< x}^* \scrE_{\ux}^{\scrL},
  \end{equation}
  where $j_{< x} : \eFl_{< x} \to \eFl$ is the inclusion map.
  Note that $j_{< x *} j_{< x}^* \scrE_{\ux}^{\scrL} \star \IC_{s_k}^{\scrL}$ is supported on $\eFl_{< w}$. 
  As a result, when we apply $j_w^* \left( (-) \star \IC_{s_k}^{\scrL} \right)$ to (\ref{eq:existence_of_parity_1}), we get an isomorphism $j_w^* (\Delta_x^{\scrL} \star \IC_{s_k}^{\scrL}) \stackrel{\sim}{\to} j_w^* \scrE_{\uw}^{\scrL}$.
  By Lemma \ref{lem:conv_ICs_with_stds}, there is an isomorphism $j_w^* (\Delta_x^{\scrL} \star \IC_{s_k}^{\scrL}) \cong \scrK_w^{\scrL} [\ell (w)]$.
  
  The result follows from taking the indecomposable summand $\scrE_w^{\scrL}$ of $\scrE_{\uw}^{\scrL}$ such that $j_w^* \scrE_w^{\scrL} \cong \scrK_w^{\scrL} [\ell (w)]$. 
\end{midsecproof}

\begin{midsecproof}{Theorem \ref{thm:conv_preserves_parity}}
    Let $\beta \in \uW{\scrL'}{\scrL}$ and $\gamma \in \uW{\scrL''}{\scrL'}$.
  By Theorem \ref{thm:existence_of_parity} and Proposition \ref{prop:uniqueness_of_parity_sh}, one has that $\Parity{\scrL'}{\scrL}^\beta (\k)$ is generated under direct sums and
  shifts by $\scrE_w^{\scrL}$ for $w \in \beta$.
  It suffices to check that for $w \in \beta$ and $v \in \gamma$, one has that $\scrE_v^{\scrL'} \star \scrE_w^{\scrL}$ is parity.
  Let $\uw = (s_1, \ldots, s_k)$ be a reduced expression for $w$. The construction of $\scrE_w^{\scrL}$ in the proof of Theorem \ref{thm:existence_of_parity} gives that
   $\scrE_v^{\scrL'} \star \scrE_w^{\scrL}$ is a direct summand of a shift of $\scrE_v^{\scrL'} \star \IC_{s_1}^{s_2 \ldots s_k (\scrL)} \star \ldots \star \IC_{s_k}^{\scrL}$.
  The latter is a parity sheaf by Lemma \ref{lem:conv_with_ICs_bad_s_parity} and Corollary \ref{cor:conv_with_ICs_good_s}. Since any direct summand of a parity sheaf is
  also a parity sheaf, we obtain that $\scrE_v^{\scrL'} \star \scrE_w^{\scrL}$ is parity.
\end{midsecproof}

\subsection{Right Equivariant Hecke Category}

Recall that we have fixed a $W$-orbit $\fr{o}$ in $\Ch^{\circ} (T, \k)$.
Let $\scrL \in \fr{o}$.
We call $\DE{\scrL} (\k) \coloneq D_{\cons} (\eFl /_{\scrL} B, \k)$ the \emph{Hecke category with right equivariant monodromy $\scrL$}.
For $w \in W$, we write $\DE{\scrL} (w, \k) = D_{\cons} (\eFl_w /_{\scrL} B, \k)$.
As discussed in \cite[\S3.2]{Gou}, there is a direct sum decomposition of triangulated categories
\[\DE{\scrL} (\k) \cong \bigoplus_{\scrL' \in \fr{o}} D_{\cons} (T \fatbslash_{\scrL'} \UGU /_{\scrL} T, \k),\]
where $D_{\cons} (T \fatbslash_{\scrL'} \UGU /_{\scrL} T, \k)$ are the full subcategories of $\DE{\scrL} (\k)$ generated by the essential image of the forgetful functors
\[\ForME{\scrL'} : \D{\scrL'}{\scrL} (\k) \to \DE{\scrL} (\k).\]

\begin{lemma}\label{lem:right_equiv_mon_hecke_is_parity}
  The Hecke category with right equivariant monodromy $\scrL$ satisfies the following properties:
  \begin{enumerate}
    \item There are equivalences of categories
      \[\DE{\scrL} (w, \k) \cong D_{\cons}(T /_{\scrL} T, \k) \cong D^b(\mod{\k}^{\fg})\]
      obtained by taking the stalk at $\dot{w}$ for a lift $\dot{w} \in G$ of $w$.
    \item The category $\DE{\scrL} (\k)$ satisfies the parity conditions.
  \end{enumerate}
\end{lemma}
\begin{proof}
  Since $\eFl_w$ is $T\times T$-equivariantly isomorphic to $\A^{\ell (w)} \times \dot{w} T$, the first isomorphism of  (1) easily follows.
  The second isomorphism is standard, for example see the proof of \cite[Lemma 7.4.1]{Gou}. Statement (2) is routine, and is similar to that of
  Proposition \ref{prop:mhc_strat_and_parity} (2).
\end{proof}

Routine modifications in the definition of the biequivariant convolution functor produces a right action of $\D{\scrL}{\scrL} (\k)$ on $\DE{\scrL} (\k)$. More generally, there are bifunctors
\[(-) \star (-) : \DE{\scrL'} (\k) \times \D{\scrL'}{\scrL} (\k) \to \DE{\scrL} (\k), \]
which are suitably associative.

The following shows how the biequivariant convolution and the above action is intertwined by the forgetful functor. Its proof follows readily from definitions.
\begin{lemma}\label{lem:top_mon_and_conv}
  Let $\scrL,\scrL',\scrL'' \in \fr{o}$.
  For all $\scrF \in \D{\scrL''}{\scrL'} (\k)$ and $\scrG \in \D{\scrL'}{\scrL} (\k)$, there is a canonical isomorphism
  \[\ForME{\scrL''} (\scrF \star \scrG) \cong \ForME{\scrL''} (\scrF) \star \scrG.\]
\end{lemma}

Let $\k$ be a complete local ring or a field.
By Lemma \ref{lem:right_equiv_mon_hecke_is_parity}, $\DE{\scrL} (\k)$ satisfies the parity conditions.
We will write $\PE{\scrL} (\k)$ for the category of parity sheaves in $\DE{\scrL} (\k)$.
The forgetful functor $\ForME{\scrL'} : \D{\scrL'}{\scrL} (\k) \to \DE{\scrL} (\k)$ takes parity sheaves to parity sheaves.
The following lemma then follows from Theorems \ref{thm:existence_of_parity} and \ref{thm:conv_preserves_parity}.

\begin{lemma}
  Suppose that $\k$ is a complete local ring or a field.
    For every $w \in W$, there is a unique indecomposable parity extension $\scrE_w^{\scrL}$ of $\scrK_w^{\scrL} [\ell (w)]$ in $\PE{\scrL} (\k)$.
    Moreover, convolution restricts to a bifunctor of parity sheaves
\[(-) \star (-) : \PE{\scrL'} (\k) \times \Parity{\scrL'}{\scrL} (\k) \to \PE{\scrL} (\k).\]
\end{lemma}

Let $\beta \in \uW{\scrL'}{\scrL}$ be a block.
We will also consider the full subcategory $\PME{\scrL'}{\scrL}^{\beta} (\k)$ of $\PE{\scrL} (\k)$ which is generated under shifts and direct sums by $\scrE_w^{\scrL}$ for $w \in \beta$.
When $\beta$ is the neutral block, we will write $\PME{\scrL}{\scrL}^{\circ} (\k) \coloneq \PME{\scrL}{\scrL}^{\beta} (\k)$.

\subsection{Monodromic Bott--Samelson Hecke Category}

The category of sheaves generated by $\IC_s^{\scrL}$ for $s \in S$ under convolution in the monodromic Hecke category are still worthwhile to study even when $\k$ is not a field or a complete local ring.
The main obstacle that presents itself is the fact that $\D{\scrL'}{\scrL} (\k)$ need not be idempotent complete, so parity extensions may not exist.
To worry less about whether idempotents exist when $\k$ is not a complete local ring, we introduce the monodromic Bott--Samelson category.

Let $\uw = (s_1, \ldots, s_k)$ be an expression of $w \in W$. We define the \emph{Bott--Samelson sheaf},
\[\scrE_{\uw}^{\scrL} \coloneq \IC_{s_1}^{s_{2} \ldots s_k \scrL} \star \ldots \star \IC_{s_k}^{\scrL},\]
which is regarded as either an object of $\D{w \scrL}{\scrL} (\k)$ or $\DE{\scrL} (\k)$.
Note that when $w \in \beta$ for some block $\beta \in \uW{w\scrL}{\scrL}$, then $\scrE_{\uw}^{\scrL} \in \D{w \scrL}{\scrL}^{\beta} (\k)$.

\begin{definition}
  Let $\scrL,\scrL' \in \fr{o}$.
  Let $\Parity{\scrL'}{\scrL}^{\BS} (\k)$ denote the full replete subcategory of $\D{\scrL'}{\scrL} (\k)$ generated by objects of the form $\scrE_{\uw}^{\scrL} [n]$ where $\uw$ is an expression for $w \in \W{\scrL'}{\scrL}$ and $n \in \Z$.
  We call $\Parity{\scrL'}{\scrL}^{\BS} (\k)$ the \emph{monodromic Bott--Samelson category}.

  Let $\beta \in \uW{\scrL'}{\scrL}$ be a block. We can define a full replete subcategory $\Parity{\scrL'}{\scrL}^{\BS, \beta} (\k)$ of $\Parity{\scrL'}{\scrL}^{\BS} (\k)$ generated by objects of the form $\scrE_{\uw}^{\scrL} [n]$ where $\uw$ is an expression for $w \in \beta$ and $n \in \Z$.
\end{definition}

The Bott--Samelson categories inherit a convolution bifunctor
\[(-) \star (-) : \Parity{\scrL''}{\scrL'}^{\BS, \gamma} (\k) \times \Parity{\scrL'}{\scrL}^{\BS, \beta} (\k) \to \Parity{\scrL''}{\scrL}^{\BS, \gamma \beta} (\k),\]
\[\scrE_{\uw}^{\scrL'} [n] \star \scrE_{\ux}^{\scrL} [m] = \scrE_{\uw \ux}^{\scrL} [n+m].\]

Let $x,y \in \W{\scrL'}{\scrL}$. Let $\ux$ and $\uy$ be expressions for $x$ and $y$.
We will write 
\[\Hom_{\Parity{\scrL'}{\scrL}^{\BS} (\k)}^\bullet (\scrE_{\ux}^{\scrL} [n], \scrE_{\uy}^{\scrL} [m] ) \coloneq \bigoplus_{k \in \Z} \Hom_{\Parity{\scrL'}{\scrL}^{\BS} (\k)} (\scrE_{\ux}^{\scrL} [n], \scrE_{\uy}^{\scrL} [m+k]).\]

\begin{lemma}
    If $\k$ is a field or a complete local ring, then the Karoubian envelope of the additive hull of $\Parity{\scrL'}{\scrL}^{\BS, \beta} (\k)$ is equivalent to $\Parity{\scrL'}{\scrL}^{\beta} (\k)$.
\end{lemma}
\begin{proof}
  It follows from Theorem \ref{thm:existence_of_parity} that $\Parity{\scrL'}{\scrL}^{\beta} (\k)$ is a full subcategory of the Karoubian envelope of the additive hull of $\Parity{\scrL'}{\scrL}^{\BS, \beta} (\k)$.
  By Proposition \ref{prop:block_decomp_of_mhc}, any summand of the idempotent completion of $\Parity{\scrL'}{\scrL}^{\BS, \beta} (\k)$ is also in the $\beta$-block. The lemma readily follows.
\end{proof}

We can now describe how extension of scalars behaves with respect to the Bott--Samelson category. The following lemma follows from standard properties for parity sheaves (cf., \cite[Lemma 2.2 (2)]{MauR}).
\begin{lemma}\label{lem:eos_for_biequiv}
  Let $\k \to \k'$ be a ring homomorphism. Let $\ux$ and $\uy$ be expressions of elements in $\W{\scrL'}{\scrL}$.
  \begin{enumerate}
    \item  The graded left $H_T^\bullet (\pt; \k)$-module
    \[\Hom_{\Parity{\scrL'}{\scrL}^{\BS} (\k)}^{\bullet} (\scrE_{\ux}^{\scrL} [n], \scrE_{\uy}^{\scrL} [m])\]
    is free.
    \item Extension of scalars induces an isomorphism of $H_T^{\bullet} (\pt; \k')$-modules,
    \[\k' \otimes_{\k} \Hom_{\Parity{\scrL'}{\scrL}^{\BS} (\k)} (\scrE_{\ux}^{\scrL} [n], \scrE_{\uy}^{\scrL} [m]) \stackrel{\sim}{\to} \Hom_{\Parity{\k' (\scrL')}{\k' (\scrL)}^{\BS} (\k')} (\scrE_{\ux}^{\k' (\scrL )} [n], \scrE_{\uy}^{\k' (\scrL) } [m]).\] 
  \end{enumerate}
\end{lemma}

It will also be useful to define a version of the Bott--Samelson category for the right equivariant category of parity sheaves.
\begin{definition}
  Let $\scrL \in \fr{o}$.
  Let $\PE{\scrL}^{\BS} (\k)$ denote the full replete category of $\DE{\scrL} (\k)$ generated by objects of the form $\scrE_{\uw}^{\scrL}$ where $\uw$ is an expression for $w \in W$ and $n \in \Z$.
  We call $\PE{\scrL}^{\BS} (\k)$ the \emph{right $\scrL$-monodromic Bott--Samelson category}.

  If $\beta \in \uW{\scrL'}{\scrL}$ is a block, we will also consider the full replete subcategory $\PME{\scrL'}{\scrL}^{\BS, \beta} (\k)$ of $\PE{\scrL}^{\BS} (\k)$ generated by the objects $\scrE_{\uw}^{\scrL} [n]$ where $\uw$ is an expression for $w \in \beta$ and $n \in \Z$.
  If $\beta = W_{\scrL}^{\circ}$ is the neutral block, we will write $\PME{\scrL}{\scrL}^{\BS, \beta} (\k) \coloneq \PME{\scrL}{\scrL}^{\BS, \circ}$. 
\end{definition}

Let $x,y \in W$. Let $\ux$ and $\uy$ be expressions for $x$ and $y$.
We will write 
\[\Hom_{\PE{\scrL}^{\BS} (\k)}^\bullet (\scrE_{\ux}^{\scrL} [n], \scrE_{\uy}^{\scrL} [m]) \coloneq \bigoplus_{k \in \Z} \Hom_{\PE{\scrL}^{\BS} (\k)} (\scrE_{\ux}^{\scrL} [n], \scrE_{\uy}^{\scrL} [m+k]).\]

We obtain the following variant of Lemma \ref{lem:eos_for_biequiv}. It also follows from the same argument given in \cite[Lemma 2.2 (2)]{MauR}.
\begin{lemma}\label{lem:eos_for_right_equiv}
  Let $\k \to \k'$ be a ring homomorphism. Let $\ux$ and $\uy$ be expressions of elements in $\W{\scrL'}{\scrL}$.
  \begin{enumerate}
    \item The $\k$-module 
    \[\Hom_{\PE{\scrL}^{\BS} (\k)} (\scrE_{\ux}^{\scrL} [n], \scrE_{\uy}^{\scrL} [m])\]
    is free.
    \item Extension of scalars induces an isomorphism of $\k'$-modules
    \[ \Hom_{\PE{\scrL}^{\BS} (\k)} (\scrE_{\ux}^{\scrL} [n], \scrE_{\uy}^{\scrL} [m]) \stackrel{\sim}{\to} \Hom_{\PE{ \k' (\scrL)}^{\BS} (\k')} (\scrE_{\ux}^{\k' (\scrL)} [n], \scrE_{\uy}^{\k' (\scrL)} [m]).\]
    \item The forgetful functor $\ForME{\scrL'} : \D{\scrL'}{\scrL} (\k) \to \DE{\scrL} (\k)$ induces an isomorphism of graded $\k$-modules
    \[ \k \otimes_{H_T^\bullet (\pt; \k)}\Hom_{\Parity{\scrL'}{\scrL}^{\BS} (\k)} (\scrE_{\ux}^{\scrL} [n], \scrE_{\uy}^{\scrL} [m]) \stackrel{\sim}{\to} \Hom_{\PE{\scrL}^{\BS} (\k)}^{\bullet} (\scrE_{\ux}^{\scrL} [n], \scrE_{\uy}^{\scrL} [m]).\]
  \end{enumerate}
\end{lemma}

\subsection{Bicategorical Variants}

Up until this point, we have mostly concerned ourselves with the 1-categorical and monoidal structure of the monodromic Hecke category.
It is useful to reformulate this in terms of bicategorical language. 
While the monodromic Bott--Samelson categories can be assembled into a 2-category, the behavior of blocks is somewhat opaque.
In particular, the non-neutral block monodromic-endoscopic equivalence (Theorem \ref{thm:endoscopy_Kac_moody}) is stated for parity sheaves rather than Bott--Samelson sheaves.
As a result, for this section, we will impose the constraint that $\k$ is a complete local ring or a field.
We will explain the necessary foundations from category theory before explaining how to construct bicategorical versions of the monodromic Hecke category.

\subsubsection{Bicategorical Preliminaries}

Let $\fr{C}$ denote a (weak) 2-category.\footnote{Classically, the notion of a 2-category is synonymous with that of a \emph{strict} 2-category. Our notion of 2-category has classically been called a bicategory and does not require that horizontal composition be strictly associative. cf., \cite[\S10.4]{LY} for further discussion on this topic.} Let $x,y \in \Ob (\fr{C})$, we will write ${}_y \fr{C}_x$ for the morphism category $\Hom_{\fr{C}} (x,y)$.
We can associate an ordinary 1-category from $\fr{C}$, denoted $\pi_{\leq 1} \fr{C}$. It has the same object set as $\fr{C}$, but the morphism sets $\Hom_{\pi_{\leq 1} \fr{C}} (x,y)$ are given by the
isomorphism classes of objects in ${}_y \fr{C}_x$.
Let $\Gamma$ be a small groupoid. A \emph{2-category over $\Gamma$} is a 2-category $\fr{C}$ along with a functor $\omega : \pi_{\leq 1} \fr{C} \to \Gamma$.
Note that any small groupoid $\Gamma$ can be regarded as a 2-category over itself.

If $(\fr{C}, \omega)$ is a 2-category over $\Gamma$, and $x,y\in \Ob (\fr{C}), \xi \in {}_{\omega (y)} \Gamma_{\omega (x)}$, we define a full subcategory ${}_y \fr{C}_x^{\xi} $ of $ {}_y \fr{C}_x$
consisting of objects which map to $\xi$ under $\omega$.
It can be easily checked that composition restricts to a bifunctor
\[\circ : {}_z \fr{C}_y^{\eta} \times {}_y \fr{C}_x^{\xi} \to {}_z \fr{C}_x^{\eta \xi},\]
for all $x,y,z \in \Ob (\fr{C})$, $\xi \in {}_{\omega (y)} \Gamma_{\omega (x)}$, and $\eta \in {}_{\omega (z)} \Gamma_{\omega (y)}$. 
When $\xi = \id_{\omega (x)} \in {}_{\omega (x)} \Gamma_{\omega (x)}$, we will often denote ${}_x \fr{C}_x^{\circ} \coloneq {}_x \fr{C}_x^{\id_{\omega (x)}}$. 

\subsubsection{Monodromic Hecke 2-Categories}

We will now apply these categorical preliminaries to construct 2-categorical versions of the monodromic Hecke category. 
Recall that we have fixed a $W$-orbit $\fr{o}$ in $\Ch^{\circ} (T, \k)$.
Let $\Xi$ denote the groupoid defined as follows:
\begin{enumerate}
  \item $\Ob (\Xi) = \fr{o}$
  \item $\Hom_\Xi (\scrL,\scrL') = {}_{\scrL'} \Xi_{\scrL} = \{w^\beta \in W \mid \beta \in \uW{\scrL'}{\scrL} \}$.
\end{enumerate}
We can construct a 2-category $\Parity{}{} (\k)$ (resp. $\D{}{} (\k)$) over $\Xi$. Namely, the object set is
given by $\fr{o}$.
The morphisms from $\scrL$ to $\scrL'$ are given by the categories $\Parity{\scrL'}{\scrL} (\k)$ (resp. $\D{\scrL'}{\scrL} (\k)$)
where composition is given by convolution. By Proposition \ref{prop:conv_preserves_blocks} and \ref{thm:conv_preserves_parity}, this data ensures that $\Parity{}{} (\k)$ (resp. $\D{}{} (\k)$) is a well-defined 2-category over $\Xi$.

\subsection{Monodromic Hecke Algebra}\label{subsec:par_mha}

The monodromic Hecke algebra was defined by Lusztig in the study of non-unipotent character sheaves \cite{Lu19, Lu16}.
It serves as an analogue for the Hecke algebra for non-unipotent representations of finite reductive groups.
The semisimple complexes in the monodromic Hecke category give a categorification of the monodromic Hecke algebra. This gives rise to a geometric interpretation for the
canonical basis.

Some complications arise when moving from finite Weyl groups to infinite ones. Namely, the monodromic Hecke algebra is no longer well-defined due to some sums in the defining relations having infinitely many terms.
This is a shadow of the fact that the monodromic Hecke category failing to be monoidal due to its lack of a unit.
As we saw in the previous section, the monodromic Hecke category can be reformulated into a bicategory.
The decategorification of a 2-category is a 1-category, and in this way, one expects that the monodromic Hecke bicategory decategorifies into a 1-categorical enhancement of the monodromic Hecke algebra called the \emph{monodromic Hecke algebroid}.\footnote{This should not be confused by Williamson's Hecke algebroid which is the decategorification of the 2-category of singular Soergel bimodules.}
As foreshadowed by the usual characteristic $p$ story, the correct incarnation of the monodromic Hecke category will be via parity sheaves.

The monodromic Hecke algebroid possesses all the usual features of the Hecke algebra. 
Namely, our categorification provides a theory of $p$-canonical bases and along with it $p$-Kazhdan--Lusztig cells. 
If $G$ is a reductive group, then we will be able to reduce to a 1-categorification of Lusztig's monodromic Hecke algebra. 

\subsubsection{Definitions}

\begin{definition}
  The \emph{monodromic Hecke algebroid} of $W$ with monodromy $\fr{o}$, denoted by $\mathbf{H}^{\fr{o}}$, is a $\Z [v,v^{-1}]$-linear category whose objects are elements of $\fr{o}$.
  We write ${}_{\scrL'} \mathbf{H}_{\scrL}^{\fr{o}}$ for the morphism space $\Hom_{\mathbf{H}^{\fr{o}}} (\scrL, \scrL')$.
  It is defined as the free $\Z[v,v^{-1}]$-module with basis $\{T_w^{\scrL} \}_{w \in \W{\scrL'}{\scrL}}$.
  Composition in $\mathbf{H}^{\fr{o}}$ is defined by the generating relation
  \begin{equation}
    T_s^{x \scrL} T_x^{\scrL} \coloneq \begin{cases} T_{sx}^{\scrL} & sx > x, \\ (v^{2} - 1) T_x^{\scrL} + v^2 T_{sx}^{\scrL}  & sx < x, s \in W_{x\scrL}^{\circ}, \\ v^2 T_{sx}^{\scrL} & sx < x, s \notin W_{x\scrL}^{\circ}.\end{cases}
  \end{equation}
\end{definition}

We have the following inversion formula,
\[(T_s^{\scrL})^{-1} = \begin{cases} v^{-2} T_s^{\scrL} + (v^{-2} - 1)T_e^{\scrL} & s \in W_{\scrL}^{\circ}, \\ v^{-2} T_s^{s \scrL} & s \notin W_{\scrL}^{\circ}.\end{cases}\]

\subsubsection{Modular Categorification}

Our main result of this section gives a categorification of the monodromic Hecke algebroid in terms of parity sheaves.
For this section, we will assume that $\k$ is a field or a complete local ring.

Let $\fr{C}$ be a $\Z$-linear 2-category. We can define its \emph{split Grothendieck algebroid} $K_{\oplus} (\fr{C})$ as the category whose object set is the same as $\fr{C}$.
The morphisms are given by taking the split Grothendieck group of the morphism categories in $\fr{C}$. 

\begin{theorem}\label{thm:categorification}
 There is an equivalence of categories 
 \[ \ch : K_\oplus (\Parity{}{} (\k)) \to \mathbf{H}^{\fr{o}},\] 
 defined by
  \begin{equation}\label{eq:categorification_1}
    \ch ([\scrF]) = \sum_{\stackrel{w \in \W{\scrL'}{\scrL}}{i \in \Z}} (\textnormal{rank} H^i (j_w^* \scrF)) v^i T_w^{\scrL},
  \end{equation}
  for $\scrF \in \Parity{\scrL'}{\scrL} (\k)$. The equivalence satisfies $\ch ([\scrF [1]]) = v^{-1} \ch ([\scrF])$. Moreover, the set $\{ \ch ([\scrE_w^{\scrL}])\}_{w \in \W{\scrL'}{\scrL}}$ forms a $\Z[v,v^{-1}]$-basis for ${}_{\scrL'} \mathbf{H}_{\scrL}^{\fr{o}}$.
\end{theorem}
\begin{proof}
  Almost all the work in proving the theorem is in establishing that $\ch$ is actually a functor. We will delay the proof of this fact until the end.
  It will be useful to also consider $\ch : K_{\oplus} (\D{}{} (\k)) \to \mathbf{H}^{\fr{o}}$ (not necessarily a functor) defined by the same formula as in (\ref{eq:categorification_1}).

  Suppose $\ch$ is a (not necessarily $\Z [v,v^{-1}]$-linear) functor. Clearly, $\ch$ is essentially surjective. 
  At the level of morphisms, $\ch$ is a well-defined homomorphism of abelian groups which satisfies $\ch ([\scrF[1]]) =
  v^{-1} \ch (\scrF)$. As a result, $\ch$ is a homomorphism of $\Z[v,v^{-1}]$-modules where $v^n [ \scrF ] = [\scrF [-n]]$.
  In particular, $\ch$ is $\Z [v,v^{-1}]$-linear.
  By Proposition \ref{thm:existence_of_parity}, we have that
  \[\ch ([\scrE_w^{\scrL}]) \in v^{-\ell (w)} T_w^{\scrL} + \sum_{\stackrel{x \in \W{w \scrL}{\scrL}}{x < w}} \Z[v,v^{-1}] T_x^{\scrL}.\]
  Since $\{ T_x^{\scrL} \}_{w \in {}_{\scrL'} W_{\scrL}}$ is a $\Z[v,v^{-1}]$-basis for ${}_{\scrL'} \mathbf{H}_{\scrL}^{\fr{o}}$, by the above equation we have that $\{\ch ([\scrE_w^{\scrL}])\}_{w \in \W{\scrL'}{\scrL}}$ is a $\Z[v,v^{-1}]$-basis for ${}_{\scrL'} \mathbf{H}_{\scrL}^{\fr{o}}$.
  Therefore, $\ch$ is fully faithful and an equivalence of categories.

  We will now show that $\ch$ is actually a functor.

  \emph{Step 1. Let $\scrF' \to \scrF \to \scrF'' \to$ be a distinguished triangle of $*$-even (resp. $*$-odd) complexes. Then $\ch ([\scrF]) = \ch ([\scrF']) + \ch ([\scrF''])$.}
  Since all three terms are $*$-even (resp. $*$-odd), after applying $H^i (j_w^* (-))$ to the distinguished triangle, we get short exact sequences
  \[0 \to H^i (j_w^* \scrF') \to H^i (j_w^* \scrF) \to H^i (j_w^* \scrF'') \to 0.\]
  The claim follows from the rank being additive on short exact sequences.

  \emph{Step 2. Let $\scrF \in \D{\scrL'}{s \scrL}$ be a $*$-even object. Then we have $\ch ([\scrF \star \IC_s^{\scrL}]) = \ch ([\scrF]) \ch ([\IC_s^{\scrL}])$ for $s \in W$ a
  simple reflection}.
  By Step 1 and Proposition \ref{prop:mhc_strat_and_parity}, it suffices to take $\scrF = j_{w!} \scrK_w^{s \scrL} [2n]$ for $w \in \W{\scrL'}{s \scrL}$ and $n \in \Z$.

  Let $k = \ell (ws) - \ell (w)$. If $s \notin W_{\scrL}^\circ$, then by Proposition \ref{prop:structure_of_min_IC},
  \[\ch ([j_{w!} \scrK_w^{s \scrL} [2n]] \star \IC_s^{\scrL}) = \ch ([j_{ws!} \scrK_{ws}^{\scrL} [2n+k]]) = v^{-2n-k} T_{ws}^{\scrL}.\]
  On the other hand, directly from the definitions we see that
  \[\ch ([j_{w!} \scrK_w^{s \scrL} [2n]]) = v^{-2n} T_w^{s \scrL} \hspace{0.5cm} \textnormal{and} \hspace{0.5cm} \ch([\IC_s^{\scrL}]) = v^{-1} T_s^{\scrL}.\]
  Using the defining relations for the monodromic Hecke algebroid, we get equalities
  \begin{equation*}
    \ch ([j_{w!} \scrK_w^{s \scrL} [2n]]) \ch([\IC_s^{\scrL}]) = \begin{cases} v^{-2n-1} T_{ws}^{\scrL} & k = 1, \\ v^{-2n+1} T_{ws}^{\scrL} & k = -1. \end{cases}
  \end{equation*}

  Suppose $s \in W_{\scrL}^\circ$. By Lemma \ref{lem:conv_ICs_with_stds}, we have that
  \[\ch ([j_{w!} \scrK_w^{\scrL} [2n]] \star \IC_s^{\scrL}) = v^{-2n-k} (T_w^{\scrL} + T_{ws}^{\scrL}).\]
  Using the defining relations in the monodromic Hecke algebroid, we get equalities
  \begin{equation*}
    \ch ([j_{w!} \scrK_w^{\scrL} [2n]]) \ch([\IC_s^{\scrL}]) = \begin{cases}  v^{-2n-1} (T_{ws}^{\scrL} + T_w^{\scrL})  & k =1, \\  v^{-2n+1} (T_{ws}^{\scrL} + T_w^{\scrL})& k = -1. \end{cases}
  \end{equation*}
  Therefore, for all simple reflections $s$, we have that
  \[\ch ([j_{w!} \scrK_w^{s \scrL} [2n]]) \ch([\IC_s^{\scrL}]) = \ch ([j_{w!} \scrK_w^{s \scrL} [2n] \star \IC_s^{\scrL}]),\]
  which completes step 2.

  \emph{Step 3. Let $\scrF \in \Parity{\scrL''}{\scrL'} (\k)$ and $\scrF' \in \Parity{\scrL'}{\scrL} (\k)$, we have that $\ch ([\scrF' \star \scrF]) = \ch ([\scrF']) \ch ([\scrF]).$ Therefore, $\ch$ is a functor.}
  Since $\scrF'$ is a direct sum of a $*$-even object and $*$-object, without loss of generality, we may assume that $\scrF'$ is itself $*$-even (the $*$-odd case can be
  recovered from $\Z [v,v^{-1}]$-linearity).
  It is also enough to prove the claim for $\scrF = \scrE_w^{\scrL}$. We will induct on $\ell (w)$.
  If $\ell (w) = 0$, i.e., $w = e$, then since $\ch ([\IC_e^{\scrL}]) = T_e^{\scrL}$ the claim is obvious.
  If $\ell (w) = 1$, i.e., $w = s$ is a simple reflection, the claim follows from Step 2.

  Suppose $\ell (w) > 1$. Pick a reduced expression $w = s_1 \ldots s_k$. By Theorem \ref{thm:existence_of_parity}, we have a decomposition
  \[\IC_{s_1}^{s_2 \ldots s_k \scrL} \star \ldots \star \IC_{s_k}^{\scrL} = \IC_w^{\scrL} \oplus \scrF_{< w},\]
  where $\scrF_{< w}$ is a direct sum of shifts of $\scrE_x^{\scrL}$ for $x < w$ and $x \in \W{\scrL'}{\scrL}$.
  By induction, we have
  \[\ch ([\scrF' \star \scrF_{< w}]) = \ch ([\scrF']) \ch (\scrF_{< w}).\]
  On the other hand, by iterating Step 2, we have that
  \[\ch ([\scrF' \star \IC_{s_1}^{s_2 \ldots s_k \scrL} \star \ldots \star \IC_{s_k}^{\scrL}]) = \ch ([\scrF']) \ch ([\IC_{s_1}^{s_2 \ldots s_k \scrL}]) \ldots \ch ([\IC_{s_k}^{\scrL}]).\]
  As a result,
  \begin{align*}
    \ch ([\scrF' \star \scrE_w^{\scrL}]) &= \ch ([\scrF' \star \IC_{s_1}^{s_2 \ldots s_k \scrL} \star \ldots \star \IC_{s_k}^{\scrL}]) - \ch ([\scrF' \star \scrF_{< w}]) \\
    &= \ch ([\scrF']) \ch ([\IC_{s_1}^{s_2 \ldots s_k \scrL} \star \ldots \star \IC_{s_k}^{\scrL}]) - \ch ([\scrF']) \ch([\scrF_{< w}]) \\
    &= \ch ([\scrF']) (\ch ([\IC_{s_1}^{s_2 \ldots s_k \scrL} \star \ldots \star \IC_{s_k}^{\scrL}]) - \ch([\scrF_{< w}])) \\
    &= \ch ([\scrF']) \ch ([\scrE_w^{\scrL}]).
  \end{align*}
  This completes the proof of the theorem.
\end{proof}

Let $p$ be 0 or a prime number, and let $\k$ be either a field of characteristic $p$ or a complete noetherian local ring of finite global dimension and residue characteristic $p$.
Let
\[{}^p \uH_w^{\scrL} \coloneq \ch ([\scrE_w^{\scrL}]) \in {}_{w \scrL} \mathbf{H}_{\scrL}^{\fr{o}}.\]
By Theorem \ref{thm:categorification}, the set $\{ {}^p \uH_w^{\scrL} \}_{w \in \W{\scrL'}{\scrL}}$ is a basis for ${}_{\scrL'} \mathbf{H}_{\scrL}^{\fr{o}}$.
We call this basis the \emph{monodromic $p$-Kazhdan--Lusztig basis}. After writing in terms of the standard basis, we obtain expressions
\[{}^p \uH_w^{\scrL} = v^{-\ell (w)} \sum_{x \in \W{\scrL'}{\scrL}} {}^p P_{x,w}^{\scrL} T_w^{\scrL}.\]
The coefficients ${}^p P_{x,w}^{\scrL} \in \Z[v,v^{-1}]$ are called the \emph{monodromic $p$-Kazhdan--Lusztig polynomials}. 
If $p = 0$, we will write $\underline{H}_w^{\scrL} := {}^0 \underline{H}_w^{\scrL}$ and $P_{x,w}^{\scrL} := {}^0 P_{x,w}^{\scrL}$.

\begin{remark}
  Other notions from $p$-Kazhdan--Lusztig theory should have obvious analogues in the monodromic setting. Most notably, there is a theory of $p$-Kazhdan--Lusztig cells.
  Just as Theorem \ref{thm:endoscopy_Kac_moody} dictates how monodromic $p$-Kazhdan--Lusztig polynomials are computed from non-monodromic $p$-Kazhdan--Lusztig polynomials,
  it also gives a means of computing $p$-cells in the monodromic setting.
\end{remark}

\subsubsection{Reduction for Reductive Groups}

We will now explain how when $G$ is a reductive group or more generally if $\fr{o}$ is a finite $W$-orbit, then Theorem \ref{thm:categorification} can be simplified to a categorification of Lusztig's monodromic Hecke algebra.

\begin{definition}
  Let $\fr{o}$ be a finite order $W$-orbit in $\Ch^{\circ} (T, \k)$.
  The \emph{monodromic Hecke algebra} of $W$ with monodromy $\fr{o}$, denoted by $\scrH^{\fr{o}}$, is a unital associative $\Z[v,v^{-1}]$-algebra defined as follows.
  The generators of $\scrH^{\fr{o}}$ are $T_w (w \in W)$ and $1_{\scrL}$ ($\scrL \in \fr{o}$). The multiplication is defined by the relations
  \begin{align}
    1_{\scrL} 1_{\scrL'} &= \delta_{\scrL, \scrL'} 1_{\scrL}, & \text{for } \scrL, \scrL' \in \fr{o}; \\
    T_x T_y &= T_{xy}, & \text{if } x,y \in W \text{ and } \ell (xy) = \ell (x) + \ell (y); \\
    T_x 1_{\scrL} &= 1_{x\scrL} T_x, & \text{for } x \in W, \scrL \in \fr{o}; \\
    T_s^2 &= v^2 T_e + (v^2 - 1) \sum_{\stackrel{\scrL \in \fr{o} \textnormal{ s.t.}}{s \in W_{\scrL}^\circ}} T_s 1_{\scrL}, & \text{ for simple reflections } s \in W; \\
    T_e &= 1 = \sum_{\scrL \in \fr{o}} 1_{\scrL}.  &
  \end{align}
\end{definition}

\begin{notation}
  The $\fr{o}$-monodromic Hecke 1-category is defined as the direct sum of additive categories
  \[\OneParity{}{} (\k)= \bigoplus_{\scrL, \scrL' \in \fr{o}} \Parity{\scrL'}{\scrL} (\k).\]
  By Theorem \ref{thm:conv_preserves_parity}, $\OneParity{}{}$ is a monoidal category.
\end{notation}

As a $\Z[v,v^{-1}]$-module, $\scrH^{\fr{o}}$ is a free module with basis $\{ T_w 1_{\scrL} \mid (w, \scrL) \in W \times \fr{o} \}$.
The monodromic Hecke algebra has an involution $\bar{\:} : \scrH^{\fr{o}} \to \scrH^{\fr{o}}$ defined by $\overline{v} = v^{-1}$, $\overline{T_w} =
T_{w^{-1}}^{-1}$, and $\overline{1_{\scrL}} = 1_{\scrL}$ for all $w \in W$ and $\scrL \in \fr{o}$.

Since $\fr{o}$ is finite, we have that $\bigoplus_{\scrL, \scrL' \in \fr{o}} \Hom_{\mathbf{H}^{\fr{o}}} (\scrL, \scrL')$ is a ring with unit given by $\sum_{\scrL \in \fr{o}} T_e^{\scrL}$.
The following lemma follows from definitions of the monodromic Hecke algebra/algebroid.

\begin{lemma}\label{lem:algebroid_vs_algebra}
  The $\Z[v,v^{-1}]$-module map
  \[\varphi : \scrH^{\fr{o}} \to \bigoplus_{\scrL, \scrL' \in \fr{o}} \Hom_{\mathbf{H}^{\fr{o}}} (\scrL, \scrL'),\]
  \[T_w 1_{\scrL} \mapsto T_w^{\scrL} \in \Hom_{\mathbf{H}^{\fr{o}}} (\scrL, w\scrL)\]
  is a ring isomorphism.
\end{lemma}

\begin{corollary}\label{cor:finite_categorification}
  The equivalence of categories 
  \[ \ch : K_\oplus (\Parity{}{} (\k)) \to \mathbf{H}^{\fr{o}}\] 
  induces a ring isomorphism
  \[ \ch : K_{\oplus} (\OneParity{}{} (\k)) \to \scrH^{\fr{o}}.\]
   The set $\{ \ch ([\scrE_w^{\scrL}])\}_{w \in W, \scrL \in \fr{o}}$ forms a $\Z[v,v^{-1}]$-basis for $\scrH^{\fr{o}}$.
\end{corollary}
\begin{proof}
  Since $\fr{o}$ is finite and by Theorem \ref{thm:categorification}, we have an isomorphism of rings
\[K_{\oplus} (\OneParity{}{} (\k)) \cong \bigoplus_{\scrL, \scrL' \in \fr{o}} \Hom_{\mathbf{H}^{\fr{o}}} (\scrL, \scrL').\]
The result, then follows from Lemma \ref{lem:algebroid_vs_algebra}.
\end{proof}

\begin{remark}
  Note that Verdier duality for parity sheaves does not give rise to the bar involution on the monodromic Hecke algebra.
  In particular, if $\scrL \in \fr{o}$, it can be the case that $\scrL^{-1} \notin \fr{o}$.
  Nonetheless, the $p$-canonical basis is self-dual under the bar involution (see for example, \cite[\S3.14]{LY}).
  In order to give a geometric interpretation, consider the map 
  \[\textnormal{inv} :  T \backslash_{\scrL'} \eFl /_{\scrL} B \to T \backslash_{(\scrL')^{-1}} \eFl /_{\scrL^{-1}} B \]
  given by taking the identity on $\eFl$ and $\varphi : T \times B \to T \times B$ defined by $\varphi (a,b) = (a^{-1}, b^{-1})$.
  Define a functor $\overline{\DD} \coloneq \textnormal{inv}^* \DD : (\DEE{\scrL}{\scrL'} (\k))^{\op} \to \DEE{\scrL}{\scrL'} (\k)$.
  It is easy to check that $\overline{\DD} (\IC_s^{\scrL}) \cong \IC_s^{\scrL}$ for all $s \in S$.
  As a consequence, $\overline{\DD} (\scrE_w^{\scrL} ) \cong \scrE_w^{\scrL}$ for all $w \in \W{\scrL'}{\scrL}$.
  It then follows that under Corollary \ref{cor:finite_categorification} that for $\scrE \in \Parity{\scrL'}{\scrL} (\k)$ we have $\ch ([\overline{\DD} (\scrE)]) = \overline{\ch([\scrE])}$.
\end{remark}

\subsubsection{Graded Hom Formulas}

Let $s$ be a simple reflection. Define an element of ${}_{s \scrL} \mathbf{H}_{\scrL}$,
\[\underline{H}_s^{\scrL} = \begin{cases} v^{-1} T_s + v^{-1} T_e & s \in W_{\scrL}^\circ, \\ v^{-1} T_s & s \notin W_{\scrL}^\circ .\end{cases}\]
For an expression $\uw = (s_1, \ldots, s_k)$, we define
\[\underline{H}_{\uw}^{\scrL} = \underline{H}_{s_1}^{s_2 \ldots s_k\scrL} \ldots \underline{H}_{s_k}^{\scrL}.\]
Note that as defined in Theorem \ref{thm:categorification}, $\underline{H}_s^{\scrL} = \ch ([\scrE_s^{\scrL}])$.

We can define the pairing
\[\langle - , - \rangle : {}_{\scrL'} \mathbf{H}_{\scrL} \times {}_{\scrL'} \mathbf{H}_{\scrL} \to \Z[v,v^{-1}]\]
by the relation $\langle T_x, T_y \rangle = \delta_{x,y}$. 
The pairing $\langle - , - \rangle$ is called the \emph{standard form} on ${}_{\scrL'} \mathbf{H}_{\scrL}$. It is a $\Z[v,v^{-1}]$-linear form.
It satisfies a version of biadjunction: for all $A \in {}_{\scrL'} \mathbf{H}_{\scrL}$, $B \in {}_{s (\scrL')} \mathbf{H}_{\scrL}$, and simple reflections $s$, we have 
\begin{equation}\label{eq:biadj_of_std_form}
  \langle \underline{H}_s^{\scrL'} A, B \rangle = \langle A, \underline{H}_s^{s (\scrL')} B \rangle.
\end{equation}

The following is a monodromic analogue of Soergel's Hom formula for parity sheaves \cite[\S3.10]{MauR}.
\begin{proposition}\label{prop:soergel_hom_v1}
  For any expressions $\ux, \uy$ of elements in $\W{\scrL'}{\scrL}$, we have that $\Hom (\scrE_{\ux}^{\scrL}, \scrE_{\uy}^{\scrL} [n])$ is free of finite rank.
  Moreover, the graded rank of these Hom spaces is given by
  \[\sum_{n \in \Z} \textnormal{rank}_{\k} \Hom (\scrE_{\ux}^{\scrL}, \scrE_{\uy}^{\scrL} [n]) v^n = \langle \underline{H}_{\ux}^{\scrL}, \underline{H}_{\uy}^{\scrL} \rangle.\] 
\end{proposition}
\begin{proof}
  It follows from Lemma \ref{lem:eos_for_biequiv}
  that $\Hom^{\bullet} (\scrE_{\ux}^{\scrL}, \scrE_{\uy}^{\scrL})$ is a graded free left $R$-module of finite rank.
  As a result, in order to compute the graded $\k$-rank of the Hom space, it suffices to take $\k$ to be a field.

  By (\ref{eq:biadj_of_std_form}) and Corollary \ref{cor:biadjoint_of_conv_with_ICs}, it suffices to consider $\ux = \emptyset$ and $\uy$ an expression for an element in $\W{\scrL}{\scrL}$.
  A routine argument produces isomorphisms of graded $\k$-modules \cite[\S3.10]{MauR},
  \[\Hom^\bullet ( \scrE_{\emptyset}^{\scrL}, \scrE_{\uy}^{\scrL}) \cong H^\bullet (j_e^! \scrE_{\uy}^{\scrL}).\]
  Moreover, by Verdier self-duality of $\IC$-sheaves, we have that $H^{n} (j_e^! \scrE_{\uy}^{\scrL})$ and \newline $H^{-n} (j_e^* \scrE_{\uy}^{\scrL})$ have the same ranks.
  As a result, by Theorem \ref{thm:categorification}, we have that
  \begin{align*}
    \sum_{n \in \Z} \textnormal{rank}_{\k} \Hom (\scrE_{\emptyset}^{\scrL}, \scrE_{\uy}^{\scrL} [n]) v^n &= \langle T_e^{\scrL}, \ch ([\scrE_{\uy}^{\scrL}]) \rangle \\
    &= \langle T_e^{\scrL}, \underline{H}_{\uy}^{\scrL} \rangle.
  \end{align*}
\end{proof}

\subsubsection{Decategorified Endoscopy}

\begin{definition}
  The \emph{endoscopic Hecke algebroid} of $W$, denoted by $\mathbf{E}^{\fr{o}}$, is a $\Z [v,v^{-1}]$-linear category whose objects are $\fr{o}$.
  Write ${}_{\scrL'} \mathbf{E}_{\scrL}$ for the morphism space $\Hom_{\mathbf{E}^{\fr{o}}} (\scrL, \scrL')$. 
  As a $\Z [v,v^{-1}]$-module it is defined by 
  \[ {}_{\scrL'} \mathbf{E}_{\scrL} = \bigoplus_{\beta \in \uW{\scrL'}{\scrL}} {}_{\scrL'} \mathbf{E}_{\scrL}^{\beta} \]
  in which ${}_{\scrL'} \mathbf{E}_{\scrL}^{\beta} = \scrH (W_{\scrL}^{\circ})$ and $\scrH (W_{\scrL}^{\circ})$ denotes the usual Hecke algebra for the endoscopic Weyl group $W_{\scrL}^{\circ}$.
  Composition is defined via
  \[ {}_{\scrL''} \mathbf{E}_{\scrL'}^{\gamma} \times {}_{\scrL'} \mathbf{E}_{\scrL}^{\beta} \to {}_{\scrL''} \mathbf{E}_{\scrL}^{\gamma \beta},\]
  \[(A, B) \mapsto \beta^{-1} (A) \cdot B,\]
  where $\cdot$ is the product in $\scrH_{\scrL'}^{\circ}$ and $\beta^{-1} : \scrH (W_{\scrL'}^{\circ}) \to \scrH (W_{\scrL}^{\circ})$ is the algebra isomorphism induced by the isomorphism of Coxeter groups $W_{\scrL'}^{\circ} \to W_{\scrL}^{\circ}$ given by conjugation by $w^{\beta, -1}$.
\end{definition}

The endoscopic Hecke algebroid was essentially introduced in \cite[\S1.6]{Lu19} for finite Weyl groups.\footnote{In \emph{loc. cit.}, Lusztig defines an algebra rather than a category. There is really no difference, and Lusztig's definition can be recovered from ours via the same process as Lemma \ref{lem:algebroid_vs_algebra}.} 
In particular, the work in checking that it is well-defined is covered in \emph{loc. cit.}

Note that ${}_{\scrL'} \mathbf{E}_{\scrL}^{\beta}$ is a free $\Z[v,v^{-1}]$-module with basis $\tilde{T}_w^{\scrL}$ for $w \in \beta$. Here $\tilde{T}_w^{\scrL}$ denotes the standard basis element indexed by $w$ in $\scrH (W_{\scrL}^{\circ})$.

We recall the decategorified version of the endoscopic equivalence we will later encounter in Theorem \ref{thm:endoscopy_Kac_moody}. The result is due to Lusztig for monodromic Hecke algebras (i.e., finite Weyl groups).
Lusztig's argument translates without any changes, so we omit it.

\begin{proposition}[{\cite[1.6(a)]{Lu19}}]\label{prop:endoscopic_equiv_for_Hecke_algebras}
  There is a well-defined functor
  \[\theta : \mathbf{E}^{\fr{o}} \to \mathbf{H}^{\fr{o}}\]
  defined by
  \[\tilde{T}_w^{\scrL} \mapsto T_w^{\scrL}.\]
  Moreover, $\theta$ is an isomorphism of categories.
\end{proposition}

\begin{corollary}\label{cor:inner_form_properties}
  Let $\langle - , - \rangle_{\beta} : {}_{\scrL'} \mathbf{E}_{\scrL}^{\beta} \times {}_{\scrL'} \mathbf{E}_{\scrL}^{\beta} \to \Z[v,v^{-1}]$ denote the standard form of $\scrH (W_{\scrL}^{\circ})$.
  Then for all $A, B \in  {}_{\scrL'} \mathbf{E}_{\scrL}^{\beta}$, we have equalities
  \[\langle \theta (A), \theta (B) \rangle = \langle A, B \rangle_{\beta}. \]
\end{corollary}

Let $t$ be a simple reflection in $W_{\scrL}^{\circ}$ which is not necessarily simple in $W$. Define an element of ${}_{\scrL} \mathbf{E}_{\scrL}$,
\[\underline{\tilde{H}}_t^{\scrL} = v^{-1} T_t + v^{-1} T_e.\]
For an expression $\uw = (t_1, \ldots, t_k)$ of an element $ w \in W_{\scrL}^{\circ}$, we define
\[\underline{\tilde{H}}_{\uw}^{\scrL} = \underline{\tilde{H}}_{t_1}^{\scrL} \ldots \underline{\tilde{H}}_{t_k}^{\scrL}.\]
Let $\uw'$ be an expression of $w$ in terms of simple reflections in $W$ obtained by replacing the $t_i$ in $\uw$ with reduced expressions in $W$. 
It then follows from Proposition \ref{prop:endoscopic_equiv_for_Hecke_algebras} that
\[\theta \left(\underline{\tilde{H}}_{\uw}^{\scrL} \right) = \underline{H}_{\uw'}^{\scrL}.\]

\begin{corollary}\label{cor:soergel_hom_v2}
  Let $\ux$ and $\uy$ be expressions of simple reflections in $W_{\scrL}^{\circ}$.
  Let $\ux'$ and $\uy'$ be expressions of simple reflections in $W$ obtained from $\ux$ and $\uy$ by substituting simple reflections in $W_{\scrL}^{\circ}$ with reduced expressions in $W$.
  Then 
  \[\sum_{n \in \Z} \textnormal{rank}_{\k} \Hom (\scrE_{\ux'}^{\scrL}, \scrE_{\uy'}^{\scrL} [n]) v^n = \langle \underline{\tilde{H}}_{\ux}^{\scrL}, \underline{\tilde{H}}_{\uy}^{\scrL} \rangle.\] 
\end{corollary}

\subsection{Mixed Categories}\label{subsec:mixed_hecke_cats_par}

The goal of this section is to construct a mixed version of the monodromic Hecke category.
We will prove various mixed versions of statements from \S\ref{sec:equiv_conv}.
Throughout this section we assume that $\k$ is a field or a complete local ring.

\subsubsection{Biequivariant Category}

Consider the mixed derived category defined in \S\ref{sec:mixed_cats},
\[\D{\scrL'}{\scrL}^m (\k) \coloneq K^b \Parity{\scrL'}{\scrL} (\k).\]
It inherits a convolution product 
\[ \star : \D{\scrL''}{\scrL'}^m (\k) \times \D{\scrL'}{\scrL}^m (\k) \to \D{\scrL''}{\scrL}^m (\k)\] 
inherited from the convolution of parity sheaves (Theorem \ref{thm:conv_preserves_parity}).
For all $w \in \W{\scrL'}{\scrL}$, we can define the \emph{standard} and \emph{costandard} sheaves
\[\underline{\Delta}_w^{\scrL} \coloneq j_{w!} \scrK_w^{\scrL} (\ell (w)) \hspace{0.5cm} \text{and} \hspace{0.5cm} \underline{\nabla}_w^{\scrL} \coloneq j_{w*} \scrK_w^{\scrL} (\ell (w)).\]

When $s \in S$ is a simple reflection, we can give an explicit description of the chain complex for $\underline{\Delta}_s^{\scrL}$ and $\underline{\nabla}_s^{\scrL}$ using the construction given in \cite[Lemma 2.4]{AR2}.
If $s \not\in W_{\scrL}^{\circ}$, there are isomorphisms
\[\underline{\Delta}_s^{\scrL} \cong \scrE_s^{\scrL} \cong \underline{\nabla}_s^{\scrL}.\]
If $s \in W_{\scrL}^{\circ}$, there are isomorphisms
\begin{equation}\label{eq:ch_cplxs_for_Delta_s}
  \underline{\Delta}_s^{\scrL} \cong \left(
\begin{tikzpicture}[baseline=(current bounding box.center)]
\node (A) at (0,-0.6) {$\scrE_s^{\scrL}$};
\node (B) at (0,0.6) {$\scrE_{\emptyset}^{\scrL} (1)$};
\node (C) at (-0.5,-0.6) {$\bullet$};
\draw[->] (A) -- (B);
\end{tikzpicture} \right) \qquad \text{and} \qquad \underline{\nabla}_s^{\scrL} \cong \left(
\begin{tikzpicture}[baseline=(current bounding box.center)]
\node (A) at (0,0.6) {$\scrE_s^{\scrL}$};
\node (B) at (0,-0.6) {$\scrE_{\emptyset}^{\scrL} (-1)$};
\node (C) at (-0.5,0.6) {$\bullet$};
\draw[->] (B) -- (A);
\end{tikzpicture} \right),
\end{equation}
where the bullet marks indicate the terms in cohomological degree 0, and the morphisms are those induced by the adjunction unit/counit morphisms.

\subsubsection{Parabolic Category}

Recall that we have fixed a $W$-orbit in $\Ch^{\circ} (T, \k)$.
Let $\scrL, \scrL' \in \fr{o}$
Let $s \in S$ be a simple reflection such that $s \in W_{\scrL}^{\circ}$.
We can then extend $\scrL$ to a multiplicative local system $\scrL^s \in \Ch (L_s, \k)$ via Lemma \ref{lem:extending_character_sheaves_to_levis}.
Consider the mixed derived category,
\[\paraD{\scrL'}{\scrL}{s}^{m} (\k) \coloneq K^b \left( \paraParity{\scrL'}{\scrL}{s} (\k) \right),\]
where $\paraParity{\scrL'}{\scrL}{s} (\k)$ is the full subcategory of $\paraD{\scrL'}{\scrL}{s} (\k)$ whose objects are parity sheaves.
For all $\overline{w} \in {}_{\scrL'} (W/\langle s \rangle)_{\scrL}$, we can define the \emph{standard} and \emph{costandard} sheaves
\[\underline{\Delta}_{\overline{w}}^{s, \scrL} \coloneq j_{\overline{w}!} \scrK_{\overline{w}}^{s, \scrL} (\ell (w)) \hspace{0.5cm} \text{and} \hspace{0.5cm} \underline{\nabla}_{\overline{w}}^{s, \scrL} \coloneq j_{\overline{w}*} \scrK_{\overline{w}}^{s, \scrL} (\ell (w)).\]

Recall that there is a proper, smooth, and even morphism $\pi_s : \eFl /_{\scrL} B \to \eFl /_{\scrL^s} P_s$ of twisted ind-algebraic stacks. 
This gives to functors (see Lemma \ref{lem:pi_s_and_parity}),
\[\pi_{s}^* : \paraD{\scrL'}{\scrL}{s}^m (\k) \rightleftarrows  \D{\scrL'}{\scrL}^{m} (\k) : \pi_{s*}.\]
Moreover, we have that $\pi_s^*$ has left adjoint $\pi_{s*} (-2)$. 

\subsubsection{Convolution of Standards and Costandards}

\begin{lemma}\label{lem:mixed_block_translation}
  Let $\beta \in \uW{\scrL''}{\scrL'}$ be a block. There is an equivalence of categories
  \[\uDelta_{w^{\beta}}^{\scrL'} \star (-) : \D{\scrL'}{\scrL}^m (\k) \to \D{\scrL''}{\scrL}^m (\k). \]
  Moreover, for all $w \in \W{\scrL'}{\scrL}$, there are isomorphisms
  \[ \uDelta_{w^{\beta}}^{\scrL'} \star \uDelta_w^{\scrL} \cong \uDelta_{w^{\beta} w}^{\scrL} \hspace{0.5cm} \text{and} \hspace{0.5cm} \uDelta_{w^{\beta}}^{\scrL'} \star \unabla_w^{\scrL} \cong \unabla_{w^{\beta} w}^{\scrL}.\]
\end{lemma}
\begin{proof}
  It is obvious from Proposition \ref{prop:structure_of_min_IC} that $\uDelta_{w^{\beta}}^{\scrL'} \star (-)$ is an equivalence of categories.

  We will now prove that $\uDelta_{w^{\beta}}^{\scrL'} \star \uDelta_w^{\scrL} \cong \uDelta_{w^{\beta} w}^{\scrL}$. The costandard version follows from a similar argument.
  Let $\scrF^i$ denote the degree $i$ term in the complex for $\uDelta_w^{\scrL}$.
  By Proposition \ref{prop:structure_of_min_IC}, we have that 
  \[ \uDelta_{w^{\beta}}^{\scrL'} \star \scrE_x^{\scrL} \cong \scrE_{w^{\beta} x}^{\scrL}, \]
  for all $x \in \W{\scrL'}{\scrL}$. As a result, for each $s \in \W{\scrL'}{\scrL}$, by computing on each term of $j_{w^{\beta} x}^* (\uDelta_{w^{\beta}}^{\scrL'} \star \uDelta_w^{\scrL})$, we obtain an isomorphism
  \begin{align*}
    \Hom (\uDelta_{w^{\beta}}^{\scrL'} \star \uDelta_w^{\scrL} , \unabla_{w^{\beta} x}^{\scrL} (i) [n]) &\cong \Hom (\uDelta_w^{\scrL} , \unabla_{x}^{\scrL} (i) [n]) \\
    &\cong \begin{cases} \k & x=w, n=0, i\geq 0 \text{ even}, \\ 0 & \text{otherwise}.\end{cases}
  \end{align*}
  By recollement, this gives an isomorphism $\uDelta_{w^{\beta}}^{\scrL'} \star \uDelta_w^{\scrL} \cong \uDelta_{w^{\beta} w}^{\scrL}$. 
\end{proof}

\begin{lemma}\label{lem:mixed_pis_dt}
  Let $s$ be a simple reflection such that $s \in W_{\scrL}^{\circ}$. Let $w \in W$ such that $ws < w$. 
  \begin{enumerate}
    \item There are isomorphisms 
    \[ \pi_{s*} \uDelta_w^{\scrL} \cong \uDelta_{\overline{w}}^{s, \scrL} (-1) \hspace{0.5cm} \text{and} \hspace{0.5cm} \pi_{s*} \uDelta_{ws}^{\scrL} \cong \uDelta_{\overline{w}}^{s, \scrL}.\]
    \item Let $\eta : \uDelta_w^{\scrL} \to \pi_s^* \uDelta_{\overline{w}}^{s, \scrL} (1)$ and $\epsilon : \pi_s^* \uDelta_{\overline{w}}^{s, \scrL} (1) \to \uDelta_{ws}^{\scrL} (1)$ be the maps induced from the isomorphisms in (1) and the unit and counit maps for $\pi_{s*}$, respectively.
    Then there exists a morphism $f : \uDelta_{ws}^{\scrL} (1) \to \uDelta_w^{\scrL} [1]$ such that
    \[\uDelta_w^{\scrL} \stackrel{\eta}{\to} \pi_s^* \uDelta_{\overline{w}}^{s, \scrL} (1) \stackrel{\epsilon}{\to} \uDelta_{ws}^{\scrL} (1) \stackrel{f}{\to}\]
    is a distinguished triangle.
  \end{enumerate}
\end{lemma}
\begin{proof}
  The first statement follows from the same argument given in Lemma \ref{lem:pi_s_pushforward_of_stds}.

  We will now prove the second statement. 
  Consider $\scrK_{w,ws}^{\scrL} \coloneq \pi_s^* \scrK_{\overline{w}}^{s, \scrL}$. Since $\pi_s^*$ preserves parity, $\scrK_{w,ws}^{\scrL}$ is a parity sheaf.
  Let $j : \eFl_w \to \eFl_{\overline{w}}^s$, $i : \eFl_{ws} \to \eFl_{\overline{w}}^s$, and $h : \eFl_{\overline{w}}^s \to \eFl$ denote the inclusion maps. 
  As discussed in the proof of Lemma \ref{lem:conv_ICs_with_stds}, there are isomorphisms $j^* \scrK_{w,ws}^{\scrL} \cong \scrK_w^{\scrL} (\ell (w))$ and $i^* \scrK_{w,ws}^{\scrL} \cong \scrK_{ws}^{\scrL} (\ell (w))$. 
  Note that here $j^*$ and $i^*$ can be safely interpreted as restrictions in the non-mixed category since they are restrictions onto single strata, and hence, they preserve parity.
  If we apply the functorial distinguished triangle $h_! j_! j^* \to h_! \to h_! i_! i^* \to$ to $\scrK_{w,ws}^{\scrL} (\ell (w))$, we obtain a distinguished triangle
  \begin{equation}\label{eq:mixed_pis_dt}
    \uDelta_w^{\scrL} \to \pi_s^* \uDelta_{\overline{w}}^{\scrL} (1) \to \uDelta_{ws}^{\scrL} (1) \to.
  \end{equation}
  By appropriately rescaling the morphisms in (\ref{eq:mixed_pis_dt}), we obtain the desired distinguished triangle (cf., \cite[Lemma 4.1]{AR2} on why this can be done).
\end{proof}

The following is a mixed analogue of Proposition \ref{prop:conv_rules}.
\begin{proposition}\label{prop:mixed_conv_rules_par}
  We have natural isomorphisms
  \begin{enumerate}
    \item $\underline{\Delta}_{xy}^{\scrL} \cong \underline{\Delta}_{x}^{y \scrL} \star \underline{\Delta}_y^{\scrL}$ if $\ell (xy) = \ell (x) + \ell (y)$;
    \item $\underline{\nabla}_{xy}^{\scrL} \cong \underline{\nabla}_x^{y \scrL} \star \underline{\nabla}_y^{\scrL}$ if $\ell (xy) = \ell (x) + \ell (y)$;
    \item $\underline{\nabla}_{x^{-1}}^{x \scrL} \star \underline{\Delta}_x^{\scrL} \cong \underline{\Delta}_e^{\scrL} \cong \underline{\Delta}_{x^{-1}}^{x \scrL} \star \underline{\nabla}_x^{\scrL}$.
  \end{enumerate}
\end{proposition}
\begin{proof}
  By Verdier duality, we can reduce to proving just (1) and the first isomorphism in (3).
  We can further reduce to the case where $y = s$ is a simple reflection in (1) and $x=s$ a simple reflection in (2). 
  If $s \notin W_{\scrL}^{\circ}$, both isomorphisms follow from Lemma \ref{lem:mixed_block_translation}. 
  As a result, we will assume that $s \in W_{\scrL}^{\circ}$.
  By the description of $\uDelta_s^{\scrL}$ given in (\ref{eq:ch_cplxs_for_Delta_s}), there is a distinguished triangle
  \begin{equation}\label{eq:mixed_conv_rules_1}
    \uDelta_s^{\scrL} \to \scrE_s^{\scrL} \to \uDelta_e^{\scrL} (1).
  \end{equation}

  We can now prove (1).
  Apply $\Delta_x^{\scrL} \star (-)$ to (\ref{eq:mixed_conv_rules_1}) to obtain a distinguished triangle
  \[\uDelta_x^{\scrL} \star \uDelta_s^{\scrL} \to \uDelta_x^{\scrL} \star \scrE_s^{\scrL} \to \uDelta_x^{\scrL} (1).\]
  By Lemma \ref{lem:conv_with_ICs_rewritting} and Lemma \ref{lem:mixed_pis_dt} (1), this distinguished triangle is of the form given in Lemma \ref{lem:mixed_pis_dt} (2).
  Therefore, there is an isomorphism $\uDelta_x^{\scrL} \star \uDelta_s^{\scrL} \cong \uDelta_{xs}^{\scrL}$.

  For (3), we apply $\unabla_s^{\scrL} (-1) \star (-) $ to  (\ref{eq:mixed_conv_rules_1}) to obtain a distinguished triangle
  \begin{equation}\label{eq:mixed_conv_rules_2}
    \unabla_s^{\scrL} \star \uDelta_s^{ \scrL} (-1) \to \unabla_s^{\scrL} \star \scrE_s^{\scrL} (-1) \to \unabla_s^{\scrL}.
  \end{equation}
  By a Verdier dual version of Lemma \ref{eq:mixed_pis_dt} (1) and Lemma \ref{lem:conv_with_ICs_rewritting}, we have an isomorphism $\unabla_s^{\scrL} \star \scrE_s^{\scrL} (-1) \cong \pi_s^* \unabla_{\overline{s}}^{s,\scrL} (1)$.
  It can be readily checked that $\unabla_{\overline{s}}^{s,\scrL} \cong \IC_{\overline{s}}^{s,\scrL}$ since $\eFl_{\overline{s}}^s$ is closed in $\eFl$.
  This observation combined with Lemma \ref{lem:pi_s_descent} yields an isomorphism $\unabla_s^{\scrL} \star \scrE_s^{\scrL} (-1) \cong \scrE_s^{\scrL}$. We can then rewrite (\ref{eq:mixed_conv_rules_2}) as follows:
  \[\unabla_s^{\scrL} \star \uDelta_s^{ \scrL} (-1) \to \scrE_s^{\scrL} \to \unabla_s^{\scrL}.\]
  This triangle is Verdier dual to (\ref{eq:mixed_conv_rules_1}) (after switching $\scrL$ with $\scrL^{-1}$).  
  As a result, we obtain an isomorphism $\uDelta_s^{ \scrL} \star \unabla_s^{\scrL} \cong \uDelta_e^{\scrL}$.
\end{proof}

\begin{corollary}\label{cor:endo_mixed_conv_rules}
  Let $\uw = (s_1, \ldots, s_k)$ be a reduced expression of endosimple reflections for $w \in W_{\scrL}^{\circ}$. Then there are isomorphisms
  \[\uDelta_{w}^{\scrL} \cong \uDelta_{s_1}^{\scrL} \star \ldots \star \uDelta_{s_k}^{\scrL} \hspace{0.5cm} \text{and} \hspace{0.5cm} \unabla_{w}^{\scrL} \cong \unabla_{s_1}^{\scrL} \star \ldots \star \unabla_{s_k}^{\scrL}.\]
\end{corollary}
\begin{proof}
  The statement follows from a routine inductive argument on $\ell_{\scrL} (w)$ using Proposition \ref{prop:mixed_conv_rules_par} and Lemma \ref{lem:mixed_block_translation}.
\end{proof}

\subsubsection{Right Equivariant Category}

Consider the mixed derived category,
\[\DE{\scrL}^m (\k) \coloneq K^b \PE{\scrL} (\k).\]
For all $w \in W$, we can define the \emph{standard} and \emph{costandard} sheaves
\[\underline{\Delta}_w^{\scrL} \coloneq j_{w!} \scrK_w^{\scrL} (\ell (w)) \qquad \text{and} \qquad \underline{\nabla}_w^{\scrL} \coloneq j_{w*} \scrK_w^{\scrL} (\ell (w)).\]
The forgetful functor $\ForME{\scrL'} : \Parity{\scrL'}{\scrL} (\k) \to \PE{\scrL} (\k)$
induces a forgetful functor on mixed sheaves,
\[\ForME{\scrL'} : \D{\scrL'}{\scrL}^m (\k) \to \DE{\scrL}^m (\k).\]
It can be easily checked that $\ForME{\scrL'}$ commutes with $j_{w*}$ and $j_{w!}$ for any $w \in \W{\scrL'}{\scrL}$ (cf., \cite[(3.4)]{AR2}).
As a result, there are natural isomorphisms
\begin{equation}\label{eq:For_and_mixed_std_costds}
  \ForME{\scrL'} (\underline{\Delta}_w^{\scrL}) \cong \underline{\Delta}_w^{\scrL} \qquad \text{and} \qquad \ForME{\scrL'} (\underline{\nabla}_w^{\scrL}) \cong \underline{\nabla}_w^{\scrL},
\end{equation}
for all $w \in \W{\scrL'}{\scrL}$.

        \section{Soergel Theory for Parabolics of Finite Type}\label{sec:parabolic_soergel}

        The main result of this section will give an equivalence between a parabolic version of the neutral block of the monodromic Hecke category of parity sheaves with characteristic 0 coefficients and the category of Soergel bimodules associated to the endoscopic group.
This can be thought of as a simplified version of the more general neutral block endoscopic equivalence (Theorem \ref{thm:endoscopy_neutral_block_Kac_moody}).
We restrict to finite parabolic subgroups of $W_{\scrL}^\circ$ and to coefficients in a field of characteristic 0. These have the following benefits:
\begin{enumerate}
  \item The finite conditions allow us to define the $\mathbb{H}$-functor using maximal IC sheaves, $\mathbb{H}_{\Theta}^{\scrL} : \Parity{\scrL}{\scrL}^{\Theta, \circ} (\k) \to \SBim_{W_{\scrL, \Theta}}$.
  \item The characteristic 0 coefficients condition ensures that the $\mathbb{H}$-functor is an equivalence of monoidal categories.
\end{enumerate}
We will later compare these parity sheaves with the Elias--Williamson diagrammatic Hecke category which will remove the above two restrictions.
Nonetheless, in the diagrammatic version, we will use the characteristic 0 Soergel theory in checking the defining relations of the diagrammatic Hecke category are as specified.

\subsection{Soergel Bimodules}

We will provide a reminder on the general theory of Soergel bimodules. 
For the most part, the ambient realization will be fixed, namely, it will be the Cartan subalgebra $\fr{h}$; however, the Coxeter group with change throughout.
As a result, we will provide some discussion on realizations.
All the results in this section are well-known, and as such, we will omit proofs.

\subsubsection{Realizations}

Let $(W,S)$ be a Coxeter system. Let $\k$ be a noetherian domain of finite global dimension.

\begin{definition}
  A \emph{realization} of $(W, S)$ over $\k$ is a triple $(\fr{h}, \{\alpha_s\}_{s \in S}, \{\alpha_s^\vee\}_{s \in S})$, where $\fr{h}$ is a finite rank free
  $\k$-module, together with subsets
  \[\{\alpha_s^\vee \}_{s \in S} \subseteq \fr{h} \qquad\text{and}\qquad \{\alpha_s\}_{s\in S} \subseteq \fr{h}^*.\]
  We call $\alpha_s$ simple roots and $\alpha_s^\vee$ simple coroots of the realization. A realization must satisfy the following conditions:
  \begin{enumerate}
    \item $\alpha_s (\alpha_s^\vee) = 2$ for all $s \in S$;
    \item the assignment
      \[S \times \fr{h} \to \fr{h} \hspace{1cm} (s,v) \mapsto v - \alpha_s (v) \alpha_s^\vee\]
      extends to a $W$-action on $\fr{h}$;
    \item for any pair $(s,t)$ of distinct simple reflections such that $m_{st} < \infty$, we have
      \[[m_{st}]_s = [m_{st}]_{t} = 0,\]
      where $[m_{st}]_s, [m_{st}]_{t}$ are the 2-colored quantum numbers for $m_{st}$ at $x=-\langle \alpha_s^\vee, \alpha_{t} \rangle$ and $y=-\langle \alpha_{t}^\vee,
      \alpha_s \rangle$ respectively.
  \end{enumerate}
\end{definition}
Condition (3) is quite technical and will usually be satisfied. See \cite[\S3.1]{EW} for a detailed discussion.

It is useful to construct a category of realizations. Namely, we define a category, $\Realize (\k)$ whose objects consist of pairs $((W,S), (\fr{h}, \{\alpha_s\}_{s \in S},
\{\alpha_s^\vee\}_{s \in S}))$ where $(\fr{h}, \{\alpha_s\}_{s \in S}, \{\alpha_s^\vee\}_{s \in S})$ is a realization of $(W,S)$ over $\k$.
We will abuse notation and often write $(W, \fr{h})$ to denote this pair when no ambiguity is present.
A morphism between realizations $\phi : (W, \fr{h}) \to (W', \fr{h}')$ consists of an isomorphism of underlying vector spaces $\phi : \fr{h} \stackrel{\sim}{\to} \fr{h}'$ along with a 
morphism of Coxeter groups $\varphi : W \to W'$ (i.e. $\varphi$ must take simple reflections to either simple reflections or the identity) such that $\phi (w \cdot x) =
\varphi (w) \cdot \phi (x)$ for $w \in W$ and $x \in \fr{h}$.

\begin{definition}\label{def:nice_realization_stuff}
  A realization $(\fr{h}, \{\alpha_s\}_{s \in S}, \{\alpha_s^\vee\}_{s \in S})$ of $(W,S)$ is said to be:
  \begin{enumerate}
    \item  \emph{Demazure surjective} if the maps
      \[\alpha_s : \fr{h} \to \k, \qquad\text{and}\qquad \alpha_s^\vee : \fr{h}^* \to \k\]
      are surjective for all $s \in S$;
    \item \emph{faithful} if the representation
      \[W \to \GL (\fr{h}) \qquad\qquad w \mapsto w \cdot (-)\]
      is faithful;
    \item \emph{reflection faithful} if it is faithful, and for all $w \in W$, the fixed point set $\fr{h}_w \subseteq \fr{h}$ has codimension one if and only if $w$
      is a reflection in $W$;
    \item \emph{balanced} if for any pair of distinct simple reflections $(s,t)$, we have
      \[[m_{st} - 1]_s = [m_{st}-1]_{t} = 1.\]
  \end{enumerate}
\end{definition}

All of our realizations will be balanced and Demazure surjective. 
Other than in \S\ref{sec:endo}, our realizations will also be reflection faithful.
We will write $\NiceRealize (\k)$ for the full subcategory of $\Realize (\k)$ consisting of
such Demazure surjective, balanced, and reflection faithful realizations.
We will often call realizations in $\NiceRealize (\k)$ ``nice'' to simplify terminology.

Finally, if $\Theta \subset S$, we can associate a parabolic subgroup $W_\Theta$ of $W$ generated by $\Theta$.
For a nice realization $(\fr{h}, \{\alpha_s\}_{s \in S}, \{\alpha_s^\vee\}_{s \in S})$ of $(W,S)$, we can consider the \emph{restricted} realization for $(W_\Theta, \Theta)$.
The restricted realization consists of the triple $(\fr{h}, \{\alpha_s\}_{s \in \Theta}, \{\alpha_s^\vee\}_{s\in \Theta})$.

\subsubsection{Basic Definitions}

Let $(\fr{h}, \{\alpha_s\}_{s \in S}, \{\alpha_s^\vee\}_{s \in S})$ be a nice realization of $(W,S)$. We will write $R (\fr{h})$ for the symmetric algebra on $\fr{h}^*$
which is graded with $\deg \fr{h}^* = 2$.
If the realization is fixed or obvious from context, then we will simply write $R$ for this symmetric algebra. The contragradient action of $W$ on $\fr{h}^*$ gives rise
to an action of $W$ on $R$.

For $w \in W$, define a graded $R$-bimodule $R_w$ as the quotient of $R \otimes R$ generated by the ideal $w(r) \otimes 1 - 1 \otimes r$ for $r \in R$.
Alternatively, we can think of $R_w$ as the $w$-twisting of the tautological $R$-bimodule.

For $I \subseteq S$, we define the graded algebra $R^I$ consisting of the $W_I$-invariants of $R$. If $s \in S$, we will write $R^s = R^{\{s\}}$.
We will work in the abelian category $\grbim{R}$ of graded $R$-bimodules. For each $s \in S$, let $B_s$ denote the graded $R$-bimodule given by $B_s = R
\otimes_{R^s} R(1)$.
Define the category of \emph{Bott--Samelson bimodules} for $W$, denoted $\BSBim_W (\fr{h})$, as the smallest full monoidal subcategory of $\grbim{R}$ closed under $\otimes_R$ containing the objects $(B_s)_{s \in S}$.
If $\k$ is a field or a complete local ring, we can consider the idempotent completion of the additive hull of $\BSBim_W (\fr{h})$, denoted $\SBim_W (\fr{h})$, which is called the category of \emph{Soergel bimodules}.
If the choice of realization is obvious from context, we will write $\BSBim_W$ (resp. $\SBim_W$) for the category of Bott--Samelson bimodules (resp. Soergel bimodules).

\subsubsection{Functorality of Realizations}

Assume that $\k$ is a complete local ring or a field.
The category of Soergel bimodules admits a certain degree of functorality over $\NiceRealize (\k)$.
Let $(W,\fr{h})$ and $(W',\fr{h}')$ be realizations in $\NiceRealize (\k)$. Let $\phi : (W, \fr{h}) \to (W', \fr{h}')$ be a morphism of realizations. We can define
${}^{\phi}R_{\fr{h}', \fr{h}}$ as the quotient of $R(\fr{h}') \otimes_{\k} R (\fr{h})$ by the ideal generated by $\phi (r) \otimes 1 - 1 \otimes r$ for $r \in R (\fr{h})$.
Similarly, we can define
$R_{\fr{h}, \fr{h}'}^{\phi} $ as the quotient of $R(\fr{h}) \otimes_{\k} R (\fr{h}')$ by the ideal generated by $r \otimes 1 - 1 \otimes \phi (r)$ for $r \in R (\fr{h})$.
Note that ${}^{\phi}R_{\fr{h}', \fr{h}} $ (resp. $R_{\fr{h}, \fr{h}'}^{\phi}$) is a graded left (resp. right) $R (\fr{h}')$-module and a graded right (resp. left) $R(\fr{h})$-module, both of rank 1. 

We can then consider the functor
\[\phi_* : \grbim{R(\fr{h})} \to \grbim{R(\fr{h}')},\]
\[M \mapsto {}^{\phi} R_{\fr{h}', \fr{h}} \otimes_{R (\fr{h})} M \otimes_{R (\fr{h})} R_{\fr{h}, \fr{h}'}^{\phi}.\]

The equivariance constraint on morphisms of realizations guarantees that $\phi_*$ restricts to a functor
\[\phi_* : \SBim_W (\fr{h}) \to \SBim_{W'} (\fr{h}').\]

Let $\Cat$ denote the 2-category of small categories.
The following lemma follows easily from definitions.
\begin{lemma}\label{lem:equiv_of_realizations}
  There is a 2-functor
  \[\SBim : \NiceRealize (\k) \to \Cat \]
  such that for all morphisms of realizations $\phi : (W, \fr{h}) \to (W', \fr{h}')$, $\SBim (\phi) = \phi_*$.

  In particular, if $\phi$ is an isomorphism of realizations, then
  \[\phi_* : \SBim_W (\fr{h}) \to \SBim_{W'} (\fr{h}')\]
  is an equivalence of categories.
\end{lemma}

\subsection{Finite Parabolics}

We will now return to the example of $W$ arising from a Kac--Moody group. Let $\fr{h}_{\k}^* = \k \otimes_{\Z} \bfY (T)$. We have that $R_{\k} \coloneq R (\fr{h}_{\k}) \cong H_T (\pt; \k)$.
When $\k$ is clear from context, we will occasionally write $\fr{h}^* \coloneq \fr{h}_{\k}^*$ and $R = R_{\k}$.
For the remainder of the section, we will fix some $\scrL \in \Ch^{\circ} (T, \k)$.

The group $W_{\scrL}^\circ$ is a Coxeter group with simple reflections
\[S_{\scrL}^\circ = \{s \in W_{\scrL}^\circ \mid \ell_{\scrL} (s) = 1 \}.\]
Let $\Theta \subseteq S_{\scrL}^\circ$ be a subset. We denote by $W_{\scrL, \Theta}^\circ$ the standard parabolic subgroup of $W_{\scrL}^\circ$ generated by $\Theta$.
We say that $\Theta$ is of \emph{finite type} if $W_{\scrL, \Theta}^\circ$ is finite.
If $\Theta$ is of finite type, then $W_{\scrL, \Theta}^\circ$ has a unique maximal element $w_{\scrL, \Theta}$.

Let $\beta \in \uW{\scrL'}{\scrL}$ be a block. For a subset $\Theta \subseteq S_{\scrL}^\circ$, we can define a subset $\beta (\Theta) = w^\beta \Theta w^{\beta, -1} \subseteq S_{\scrL'}^\circ$.
It is clear that if $\Theta$ is of finite type, then $\beta (\Theta)$ will be of finite type as well.

The isomorphism of \ref{lem:beta_conj_of_endo_weyl_gps} gives rise to an isomorphism of realizations $(W_{\scrL}^\circ, \fr{h}_{\k}) \to (W_{w^\beta (\scrL)}^\circ, \fr{h}_{\k})$.
The following lemma is an immediate application of Lemma \ref{lem:equiv_of_realizations}.

\begin{lemma}\label{lem:beta_conj_for_SBim}
  Let $\beta \in \uW{\scrL'}{\scrL}$. Then there is an equivalence of monoidal categories
  \[{}^{\beta} (-) : \SBim_{W_{\scrL}^\circ}^\Theta (\fr{h}_{\k}) \to \SBim_{W_{\scrL'}^\circ }^{\beta(\Theta)} (\fr{h}_{\k}).\]
  \[{}^{\beta} (M) = R_{w^\beta} \otimes_R M \otimes_R R_{w^{\beta, -1}}\]
  such that for all $w \in W_{\scrL, \Theta}^\circ$,
  \[{}^{\beta} (B_w) \cong B_{w^\beta w w^{\beta,-1}}.\]
\end{lemma}

\subsection{Endosimple Reflections}\label{subsec:endosimps}

A reoccurring problem in working with simple reflections in $W_{\scrL}^\circ$ is that the simple reflections in $W_{\scrL}^\circ$ need not be simple reflections in $W$.
Recall that the simple reflections $W_{\scrL}^\circ$ are called \emph{endosimple reflections}. 
Unless explicitly stated, we will only call a reflection $s$ \emph{simple} if it is simple in $W$. Of course, if $s \in W_{\scrL}^{\circ}$ is simple, then it is automatically endosimple.

Let $\beta \in \uW{\scrL'}{\scrL}$.
Let $w \in \beta$ and let $\uw = (s_1, \ldots, s_k)$ be an expression for $w$ in $(W, S)$.
We say that the expression $\uw$ is \emph{endo-reduced} if 
\[ \ell_{\beta} (w) = \# \{ 1 \leq i \leq k \mid s_i \ldots s_k \scrL = s_{i+1} \ldots s_k \scrL \}.\]
If $\uw$ is a reduced, then by Lemma \ref{lem:length_for_blocks}, $\uw$ is endo-reduced. However, the converse need not be true. For example, if $s \in S$ satisfies $s \scrL \neq \scrL$, then $(s,s)$ is not reduced, but it is endo-reduced. 

\begin{lemma}[{\cite[Lemma 5.1.5]{Gou}}]\label{lem:make_refln_simple}
  Let $s \in W_{\scrL}^\circ$ be an endosimple reflection. Let $\us$ be an endo-reduced expression for $s$ in $W$. Then
  \begin{itemize}
    \item $\us$ has odd length and is palindromic, i.e., $\us = (s_1, \ldots s_k, t, s_{k}, \ldots, s_1)$;
    \item  $s_k \ldots s_1$ is minimal in the block $\beta \in \uW{s_k \ldots s_1\scrL }{\scrL}$ containing $s_k \ldots s_1$;
    \item $t$ is a simple reflection in $W_{w^{\beta} \scrL}^{\circ}$.
  \end{itemize}
\end{lemma}

In the study of unipotent Hecke categories, the $\IC$ attached to a simple reflection carries a canonical Frobenius algebra structure.
We wish to extend this to the monodromic case. There is an initial obstruction in ensuring that the Frobenius algebra is canonical, namely, that $\IC_s^{\scrL}$ with $s \in S_{\scrL}^{\circ}$ might only be defined up to a non-unique isomorphism.
Note that when $s$ is externally simple, then $\IC_s^{\scrL}$ has a canonical representative, namely the extension $\scrL^s$ of $\scrL$ to $U_s \backslash L_s \cong \eFl_{L_s}$.
This turns out to not be a problem since we can rigidify at the stalk at $e$. In particular, for any $\scrF \in \D{\scrL}{\scrL}^{\circ} (\k)$ in the isomorphism class of $\IC_s^{\scrL}$, the unit of the adjunction $\id \to j_{e*} j_e^*$ induces a morphism
\[\epsilon_{\scrF} : \scrF \to \IC_e^{\scrL} [1].\]

\begin{lemma}\label{lem:rig_of_ICs}
  Let $s \in S_{\scrL}^{\circ}$. 
  If $\scrF, \scrF' \in \D{\scrL}{\scrL}^{\circ} (\k)$ are both in the isomorphism class of $\IC_s^{\scrL}$, then there exists a unique isomorphism $\scrF \stackrel{\sim}{\to} \scrF'$ such that the following diagram commutes
  \[\begin{tikzcd}
\scrF \arrow[rr, "\sim"] \arrow[rd, "\epsilon_{\scrF}"'] &                      & \scrF' \arrow[ld, "\epsilon_{\scrF'}"] \\
                                                 & {\IC_e^{\scrL} [1].} &                               
\end{tikzcd}\]
\end{lemma}
\begin{proof}
  By Corollary \ref{cor:soergel_hom_v2}, the map
  \[\epsilon_{\scrF'} \circ (-): \Hom (\scrF, \scrF') \stackrel{\epsilon_{\scrF'} \circ (-)}{\longrightarrow} \Hom (\scrF, \IC_e^{\scrL} [1])\]
  is a morphism of rank 1 free $\k$-modules. We claim that $\epsilon_{\scrF'} \circ (-)$ is an isomorphism. 
  Let $\F$ be a field with a morphism $\k \to \F$. Note that $\epsilon_{\F (\scrF')} = \F (\epsilon_{\scrF'})$.
  Therefore, to check that $\epsilon_{\scrF'} \circ (-)$ is an isomorphism, it suffices to check that $\epsilon_{\F (\scrF')} \circ (-)$ is an isomorphism. But this is obvious since $\F$ is a field.
\end{proof}

In light of Lemma \ref{lem:rig_of_ICs}, for each $s \in S_{\scrL}^{\circ}$, we can fix $\IC_s^{\scrL}$ once and for all.

\subsubsection{Unit and Counit}

The unit of adjunction $\id \to j_{e*} j_e^*$ produces a map
\begin{equation}\label{eq:counit_map}
  \epsilon_s : \IC_{s}^{\scrL} \to j_{e*} j_e^* \IC_{s}^{\scrL} \cong \IC_e^{\scrL} [1].
\end{equation}
Similarly, the counit of adjunction $j_{e*} j_e^! \to \id$ produces a map
\begin{equation}\label{eq:unit_map}
  \eta_s : \IC_e^{\scrL} \cong j_{e*} j_e^! \IC_{s}^{\scrL} [1] \to \IC_{s}^{\scrL} [1].
\end{equation}

\subsubsection{Multiplication and Comultiplication}

\begin{lemma}\label{lem:conv_simple_with_simple}
  Let $s \in S_{\scrL}^{\circ}$ be an endosimple reflection. Then there is an isomorphism 
  \[\IC_s^{\scrL} \star \IC_s^{\scrL} \cong \IC_s^{\scrL} [-1] \oplus \IC_s^{\scrL} [1].\]
\end{lemma}
\begin{proof}
  By Proposition \ref{prop:structure_of_min_IC} and Lemma \ref{lem:make_refln_simple}, we may assume that $s$ is externally simple.
 Consider the distinguished triangle
 \[\Delta_s^{\scrL} \to \IC_s^{\scrL} \to \Delta_e^{\scrL} [1] \to.\]
 We can apply $(-) \star \IC_s^{\scrL}$ to obtain a distinguished triangle
 \begin{equation}\label{eq:conv_ICs_with_ICs}
  \Delta_s^{\scrL} \star \IC_s^{\scrL} \to \IC_s^{\scrL} \star \IC_s^{\scrL} \to \IC_s^{\scrL} [1] \to.
 \end{equation}
 By Lemmas \ref{lem:pi_s_pushforward_of_stds}, \ref{lem:pi_s_descent}, and \ref{lem:conv_with_ICs_rewritting}, we have an isomorphism $\Delta_s^{\scrL} \star \IC_s^{\scrL} \cong \IC_s^{\scrL} [-1]$.
 The connecting morphism $\IC_s^{\scrL} [1] \to \IC_s^{\scrL}$ in (\ref{eq:conv_ICs_with_ICs}) must be zero since $\IC_s^{\scrL}$ is perverse.
 Therefore, the distinguished triangle (\ref{eq:conv_ICs_with_ICs}) splits.
\end{proof}

\begin{lemma}\label{lem:trivalent_mors}\quad
  \begin{enumerate}
    \item The morphism
      \[b_1 : \IC_s^{\scrL} \star \IC_s^{\scrL} [1] \stackrel{\id \star \epsilon_s }{\longrightarrow } \IC_s^{\scrL} \star \IC_e^{\scrL} \cong \IC_s^{\scrL}\]
      can be identified under an isomorphism from Lemma \ref{lem:conv_simple_with_simple} with the projection
      \[\IC_s^{\scrL} \oplus \IC_s^{\scrL} [2] \twoheadrightarrow \IC_s^{\scrL}.\]
      Moreover, composition with $b_1$ induces an isomorphism
      \[\tilde{b}_1 : \Hom (\IC_s^{\scrL}, \IC_s^{\scrL} \star \IC_s^{\scrL} [-1]) \stackrel{\sim}{\to} \Hom (\IC_s^{\scrL}, \IC_s^{\scrL}).\]
    \item The morphism
      \[b_2 :  \IC_s^{\scrL}[-1] \cong \IC_s^{\scrL} \star \IC_e^{\scrL} [-1]  \stackrel{\id \star \eta_s }{\longrightarrow } \IC_s^{\scrL} \star \IC_s^{\scrL}  \]
      can be identified under an isomorphism from Lemma \ref{lem:conv_simple_with_simple} with the inclusion
      \[\IC_s^{\scrL} [-1] \hookrightarrow \IC_s^{\scrL} [1] \oplus \IC_s^{\scrL} [-1] .\]
      Moreover, pre-composition with $b_2$ induces an isomorphism
      \[\tilde{b}_2 : \Hom (\IC_s^{\scrL} \star \IC_s^{\scrL}, \IC_s^{\scrL} [-1]) \stackrel{\sim}{\to} \Hom (\IC_s^{\scrL} [-1], \IC_s^{\scrL} [-1]).\]
  \end{enumerate}
\end{lemma}
\begin{proof}
  We will only prove (1) as the proof of (2) can be obtained by a similar argument.
  Since $\IC_s^{\scrL}$ are perverse, there are no negative self-extensions, so we have an isomorphism
  \[\Hom (\IC_s^{\scrL} \oplus \IC_s^{\scrL} [2], \IC_s^{\scrL}) \cong \Hom (\IC_s^{\scrL}, \IC_s^{\scrL}).\]
  By Corollary \ref{cor:biadjoint_of_conv_with_ICs} and Lemma \ref{lem:conv_simple_with_simple},
  \begin{align*}
    \Hom (\IC_s^{\scrL}, \IC_s^{\scrL}) &\cong \Hom (\IC_s^{\scrL} \star \IC_s^{\scrL}, \IC_e^{\scrL}) \\
    &\cong \Hom (\IC_s^{\scrL} [-1] \oplus \IC_s^{\scrL} [1], \IC_e^{\scrL}) \\
    &\cong \k.
  \end{align*}
  Since $b_1$ corresponds to an invertible element of $\Hom (\IC_s^{\scrL}, \IC_s^{\scrL}) \cong \k$, by potentially scaling the isomorphism from Lemma \ref{lem:conv_simple_with_simple}, we can ensure that $b_1$ indeed identifies with the projection map.
  Under this identification, it easily follows that $\tilde{b}_1$ is an isomorphism.
\end{proof}

Following Lemma \ref{lem:trivalent_mors}, we can define morphisms
\[\nu_s = \tilde{b}_1^{-1} (\id) : \IC_{s}^{\scrL} \to \IC_s^{\scrL} \star \IC_s^{\scrL} [-1] \quad\text{and}\quad \mu_s = \tilde{b}_2^{-1} (\id): \IC_s^{\scrL} \star \IC_s^{\scrL} \to \IC_s^{\scrL} [-1].\]

\begin{lemma}\label{lem:Frob_alg_stuff}
  The maps $(\mu_s, \nu_s, \epsilon_s, \eta_s)$ make $\IC_s^{\scrL}$ into a graded Frobenius algebra of degree 1.
\end{lemma}
\begin{proof}
  The lemma follows from chasing definitions and using the characterizing criteria from Lemma \ref{lem:trivalent_mors}.
\end{proof}

\subsubsection{Parabolic Monodromic Subcategories}

Let $\Theta \subseteq S_{\scrL}^\circ$.
We define a full triangulated subcategory $\D{\scrL}{\scrL}^{\circ, \Theta} (\k)$ of $\D{\scrL}{\scrL}^{\circ} (\k)$ generated by $\Delta_w^{\scrL}$ (or $\nabla_w^{\scrL}$) for $w \in W_{\scrL, \Theta}^{\circ}$.
It is not hard to check from Proposition \ref{prop:conv_rules} and Lemma \ref{lem:conv_ICs_with_stds} that $\D{\scrL}{\scrL}^{\circ, \Theta} (\k)$ is closed under convolution.

Suppose that $\k$ is a field or a complete local ring.
We define a full subcategory $\Parity{\scrL}{\scrL}^{\circ, \Theta} (\k)$ of $\Parity{\scrL}{\scrL}^{\circ} (\k)$ generated under direct sums and shifts by the objects $\scrE_w^{\scrL}$ for $w \in W_{\scrL, \Theta}^{\circ}$.
It is clear from the construction of $\Parity{\scrL}{\scrL}^{\circ, \Theta} (\k)$ that
\[\Parity{\scrL}{\scrL}^{\circ, \Theta} (\k) = \Parity{\scrL}{\scrL}^{\circ} (\k) \cap \D{\scrL}{\scrL}^{\circ, \Theta} (\k).\]
By Theorem \ref{thm:conv_preserves_parity}, $\Parity{\scrL}{\scrL}^{\circ, \Theta} (\k)$ is closed under convolution.

We call these subcategories \emph{parabolic monodromic subcategories}. We are particularly interested when $\Theta$ is of finite type.
In such a situation, $\D{\scrL}{\scrL}^{\circ,\Theta} (\k)$ contains only finitely many isomorphism classes of simple perverse sheaves.

\subsection{Maximal IC Sheaves}

The main benefit in working in the parabolic monodromic subcategories is they allow us to construct maximal IC sheaves. In this section, we will study these sheaves
along with rigidified variants.
In order for maximal IC sheaves to exist, we will always assume that $\Theta \subseteq S_{\scrL}^\circ$ is of finite type.

We will also start to make the assumption that $\k$ is a field of characteristic 0 which will continue until \S\ref{sec:endo}.
It can be checked using the decomposition theorem that with characteristic 0 coefficients the category of parity sheaves coincides with the category of semisimple complexes. 

Let $\beta \in \uW{\scrL'}{\scrL}$.
We define the $\Theta$-maximal IC sheaf as $\IC_{w_{\beta, \Theta}}^{\scrL}$. When $\Theta = S$ is of finite type, then the $\Theta$-maximal IC sheaf is a maximal IC sheaf in the sense of \cite{LY}.

When $\scrL, \scrL' = \uk_T$, $\Theta$ corresponds to an actual parabolic subgroup $P_\Theta$ of $G$.
In this case, the $\Theta$-maximal IC sheaf corresponds to the shifted constant sheaf $\underline{\k}_{B \backslash P_\Theta / B} [ \dim P_\Theta / B]$ which is
$!$-extended to all of $\BGB$.
The constant sheaf $\underline{\k}_{B \backslash P_\Theta / B}$ has two crucial properties in defining a restricted parabolic variant of Soergel theory.
\begin{enumerate}
  \item The stalks and costalks at $\dot{w}$ of $\underline{\k}_{B \backslash P_\Theta / B}$ are one-dimensional when $w \in W_\Theta$ and 0 when $w \notin W_\Theta$.
  \item Any convolution $\underline{\k}_{B \backslash P_\Theta / B} \star \IC_w^{\uk}$ for $w \in W_\Theta$ yields direct sums of shifts of $\underline{\k}_{B \backslash P_\Theta / B}$.
\end{enumerate}
These properties allow one to define a (canonical) coalgebra structure on $\underline{\k}_{B \backslash P_\Theta / B}$. As a result, the Soergel functor
\[\Hom^\bullet (\underline{\k}_{B \backslash P_\Theta / B}, -) : \Parity{\uk}{\uk}^{\circ, \Theta} (\k) \to \grbim{\End^\bullet (\IC_e^{\uk}) }\]
is lax-monoidal. The goal of this section is to develop versions of these facts for $\Theta$-maximal IC sheaves. As a result, we can develop a monodromic version of the Soergel functor.
We note that when $W$ is finite and $\Theta = S$, there are similar results covered in \cite[\S6]{LY}.
However, we need to modify some of their arguments to make the parabolic variants work. 

\subsubsection{Convolution with Maximal IC Sheaves}

\begin{lemma}\label{lem:conv_ICw_and_ICs}
  Let $s \in W_{\scrL}^{\circ}$ be a simple reflection. Let $\beta \in \uW{\scrL'}{\scrL}$ and $w \in \beta$.
  \begin{enumerate}
    \item If $w < ws$, then there is an isomorphism 
    \begin{equation}\label{eq:conv_ICw_and_ICs_1}
      \IC_{w}^{\scrL} \star \IC_s^{\scrL} \cong \IC_{ws}^{\scrL} \oplus \scrP_{<_{\beta} ws},
    \end{equation} 
    where $\scrP_{<_{\beta} ws}$ is a semisimple perverse sheaf of the form
      \[\scrP_{<_{\beta} ws} = \bigoplus_{x <_{\beta} ws} \IC_x^{\scrL} \otimes V_x\]
      and $V_x$ is a $\k$-vector space. Moreover, $V_x = 0$ if $\ell (x) = \ell (ws) (\textnormal{mod } 2)$.
    \item If $w > ws$, then there is an isomorphism
    \begin{equation}\label{eq:conv_ICw_and_ICs_2}
      \IC_{w}^{\scrL} \star \IC_s^{\scrL} \cong \IC_{w}^{\scrL} [-1] \oplus \IC_{w}^{\scrL} [1].
    \end{equation} 
  \end{enumerate}
\end{lemma}
\begin{proof}
  \emph{(1): } Assume $w < ws$. We will first show that $\IC_{w}^{\scrL} \star \IC_s^{\scrL}$ is perverse. To simplify notation, we will write $\left( {}^p D^{\leq 0}, {}^p D^{\geq 0}\right)$ for the perverse $t$-structure on $\D{\scrL'}{\scrL} (\k)$.
  In particular, we will show that $\IC_w^{\scrL} \star \IC_s^{\scrL} \in {}^p D^{\leq 0}$. The case of $\IC_w^{\scrL} \star \IC_s^{\scrL} \in {}^p D^{\geq 0}$ follows from Verdier duality.
  By Theorem \ref{thm:categorification}, it follows from a standard combinatorial argument that 
  \begin{equation}\label{eq:conv_ICw_and_ICs_3} 
    \langle  \underline{H}_w^{\scrL}  \underline{H}_s^{\scrL}, T_x^{\scrL} \rangle = \begin{cases} v^{-\ell (w) - 1} \left(  P_{x,w}^{\scrL} + v^2 P_{xs,w}^{\scrL}\right) & \text{if } x < xs, \\ v^{-\ell (w) - 1} \left(  P_{xs,w}^{\scrL} + v^2 P_{x,w}^{\scrL}\right) & \text{if } x > xs. \end{cases}
  \end{equation}
  In particular $\langle \underline{H}_w^{\scrL} \underline{H}_s^{\scrL}, T_x^{\scrL} \rangle$ has degree $\leq - \ell (x)$ for all $x \in \beta$.
  As a consequence, $\dim H^i (j_x^* (\IC_w^{\scrL} \star \IC_s^{\scrL})) = 0$ when $i > - \ell (x)$. Therefore, $\IC_w^{\scrL} \star \IC_s^{\scrL} \in {}^p D^{\leq 0}$.
  
  Moreover, (\ref{eq:conv_ICw_and_ICs_3}) implies that the restrictions of $\IC_w^{\scrL} \star \IC_s^{\scrL}$ along $\eFl_x$ is only nonzero if $x \in \{w, ws\}$ or $x \leq_{\beta} w$. 
  In particular, the stalks can only be nonzero when $x \leq_{\beta} ws$. 
  This stalk computation along with Theorem \ref{thm:conv_preserves_parity} implies that $\IC_w^{\scrL} \star \IC_s^{\scrL}$ admits a decomposition,
  \[ \IC_{w}^{\scrL} \star \IC_s^{\scrL} \cong \IC_{ws}^{\scrL} \oplus \scrP_{<_{\beta} ws},\]
  where $\scrP_{<_{\beta} ws}$ is a semisimple perverse sheaf of the form
  \[\scrP_{<_{\beta} ws} = \bigoplus_{x <_{\beta} ws} \IC_x^{\scrL} \otimes V_x.\]
  If $\ell (w)$ is even (resp. $\ell (w)$ is odd), then $\IC_w^{\scrL}$ is even (resp. odd). 
  As a result, the proof of Lemma \ref{lem:pi_s_and_parity} ensures that $\IC_{w}^{\scrL} \star \IC_s^{\scrL}$ has the opposite parity of $\IC_w^{\scrL}$.
  Therefore, $V_x = 0$ if $\ell (x) = \ell (ws) (\textnormal{mod } 2)$.

  \emph{(2): } Assume $w > ws$. By part (1), there is an isomorphism
  \begin{equation}\label{eq:conv_ICw_and_ICs_4}
    \IC_{ws}^{\scrL} \star \IC_s^{\scrL} \star \IC_s^{\scrL} \cong \IC_{w}^{\scrL} \star \IC_s^{\scrL} \oplus \scrP_{<_{\beta} w} \star \IC_s^{\scrL}.
  \end{equation}
  On the other hand, part (1) along with Lemma \ref{lem:conv_simple_with_simple} yields an isomorphism
  \begin{equation}\label{eq:conv_ICw_and_ICs_5}
    \IC_{ws}^{\scrL} \star \IC_s^{\scrL} \star \IC_s^{\scrL} \cong \IC_{w}^{\scrL} [-1] \oplus \IC_w^{\scrL} [1] \oplus \scrP_{<_{\beta} w} [-1] \oplus \scrP_{<_{\beta} w} [-1].
  \end{equation}
  The parity conditions on $\scrP_{<_{\beta} w}$ ensures that if $\IC_x^{\scrL}$ is a summand of $\scrP_{<_{\beta} w}$, we have that $j_w^* (\IC_x^{\scrL} \star \IC_s^{\scrL}) = 0$.
  Therefore, by comparing (\ref{eq:conv_ICw_and_ICs_4}) with (\ref{eq:conv_ICw_and_ICs_5}), we obtain an isomorphism 
  \[\IC_{w}^{\scrL} \star \IC_s^{\scrL}  \cong \IC_{w}^{\scrL} [-1] \oplus \IC_w^{\scrL} [1].\]
\end{proof}

\begin{proposition}\label{prop:conv_with_maximal_IC_neutral}
  Let $w \in W_{\scrL, \Theta}^\circ$.
  \begin{enumerate}
    \item The convolution $\IC_{w_{\scrL, \Theta}}^{\scrL} \star \IC_w^{\scrL}$ is isomorphic to a direct sum of shifts of $\IC_{w_{\scrL, \Theta}}^{\scrL}$.
    \item The perverse cohomology ${}^p H^i (\IC_{w_{\scrL, \Theta}}^{\scrL} \star \IC_w^{\scrL}) = 0$ unless $-\ell_{\scrL} (w) \leq i \leq \ell_{\scrL} (w)$ and $i \equiv \ell_{\scrL} (w) \textnormal{ } (\textnormal{mod }2)$.
    \item There is an isomorphism
      \[{}^p H^{\pm \ell_{\scrL} (w)} (\IC_{w_{\scrL, \Theta}}^{\scrL} \star \IC_w^{\scrL}) \cong \IC_{w_{\scrL, \Theta}}^{\scrL}.\]
  \end{enumerate}
\end{proposition}
\begin{proof}
  We will argue by induction on $\ell_{\scrL} (w)$.
  We can first consider the base cases of $\ell_{\scrL} (w) = 0$ and $\ell_{\scrL} (w) = 1$. When $\ell_{\scrL} (w) = 0$, we have that $w = e$, and so (1), (2), (3) follow
  $\IC_e^{\scrL}$ being the monoidal unit.
  Now suppose that $\ell_{\scrL} (w) = 1$, i.e., $w = s \in \Theta$ an endosimple reflection.
  By Lemma \ref{lem:make_refln_simple}, we can find some block $\beta \in \uW{\scrL}{\scrL'}$ such that $s = w^\beta t w^{\beta, -1}$ for $t$ a simple reflection in $W_{\scrL'}^{\circ}$.
  By Proposition \ref{prop:structure_of_min_IC} there is an isomorphism
  \[ \IC_{w_{\scrL, \Theta}}^{\scrL} \star \IC_s^{\scrL} \cong \IC_{w_{\scrL, \Theta} w^{\beta} }^{\scrL'} \star \IC_{t}^{\scrL'} \star \IC_{w^{\beta, -1}}^{\scrL}. \]
  By maximality in $W_{\scrL, \Theta}^{\circ}$, $\ell_{\scrL} (w_{\scrL, \Theta} s) < \ell_{\scrL} (w_{\scrL, \Theta})$. As a result, by (\ref{eq:endo_vs_bruhat_orders}), we have that $\ell (w_{\scrL, \Theta} w^{\beta} t) < \ell (w_{\scrL, \Theta} w^\beta)$.
  By Lemma \ref{lem:conv_ICw_and_ICs} (2), we obtain an isomorphism
  \[  \IC_{w_{\scrL, \Theta} w^{\beta} }^{\scrL'} \star \IC_{t}^{\scrL'} \cong \IC_{w_{\scrL, \Theta} w^\beta}^{\scrL'} [-1] \oplus \IC_{w_{\scrL, \Theta} w^\beta}^{\scrL'}[1]. \]
  We can convolve the above isomorphisms with $\IC_{w^{\beta, -1}}^{\scrL}$ to obtain
  \begin{equation} \IC_{w_{\scrL, \Theta}}^{\scrL} \star \IC_s^{\scrL}  \cong \IC_{w_{\beta, \Theta}}^{\scrL}  [-1] \oplus \IC_{w_{\beta, \Theta}}^{\scrL} [1].
    \label{eq:conv_with_maximal_IC_netural_1}
  \end{equation}

  Let $\ell_{\scrL} (w) > 1$. We can then find some $s \in \Theta$ such that $w' = ws$ with $\ell_{\scrL} (w') < \ell_{\scrL} (w)$.
  By the inductive hypothesis,
  \begin{equation} \IC_{w_{\scrL, \Theta}}^{\scrL} \star \IC_{w'}^{\scrL} \cong \IC_{w_{\scrL, \Theta}}^{\scrL} \otimes V', \label{eq:conv_with_maximal_IC_netural_2}
  \end{equation}
  where $V' = \bigoplus_{n=-\ell_{\scrL} (w')}^{\ell_{\scrL} (w')} V_n' [-n]$ where $\dim V_n' = 1$ if $n = \pm \ell_{\scrL} (w')$.
  From (\ref{eq:conv_with_maximal_IC_netural_1}) and (\ref{eq:conv_with_maximal_IC_netural_2}), we obtain an isomorphism
  \begin{equation}
    \IC_{w_{\scrL, \Theta}}^{\scrL} \star \IC_{w'}^{\scrL} \star \IC_s^{\scrL} \cong \IC_{w_{\scrL, \Theta}}^{\scrL} \otimes V'', \label{eq:conv_with_maximal_IC_netural_3}
  \end{equation}
  where $V'' = V' [-1] \oplus V' [1]$. By Lemma \ref{lem:conv_ICw_and_ICs} and Proposition \ref{prop:structure_of_min_IC}, we have an isomorphism $\IC_{w'}^{\scrL} \star \IC_s^{\scrL} \cong \IC_w^{\scrL} \oplus
  \scrP_{<_{\scrL} w}$ where $\scrP_{<_{\scrL} w}$ is a semisimple perverse sheaf supported on $\eFl_{< w}$.
  As a result,
  \[ \IC_{w_{\scrL, \Theta}}^{\scrL} \star \IC_w^{\scrL} \oplus \IC_{w_{\scrL, \Theta}}^{\scrL} \star \scrP_{<_{\scrL} w} \cong \IC_{w_{\scrL, \Theta}}^{\scrL} \otimes V''.\]
  It is then clear that (1) and (2) hold. For $x \in W_{\scrL, \Theta}^\circ$ with $\ell_{\scrL} (x) < \ell_{\scrL} (w)$, the inductive hypothesis yields an isomorphism
  \[\IC_{w_{\scrL, \Theta}}^{\scrL} \star \IC_x^{\scrL} \cong \IC_{w_{\scrL, \Theta}}^{\scrL} \otimes V_x,\]
  where $V_x = \bigoplus_{n = -\ell_{\scrL} (x)}^{\ell_{\scrL} (x)} V_{x,n} [-n]$. Since all $\IC_x^{\scrL}$-summands of $\scrP_{<_{\scrL} w}$ satisfy that $\ell_{\scrL} (x) < \ell_{\scrL}
  (w)$ we have that
  \begin{equation} \IC_{w_{\scrL, \Theta}}^{\scrL} \star \scrP_{<_{\scrL} w} \cong \IC_{w_{\scrL, \Theta}}^{\scrL} \otimes V''', \label{eq:conv_with_maximal_IC_netural_4}
  \end{equation}
  where $V'''$ is a formal chain complex of $\k$-vector spaces concentrated in degrees $n$ where $-\ell_{\scrL} (w) < n < \ell_{\scrL} (w)$.
  By comparing the $V''$ from (\ref{eq:conv_with_maximal_IC_netural_3}) and the $V'''$ from (\ref{eq:conv_with_maximal_IC_netural_4}), we see that
  \[{}^p H^{n} (\IC_{w_{\scrL, \Theta}}^{\scrL} \star \IC_w^{\scrL}) \cong \IC_{w_{\scrL, \Theta}}^{\scrL},\]
  when $n = \pm \ell_{\scrL} (w)$ which completes the proof of the lemma.
\end{proof}

\subsubsection{Stalks of Maximal IC Sheaves}

\begin{proposition}\label{prop:stalks_of_maximal_IC}
  For $w \in W_{\scrL, \Theta}^{\circ}$, we have isomorphisms
  \begin{enumerate}
    \item  $j_w^* \IC_{w_{\scrL, \Theta}}^{\scrL} \cong \scrK_w^{\scrL} [\ell_{\scrL} (w_{\scrL, \Theta}) + \ell (w) - \ell_{\scrL} (w)], $
    \item $ j_w^! \IC_{w_{\scrL, \Theta}}^{\scrL} \cong \scrK_w^{\scrL} [-\ell_{\scrL} (w_{\scrL, \Theta}) + \ell (w) + \ell_{\scrL} (w)].$
  \end{enumerate}
\end{proposition}
\begin{proof}
  It suffices to just prove the first isomorphism as the second will follow from Verdier duality.

  Since $\IC_{w_{\scrL, \Theta}}^{\scrL}$ is a parity sheaf, we have that $j_w^* \IC_{w_{\scrL, \Theta}}^{\scrL} \cong V_w \otimes_{H_T^\bullet (\pt; \k)} \scrK_w^{\scrL}$ where $V_w = \bigoplus_n V_w^n [n]$ is a graded $H_{T}^\bullet (\pt; \k)$-module\footnote{Here we are using the identification given in Proposition \ref{prop:mhc_strat_and_parity} to view sheaves on $\eFl_w$ as equivariant sheaves on a point.} which can be computed by
  \[V_w^n = \Hom (j_w^* \IC_{w_{\scrL, \Theta}}^{\scrL}, \scrK_w^{\scrL} [-n]).\]
  The proposition can then be reformulated as
  \begin{equation}\label{eq:stalks_of_maximal_IC_1}
    \Hom^\bullet (\IC_{w_{\scrL, \Theta}}^{\scrL}, \nabla_w^{\scrL}) \cong H^\bullet_T (\pt; \k) [-\ell_{\scrL} (w_{\scrL, \Theta}) + \ell_{\scrL} (w)].
  \end{equation}

  We will argue by backwards induction on $\ell_{\scrL} (w)$. If $\ell_{\scrL} (w)$ is maximal, i.e., $w = w_{\scrL, \Theta}$, then the isomorphism is obvious.
  If $\ell_{\scrL} (w)$ is not maximal, then we can find some $s \in \Theta$ such that $w <_{\scrL} ws$.
  By Lemma \ref{lem:make_refln_simple}, we can find some block $\beta \in \uW{\scrL}{\scrL'}$ such that $s = w^\beta t w^{\beta, -1}$ for $t$ a simple reflection in $W_{\scrL'}^{\circ}$.
  By Proposition \ref{prop:structure_of_min_IC}, we have an isomorphism
  \[ \Hom^\bullet (\IC_{w_{\scrL, \Theta}}^{\scrL}, \nabla_w^{\scrL}) \cong \Hom^\bullet (\IC_{w_{\scrL, \Theta} w^{\beta}}^{\scrL'}, \nabla_{ww^{\beta}}^{\scrL'}). \]
   Since $\ell_{\beta} (w_{\scrL, \Theta} w^{\beta} t) < \ell_{\beta} (w_{\scrL, \Theta} w^{\beta})$, we can apply Lemma \ref{lem:pi_s_descent} to show that the stalk of $\IC_{w_{\scrL, \Theta} w^{\beta}}^{\scrL'}$ at lifts of $ww^{\beta}$ and $w w^{\beta} t$ are isomorphic up to a shift.
   This induces an isomorphism of Hom spaces,
    \[ \Hom^\bullet (\IC_{w_{\scrL, \Theta} w^{\beta}}^{\scrL'}, \nabla_{ww^{\beta}}^{\scrL'}) \cong \Hom^\bullet (\IC_{w_{\scrL, \Theta} w^{\beta}}^{\scrL'}, \nabla_{ww^{\beta} t}^{\scrL'} [-1]).\]
  Finally, by another application of Proposition \ref{prop:structure_of_min_IC} and the inductive hypothesis, we get isomorphisms
  \begin{align*}
    \Hom^\bullet (\IC_{w_{\scrL, \Theta}}^{\scrL}, \nabla_w^{\scrL}) &\cong \Hom^\bullet (\IC_{w_{\scrL, \Theta} w^{\beta}}^{\scrL'}, \nabla_{ww^{\beta} t}^{\scrL'} [-1]) \\
    &\cong \Hom^\bullet (\IC_{w_{\scrL, \Theta}}^{\scrL}, \nabla_{ws}^{\scrL} [-1]) \\
    &\cong H_T^\bullet (\pt; \k) [-\ell_{\scrL} (w_{\scrL, \Theta}) + \ell_{\scrL} (ws) - 1] \\
    &\cong H_T^\bullet (\pt; \k) [-\ell_{\scrL} (w_{\scrL, \Theta}) + \ell_{\scrL} (w) ]. 
  \end{align*}
  This completes the proof by the characterization of the proposition given in (\ref{eq:stalks_of_maximal_IC_1}).
\end{proof}

\begin{remark}
  There are extensions of Proposition \ref{prop:conv_with_maximal_IC_neutral} and Proposition \ref{prop:stalks_of_maximal_IC} to sheaves outside the neutral block.
  Such generalizations necessitate a more general definition of the parabolic monodromic subcategories. For example, if $\beta \in \uW{\scrL'}{\scrL}$, one could consider the IC-sheaf $\IC_{w^{\beta} w_{\scrL, \Theta}}^{\scrL}$ which behaves like a maximal IC sheaf for the block $\beta$.
  We do not need these generalizations for this paper, but similar results in this direction when $W$ is finite can be found in \cite[\S6]{LY}.
\end{remark}

\subsubsection{Coalgebra Structure on Maximal IC Sheaves}

Let $\Theta \subset S_{\scrL}^{\circ}$ be of finite type.
Let $\scrF \in \D{\scrL}{\scrL}^{\circ} (\k)$ be an object whose isomorphism class is $\IC_{w_{\scrL, \Theta} }^{\scrL} [-\ell_{\scrL} (w_{\scrL, \Theta})]$.
By Proposition \ref{prop:stalks_of_maximal_IC}, there is a nonzero morphism
\[\epsilon_{\scrF} : \scrF \to \IC_e^{\scrL}\]
induced by the unit map $\id \to j_{e*} j_e^*$. 

\begin{lemma}\label{lem:rig_of_max_IC}
  If $\scrF, \scrF' \in \D{\scrL}{\scrL}^{\circ} (\k)$ are both isomorphic to $\IC_{w_{\scrL, \Theta} }^{\scrL} [-\ell_{\scrL} (w_{\scrL, \Theta})]$, then there exists a unique isomorphism $\scrF \stackrel{\sim}{\to} \scrF'$ such that the following diagram commutes
  \[\begin{tikzcd}
\scrF \arrow[rr, "\sim"] \arrow[rd, "\epsilon_{\scrF}"'] &                      & \scrF' \arrow[ld, "\epsilon_{\scrF'}"] \\
                                                 & {\IC_e^{\scrL}.} &                               
\end{tikzcd}\]
\end{lemma}
\begin{proof}
  By Proposition \ref{prop:stalks_of_maximal_IC}, the map
  \[\Hom (\scrF, \scrF') \stackrel{\epsilon_{\scrF'} \circ (-)}{\longrightarrow} \Hom (\scrF, \IC_e^{\scrL})\]
  is a morphism of 1-dimensional $\k$-vector spaces. Clearly this map is nonzero, and hence, must be an isomorphism.
\end{proof}

In light of Lemma \ref{lem:rig_of_max_IC}, we can fix the object $\IC_{w_{\scrL, \Theta} }^{\scrL} [-\ell_{\scrL} (w_{\scrL, \Theta})]$ once and for all.
We will denote the fixed object by $\scrK^{\scrL}_{\Theta} \in \D{\scrL}{\scrL}^{\circ} (\k)$. This sheaf should be thought of as a replacement for the constant sheaf in the twisted-equivariant setting.
It comes equipped with a canonical adjunction morphism $\epsilon_{\scrL} : \scrK_{\Theta}^{\scrL} \to \IC_e^{\scrL}$.

\begin{proposition}\label{prop:coalg_on_max_IC}
    There is a unique coalgebra structure on $\scrK_{\Theta}^{\scrL}$ with counit map $\epsilon_{\scrL}$.
\end{proposition}

Before proving Proposition \ref{prop:coalg_on_max_IC}, we must introduce some notation for handling iterated convolutions of minimal and maximal IC's. For each $n \in \Z_{>0}$, we write
\[
  (\scrK_{\Theta}^{\scrL})^{\star n} \coloneq \underbrace{\scrK_{\Theta}^{\scrL} \star \ldots \star \scrK_{\Theta}^{\scrL}}_{n \textnormal{-fold}} \hspace{0.5cm}\text{and}\hspace{0.5cm} (\IC_e^{\scrL})^{\star n} \coloneq \underbrace{\IC_{e}^{\scrL} \star \ldots \star \IC_e^{\scrL}}_{n \textnormal{-fold}}.
\]
By convention, we set $(\scrK_{\Theta}^{\scrL})^{\star 0} = \IC_e^{\scrL} = (\IC_e^{\scrL})^{\star 0}$.
Similarly, we can define a morphism $\epsilon_{\scrL}^n : (\scrK_{\Theta}^{\scrL})^{\star n} \to (\IC_{e}^{\scrL})^{\star n}$ as the $n$-fold convolution of $\epsilon_{\scrL}$.
Of course, there are canonical isomorphisms $(\IC_{e}^{\scrL})^{\star n} \cong \IC_{e}^{\scrL}$.

\begin{lemma}\label{lem:uniqueness_of_comult}
    For each $n \in \Z_{\geq 0}$, there is a unique morphism
    \[\mu_{\scrL}^n : \scrK_{\Theta}^{\scrL} \to  (\scrK_{\Theta}^{\scrL})^{\star n}\]
    making the following diagram commute
    \begin{equation}
        \begin{tikzcd}\label{eq:uniqueness_of_comult_0}
            {\scrK_{\Theta}^{\scrL}} \arrow[r, "\mu_{\scrL}^n"] \arrow[d, "\epsilon_{\scrL}"] &  (\scrK_{\Theta}^{\scrL})^{\star n} \arrow[d, "\epsilon_{\scrL}^n"] \\
            \IC_{e}^{\scrL} \arrow[r, "\sim"]                                                  & (\IC_{e}^{\scrL})^{\star n}.                             
            \end{tikzcd}
    \end{equation}
\end{lemma}
\begin{proof}
    We can iterate Proposition \ref{prop:conv_with_maximal_IC_neutral} (3) to obtain an isomorphism
    \[(\scrK_{\Theta}^{\scrL})^{\star n} \cong \scrK_{\Theta}^{\scrL} \otimes V\]
    where 
    \[V = \bigoplus_{i=0}^{2n \cdot \ell_{\scrL} (w_{\scrL, \Theta})} V_{-i} [i]\]
    and $\dim V_0 = 1$. As a result, we can see that $\scrK_{\Theta}^{\scrL}$ occurs with multiplicity 1 in $(\scrK_{\Theta}^{\scrL})^{\star n}$.
    Moreover, since $V_i = 0$ whenever $i > 0$, it follows that
    \[\dim \Hom (\scrK_{\Theta}^{\scrL}, (\scrK_{\Theta}^{\scrL})^{\star n}) = 1.\]
    Let $\mu^n_{\scrL} : \scrK_{\Theta}^{\scrL} \to (\scrK_{\Theta}^{\scrL})^{\star n}$ be any nonzero morphism.

    We claim that the composition
    \begin{equation}\label{eq:uniqueness_of_comult_1}
      \scrK_{\Theta}^{\scrL} \stackrel{\mu_{\scrL}^n}{\to} (\scrK_{\Theta}^{\scrL})^{\star n} \stackrel{\epsilon_{\scrL}^n}{\to} (\IC_{e}^{\scrL})^{\star n} \cong \IC_{e}^{\scrL}
    \end{equation}
    is nonzero. We will prove this claim for $n = 2$ as the cases of $n > 2$ can be obtained by a routine induction argument (cf., \cite[Proposition 6.6]{LY}).
    Since $\epsilon_{\scrL}^{2} = (\epsilon_{\scrL} \star \id_{\IC_e^{\scrL}}) \circ (\id_{\scrK_{\Theta}^{\scrL}} \star \epsilon_{\scrL})$ and from Proposition \ref{prop:conv_with_maximal_IC_neutral}, we obtain that 
    ${}^p H^0 (\epsilon_{\scrL}^2)$ is surjective. As we remarked before, ${}^p H^0 ((\scrK_{\Theta}^{\scrL})^{\star 2}) \cong \scrK_{\Theta}^{\scrL}$, hence, ${}^p H^0 (\mu_{\scrL}^2)$ is an isomorphism.
    In particular, ${}^p H^0 (\mu_{\scrL}^2 \circ \epsilon_{\scrL}^2)$ is nonzero which proves the claim.
  
    Since $\Hom (\scrK_{\Theta}^{\scrL}, \IC_e^{\scrL})$ is one-dimensional and since (\ref{eq:uniqueness_of_comult_1}) is nonzero, $\mu_{\scrL}^n$ is uniquely determined by some scalar which makes diagram (\ref{eq:uniqueness_of_comult_0}) commute.
\end{proof}

\begin{midsecproof}{Proposition \ref{prop:coalg_on_max_IC}}
    In light of Lemma \ref{lem:uniqueness_of_comult}, we only need to show that $\scrK_{\Theta}^{\scrL}$ is a coalgebra with comultiplication $\mu_{\scrL}^2$ and counit $\epsilon_{\scrL}$.
    This argument is essentially given in \cite[Proposition 6.6]{LY}. 
    The coassociativity axioms follows from the uniqueness of $\mu_{\scrL}^3$.
    The counit axioms follow from the commutativity of (\ref{eq:uniqueness_of_comult_0}).
\end{midsecproof}

\subsection{Soergel \texorpdfstring{$\H$}{H}-functor}

\subsubsection{Definition}

Let $\scrF, \scrG \in \D{\scrL}{\scrL}^\circ (\k)$. We have an isomorphism
\[\Hom^\bullet_{\D{\scrL}{\scrL}^{\circ}} (\scrF, \scrG) \cong H_{T \times T}^\bullet (\eFl, \RHom (\scrF, \scrG)).\]
As a result, there is an action of $R_{\k} = H_T^\bullet (\pt; \k)$ on both the left and right of $\Hom^\bullet (\scrF, \scrG)$ which makes $\Hom^\bullet (\scrF, \scrG)$ a graded $R_{\k}$-bimodule.

For a subset $\Theta \subset S_{\scrL}^\circ$ of finite type, we define the Soergel $\H$-functor,
\[\H_{\Theta}^{\scrL} : \D{\scrL}{\scrL}^{\circ, \Theta} (\k) \to \grbim{R_{\k}}, \hspace{1cm} \scrF \mapsto \Hom^{\bullet} (\scrK_{\Theta}^{\scrL}, \scrF).\]

\subsubsection{Behavior on Unit Object}

We first will describe the behavior of the $\H$-functor on the unit object.

\begin{lemma}\label{lem:H_functor_on_costds}
  There is a canonical isomorphism of graded $R$-bimodules,
  \[\H_{\Theta}^{\scrL} (\IC_e^{\scrL}) \cong R\]
  such that the map $\epsilon_{\scrL} : \scrK_{\Theta}^{\scrL} \to \IC_{e}^{\scrL}$ corresponds to $1 \in R$.
\end{lemma}
\begin{proof}
  By adjunction and Proposition \ref{prop:stalks_of_maximal_IC}, we obtain canonical isomorphisms of graded $R$-bimodules,
  \[\H_{\Theta}^{\scrL} (\IC_e^{\scrL}) \cong \Hom^\bullet (\scrK_e^{\scrL}, \scrK_e^{\scrL}) \cong H_{T}^\bullet (\pt; \k).\]
  The definition of $\epsilon_{\scrL}$ ensures that it identifies with $1 \in R$.
\end{proof}

\subsubsection{Monoidal Structure}

In this section, we will show that the $\H$-functor is monoidal after restricting to the category of parity sheaves. 

By Proposition \ref{prop:coalg_on_max_IC}, $\H_{\Theta}^{\scrL}$ is lax-monoidal. It will be useful to unpack what this entails.
Let $\scrF, \scrG \in \D{\scrL}{\scrL}^{\circ, \Theta} (\k)$.
For all $i, j \in \Z$, we can define the map
\begin{align*}
  \Hom (\scrK_{\Theta}^{\scrL}, \scrG [i]) \times \Hom (\scrK_{\Theta}^{\scrL}, \scrF [j]) &\stackrel{\star}{\to}  \Hom (\scrK_{\Theta
  }^{\scrL} \star \scrK_{\Theta}^{\scrL}, \scrG \star \scrF [i + j]) \\
  &\stackrel{(-) \circ \mu_{\scrL}^2 }{\longrightarrow} \Hom (\scrK_{\Theta}^{\scrL}, \scrG \star \scrF [i+j]).
\end{align*}
By taking direct sums over all $i,j$, we get an induced pairing
\[\H_{\Theta}^{\scrL} (\scrG) \times \H_{\Theta}^{\scrL} (\scrF) \to \H_{\Theta}^{\scrL} (\scrG \star \scrF).\]
The pairing is $R$-balanced and hence induces a morphism of graded $R$-bimodules
\begin{equation}\label{eq:H_mononoidal_map}
  c_{\Theta}^{\scrL} (\scrG, \scrF) : \H_{\Theta}^{\scrL} (\scrG) \otimes_R \H_{\Theta}^{\scrL} (\scrF) \to \H_{\Theta}^{\scrL} (\scrG \star \scrF).
\end{equation}
Note that $ c_{\Theta}^{\scrL} (\scrG, \scrF)$ is functorial in $\scrF, \scrG$, so we obtain a natural transformation of bifunctors
\[  c_{\Theta}^{\scrL} : \H_{\Theta}^{\scrL} (-) \otimes_R \H_{\Theta}^{\scrL} (-) \implies \H_{\Theta}^{\scrL} (- \star -) .\]
Since $\H_{\Theta}^{\scrL}$ is only lax-monoidal, the morphisms $c_{\Theta}^{\scrL} (\scrG, \scrF)$ are generally not isomorphisms. 
Our next goal will be to build up increasingly more general classes of objects
in which $c_{\Theta}^{\scrL} (\scrG, \scrF)$ is an isomorphism.

\begin{lemma}\label{lem:beta_conj_for_H_functor}
  Let $\beta \in \uW{\scrL'}{\scrL}$ be a block. Pick sheaves $\delta_{\beta}^{\scrL}$ and $\delta_{\beta^{-1}}^{\scrL'}$ in the isomorphism classes of $\IC_{w^{\beta}}^{\scrL}$ and $\IC_{w^{\beta, -1}}^{\scrL'}$ respectively.
  \begin{enumerate}
    \item The following diagram commutes up to natural isomorphism 
    \begin{equation}\label{eq:beta_conj_for_H_functor_1}
      \begin{tikzcd}
        {\D{\scrL}{\scrL}^{\circ, \Theta} (\k)} \arrow[r, "\H_{\Theta}^{\scrL}"] \arrow[d, "\delta_{\beta}^{\scrL} \star (-) \star \delta_{\beta^{-1}}^{\scrL'} "'] & \grbim{R_{\k}} \arrow[d, "{}^\beta (-)"] \\
        {\D{\scrL'}{\scrL'}^{\circ, \beta(\Theta)} (\k)} \arrow[r, "\H_{\beta (\Theta)}^{\scrL'}"]                                                                    & \grbim{R_{\k} }.                       
        \end{tikzcd}
    \end{equation}
     \item For $\scrF, \scrG \in \D{\scrL}{\scrL}^{\circ, \Theta} (\k)$, ${}^{\beta} (c_{\Theta}^{\scrL} (\scrG, \scrF))$ identifies with $c_{\Theta}^{\scrL'} (\delta_{\beta}^{\scrL} \star \scrG \star \delta_{\beta^{-1}}^{\scrL'}, \delta_{\beta}^{\scrL} \star \scrF \star \delta_{\beta^{-1}}^{\scrL'})$ via the commutativity of (\ref{eq:beta_conj_for_H_functor_1})
  \end{enumerate}
\end{lemma}
\begin{proof}
  Note that there is a unique isomorphism $\varphi : \scrK_{\beta (\Theta)}^{\scrL'} \to \delta_{\beta}^{\scrL} \star \scrK_{\Theta}^{\scrL} \star \delta_{\beta^{-1}}^{\scrL'}$ which makes the following diagram commute
  \begin{equation}
    \begin{tikzcd}
      \scrK_{\beta (\Theta)}^{\scrL'} \arrow[r] \arrow[d, "\epsilon_{\scrL'}"] & { \delta_{\beta}^{\scrL} \star \scrK_{\Theta}^{\scrL} \star \delta_{\beta^{-1}}^{\scrL'}} \arrow[d, "\id \star \epsilon_{\scrL} \star \id"] \\
      \IC_e^{\scrL'} \arrow[r, "\sim"]                                         & { \delta_{\beta}^{\scrL} \star \IC_e^{\scrL} \star \delta_{\beta^{-1}}^{\scrL'}.}                                                           
      \end{tikzcd}
  \end{equation} 
  After unpacking definitions, we see that $\varphi$ is enough to prove (1).

  By Proposition \ref{prop:coalg_on_max_IC} and the construction of $\varphi : \scrK_{\beta (\Theta)}^{\scrL'} \to \delta_{\beta}^{\scrL} \star \scrK_{\Theta}^{\scrL} \star \delta_{\beta^{-1}}^{\scrL'}$ in (1), we have that $\varphi$ is an isomorphism of coalgebras.
  This is enough to prove (2).
\end{proof}

\begin{lemma}\label{lem:xi_s_map}
  There is a unique map $\xi_s : \scrK_{\Theta}^{\scrL} \to \IC_s^{\scrL} [-1]$ making the following diagram commute
  \[\begin{tikzcd}
\scrK_{\Theta}^{\scrL} \arrow[rr, "\xi_s"] \arrow[rd, "\epsilon_{\scrL}"'] &               & {\IC_s^{\scrL} [-1]} \arrow[ld, "{\epsilon_s [-1]}"] \\
                                                                           & \IC_e^{\scrL}. &                                                     
\end{tikzcd}\]
\end{lemma}
\begin{proof}
  We can compute 
  \[\Hom (\scrK_{\Theta}^{\scrL}, \IC_s^{\scrL} [-1]) \cong \Hom (\scrK_{\Theta}^{\scrL} \star \IC_s^{\scrL}, \IC_e^{\scrL} [-1]) \cong \Hom (\scrK_{\Theta}^{\scrL} [-1] \oplus \scrK_{\Theta}^{\scrL} [1], \IC_e^{\scrL}) \cong \k.\]
  The first isomorphism follows from Corollary \ref{cor:biadjoint_of_conv_with_ICs}, the second isomorphism by Lemma \ref{lem:conv_ICw_and_ICs}, and the third isomorphism from Proposition \ref{prop:stalks_of_maximal_IC}.
  Now consider the map
  \begin{equation}\label{eq:xi_s_map_1}
    \Hom (\scrK_{\Theta}^{\scrL}, \IC_s^{\scrL} [-1]) \to \Hom (\scrK_{\Theta}^{\scrL}, \IC_e^{\scrL}).
  \end{equation}
  By Proposition \ref{prop:stalks_of_maximal_IC}, this is a morphism of 1-dimensional $\k$-vector spaces.
  The composition
  \[\scrK_{\Theta}^{\scrL} [-1] \oplus \scrK_{\Theta}^{\scrL} \twoheadrightarrow \scrK_{\Theta}^{\scrL} \stackrel{\epsilon_{\scrL}}{\to} \IC_e^{\scrL}\]
  under the biadjunction of $(-) \star \IC_s^{\scrL}$ becomes a map $\scrK_{\Theta}^{\scrL} \to \IC_e^{\scrL}$ whose image under (\ref{eq:xi_s_map_1}) is nonzero. Therefore, (\ref{eq:xi_s_map_1}) is an isomorphism.
\end{proof}

\begin{lemma}\label{lem:H_functor_for_internally_simple_reflns}
  Let $s \in W_{\scrL, \Theta}^\circ$ be an endosimple reflection.
  \begin{enumerate}
    \item Let $\scrF \in \D{\scrL}{\scrL}^{\circ, \Theta} (\k)$. There is an isomorphism of graded $R_{\k}$-bimodules
      \[\H_{\Theta}^{\scrL} (\scrF) \otimes_{R^s} R (1) \cong \H_{\Theta}^{\scrL} (\scrF \star \IC_s^{\scrL}).\]
    \item There is an isomorphism of graded $R_{\k}$-bimodules
      \[\H_{\Theta}^{\scrL} (\IC_s^{\scrL}) \cong B_s.\]
      \item The canonical map $\xi_s : \scrK_{\Theta}^{\scrL} \to \IC_s^{\scrL} [-1]$ corresponding to the counit map $\scrK_{\Theta}^{\scrL} [-1] \oplus \scrK_{\Theta}^{\scrL} [1] \twoheadrightarrow \scrK_{\Theta}^{\scrL} [-1] \stackrel{\epsilon_{\scrL}}{\to} \IC_e^{\scrL}$ 
      under biadjunction of $\IC_s^{\scrL}$ identifies with $1 \otimes 1$ in $B_s [-1]$.
  \end{enumerate}
\end{lemma}
\begin{proof}
  By Lemma \ref{lem:make_refln_simple} and Lemma \ref{lem:beta_conj_for_H_functor}, it suffices to consider $s \in W_{\scrL, \Theta}^{\circ}$ a simple reflection. 
  Let $\scrF \in \D{\scrL}{\scrL}^{\circ, \Theta} (\k)$. By assumption, $\scrF$ is supported on $\eFl_{\leq w_{\scrL, \Theta}}$. We will write $X = \eFl_{\leq w_{\scrL, \Theta}} / T$ and $X^s = \eFl_{\leq w_{\scrL, \Theta}} / L_s$.

  \emph{(1):}
  We can extend $\scrL$ uniquely to $\scrL^s \in \Ch (L_s, \k)$ via Lemma \ref{lem:extending_character_sheaves_to_levis}.
  By Lemma \ref{lem:pi_s_descent}, we get an isomorphism,
  \begin{align}
    \H_{\Theta}^{\scrL} (\scrF) &\cong \Hom^\bullet (\IC_{w_{\scrL, \Theta}}^{\scrL^s} [-\ell_{\scrL} (w_{\scrL, \Theta}) + 1], \pi_{s*} \scrF) \notag \\
    &\cong H^\bullet_{T} ( X^s, \RHom (\IC_{w_{\scrL,\Theta}}^{\scrL^s} [-\ell_{\scrL} (w_{\scrL, \Theta}) + 1], \pi_{s*} \scrF)). \label{eq:H_functor_for_simple_reflns_1}
  \end{align}
  In particular, $\H_{\Theta}^{\scrL} (\scrF)$ naturally has the structure of a graded $(R, R^s)$-bimodule via the identification $R^s \cong H^\bullet_{L_s} (\pt; \k)$.

  For $\scrG \in D (T \backslash X^s , \k)$, the pullback along the morphism $\pi_s : X \to X^s$ induces an isomorphism of graded $R$-bimodules,
  \[H^\bullet_{T} (X, \pi_s^* \scrG) \cong H^\bullet_T (X^s, \scrG \otimes^L \pi_{s*} \underline{\k}) \cong H^\bullet_T (X^s, \scrG) \otimes_{R^s} R .\]
  We can apply the above isomorphism to (\ref{eq:H_functor_for_simple_reflns_1}) to obtain isomorphisms
  \begin{align*}
    \H_{\Theta}^{\scrL} (\scrF) \otimes_{R^s} R (1) &\cong H^\bullet_{T} (X^s, \pi_s^* \RHom (\IC_{w_{\scrL,\Theta}}^{\scrL^s} [-\ell_{\scrL} (w_{\scrL, \Theta}) + 1], \pi_{s*} \scrF [1])) \\
    &\cong H^\bullet_{T} (X, \RHom (\pi_s^* \IC_{w_{\scrL, \Theta}}^{\scrL^s} [-\ell_{\scrL} (w_{\scrL, \Theta}) + 1], \pi_s^* \pi_{s*} \scrF [1]) ) \\
    &\cong H^\bullet_{T} (X, \RHom (\scrK_{\Theta}^{\scrL}, \scrF \star \IC_s^{\scrL}) ) \\
    &\cong \H_{\Theta}^{\scrL} (\scrF \star \IC_s^{\scrL} ).
  \end{align*}
  Note here we are using the push-pull lemma (Lemma \ref{lem:conv_with_ICs_rewritting}).

  \emph{(2): } The statement easily follows from (1) and Lemma \ref{lem:H_functor_on_costds} by taking $\scrF = \IC_e^{\scrL}$.
  
  \emph{(3): } By Lemma \ref{lem:xi_s_map}, $\xi_s$ is the unique morphism $\scrK_{\Theta}^{\scrL} \to \IC_s^{\scrL} [-1]$ which satisfies $\epsilon_s \circ \xi_s = \epsilon_{\scrL}$.
  Then $\H_{\Theta}^{\scrL} (\xi_s)$ is the unique element of $R \otimes_{R^s} R$ such that $m_s (\xi_s) = 1$ where $m : R \otimes_{R^s} R \to R$ is the multiplication map.
  As a result, $\H_{\Theta}^{\scrL} (\xi_s) = 1 \otimes 1$.
\end{proof}

\begin{lemma}\label{lem:monoidality_for_ICs}
  Let $s \in W_{\scrL, \Theta}^\circ$ be an endosimple reflection. Let $\scrF \in \D{\scrL}{\scrL}^{\circ, \Theta} (\k)$.
  Then the map $c_\Theta^{\scrL} (\scrF, \IC_s^{\scrL}) : \H_{\Theta}^{\scrL} (\scrF) \otimes_R \H_{\Theta}^{\scrL} (\IC_s^{\scrL}) \to \H_{\Theta}^{\scrL} (\scrF \star \IC_s^{\scrL})$ is an isomorphism.
\end{lemma}
\begin{proof}
  By Lemma \ref{lem:H_functor_for_internally_simple_reflns}, we have an isomorphism
  \[\mu_{\scrF, s} : \H_{\Theta}^{\scrL} (\scrF) \otimes_R \H_{\Theta}^{\scrL} (\IC_s^{\scrL}) \cong \H_{\Theta}^{\scrL} (\scrF) \otimes_{R^s} R(1) \cong \H_{\Theta}^{\scrL} (\scrF
  \star \IC_s^{\scrL}).\]
  It remains to show that $\mu_{\scrF, s}$ agrees with $c_{\Theta}^{\scrL} (\scrF, \IC_s^{\scrL})$.
  By Lemma \ref{lem:beta_conj_for_H_functor}, it suffices to show this for $s \in S$ a simple reflection.

  The proof then follows from the same argument given in \cite[Lemma 7.7]{LY} after the obvious replacements are made. 
\end{proof}

\begin{proposition}\label{prop:monoidality_of_H_functor}
  Let $\scrF, \scrG \in \D{\scrL}{\scrL}^{\circ, \Theta} (\k)$.
  The map $c_{\Theta}^{\scrL} (\scrG, \scrF)$ is an isomorphism if either $\scrF$ or $\scrG$ is a parity sheaf.
  As a result, the restriction of the Soergel functor 
  \[ \H_{\Theta}^{\scrL} : \Parity{\scrL}{\scrL}^{\circ, \Theta} (\k) \to \grbim{R_{\k}} \]
  is monoidal.
\end{proposition}
\begin{proof}
  Since the statement is symmetric, we may assume that $\scrG$ is a parity sheaf. 
  By Theorem \ref{thm:existence_of_parity}, it suffices to prove the isomorphism for $\scrG = \IC_{s_1}^{\scrL} \star \ldots \star \IC_{s_k}^{\scrL}$ where $s_1, \ldots, s_k$ are endosimple reflections in $W_{\scrL, \Theta}^{\circ}$.
  The proposition then follows from a simple inductive argument on $k$ using Lemma \ref{lem:monoidality_for_ICs}
\end{proof}

\subsection{Comparison with Soergel Bimodules}

We are ready to state the main result of the section.
 
\begin{theorem}\label{thm:H_functor_properties_par}
  The $\H$-functor restricts to a monoidal equivalence of categories
  \[\H_{\Theta}^{\scrL} : \Parity{\scrL}{\scrL}^{\circ, \Theta} (\k) \to \SBim_{W_{\scrL, \Theta}^\circ} (\fr{h}_{\k})\]
  such that $\H_{\Theta}^{\scrL} (\IC_w^{\scrL} [n]) \cong B_w (n)$ for all $w \in W_{\scrL, \Theta}^{\circ}$ and $n \in \Z$.
\end{theorem}

Note that by Proposition \ref{prop:monoidality_of_H_functor} and Lemma \ref{lem:H_functor_for_internally_simple_reflns}, 
\[ \H_{\Theta}^{\scrL} : \Parity{\scrL}{\scrL}^{\circ, \Theta} (\k) \to \grbim{R_{\k}} \] 
is monoidal and satisfies $\H_{\Theta}^{\scrL} (\IC_s^{\scrL}) \cong B_s$ for all $s \in S_{\scrL}^{\circ}$.
These observations together imply that $\H_{\Theta}^{\scrL} (\IC_w^{\scrL}) \cong B_w$ for all $w \in W_{\scrL, \Theta}^{\circ}$.
As a result, to prove Theorem \ref{thm:H_functor_properties_par}, it suffices to show that $\H_{\Theta}^{\scrL}$ is fully faithful.
Before proving full faithfulness, we need some preliminary lemmas.

\begin{lemma}\label{lem:par_H_functor_ff_base_case}
  Let $\scrG \in \Parity{\scrL}{\scrL}^{\circ, \Theta} (\k)$. The natural map
  \[ \Hom (\IC_e^{\scrL}, \scrG) \to \Hom (R, \H_{\Theta}^{\scrL} (\scrG))\]
  induced by $\H_{\Theta}^{\scrL}$ is an isomorphism.
\end{lemma}
\begin{proof}
  The lemma follows from the proof of \cite[Lemma 7.10]{LY} after the obvious modifications are made.
\end{proof}

\begin{lemma}\label{lem:par_compat_of_self_adjointness} 
  Let $\scrF, \scrG \in \Parity{\scrL}{\scrL}^{\circ, \Theta} (\k)$. Let $s \in \Theta$. There is a natural commutative diagram
  \begin{equation}\label{eq:par_compat_of_self_adjointness_1}
    \begin{tikzcd}
      {\Hom (\scrF \star \IC_s^{\scrL}, \scrG)} \arrow[rr, "\sim"] \arrow[d, "{\H_{\Theta}^{\scrL}}"]                                                               &  &
      {\Hom (\scrF, \scrG \star \IC_s^{\scrL})} \arrow[d, "{\H_{\Theta}^{\scrL}}"]                                                                               \\
      {\Hom (\H_{\Theta}^{\scrL} (\scrF \star \IC_s^{\scrL}), \H_{\Theta}^{\scrL} (\scrG))} \arrow[d, "{(-)\circ c_{\Theta}^{\scrL} (\scrF, \IC_s^{\scrL})}"] &  &
      {\Hom (\H_{\Theta}^{\scrL} (\scrF ), \H_{\Theta}^{\scrL} (\scrG \star \IC_s^{\scrL}) )} \arrow[d, "{c_\Theta^{\scrL} (\scrG, \IC_s^{\scrL})^{-1} \circ (-)}"] \\
      {\Hom (\H_{\Theta}^{\scrL} (\scrF) \otimes_R B_s, \H_{\Theta}^{\scrL} (\scrG))} \arrow[rr, "\sim"]                                                             &  &
      {\Hom (\H_{\Theta}^{\scrL} (\scrF) , \H_{\Theta}^{\scrL} (\scrG) \otimes_R B_s),}
    \end{tikzcd}
  \end{equation}
  where the horizontal maps are given by the self-adjunction of $(-) \star \IC_s^{\scrL}$ and $(-) \otimes_R B_s$.
\end{lemma}
\begin{proof}
  We can factor the diagram (\ref{eq:par_compat_of_self_adjointness_1}) as follows:
  \[\adjustbox{scale=0.7, center}{\begin{tikzcd} 
    {\Hom (\scrF \star \IC_s^{\scrL}, \scrG)} \arrow[d, "\H_{\Theta}^{\scrL}"'] \arrow[r, "(-) \star \IC_s^{\scrL}"]                                                                                                  & {\Hom^\bullet (\scrF \star \IC_s^{\scrL} \star \IC_s^{\scrL}, \scrG \star \IC_s^{\scrL})} \arrow[r] \arrow[d, "\H_{\Theta}^{\scrL}"]                       & {\Hom^\bullet (\scrF, \scrG \star \IC_s^{\scrL})} \arrow[d, "\H_{\Theta}^{\scrL}"]                                                                                    \\
    {\Hom (\H_{\Theta}^{\scrL} (\scrF \star \IC_s^{\scrL}), \H_{\Theta}^{\scrL} (\scrG))} \arrow[d, "{(-)\circ c_{\Theta}^{\scrL} (\scrF, \IC_s^{\scrL})}"'] \arrow[r, "\H_{\Theta}^{\scrL} (- \star \IC_s^{\scrL})"] & {\Hom^\bullet (\H_{\Theta}^{\scrL} (\scrF \star \IC_s^{\scrL} \star \IC_s^{\scrL}), \H_{\Theta}^{\scrL} ( \scrG \star \IC_s^{\scrL}))} \arrow[r] \arrow[d] & {\Hom^\bullet (\H_{\Theta}^{\scrL} (\scrF ), \H_{\Theta}^{\scrL} (\scrG \star \IC_s^{\scrL}) )} \arrow[d, "{c_\Theta^{\scrL} (\scrG, \IC_s^{\scrL})^{-1} \circ (-)}"] \\
    {\Hom (\H_{\Theta}^{\scrL} (\scrF) \otimes_R B_s, \H_{\Theta}^{\scrL} (\scrG))} \arrow[r, "(-) \otimes_R B_s"]                                                                                                    & {\Hom^\bullet (\H_{\Theta}^{\scrL} (\scrF) \otimes_R B_s \otimes_R B_s, \H_{\Theta}^{\scrL} (\scrG) \otimes_R B_s)} \arrow[r]                              & {\Hom^\bullet (\H_{\Theta}^{\scrL} (\scrF) , \H_{\Theta}^{\scrL} (\scrG) \otimes_R B_s),}                                                                             
    \end{tikzcd}}\]
    where the horizontal morphisms on the right are given by the unit morphisms of adjunction.
    The top two squares commute by functorality of $\H_{\Theta}^{\scrL}$. The bottom two squares commute by Proposition \ref{prop:monoidality_of_H_functor}.
\end{proof}

\begin{midsecproof}{Theorem \ref{thm:H_functor_properties_par}} 
  Consider the morphism
  \[\alpha (\scrF, \scrG) : \Hom (\scrF, \scrG) \to \Hom (\H_{\Theta}^{\scrL} (\scrF), \H_{\Theta}^{\scrL} (\scrG)).\]
  For each expression $\uw = (s_1, \ldots, s_k)$ of endosimple reflections in $\Theta$, define a parity sheaf
  \[\scrF_{\uw}^{\scrL} \coloneq \IC_{s_1}^{\scrL} \star \ldots \star \IC_{s_k}^{\scrL}.\]
  By Theorem \ref{thm:existence_of_parity} and Theorem \ref{thm:conv_preserves_parity}, every indecomposable parity sheaf occurs as a direct summand of some $\scrF_{\uw}^{\scrL}$. 
  It suffices to prove that $\alpha (\scrF, \scrG)$ is an isomorphism for $\scrF = \scrF_{\uw}^{\scrL}$.
  We will prove $\alpha (\scrF_{\uw}^{\scrL}, \scrG)$ is an isomorphism by induction on $k$. The base case of $k = 0$ is covered by Lemma \ref{lem:par_H_functor_ff_base_case}.

  Let $\uw' = (s_1, \ldots, s_{k-1})$ and write $s = s_k$. 
  Note that $\scrF_{\uw}^{\scrL} = \scrF_{\uw'}^{\scrL} \star \IC_{s}^{\scrL}$.
  By Lemma \ref{lem:par_compat_of_self_adjointness}, we get a commutative diagram
  \begin{equation}
    \begin{tikzcd}
      {\Hom (\scrF_{\uw'}^{\scrL} \star \IC_s^{\scrL}, \scrG)} \arrow[rr, "\sim"] \arrow[d, "{\alpha (\scrF_{\uw}^{\scrL}, \scrG)}"']
      &  & {\Hom (\scrF_{\uw'}^{\scrL}, \scrG \star \IC_s^{\scrL})} \arrow[d, "{\alpha (\scrF_{\uw'}^{\scrL}, \scrG \star \IC_s^{\scrL})}"]                                                                               \\
      {\Hom (\H_{\Theta}^{\scrL} (\scrF_{\uw'}^{\scrL} \star \IC_s^{\scrL}), \H_{\Theta}^{\scrL} (\scrG))} \arrow[d, "\sim", "\ref{prop:monoidality_of_H_functor}"'] &  &
      {\Hom (\H_{\Theta}^{\scrL} (\scrF_{\uw'}^{\scrL} ), \H_{\Theta}^{\scrL} (\scrG \star \IC_s^{\scrL}) )} \arrow[d, "\sim", "\ref{prop:monoidality_of_H_functor}"'] \\
      {\Hom (\H_{\Theta}^{\scrL} (\scrF_{\uw'}^{\scrL}) \otimes_R B_s, \H_{\Theta}^{\scrL} (\scrG))} \arrow[rr, "\sim"]
      &  & {\Hom (\H_{\Theta}^{\scrL} (\scrF_{\uw'}^{\scrL}) , \H_{\Theta}^{\scrL} (\scrG) \otimes_R B_s).}
    \end{tikzcd}
  \end{equation}
  The top right vertical map is an isomorphism by induction; therefore, we must have that $\alpha (\scrF_{\uw}^{\scrL}, \scrG)$ is an isomorphism. We can then conclude that $\H_{\Theta}^{\scrL}$ is fully faithful.
\end{midsecproof}

\subsection{Right Equivariant Soergel Functor}

We again keep the assumption that $\k$ is a characteristic 0 field. 
Let $\Theta \subseteq S_{\scrL}^{\circ}$. We define a full subcategory $\PME{\scrL}{\scrL}^{\circ, \Theta} (\k)$ of $\PME{\scrL}{\scrL}^{\circ} (\k)$  generated under direct sums and shifts by the objects $\scrE_{w}^{\scrL}$ for $w \in W_{\scrL, \Theta}^{\circ}$.
We define the right-equivariant Soergel functor
\[\tilde{\H}_{\Theta}^{\scrL} : \PME{\scrL}{\scrL}^{\circ, \Theta} (\k) \to \rgrmod{R}, \qquad \scrF \mapsto \Hom_{\PME{\scrL}{\scrL}^{\circ, \Theta} (\k)}^{\bullet} (\ForME{\scrL} (\scrK_{\Theta}^{\scrL}), \scrF ).\]
It follows from definitions that the following diagram commutes up to natural isomorphisms
\begin{equation}\label{eq:soergel_and_for}
  \begin{tikzcd}
{\PEE{\scrL}{\scrL}^{\circ, \Theta} (\scrQ)} \arrow[d, "\ForME{\scrL}"'] \arrow[r, "\H_{\Theta}^{\scrL}"] & \grbim{R} \arrow[d, "\k \otimes_{R}(-)"] \\
{\PME{\scrL}{\scrL}^{\circ, \Theta} (\scrQ)} \arrow[r, "\tilde{\H}_{\Theta}^{\scrL}"]                     & \rgrmod{R}                                                            
\end{tikzcd}
\end{equation}

\begin{proposition}\label{prop:right_equiv_H_functor_properties}
 The $\tilde{\H}$-functor
  \[\tilde{\H}_{\Theta}^{\scrL} : \PME{\scrL}{\scrL}^{\circ, \Theta} (\k) \to \rgrmod{R}\]
  is fully faithful and satisfies $\H_{\Theta}^{\scrL} (\IC_w^{\scrL} [n]) \cong \k \otimes_R B_w (n)$ for all $w \in W_{\scrL, \Theta}^{\circ}$ and $n \in \Z$.
\end{proposition}
\begin{proof}
  The fact that $\H_{\Theta}^{\scrL} (\IC_w^{\scrL} [n]) \cong \k \otimes_R B_w (n)$ for all $w \in W_{\scrL, \Theta}^{\circ}$ and $n \in \Z$ follows from Theorem \ref{thm:H_functor_properties_par} and (\ref{eq:soergel_and_for}).
  It remains to show that $\tilde{\H}_{\Theta}^{\scrL}$ is fully faithful.

  Note that $\ForME{\scrL} : \PEE{\scrL}{\scrL}^{\circ, \Theta} (\k) \to \PME{\scrL}{\scrL}^{\circ, \Theta} (\k)$ is essentially surjective. In particular, it suffices to prove that
  \begin{align*}
    \Hom (\ForME{\scrL} (\scrF), \ForME{\scrL} (\scrG)) &\to \Hom_{\rgrmod{R}} (\tilde{\H}_{\Theta}^{\scrL} (\ForME{\scrL} (\scrF)), \tilde{\H}_{\Theta}^{\scrL} (\ForME{\scrL} (\scrG))) \\
    &\to \Hom_{\rgrmod{R}} (\k \otimes_R \H_{\Theta}^{\scrL} (\scrF), \k \otimes_R \H_{\Theta}^{\scrL} (\scrG))
  \end{align*}
  is an isomorphism. By Lemma \ref{lem:eos_for_right_equiv} and \cite[Theorem 5.15]{So07}, this map can be identified with the map
  \[\k \otimes_R \H_{\Theta}^{\scrL} : \Hom_{\Parity{\scrL}{\scrL}^{\circ, \Theta} (\k)} (\scrF, \scrG) \to  \k \otimes_R \Hom_{\grbim{R}} (\H_{\Theta}^{\scrL} (\scrF), \H_{\Theta}^{\scrL} (\scrG)).\]
  This is an isomorphism by Theorem \ref{thm:H_functor_properties_par}.
\end{proof}

        \section{Endoscopic-Monodromic Equivalences}\label{sec:endo}

\subsection{Coefficient Rings}\label{subsec:integral_lifts}

Let $\k$ be a noetherian ring of finite global dimension which admits a ring homomorphism $\Z' \to \k$. 
Our goal is to find a characteristic 0 noetherian ring $\scrZ$ of finite global dimension which acts like an ``integral'' version $\k$.
Of course, one could take $\scrZ = \Z'$, but this introduces a problem: most local systems $\scrL \in \Ch^{\circ} (T, \k)$ fail to lift to local systems in $\Ch^{\circ} (T, \Z')$.
To fix this, we will construct a characteristic 0 noetherian domain of finite global dimension $\scrZ$ and ring homomorphism $\scrZ \to \k$ such that $\scrL$ lifts to a good $\scrZ$-local system.

Fix $\scrL \in \Ch^{\circ} (T, \k)$. The data of this local systems is equivalent to a group homomorphism
\[\rho_{\scrL} : \bfY = \pi_1 (T) \to \k^{\times}.\]
Denote by $K$ the kernel of $\rho_{\scrL}$. Since $\k$ is a domain, there is a non-canonical isomorphism $\bfY/K \cong \Z^r \times C_m$ where $r,m \in \Z_{\geq 0}$ and $C_m$ denotes the cyclic group of order $m$.
We can then pick an injective group homomorphism $\iota : \bfY/K \hookrightarrow \C^{\times}$. Explicitly, this is given by picking $r$ algebraically independent transcendental numbers in $\C$ and a primitive $m$-th root of unity $\zeta_m$.
We can then define $\scrZ'$ as the sub-$\Z'$-algebra of $\C$ generated by the image of $\iota$. 
Under the isomorphism $\bfY/K \cong \Z^r \times C_m$, we can describe $\scrZ'$ as follows,
\[\scrZ' \cong \Z' [\zeta_m] [y_1^{\pm 1}, \ldots, y_r^{\pm}],\]
where $\Z' [\zeta_m]$ is the ring of cyclotomic integers which is localized at $2$ if $\Z' = \Z [\frac{1}{2}]$. Consider the set
\[J = \{ \alpha \in \Phi_{\re} \mid \rho_{\scrL} (\alpha^{\vee}) \neq 1 \}.\]
Denote by $\scrZ$ the localization of $\scrZ'$ at the multiplicative set generated by $\iota ([\alpha^{\vee}]) - 1$ for $\alpha \in J$ (here $[\alpha^{\vee}]$ is the coset of $\alpha^{\vee}$ in $\bfY/K$).

\begin{lemma}\label{lem:existence_of_integral_version}
  Follow the same notation and setup as above.
  \begin{enumerate}
    \item There exists a ring homomorphism $\varphi : \scrZ \to \k$ and a good local system $\scrL_{\scrZ} \in \Ch^{\circ} (T, \scrZ)$ such that $\k (\scrL_{\scrZ}) = \scrL$.
    \item The ring $\scrZ$ is a characteristic 0 noetherian domain of finite global dimension.
  \end{enumerate}
\end{lemma}
\begin{proof}
  We will first construct the ring homomorphism $\varphi : \scrZ \to \k$.
  The morphism $\rho_{\scrL}$ factors through a morphism $\rho_{\k} : \Z' [\bfY / K] \to \k$. Similarly, by construction $\iota$ factors through a morphism $\rho_{\C} : \Z' [\bfY / K] \to \scrZ'$.
  Let $\gamma \in \bfY / K$ be a generator for the (cyclic) torsion subgroup. One can readily check that the kernel of $\rho_{\C}$ is generated by the cyclotomic polynomial $\Phi_m (\gamma)$.
  Since $\rho_{\k} (\gamma)$ is also an $m$-th root of unity, we must have that $\rho_{\k} (\Phi_m (\gamma)) = 0$. As a result, $\rho_{\scrL}$ factors through a morphism $\varphi' : \scrZ' \to \k$.
  Finally, we can check that $\varphi' ([\alpha^{\vee}] - 1) \in \k^{\times}$ for $\alpha \in J$. Indeed, following definitions 
  \[\varphi' ([\alpha^{\vee}] - 1) = \rho_{\scrL} (\alpha^{\vee}) - 1. \]
Since $\scrL$ is good, we can conclude that the $\rho_{\scrL} (\alpha^{\vee}) - 1 \in \k^{\times}$. Therefore, $\varphi'$ factors through a ring homomorphism $\varphi : \scrZ \to \k$ as desired.
We summarize that the following diagram is commutative
\begin{equation}\label{eq:existence_of_integral_version_1}
\begin{tikzcd}
{\Z [\bfY]} \arrow[r] \arrow[rrrr, "\rho_{\scrL}"', bend right] & {\Z' [\bfY/K]} \arrow[r, "\rho_{\C}"] & \scrZ' \arrow[r, hook] & \scrZ \arrow[r, "\varphi"] & \k.
\end{tikzcd}
\end{equation}
The composite $\rho_{\scrL_{\scrZ}} : \Z[\bfY] \to \scrZ$ given by the arrows in (\ref{eq:existence_of_integral_version_1}) defines a local system $\scrL_{\scrZ} \in \Ch (T, \scrZ)$ such that $\k (\scrL_{\scrZ}) = \scrL$ as desired.

We will now check that $\scrL_{\scrZ}$ is good. This is equivalent to the condition that $\rho_{\scrL_{\scrZ}} (\alpha^{\vee}) = 1$ or $\rho_{\scrL_{\scrZ}} - 1 \in \scrZ^{\times}$.
Assume that  $\rho_{\scrL_{\scrZ}} (\alpha^{\vee}) \neq 1$. Note that we must have $[\alpha^{\vee}] \neq 1 \in \Z' [\bfY/K]$. In other words, $\alpha^{\vee}\notin K$, and thus, $\rho_{\scrL} (\alpha^{\vee}) \neq 1$.
Therefore, $\alpha \in J$ which implies that $\rho_{\scrL_{\scrZ}} - 1$ is invertible in $\scrZ$.

Finally, we will verify that $\scrZ$ is a characteristic 0 noetherian domain of finite global dimension.
Since $\scrZ$ is a sub-ring of $\C$, it must be a characteristic 0 domain. Now, $\Z' [\zeta_m]$ is a Dedekind domain. In particular, it is noetherian and has finite global dimension.
Since $\scrZ$ is a localization of a polynomial ring over $\Z' [\zeta_m]$, Hilbert's basis theorem and Hilbert's syzygy theorem imply that $\scrZ$ is also noetherian and has finite global dimension.
\end{proof}

Lemma \ref{lem:existence_of_integral_version} implies that the induced map $\k (-) : \Ch^{\circ} (T, \scrZ) \to \Ch^{\circ} (T, \k)$ surjects onto the $W$-orbit of $\scrL$.

\subsection{Endoscopic Groups}\label{subsec:endo_gps}

We have seen that $(W_{\scrL}^{\circ}, S_{\scrL}^{\circ})$ is a Coxeter group with realization $\fr{h}$.
We would like $W_{\scrL}^{\circ}$ to arise as the Weyl group of a Kac--Moody group with maximal torus $T$. This amounts to two conditions:
\begin{enumerate}
  \item $W_{\scrL}^{\circ}$ is crystallographic;
  \item the quadruple $(S_{\scrL}^{\circ}, \bfX, \{ \alpha_s\}_{s \in S_{\scrL}^{\circ}}, \{\alpha_s^{\vee}\}_{s \in S_{\scrL}^{\circ}})$ is a Kac--Moody root datum.
\end{enumerate} 
Condition (1) was shown to hold in \cite[Lemma 3.5]{H23}; however, it turns out that (2) does not always hold. In fact, when $W$ is of indefinite type, then $W_{\scrL}^{\circ}$ can fail to be of finite rank.

\begin{example}[{\cite[Lemma 3.11]{H23}}]\label{ex:endosimple_not_finite}
  Consider the generalized Cartan matrix given by
  \[A = \begin{bmatrix} 2 & -2 & -2 & -2 \\ -2 & 2 & -2 & -2 \\ -2 & -2 & 2 & -2 \\ -2 & -2 & -3 & 2 \end{bmatrix}.\]
   To $A$, we associate a Kac--Moody group $G$ along with Borel $B$ and maximal torus $T$. Let $\fr{h}$ denote the Lie algebra of $T$. Let $W = \langle s_1, s_2, s_3, s_4 \rangle$ denote the Weyl group of $G$.
   Note that $A$ is invertible. As a result, we may assume that $\fr{h} \cong \bigoplus_{i=1}^4 \C \alpha_i^{\vee}$ where the $\alpha_i^{\vee}$'s are the simple coroots.
  Define a $\C$-linear map $f : \fr{h} \to \C$ by $h ( \alpha_i^{\vee} ) = 1$ for all $i=1,2,3,4$.
  We can then define a group homomorphism $\chi : \bfY \to \C^{\times}$ by $\chi (\phi) = (-1)^{f (\phi)}$.
  Note that $\chi$ corresponds with a multiplicative local system $\scrL_{\chi} \in \Ch^{\circ} (T, \C)$.
  Moreover, we have that $\alpha^{\vee} \in \Phi_{\re, \scrL_{\chi}}^{\vee}$ if and only if $\chi (\alpha^{\vee}) = 1$.

  Let $W' = \langle s_1, s_2 \rangle \subset W$. We claim that for all $w \in W'$ that $ws_3 \cdot \alpha_4^{\vee} \in \Phi_{\re, \scrL_{\chi}}^{\vee}$.
  In order to see this, note that $s_3 \cdot \alpha_4^{\vee} = \alpha_4^{\vee} + 3 \alpha_3^{\vee}$. 
  Moreover, if $i = 1$ or $i=2$, then $s_i \cdot (\alpha_4^{\vee} + 3 \alpha_3^{\vee} + 2 \bfY) \subseteq (\alpha_4^{\vee} + 3 \alpha_3^{\vee} + 2 \bfY)$.
  As a result, $f (ws_3 \cdot \alpha_4^{\vee}) = 1$ and hence $ws_3 \cdot \alpha_4^{\vee} \in \Phi_{\re, \scrL_{\chi}}^{\vee}$ for all $w \in W'$.

  Next, we will show that $v = ws_3s_4s_3w^{-1}$, the reflection about $ws_3 \cdot \alpha_4^{\vee}$, is endosimple.
  For all $i \in \{1,2,3,4\}$ and $j \in \{1,2,3\}$, we have that $s_j \cdot (\alpha_i^{\vee} + 2 \bfY) \subseteq \alpha_i^{\vee} + 2 \bfY$.
  As a result, for all $u \in W$, $f(u \cdot \alpha_j^{\vee}) = -1$, and so $u\alpha_j^{\vee} \notin \Phi_{\re, \scrL}^{\vee}$.
  By \cite[Lemma 1.3.14]{Ku}, $\{ \alpha \in \Phi^{\vee, +} \mid v \alpha < 0 \}$ consists of elements of the form $u \cdot \alpha_j$ for $j \in \{1,2,3\}$ for all but one element.
  In particular, 
  \[\{ \alpha \in \Phi_{\re, \scrL}^{\vee, +} \mid v \alpha < 0 \} =  \{ \alpha \in \Phi^{\vee, +} \mid v \alpha < 0 \} \cap \Phi_{\re, \scrL} = \{ws_3 \alpha_4^{\vee}\}.\]
  By Lemma \ref{lem:endosimple_and_dyer}, we conclude that $ws_3s_4s_3w^{-1}$ is endosimple for all $w \in W'$. 
  Since $W'$ is infinite, we conclude that $W_{\scrL_{\chi}}^{\circ}$ has infinite rank as a Coxeter group. 
  By \cite[Remark 2.5 (2)]{H22}, it follows that there is no finite rank presentation of $W_{\scrL_{\chi}}^{\circ}$ either.
\end{example}

In light of Example \ref{ex:endosimple_not_finite}, whenever we discuss the endoscopic Kac--Moody group, we will assume that $G$ is either of finite or affine type.
We review the construction of endoscopic groups which can be found in \cite{LY} in finite type or in \cite{Li} in affine type.
Let $\scrL \in \Ch^{\circ} (T, \k)$. To $\scrL$ we associated a collection of real roots $\Phi_{\textnormal{re}, \scrL} \subset \bfX$, (endo)simple roots $\{\alpha_s\}_{s \in S_{\scrL}^{\circ}} \subset \bfX$, and (endo)simple coroots  $\{\alpha_s^{\vee}\}_{s \in S_{\scrL}^{\circ}} \subset \bfY$.
The subgroup $W_{\scrL}^\circ$ of $W$ generated by reflections in $\Phi_{\textnormal{re}, \scrL}$ is canonically a crystallographic Coxeter group of finite rank.
We can then consider the Kac--Moody root datum $(S_{\scrL}^{\circ}, \bfX, \{\alpha_s\}_{s \in S_{\scrL}^{\circ}}, \{\alpha_s^{\vee}\}_{s \in S_{\scrL}^{\circ}})$ for the generalized Cartan matrix $(\alpha_s (\alpha_t^{\vee}))_{t,s \in S_{\scrL}^{\circ}}$.
This root datum gives rise to a (connected) complex Kac--Moody group $H_{\scrL}^{\circ}$ with canonical Borel subgroup $B_{\scrL}$.
The maximal torus of $B_{\scrL}$ is $T$, and the Weyl group of $H_{\scrL}^{\circ}$ is $W_{\scrL}^{\circ}$.
We write $U_{\scrL}$ for the pro-unipotent radical of $B_{\scrL}$. The group $H_{\scrL}^{\circ}$ is called the \emph{endoscopic group} of $G$ corresponding to $\scrL$.

We let $\Fl_{H_{\scrL}^{\circ}} \coloneq B_{\scrL} \backslash H_{\scrL}^{\circ}$ (resp. $\eFl_{H_{\scrL}^{\circ}} \coloneq U_{\scrL} \backslash H_{\scrL}^{\circ}$) denote the flag variety (resp. enhanced flag variety) of $H_{\scrL}^{\circ}$ with respect to $B_{\scrL}$.
The Bruhat stratification gives a decomposition of the flag variety into locally closed varieties,
\[\Fl_{H_{\scrL}^{\circ}} = \bigsqcup_{w \in W_{\scrL}^{\circ}} \Fl_{H_{\scrL}^{\circ}, w},\]
where $\Fl_{H_{\scrL}^{\circ}, w}$ is isomorphic to an affine space of dimension $\ell_{\scrL} (w)$. The Bruhat stratification lifts to a stratification of the enhanced flag variety,
\[\eFl_{H_{\scrL}^{\circ}} = \bigsqcup_{w \in W_{\scrL}^{\circ}} \eFl_{H_{\scrL}^{\circ}, w},\]
where $\eFl_{H_{\scrL}^{\circ}, w} \cong \A^{\ell_{\scrL} (w)} \times T$.
We write $j_w^H : \Fl_{H_{\scrL}^{\circ}, w} \hookrightarrow \Fl_{H_{\scrL}^{\circ}}$ for the inclusion maps, and we use the same notation for the enhanced flag variety variants.
The closure of $\Fl_{H_{\scrL}^{\circ}, w}$ is given by $\Fl_{H_{\scrL}^{\circ}, \leq w} = \bigcup_{x \leq_{\scrL} w} \Fl_{H_{\scrL}^{\circ}, x}$.
It is also useful to consider the complement of the closure which is given by $\Fl_{H_{\scrL}^{\circ}, > w} = \bigcup_{x >_{\scrL} w} \Fl_{H_{\scrL}^{\circ}, x}$. There are obvious variants of these Schubert varieties for the enhanced flag variety which we denote similarly.

Let $D_{\cons} (\Fl_{H_{\scrL}^{\circ}} / B_{\scrL}, \k)$ denote the bounded derived category of $B_{\scrL}$-equivariant sheaves on $\Fl_{H_{\scrL}^{\circ}}$.\footnote{As with the monodromic Hecke category, this definition needs some justification since neither $B_{\scrL}$ nor $\Fl_{H_{\scrL}^{\circ}}$ is of finite type. By taking unipotent monodromy, one can simply use the definition from the monodromic Hecke category.}
Note that $D_{\cons} (\Fl_{H_{\scrL}^{\circ}} / B_{\scrL}, \k)$ is a monoidal category with respect to convolution $\star$.
For each $s \in S_{\scrL}^{\circ}$, we will write $\scrE_s^{H, \scrL} \coloneq \IC_s^{H, \scrL}$, i.e., the $\IC$-extension of $\uk_{\Fl_{H_{\scrL}^{\circ}}, s} [1]$.
For an expression $\uw = (s_1, \ldots, s_k)$ in $W_{\scrL}^{\circ}$, we will write $\scrE_{\uw}^{H, \scrL} = \scrE_{s_1}^{H, \scrL} \star \ldots \star \scrE_{s_k}^{H, \scrL}$.
We will write $\Par_{\BS} (\Fl_{H_{\scrL}^{\circ}} / B_{\scrL}, \k)$ for the Bott--Samelson category of parity sheaves for $\Fl_{H_{\scrL}^{\circ}}$. That is the full subcategory of $D_{\cons} (\Fl_{H_{\scrL}^{\circ}} / B_{\scrL}, \k)$ consisting of the objects $\scrE_{\uw}^{\scrL} [n]$.
Note that $\Par_{\BS} (\Fl_{H_{\scrL}^{\circ}} / B_{\scrL}, \k)$ is a monoidal category with respect to convolution.
When $\k$ is a field or a complete local ring, we can consider also consider the category of parity sheaves  for $\Fl_{H_{\scrL}^{\circ}}$, denoted $\Par (\Fl_{H_{\scrL}^{\circ}} /B_{\scrL}, \k)$.
In this case, the indecomposables of $\Par (\Fl_{H_{\scrL}^{\circ}} / B_{\scrL}, \k)$ are indexed by $W_{\scrL}^\circ \times \Z$. For all $w \in W_{\scrL}^{\circ}$, we write $\scrE_w^{H, \scrL}$ for the parity extension of $\uk_{\Fl_{H_{\scrL}^{\circ}, w}} [\ell_{\scrL} (w)]$.

We can also consider right equivariant versions of the endoscopic categories. Let $D_{\cons} (\eFl_{H_{\scrL}^{\circ}} / B_{\scrL}, \k)$ denote the bounded derived category of right $B_{\scrL}$-equivariant constructible sheaves on $\eFl_{H_{\scrL}^{\circ}}$.
We regard any sheaf on $\Fl_{H_{\scrL}^{\circ}}$ as a sheaf in $D_{\cons} (\eFl_{H_{\scrL}^{\circ}} / B_{\scrL}, \k)$ via pullback along the $T$-torsor $\eFl_{H_{\scrL}^{\circ}} \to \Fl_{H_{\scrL}^{\circ}}$.
In particular, for each expression $\uw$ in $W_{\scrL}^{\circ}$, we can regard $\scrE_{\uw}^{H, \scrL}$ as a sheaf on the enhanced flag variety.
We write $\Par_{\BS} (\eFl_{H_{\scrL}^{\circ}} / B_{\scrL}, \k)$ for the Bott--Samelson category of right-equivariant parity sheaves for $\eFl_{H_{\scrL}^{\circ}}$. 
That is the full subcategory of $D_{\cons} (\eFl_{H_{\scrL}^{\circ}} / B_{\scrL}, \k)$ consisting of objects of the form $\scrE_{\uw}^{\scrL} [n]$.
When $\k$ is a field or a complete local ring, we can consider also consider the category of parity sheaves for $\eFl_{H_{\scrL}^{\circ}}$, denoted $\Par (\eFl_{H_{\scrL}^{\circ}} / B_{\scrL}, \k)$.
For all $w \in W_{\scrL}^{\circ}$, we also write $\scrE_w^{H, \scrL}$ for the parity extension of $\uk_{\eFl_{H_{\scrL}^{\circ}, w}} [\ell_{\scrL} (w)]$.

\subsection{Neutral Block Endoscopy}

Let $\scrQ$ be a field of characteristic 0.
If $W_{\scrL}^\circ$ is finite, then we can take $\Theta = S_{\scrL}^\circ$. In which case, Theorem \ref{thm:H_functor_properties_par} simplifies to an equivalence of monoidal categories
\[\mathbb{H}_{\scrL}^{\circ} : \Parity{\scrL}{\scrL}^\circ (\scrQ) \to \SBim_{W_{\scrL}^\circ} (\fr{h}_{\scrQ}).\]
The classical theory of Soergel bimodules gives an equivalence of monoidal categories,
\[\Par (\Fl_{H_{\scrL}^{\circ}} / B_{\scrL}, \scrQ) \cong \SBim_{W_{\scrL}^\circ} (\fr{h}_{\scrQ}).\]
By composition of these equivalences with obtain an equivalence,
\[\Par ( \Fl_{H_{\scrL}^{\circ}} / B_{\scrL}, \scrQ) \cong \Parity{\scrL}{\scrL}^\circ (\scrQ),\]
which provides an analogue of \cite[Theorem 9.2]{LY}.
In the following sections, we will further extend this equivalence in three key ways:
\begin{enumerate}
  \item We will enlarge the allowable coefficient rings to any noetherian domain of finite global dimension $\k$.
  \item We will remove the finiteness condition on $W_{\scrL}^\circ$ (for the neutral block).
  \item We will extend the result beyond the neutral block.
\end{enumerate}

Let $\scrD_{\BS} (\fr{h}_{\k}, W_{\scrL}^{\circ})$ denote the Bott--Samelson diagrammatic Hecke category introduced in \cite{EW}.
The objects of $\scrD_{\BS} (\fr{h}_{\k}, W_{\scrL}^{\circ})$ are indexed by pairs $(\uw, n)$ where $\uw$ is an expression in $W_{\scrL}^{\circ} $ and $n \in \Z$. 
We denote the object in $\scrD_{\BS} (\fr{h}_{\k}, W_{\scrL}^{\circ})$ indexed by $(\uw, n)$ by $B_{\uw} (n)$.
We can also consider the category $\widetilde{\scrD}_{\BS} (\fr{h}_{\k}, W_{\scrL}^{\circ})$ obtained from $\scrD_{\BS} (\fr{h}_{\k}, W_{\scrL}^{\circ})$ by tensoring all morphism spaces as left $R_{\k}$-modules with the trivial $R_{\k}$-module $\k$.
If $\k$ is a complete local ring of a field, we will write $\scrD (\fr{h}_{\k}, W_{\scrL}^{\circ})$ for the Karoubian envelope of the additive hull of $\scrD_{\BS} (\fr{h}_{\k}, W_{\scrL}^{\circ})$.

\begin{theorem}[Monodromic-Endoscopic equivalence for the neutral block]\label{thm:endoscopy_neutral_block_Kac_moody}\qquad 
  \begin{enumerate}
    \item There is an equivalence of monoidal additive categories,
    \[\Upsilon_{\scrL}^{\circ} : \scrD_{\BS} (\fr{h}_{\k}, W_{\scrL}^{\circ}) \to \Parity{\scrL}{\scrL}^{\BS, \circ} (\k).\]
    \item Assume that $G$ is either of finite or affine type.
    There is an equivalence of monoidal categories,
    \[\Psi_{\scrL}^{\circ} : \Par_{\BS} (\Fl_{H_{\scrL}^{\circ}} /B_{\scrL}, \k) \to \Parity{\scrL}{\scrL}^{\BS, \circ} (\k).\]
    If $\k$ is a complete local ring, $\Psi_{\scrL}^{\circ}$ induces an equivalence of monoidal additive categories
    \[\Psi_{\scrL}^{\circ} : \Par (\Fl_{H_{\scrL}^{\circ}} / B_{\scrL}, \k) \to \Parity{\scrL}{\scrL}^\circ (\k)\]
    such that for all $w \in W_{\scrL}^\circ$, we have that $\Psi_{\scrL}^\circ (\scrE_w^{H, \scrL}) \cong \scrE_w^{\scrL}$.
  \end{enumerate}
\end{theorem}

Along the way, we will also prove the following variant of Theorem \ref{thm:endoscopy_neutral_block_Kac_moody} for right equivariant Hecke categories.

\begin{theorem}\label{thm:endoscopy_neutral_block_Kac_moody_right_equivariant}\qquad
  \begin{enumerate}
    \item There is an equivalence of additive categories,
    \[\tilde{\Upsilon}_{\scrL}^{\circ} : \widetilde{\scrD}_{\BS} (\fr{h}_{\k}, W_{\scrL}^{\circ}) \to \PME{\scrL}{\scrL}^{\BS, \circ} (\k).\]
    \item Assume that $G$ is either of finite or affine type. Then there is an equivalence of additive categories,
    \[\tilde{\Psi}_{\scrL}^{\circ} : \Par_{\BS} (\eFl_{H_{\scrL}^{\circ}} / B_{\scrL}, \k) \to \PME{\scrL}{\scrL}^{\BS, \circ} (\k).\]
    If $\k$ is a complete local ring, $\tilde{\Psi}_{\scrL}^{ \circ}$ induces an equivalence of additive categories
    \[\tilde{\Psi}_{\scrL}^{\circ} : \Par (\eFl_{H_{\scrL}^{\circ}} / B_{\scrL}, \k) \to \PME{\scrL}{\scrL}^\circ (\k)\]
    such that for all $w \in W_{\scrL}^\circ$, we have that $\tilde{\Psi}_{\scrL}^\circ (\scrE_w^{H, \scrL}) \cong \scrE_w^{\scrL}$.
  \end{enumerate}
\end{theorem}

\begin{remark}
  \begin{enumerate}
    \item While we only stated Theorem \ref{thm:endoscopy_neutral_block_Kac_moody_right_equivariant} for the right equivariant Hecke categories, an obvious variant exists with left equivariant Hecke categories instead.
    \item Theorems \ref{thm:endoscopy_neutral_block_Kac_moody} and \ref{thm:endoscopy_neutral_block_Kac_moody_right_equivariant} are compatible in the sense that for sheaves $\scrF \in \Par (\eFl_{H_{\scrL}^{\circ}} / B_{\scrL}, \k)$ and $\scrG \in \Par (\Fl_{H_{\scrL}^{\circ}} / B_{\scrL}, \k)$, there are natural isomorphisms
  \[\tilde{\Psi}_{\scrL}^{\circ} (\scrF \star \scrG) \cong \tilde{\Psi}_{\scrL}^{\circ} (\scrF) \star \Psi_{\scrL}^{\circ} (\scrG).\]
    \item There are derived versions of Theorems \ref{thm:endoscopy_neutral_block_Kac_moody} and \ref{thm:endoscopy_neutral_block_Kac_moody_right_equivariant} where we take the bounded homotopy category of both sides of the equivalences.
        This produces an equivalence of mixed derived categories. This perspective will be used in the proof of Theorem \ref{thm:endoscopy_neutral_block_Kac_moody_right_equivariant}.
  \end{enumerate}
\end{remark}

The proof of Theorems \ref{thm:endoscopy_neutral_block_Kac_moody} and \ref{thm:endoscopy_neutral_block_Kac_moody_right_equivariant} will occupy the remainder of the section. We will first explain how to deduce Theorem \ref{thm:endoscopy_neutral_block_Kac_moody} (2) from Theorem \ref{thm:endoscopy_neutral_block_Kac_moody} (1), and likewise for the right equivariant categories.

If $G$ is either of finite or affine type, by \cite[Theorem 10.5]{RW}, we have an equivalence of monoidal categories
\begin{equation}\label{eq:RW_main_thm}
  \scrD_{\BS} (\fr{h}_{\k}, W_{\scrL}^\circ) \cong \Par_{\BS} (\Fl_{H_{\scrL}^{\circ}} / B_{\scrL}, \k) 
\end{equation}
that intertwines (1) with [1] and sends $B_{\uw} (n)$ to $\scrE_{\uw} [n]$ for expressions $\uw$ of $w \in W_{\scrL}^\circ$ and $n \in \Z$.
If $\k$ is a field or a complete local ring, the idempotent completion of (\ref{eq:RW_main_thm}) induces an equivalence,
\[\scrD (\fr{h}_{\k}, W_{\scrL}^\circ ) \cong  \Par (\Fl_{H_{\scrL}^{\circ}} / B_{\scrL}, \k) .\]
We can then compose the equivalence (\ref{eq:RW_main_thm}) with $\Upsilon_{\scrL}^{\circ}$ to produce the desired equivalence for Theorem \ref{thm:endoscopy_neutral_block_Kac_moody} (2). 

By applying $\k \otimes_{R_{\k}} (-)$ to the Hom spaces on both sides of (\ref{eq:RW_main_thm}), we obtain an analogous equivalence of additive categories $\widetilde{\scrD}_{\BS} (\fr{h}_{\k}, W_{\scrL}^\circ) \cong \Par_{\BS} (\eFl_{H_{\scrL}^{\circ}} / B_{\scrL}, \k)$.
The previous argument can then be copied to show that Theorem \ref{thm:endoscopy_neutral_block_Kac_moody_right_equivariant} (1) implies Theorem \ref{thm:endoscopy_neutral_block_Kac_moody_right_equivariant} (2).

\subsection{Construction of \texorpdfstring{$\Upsilon_{\scrL}^\circ$}{Phi}}\label{subsec:constr_of_Phi}

Let $\scrZ \to \k$ be as in \S\ref{subsec:integral_lifts}.
We define $\scrQ = \textnormal{Frac} (\scrZ)$, the field of fractions for $\scrZ$.
We then have ring homomorphisms
\[\scrQ \leftarrow \scrZ \to \k.\]
By Lemma \ref{lem:existence_of_integral_version}, there exists some $\scrL_{\scrZ} \in \Ch^{\circ} (T, \scrZ)$ such that $\k (\scrL_{\scrZ}) = \scrL$. Denote $\scrL_{\scrQ} \coloneq \scrQ (\scrL_{\scrZ})$.
Moreover, by Lemma \ref{lem:eos_and_root_systems}, we have that $W_{\scrL_{\scrQ}}^{\circ} = W_{\scrL_{\scrZ}}^{\circ} = W_{\scrL}^{\circ}$ and these have the same Coxeter structure.
As a result, we will simplify notation and just write $W_{\scrL}^{\circ}$ to refer to any of these endoscopic Weyl groups.

The main principle of the construction of $\Upsilon_{\scrL}^\circ$ is to define the functor on the defining objects and morphisms of $\scrD_{\BS} (\fr{h}_{\scrZ}, W_{\scrL}^\circ)$ over $\scrZ$.
We will then check that the defining relations in $\scrD_{\BS} (\fr{h}_{\scrQ}, W_{\scrL}^\circ)$ are satisfied by extending scalars to $\scrQ$.
Finally, we will extend scalars to construct $\Upsilon_{\scrL}^\circ$ over $\k$.

It is easy to define  $\Upsilon_{\scrL_{\scrZ}}^\circ$  on objects. For $s \in S_{\scrL}^{\circ}$, we set
\[ \Upsilon_{\scrL_{\scrZ}}^\circ (B_s (n)) \coloneq \IC_s^{\scrL_{\scrZ}} [n]\qquad\text{and}\qquad \Upsilon_{\scrL_{\scrZ}}^{\circ} (R (n)) \coloneq \IC_e^{\scrL} [n].\]
More generally, for an expression $\uw = (s_1, \ldots, s_k)$ in $W_{\scrL}^{\circ}$, we abuse notation and write
\[\scrE_{\uw}^{\scrL} \coloneq \IC_{s_1}^{\scrL} \star \IC_{s_2}^{\scrL} \star \ldots \star \IC_{s_k}^{\scrL}.\]
We can then define
\[ \Upsilon_{\scrL_{\scrZ}}^\circ (B_{\uw} (n)) \coloneq \scrE_{\uw}^{\scrL_{\scrZ}} [n].\]

The generating morphisms in $\scrD_{\BS} (\fr{h}_{\scrZ}, W_{\scrL}^\circ)$ are given by the diagrams
\[
  \begin{tikzpicture}[squarednode/.style={rectangle, draw=black, minimum size=5mm}, dotnode/.style={circle, draw=black, fill=black, minimum size=2mm, inner sep=0pt}]
    \node (f1) at (0,-1) {};
    \node (f2) at (0,1) {};
    \node[squarednode] (f3) at (0,0) {$f$};
    \draw[dashed] (f1) -- (f3) -- (f2);

    \node[dotnode, red] (a1) at (1,0) {};
    \node (a2) at (1,1) {};
    \draw[red, thick] (a1) -- (a2);

    \node[dotnode, red] (b1) at (2,0) {};
    \node (b2) at (2,-1) {};
    \draw[red, thick] (b1) -- (b2);

    \draw[red, thick] (3,-1) -- (4,0) -- (5,-1);
    \draw[red, thick] (4,0) -- (4,1);

    \draw[red, thick] (6,1) -- (7,0) -- (8,1);
    \draw[red, thick] (7,0) -- (7,-1);

    \node (14) at (10.25, 0.75) {$\ldots$};
      \node (17) at (10.25, -0.65) {$\ldots$};
      \draw [color=red, thick] (8.75, -1) to (10,0);
      \draw [color=red, thick] (9.75, -1) to (10,0);
      \draw [color=red, thick] (10,0) to (9.25, 1);
      \draw [color=blue, thick] (8.75, 1) to (10,0);
      \draw [color=blue, thick] (9.25, -1) to (10,0);
      \draw [color=blue, thick] (10,0) to (9.75, 1);
      \draw [color=violet, thick] (11.25, -1) to (10,0);
      \draw [color=violet, thick] (10,0) to (11.25, 1);

  \end{tikzpicture}
\]
where $f \in R_{\scrZ}$ is homogeneous.

\subsubsection{Polynomials}
By adjunction, we have a canonical identification
\[\theta : \Hom^\bullet (\IC_e^{\scrL_{\scrZ}}, \IC_e^{\scrL_{\scrZ}}) \stackrel{\sim}{\to} H_T^\bullet (\pt, \scrZ) = R_{\scrZ}.\]
We then define
\[\Upsilon_{\scrL_{\scrZ}}^\circ \left( \scalebox{1}{
      \begin{tikzpicture}[baseline={([yshift=-.5ex]current bounding box.center)}, squarednode/.style={rectangle, draw=black, minimum size=1mm}, dotnode/.style={circle,
        draw=black, fill=black, minimum size=1mm, inner sep=0pt}]
        \node (f1) at (0,-0.7) {};
        \node (f2) at (0,0.7) {};
        \node[squarednode] (f3) at (0,0) {$f$};
        \draw[dashed] (f1) -- (f3) -- (f2);
\end{tikzpicture} } \right) = \theta^{-1} (f),\]
for $f \in R_{\scrZ}$. Since $\theta$ is an isomorphism of graded algebras, we have that $f \in R_{\scrZ}^{m}$ maps to a morphism $\Upsilon_{\scrL_{\scrZ}}^\circ (f) : \IC_e^{\scrL_{\scrZ}} \to \IC_e^{\scrL_{\scrZ}} [m]$.

\subsubsection{One-color Morphisms}

Recall from Lemma \ref{lem:Frob_alg_stuff} that $\IC_s^{\scrL_{\scrZ}}$ carried a graded Frobenius algebra structure.
\[\mu_{s} : \IC_s^{\scrL_{\scrZ}} \star \IC_s^{\scrL_{\scrZ}} \to \IC_s^{\scrL_{\scrZ}} [-1], \qquad\qquad \nu_{s} : \IC_s^{\scrL_{\scrZ}} \to \IC_s^{\scrL_{\scrZ}} \star \IC_s^{\scrL_{\scrZ}} [-1],\]
\[ \eta_{s} : \IC_e^{\scrL_{\scrZ}} \to \IC_s^{\scrL_{\scrZ}} [1], \qquad\qquad \epsilon_{s} : \IC_s^{\scrL_{\scrZ}} \to \IC_e^{\scrL_{\scrZ}} [1]. \]

We then define
\[
  \Upsilon_{\scrL_{\scrZ}}^\circ \left( 
      \begin{tikzpicture}[baseline={([yshift=-.5ex]current bounding box.center)}, squarednode/.style={rectangle, draw=black, minimum size=2mm}, dotnode/.style={circle,
        draw=black, fill=black, minimum size=1mm, inner sep=0pt}]
        \draw[red, thick] (6.5,-0.5) -- (7,0) -- (7.5,-0.5);
        \draw[red, thick] (7,0) -- (7,0.5);
  \end{tikzpicture}  \right) \coloneq \mu_{s}, \qquad\qquad
  \Upsilon_{\scrL_{\scrZ}}^\circ \left( 
      \begin{tikzpicture}[baseline={([yshift=-.5ex]current bounding box.center)}, squarednode/.style={rectangle, draw=black, minimum size=2mm}, dotnode/.style={circle,
        draw=black, fill=black, minimum size=1mm, inner sep=0pt}]
        \draw[red, thick] (3.5,0.5) -- (4,0) -- (4.5,0.5);
        \draw[red, thick] (4,0) -- (4,-0.5);
  \end{tikzpicture} \right) \coloneq \nu_{s},
\]
\[\Upsilon_{\scrL_{\scrZ}}^\circ \left( 
      \begin{tikzpicture}[baseline={([yshift=-.5ex]current bounding box.center)}, squarednode/.style={rectangle, draw=black, minimum size=2mm}, dotnode/.style={circle,
        draw=black, fill=black, minimum size=2mm, inner sep=0pt}]
        \node[dotnode, red] (a1) at (1,0.5) {};
        \node (a2) at (1,1) {};
        \node (a3) at (1,0) {};
        \draw[red, thick] (a1) -- (a2);
  \end{tikzpicture} \right) \coloneq \eta_{s}, \qquad\qquad \Upsilon_{\scrL_{\scrZ}}^\circ \left( 
      \begin{tikzpicture}[baseline={([yshift=-.5ex]current bounding box.center)}, squarednode/.style={rectangle, draw=black, minimum size=2mm}, dotnode/.style={circle,
        draw=black, fill=black, minimum size=2mm, inner sep=0pt}]
        \node[dotnode, red] (a1) at (1,-0.5) {};
        \node (a2) at (1,-1) {};
        \node (a3) at (1,0) {};
        \draw[red, thick] (a1) -- (a2);
\end{tikzpicture} \right) \coloneq \epsilon_{s}.\]

\subsubsection{\texorpdfstring{$2m_{st}$}{2mst}-valent vertices.}

Let $s,t \in S_{\scrL}^\circ$ with $m_{s,t} < \infty$. Write $w = sts\ldots$ with $m_{s,t}$-terms.
Let $\scrF_s^{\scrL_{\scrZ}} \coloneq \scrE_{(s, t, \ldots)}^{\scrL_{\scrZ}}$ and $\scrF_t^{\scrL_{\scrZ}} \coloneq \scrE_{(t,s, \ldots)}^{\scrL_{\scrZ}}$ where the expressions in both subscripts have $m_{s,t}$-terms.

\begin{lemma}\label{lem:mst_pullbacks}
  The pullback $j_w^*$ induces an isomorphism of $\scrZ$-modules
  \begin{equation}\label{eq:mst_pullbacks_1}
    \Hom (\scrF_s^{\scrL_{\scrZ}}, \scrF_t^{\scrL_{\scrZ}}) \stackrel{\sim}{\to} \Hom (j_w^* \scrF_s^{\scrL_{\scrZ}}, j_w^* \scrF_t^{\scrL_{\scrZ}}).
  \end{equation}
  Moreover, both of sides of (\ref{eq:mst_pullbacks_1}) are free $\scrZ$-modules of rank 1.
\end{lemma}
\begin{proof}
  By Corollary \ref{cor:soergel_hom_v2} and Lemma \ref{lem:eos_for_biequiv}, the $\scrZ$-module $\Hom (\scrF_s^{\scrL_{\scrZ}}, \scrF_t^{\scrL_{\scrZ}})$ is free of rank 1. 
  An argument similar to the one given in Theorem \ref{thm:existence_of_parity} shows that $j_w^* \scrF_s^{\scrL_{\scrZ}} \cong \scrK_w^{\scrL_{\scrZ}} \cong j_w^* \scrF_t^{\scrL_{\scrZ}}$.
  Note that $\End (\scrK_w^{\scrL_{\scrZ}})$ is also a free $\scrZ$-module of rank 1.
  Since $\eFl_w$ is open in the support of $\scrF_s^{\scrL_{\scrZ}}$ and $\scrF_t^{\scrL_{\scrZ}}$, by a variation of Proposition \ref{prop:hom_star_and_shriek}, we have that the map (\ref{eq:mst_pullbacks_1}) is surjective.
  We are now done since surjective maps of rank 1 free modules are automatically isomorphisms.
\end{proof}

Define the map $g_{s,t} : \scrF_s^{\scrL_{\scrZ}} \to \scrF_t^{\scrL_{\scrZ}}$ as the pre-image under (\ref{eq:mst_pullbacks_1}) of an isomorphism $g_{s,t}^w : j_w^* \scrF_s^{\scrL_{\scrZ}} \stackrel{\sim}{\to} j_w^* \scrF_t^{\scrL_{\scrZ}}$.
Likewise, by switching the roles of $s$ and $t$, there is a map $g_{t,s} : \scrF_t^{\scrL_{\scrZ}} \to \scrF_s^{\scrL_{\scrZ}}$ lifted from an isomorphism $g_{t,s}^w : j_w^* \scrF_t^{\scrL_{\scrZ}} \stackrel{\sim}{\to} j_w^* \scrF_s^{\scrL_{\scrZ}}$. 
We pick $g_{s,t}^w$ and $g_{t,s}^w$ such that $g_{t,s}^w \circ g_{s,t}^w = \id_{j_w^* \scrF_s^{\scrL_{\scrZ}}}$.

To define the image of the $2m_{s,t}$-valent vertices, we will need another lemma. We define
\[\epsilon_{\hat{s}} \coloneq \epsilon_{s} \star \epsilon_t \ldots  : \scrF_s^{\scrL_{\scrZ}} \to \IC_e^{\scrL_{\scrZ}} [m] \qquad\text{and}\qquad \epsilon_{\hat{t}} \coloneq \epsilon_{t} \star \epsilon_s \ldots  : \scrF_t^{\scrL_{\scrZ}} \to \IC_e^{\scrL_{\scrZ}} [m]\]

\begin{lemma}\label{lem:hom_generated_by_counit}
  For $u \in \{s,t\}$, we have
  \[\Hom_{\DE{\scrL_{\scrZ}} (\scrZ)} (\scrF_u^{\scrL_{\scrZ}}, \IC_e^{\scrL_{\scrZ}} [m]) = \scrZ \cdot \ForME{\scrL_{\scrZ}} (\epsilon_{\hat{u}}).\]
\end{lemma}
\begin{proof}
    By Lemma \ref{lem:eos_for_right_equiv} and Corollary \ref{cor:soergel_hom_v2}, the space in question is free of rank 1 over $\scrZ$.
    It then suffices to show that $\ForME{\scrL_{\scrZ}} (\epsilon_{\hat{u}})$ is nonzero after extension of scalars to any field. This is clear since extension of scalars commutes with convolution and the unit of the adjunction $\id \to j_{s*} j_s^*$ for any $s \in S_{\scrL}^{\circ}$.
\end{proof}

By Lemma \ref{lem:hom_generated_by_counit}, we have 
\begin{equation}\label{eq:cst_constants}
  \ForME{\scrL_{\scrZ}} (\epsilon_{\hat{t}} \circ g_{s,t}) = c_{s,t} \cdot \left( \ForME{\scrL_{\scrZ}} (\epsilon_{\hat{s}})\right) \qquad\text{and}\qquad \ForME{\scrL_{\scrZ}} (\epsilon_{\hat{s}} \circ g_{t,s}) = c_{t,s} \cdot \left( \ForME{\scrL_{\scrZ}} (\epsilon_{\hat{t}}) \right)
\end{equation}
for some $c_{s,t}, c_{t,s} \in \scrZ$. We can now set
\[f_{s,t} \coloneq c_{t,s} g_{s,t} : \scrF_s^{\scrL_{\scrZ}} \to \scrF_t^{\scrL_{\scrZ}} \qquad\text{and}\qquad f_{t,s} \coloneq c_{s,t} g_{t,s} : \scrF_t^{\scrL_{\scrZ}} \to \scrF_s^{\scrL_{\scrZ}},\]
and then define these to be the image of the $2m$-valent vertices under $\Upsilon_{\scrL_{\scrZ}}^{\circ}$:
\[
  \Upsilon_{\scrL_{\scrZ}}^\circ \left( \scalebox{1}{
      \begin{tikzpicture}[baseline={([yshift=-.5ex]current bounding box.center)}, squarednode/.style={rectangle, draw=black, minimum size=1mm}, dotnode/.style={circle,
        draw=black, fill=black, minimum size=1mm, inner sep=0pt}]
        \node (14) at (10.25, 0.75) {$\ldots$};
        \node (17) at (10.25, -0.65) {$\ldots$};
        \draw [color=red, thick] (8.75, -1) to (10,0);
        \draw [color=red, thick] (9.75, -1) to (10,0);
        \draw [color=red, thick] (10,0) to (9.25, 1);
        \draw [color=blue, thick] (8.75, 1) to (10,0);
        \draw [color=blue, thick] (9.25, -1) to (10,0);
        \draw [color=blue, thick] (10,0) to (9.75, 1);
        \draw [color=violet, thick] (11.25, -1) to (10,0);
        \draw [color=violet, thick] (10,0) to (11.25, 1);
  \end{tikzpicture} } \right) \coloneq f_{s,t} \qquad\text{and}\qquad \Upsilon_{\scrL_{\scrZ}}^\circ \left( \scalebox{1}{
      \begin{tikzpicture}[baseline={([yshift=-.5ex]current bounding box.center)}, squarednode/.style={rectangle, draw=black, minimum size=1mm}, dotnode/.style={circle,
        draw=black, fill=black, minimum size=1mm, inner sep=0pt}]
        \node (14) at (10.25, 0.75) {$\ldots$};
        \node (17) at (10.25, -0.65) {$\ldots$};
        \draw [color=blue, thick] (8.75, -1) to (10,0);
        \draw [color=blue, thick] (9.75, -1) to (10,0);
        \draw [color=blue, thick] (10,0) to (9.25, 1);
        \draw [color=red, thick] (8.75, 1) to (10,0);
        \draw [color=red, thick] (9.25, -1) to (10,0);
        \draw [color=red, thick] (10,0) to (9.75, 1);
        \draw [color=violet, thick] (11.25, -1) to (10,0);
        \draw [color=violet, thick] (10,0) to (11.25, 1);
  \end{tikzpicture} } \right) \coloneq f_{t,s}.\]

\subsection{Verification of Defining Relations}\label{subsec:ver_relns}

In order to check that $\Upsilon_{\scrL_{\scrZ}}^\circ$ is a well-defined functor, one must check that the image of a morphism under $\Upsilon_{\scrL_{\scrZ}}^\circ$ is invariant under the defining relations in $\scrD_{\BS} (\fr{h}_{\scrZ}, W_{\scrL}^\circ)$. 
By Lemma \ref{lem:eos_for_biequiv}, it suffices to check this after extending scalars to $\scrQ$-coefficients. Write $\Upsilon_{\scrL_{\scrQ}}^{\circ}$ for the extension of $\Upsilon_{\scrL_{\scrZ}}^{\circ}$ along $\scrZ \to \scrQ$.
Since all relations only depend on parabolic subgroups of $W_{\scrL}^\circ$ of finite-type and by Theorem \ref{thm:H_functor_properties_par}, we can further reduce to checking the defining relations after composition with $\H_{\Theta}^{\scrL_{\scrQ}}$ for some finite-type subset $\Theta \subset S_{\scrL}^\circ$.

Let $\Theta \subseteq S_{\scrL}^\circ$ with $\# \Theta \leq 3$ such that $\langle \Theta \rangle \subset W_{\scrL}^{\circ}$ is finite.
We will now calculate the image of our morphisms under $\H_{\Theta}^{\scrL_{\scrQ}}$. 
\begin{enumerate}
  \item \emph{Polynomials:} $\H_{\Theta}^{\scrL_{\scrQ}} (\Upsilon_{\scrL_{\scrQ}}^\circ (f))$ for $f \in R_{\scrQ}$ is given by multiplication by $f$ in $R_{\scrQ}$.
  \item \emph{The upper dot:} Any morphism of graded $R_{\scrQ}$-bimodules $B_s \to R_{\scrQ} (1)$ is a scalar multiple of $m_s : B_s \to R_{\scrQ} (1)$ taking $m_s (f \otimes g) = fg$.
        By Lemma \ref{lem:xi_s_map}, there is a unique morphism $\xi_s : \scrK_{\Theta}^{\scrL_{\scrQ}} \to \scrE_s^{\scrL_{\scrQ}} [-1]$ such that $(\epsilon_s [-1]) \circ \xi_s = \epsilon_{\scrL_{\scrQ}}$.
        It then follows from Lemma \ref{lem:H_functor_for_internally_simple_reflns} that under the isomorphism $\H_{\Theta}^{\scrL_{\scrQ}} (\IC_s^{\scrL_{\scrQ}}) \cong B_s$, one has that $(\H_{\Theta}^{\scrL_{\scrQ}} (\epsilon_s)) (1 \otimes 1) = 1$.
        This implies that $\H_{\Theta}^{\scrL_{\scrQ}} (\epsilon_s) = m_s$.
  \item \emph{The lower dot:} 
            The map
            \[\IC_e^{\scrL_{\scrQ}} [-1] = j_{e*} j_e^! \IC_s^{\scrL_{\scrQ}} \stackrel{\eta_s}{\to} \IC_s^{\scrL_{\scrQ}} \stackrel{\epsilon_s}{\to} j_{e*} j_e^* \IC_s^{\scrL_{\scrQ}} [1]\]
            in $\Hom (\IC_e^{\scrL_{\scrQ}} [-1], \IC_e^{\scrL_{\scrQ}} [1])$ is given by $\lambda \cdot \id_{\IC_e^{\scrL_{\scrQ}} [-1]}$ for some $\lambda \in H_T^2 (\pt; \scrQ)$.  
            We claim that $\lambda = \alpha_s$. By Lemma \ref{lem:beta_conj_for_H_functor}, it suffices to verify this when $s$ is an externally simple reflection.

            Let $L_s$ be the rank 1 Levi subgroup of $G$ corresponding to $s$ and write $j : T \hookrightarrow L_s$ for the inclusion.
            Since $\scrL_{\scrQ} s = \scrL_{\scrQ}$, there is a unique multiplicative local system $\scrL_{\scrQ}^s$ on $L_s$ which extends $\scrL_{\scrQ}$.
            There is an isomorphism of varieties $\eFl_{\leq s} \cong U_s \backslash L_s$ where $U_s = U \cap L_s$.
            Under this isomorphism, the simple perverse sheaf $\IC_s^{\scrL_{\scrQ}}$ is canonically identified with $\scrL_{\scrQ}^s$. 
            We can then see under this isomorphism and adjunction that the natural morphism of $(T \times T, \scrL_{\scrQ} \boxtimes \scrL_{\scrQ}^{-1})$-equivariant sheaves
            \begin{equation}\label{eq:lower_dot_reln_1_par}
              \scrL_{\scrQ} [-1] = j^! \scrL_{\scrQ}^s [1] \to j^* \scrL_{\scrQ}^s [1] = \scrL_{\scrQ} [1]
            \end{equation}
            is given by $\lambda \cdot \id_{\scrL_{\scrQ} [-1]}$.
            After applying $(-) \otimes^L \scrL_{\scrQ}^{-1}$ to (\ref{eq:lower_dot_reln_1_par}), we see that the natural morphism of $T\times T$-equivariant sheaves
            \[
              \underline{\scrQ}_{T} [-1] = j^! \underline{\scrQ}_{L_s} [1] \to j^* \underline{\scrQ}_{L_s} [1] = \underline{\scrQ}_T [1]
            \]
            is also given by $\lambda \cdot \id_{\underline{\scrQ}_T [-1]}$. 
            It then follows from the non-monodromic case (cf., \cite[\S10.5]{RW}) that $\lambda = \alpha_s$, and as a result, $\H_{\Theta}^{\scrL_{\scrQ}} (\epsilon_s \circ \eta_s) = \alpha_s$.

            Any morphism of graded $R_{\scrQ}$-bimodules $R_{\scrQ} \to B_s (1)$ is a scalar multiple of $\delta_s : R_{\scrQ} \to B_s(1)$ where $\delta_s (1) = \frac{1}{2} (\alpha_s \otimes 1 + 1 \otimes \alpha_s )$.
            Note that $\delta_s$ is uniquely determined by $m_s \circ \delta_s = \alpha_s$. As a consequence, we deduce from the upper dot computation that $\H_{\Theta}^{\scrL_{\scrQ}} (\eta_s) = \delta_s$. 
    \item \emph{The trivalent vertices: } It follows from classical theory, that there are decompositions $B_s \otimes_R B_s \cong B_s (1) \oplus B_s (-1)$. Moreover, we have isomorphisms of vector spaces
        \begin{equation}\label{eq:trivalent_vert_spaces}
         \Hom (B_s, B_s \otimes_R B_s (-1)) \cong \scrQ, \hspace{0.5cm}\text{and}\hspace{0.5cm} \Hom (B_s \otimes_R B_s, B_s (-1)) \cong \scrQ.
        \end{equation}
        Define a map $\partial_s : R \to R^s$ by $\partial_s (f) = (f- s(f))/\alpha_s$. 
        The morphism spaces in (\ref{eq:trivalent_vert_spaces}) are generated by the maps $t_1 : B_s \to B_s \otimes_R B_s (-1)$ defined by $f \otimes g \mapsto f \otimes 1 \otimes g$ and $t_2 : B_s \otimes_R B_s \to B_s (-1)$ defined by $f \otimes g \otimes h \mapsto f (\partial_s g) \otimes h$
        where we have identified $B_s \otimes_R B_s \cong R \otimes_{R^s} R \otimes_{R^s} R (2)$.
        Therefore, $\H_{\Theta}^{\scrL_{\scrQ}} (b_i)$ is a scalar multiple of $t_i$ for $i=1,2$.
        Note that $t_1$ (resp. $t_2$) are uniquely characterized by the identity
        \[\hspace{1cm} (m_s (-1) \otimes_R \id_{B_s}) \circ t_1 = \id_{B_s} \hspace{0.5cm} \text{resp.} \hspace{0.5cm} t_2 (1) \circ (\delta_s \otimes_R \id_{B_s}) = \id_{B_s}.\]
        Therefore, $\H_{\Theta}^{\scrL_{\scrQ}} (\nu_s) = t_1$ and $\H_{\Theta}^{\scrL_{\scrQ}} (\mu_s) = t_2$ follow from Lemma \ref{lem:trivalent_mors}.
\end{enumerate}

Computing the image of the $2m_{s,t}$-valent vertices will need some preparation. Let $s,t \in S_{\scrL}^{\circ}$, and set $m = m_{s,t}$.
We define Bott--Samelson bimodules
\[B_{\hat{s}} \coloneq \underbrace{B_s \otimes_{R_{\scrQ}} B_t \otimes_{R_{\scrQ}} \ldots}_{m\text{ factors}} \qquad\text{and}\qquad B_{\hat{t}} \coloneq \underbrace{B_t \otimes_{R_{\scrQ}} B_s \otimes_{R_{\scrQ}} \ldots}_{m\text{ factors}}.\]
We can then define morphisms
\[m_{\hat{s}} \coloneq m_{s} \otimes_{R_{\scrQ}} m_{t} \otimes_{R_{\scrQ}} \ldots : B_{\hat{s}} \to R_{\scrQ} (m),\]
\[ m_{\hat{t}} \coloneq m_{t} \otimes_{R_{\scrQ}} m_{s} \otimes_{R_{\scrQ}} \ldots : B_{\hat{t}} \to R_{\scrQ} (m).\]
Recall (cf., \cite[Proposition 4.3]{Lib08}) that $\Hom_{\SBim_{W_{\scrL, \Theta}^{\circ}} (\fr{h}_{\scrQ})} (B_{\hat{s}}, B_{\hat{t}})$ is 1-dimensional.
The same argument in \cite[Lemme 4.7]{Lib08} shows that there is a unique morphism of $R_{\scrQ}$-bimodules
\[j_{s,t} : B_{\hat{s}} \to B_{\hat{t}}\]
that acts as the identity map in degree $m$. 
Note that $m$ is the largest degree in which $B_{\hat{s}}$ and $B_{\hat{t}}$ are nonzero. We can then rephrase the condition on $j_{s,t}$ as follows: it is the unique morphism such that there is an equality of maps
\[\scrQ \otimes_{R_{\scrQ}} (m_{\hat{t}} \circ j_{s,t}) = \scrQ \otimes_{R_{\scrQ}} m_{\hat{s}} : \scrQ \otimes_{R_{\scrQ}} B_{\hat{s}} \to \scrQ (m). \]

\begin{lemma}\label{lem:cst_constants_inverses}
  The constants $c_{s,t}, c_{t,s} \in \scrZ$ defined by (\ref{eq:cst_constants}) satisfy
  \[c_{s,t} c_{t,s} = 1.\]
\end{lemma}
\begin{proof}
  Let $m = m_{s,t}$.
  It follows from the definition of $g_{s,t}, g_{t,s}$ that
  \[g_{s,t} \circ g_{t,s} \circ g_{s,t} = g_{s,t}.\]
  Applying $\ForME{\scrL_{\scrZ}} (\epsilon_{\hat{t}} \circ (-))$ both sides and using (\ref{eq:cst_constants}), we deduce that
  $c_{s,t} c_{t,s} c_{s,t} = c_{s,t}$, or alternatively that $c_{s,t} (c_{t,s} c_{s,t} - 1) = 0$.
  
  It remains to prove that $c_{s,t} \neq 0$. Since $g_{s,t}$ is a generator for $\Hom (\scrF_s^{\scrL_{\scrZ}}, \scrF_t^{\scrL_{\scrZ}})$, it suffices to show that the map
  \begin{align*}
    \Hom_{\D{\scrL_{\scrZ}}{\scrL_{\scrZ}} (\scrZ)} (\scrF_s^{\scrL_{\scrZ}}, \scrF_t^{\scrL_{\scrZ}}) &\to \Hom_{\DME{\scrL_{\scrZ}}{{\scrL_{\scrZ}}} (\scrZ)} (\scrF_s^{\scrL_{\scrZ}}, \IC_e^{\scrL_{\scrZ}} [m]) \\
    f &\mapsto \ForME{\scrL_{\scrZ}} (\epsilon_{\hat{t}} \circ f)
  \end{align*}
  is nonzero. It suffices to prove that this map is nonzero after apply $\scrQ (-)$.

  Write $\scrF_u^{\scrL_{\scrQ}} \coloneq \scrQ (\scrF_u^{\scrL_{\scrZ}})$ for $u \in \{s,t\}$. 
  Let $f : \scrF_s^{\scrL_{\scrQ}} \to \scrF_t^{\scrL_{\scrQ}}$. By the upper dot relation we have already verified and (\ref{eq:soergel_and_for}), we have a commutative diagram
  \[\begin{tikzcd}
{\Hom (\scrF_s^{\scrL_{\scrQ}}, \scrF_t^{\scrL_{\scrQ}})} \arrow[d, "\H_{\Theta}^{\scrL_{\scrQ}}"] \arrow[r, "\epsilon_{\hat{s}} \circ f"] & {\Hom (\scrF_s^{\scrL_{\scrQ}}, \IC_e^{\scrL_{\scrQ}} [m])} \arrow[d, "\H_{\Theta}^{\scrL_{\scrQ}}"] \arrow[r, "\ForME{\scrL_{\scrQ}}"] & {\Hom_{\DE{\scrL_{\scrQ}} } (\scrF_s^{\scrL_{\scrQ}}, \IC_e^{\scrL_{\scrQ}} [m])} \arrow[d, "\tilde{\H}_{\Theta}^{\scrL_{\scrQ}}"] \\
{\Hom (B_{\hat{s}}, B_{\hat{t}})} \arrow[r, "m_{\hat{s}} \circ \H_{\Theta}^{\scrL_{\scrQ}} (f)"]                                               & {\Hom (B_{\hat{s}}, R (m))} \arrow[r, "\scrQ \otimes_{R} (-)"]                                                      & {\Hom_{\rgrmod{R }} (\scrQ \otimes_R B_{\hat{s}}, \scrQ (m))}.                                                                                             
\end{tikzcd}\]
By Theorem \ref{thm:H_functor_properties_par} and Proposition \ref{prop:right_equiv_H_functor_properties}, it suffices to prove that $\scrQ \otimes_{R} (m_{\hat{s}} \circ \varphi) \neq 0$ for some $\varphi : B_{\hat{t}} \to B_{\hat{s}}$.
We can take $\varphi = j_{s,t}$ which by the discussion proceeding the lemma satisfies the desired constraint.
\end{proof}

We can now compute the image of the $2m_{s,t}$-valent vertices.
\begin{enumerate}
  \item[5.] \emph{$2m_{s,t}$-valent vertices: } Let $m = m_{s,t}$.
  By our definition for $f_{s,t}$ and Lemma \ref{lem:cst_constants_inverses}, we have $\ForME{\scrL_{\scrQ}} (\epsilon_{\hat{t}} \circ f_{s,t}) = \ForME{\scrL_{\scrQ}} (\epsilon_{\hat{s}})$.
  We can now apply $\tilde{\H}_{\Theta}^{\scrL_{\scrQ}}$ and use (\ref{eq:soergel_and_for}) to deduce that 
  \[\scrQ \otimes_{R_{\scrQ}} \tilde{\H}_{\Theta}^{\scrL_{\scrQ}} (\epsilon_{\hat{t}}) \circ \tilde{\H}_{\Theta}^{\scrL_{\scrQ}} (f_{s,t}) = \scrQ \otimes_{R_{\scrQ}} \tilde{\H}_{\Theta}^{\scrL_{\scrQ}} (\epsilon_{\hat{s}}).\]
  Since $\tilde{\H}_{\Theta}^{\scrL_{\scrQ}} (\epsilon_{\hat{u}}) = m_{\hat{u}}$ for $u \in \{s,t\}$, we conclude from the uniqueness of $j_{s,t}$ that $\H_{\Theta}^{\scrL} (f_{s,t}) = j_{s,t}$.
\end{enumerate}

To conclude the construction of the functor, we must check that the defining relations in $\scrD_{\BS} (\fr{h}_{\scrQ}, W_{\scrL}^\circ)$ are satisfied.
Each relation involves only a subset $\Theta \subseteq S_{\scrL}^\circ$ of finite type with $\abs{\Theta} \leq 3$.
By faithfulness of $\H_{\Theta}^{\scrL_{\scrQ}}$ it suffices to show that the relations are satisfied after composing with $\H_{\Theta}^{\scrL_{\scrQ}}$.
The relations have already been checked at this point using localization (see \cite[\S5.5]{EW}). As a result, we have a monoidal functor 
\[ \Upsilon_{\scrL_{\scrZ}}^\circ : \scrD_{\BS} (\fr{h}_{\scrZ}, W_{\scrL}^\circ) \to \Parity{\scrL_{\scrZ}}{\scrL_{\scrZ}}^{\BS, \circ} (\scrZ).\]

Let $\k \to \k'$ be a ring homomorphism of noetherian domains of finite global dimension.
By Lemma \ref{lem:eos_for_biequiv} and Lemma \ref{lem:existence_of_integral_version}, we can define a monoidal functor
\[ \Upsilon_{\k' (\scrL)}^\circ : \scrD_{\BS} (\fr{h}_{\k'}, W_{\scrL}^\circ) \to \Parity{\k' (\scrL)}{\k' (\scrL)}^{\BS, \circ} (\k') \]
by extending scalars along $\scrZ \to \k \to \k'$. Note this includes the case of $\Upsilon_{\scrL}^{\circ}$ where $\k' = \k$.

By Lemma \ref{lem:eos_for_right_equiv}, we can tensor the graded Hom spaces by $\k' \otimes_{R_{\k'}} (-)$ to produce a functor
\[\tilde{\Upsilon}_{\k' (\scrL)}^{\circ} : \widetilde{\scrD}_{\BS} (\fr{h}_{\k'}, W_{\scrL}^\circ) \to \PME{\k'(\scrL)}{\k'(\scrL)}^{\BS, \circ} (\k').\]
It is clear from their definitions and Proposition \ref{prop:structure_of_min_IC} that both $\Upsilon_{\k'(\scrL)}^{\circ}$ and $\tilde{\Upsilon}_{\k'(\scrL)}^{\circ}$ are essentially surjective.

\subsection{Proof of Theorems \ref{thm:endoscopy_neutral_block_Kac_moody} and \ref{thm:endoscopy_neutral_block_Kac_moody_right_equivariant}.}

We will prove that $\Upsilon_{\scrL}^{\circ}$ and $\tilde{\Upsilon}_{\scrL}^{\circ}$ are equivalences assuming the following claim.

\begin{claim}\label{claim:par_field_case}
$\tilde{\Upsilon}_{\F (\scrL)}^{\circ}$ is an equivalence of categories for all $\k \to \F$ such that $\F$ is a field.
\end{claim}

We will prove Claim \ref{claim:par_field_case} in the next section.

\begin{midsecproof}{Theorem \ref{thm:endoscopy_neutral_block_Kac_moody_right_equivariant}}
  Consider the morphism of free $\k$-modules of finite rank
  \[\alpha : \Hom_{\widetilde{\scrD}_{\BS} (\fr{h}_{\k}, W_{\scrL}^\circ)} (B_{\ux}, B_{\uy} (n)) \to  \Hom_{\PME{\scrL}{\scrL}^{\BS, \circ} (\k)} (\scrE_{\ux}^{\scrL}, \scrE_{\uy}^{\scrL} [n])\]
  given by $\tilde{\Upsilon}_{\scrL}^{\circ}$.
  By the Nakayama lemma, $\alpha$ is an isomorphism if and only if $\F \otimes_{\k} \alpha$ is an isomorphism for all fields $\F$ under $\k$.
  Note that $\F \otimes_{\k} \alpha$ is induced by $\tilde{\Upsilon}_{\F (\scrL)}^{\circ}$. 
  We then have that $\F \otimes_{\k} \alpha$ is an isomorphism by Claim \ref{claim:par_field_case}.
\end{midsecproof}

\begin{midsecproof}{Theorem \ref{thm:endoscopy_neutral_block_Kac_moody}}
  Consider the morphism of graded left $R_{\k}$-modules
  \[\beta : \Hom_{\scrD_{\BS} (\fr{h}_{\k}, W_{\scrL}^\circ)}^{\bullet} (B_{\ux}, B_{\uy}) \to  \Hom_{\Parity{\scrL}{\scrL}^{\BS, \circ} (\k)}^\bullet (\scrE_{\ux}^{\scrL}, \scrE_{\uy}^{\scrL})\]
  induced by $\Upsilon_{\scrL}^{\circ}$.
  As graded left $R_{\k}$-modules, the left-hand side of $\beta$ is free of finite rank by \cite[Theorem 6.11]{EW} and the right-hand side of $\beta$ is free of finite rank by Lemma \ref{lem:eos_for_biequiv}.
  By the graded Nakayama lemma, to show that $\beta$ is an isomorphism it suffices to show that $\k \otimes_R \beta$ is an isomorphism.
  Note that $\k \otimes_R \beta$ is the map induced from $\tilde{\Upsilon}_{\scrL}^{\circ}$.
  It follows from Theorem \ref{thm:endoscopy_neutral_block_Kac_moody_right_equivariant} that $\k \otimes_R \beta$ and hence $\beta$ is an isomorphism.
\end{midsecproof}

\subsection{Passage to Mixed Derived Categories}

Let $\F$ be a field with a ring homomorphism $\k \to \F$.
We can consider the biequivariant and right equivariant mixed derived categories on the endoscopic side,
\[D^m (\fr{h}_{\F}, W_{\scrL}^{\circ}) \coloneq K^b \scrD (\fr{h}_{\F}, W_{\scrL}^\circ) \qquad\text{and}\qquad \widetilde{D}^m (\fr{h}_{\F}, W_{\scrL}^{\circ}) \coloneq K^b \widetilde{\scrD} (\fr{h}_{\F}, W_{\scrL}^\circ).\]
These categories have been thoroughly studied by Achar--Riche--Vay \cite{ARV}. 
Despite these not necessarily being of geometric origin, they share many similar features to the mixed derived categories.
In particular, all the constructions and facts from \S \ref{sec:mixed_cats} hold after making appropriate changes.
We use the same notation as in \S\ref{subsec:endo_gps}, despite the endoscopic Kac--Moody group not necessarily being well-defined. We refer to \emph{loc. cit.} on how to make these constructions precise.

The functor $\Upsilon_{\F (\scrL)}^{\circ} : \scrD_{\BS} (\fr{h}_{\F}, W_{\scrL}^{\circ}) \to \Parity{\F (\scrL)}{\F (\scrL)}^{\BS, \circ} (\F)$ induces a monoidal functor
\[\underline{\Upsilon}_{\F (\scrL)}^{\circ} : D^m (\fr{h}_{\F}, W_{\scrL}^{\circ}) \to \D{\F (\scrL)}{\F(\scrL)}^m (\F).\] 
For each $w \in W_{\scrL}^{\circ}$, there are \emph{standard} $\underline{\Delta}_w^{H, \F (\scrL)}$ and \emph{costandard} objects $\underline{\nabla}_w^{H, \F (\scrL)}$
which we can regard as objects in either $D^m (\fr{h}_{\F}, W_{\scrL}^{\circ})$ or $\widetilde{D}^m (\fr{h}_{\F}, W_{\scrL}^{\circ})$.

\begin{lemma}\label{lem:par_Psi_on_stds}
  For all $w \in W_{\scrL}^{\circ}$, there are isomorphisms
  \[\underline{\Upsilon}_{\F (\scrL)}^{\circ} (\uDelta_w^{H, \F (\scrL)}) \cong \uDelta_w^{\F (\scrL)} \hspace{0.5cm} \text{and} \hspace{0.5cm} \underline{\Upsilon}_{\F (\scrL)}^{\circ} (\unabla_w^{H, \F (\scrL )}) \cong \unabla_w^{\F (\scrL)}.  \]
\end{lemma}
\begin{proof}
  We will just prove the first isomorphism. The second isomorphism follows from a similar argument.

  By monoidality of $\underline{\Upsilon}_{\F (\scrL)}^{\circ}$ and Corollary \ref{cor:endo_mixed_conv_rules} (see also \cite[Proposition 6.11]{ARV}), it suffices to consider the case of $s \in W_{\scrL}^{\circ}$ an endosimple reflection.
  In $D^m (\fr{h}_{\k}, W_{\scrL}^{\circ})$ there is a distinguished triangle
  \begin{equation}\label{eq:par_Psi_on_stds_1}
    \uDelta_s^{H, \F (\scrL)} \to \scrE_s^{H, \F (\scrL)} \to \Delta_e^{H, \F (\scrL)} (1) \to.
  \end{equation}
  We can then apply $\underline{\Upsilon}_{\F (\scrL)}^{\circ}$ to (\ref{eq:par_Psi_on_stds_1}) to produce a distinguished triangle
  \begin{equation}\label{eq:par_Psi_on_stds_2}
    \underline{\Upsilon}_{\F (\scrL)}^{\circ} \left( \uDelta_s^{H, \F (\scrL)} \right) \to \scrE_{s}^{\F (\scrL)} \to \uDelta_e^{\F (\scrL)} (1) \to,
  \end{equation}
  where the second map is given by the unit map $\id \to j_{e*} j_e^*$ in the non-mixed category.
  Since $j_x^* \scrE_x^{\F (\scrL)} = 0$ for all $x \neq e$ and $x \neq s$, the distinguished triangle (\ref{eq:par_Psi_on_stds_2}) coincides with the open-closed distinguished triangle for the pair ($\eFl_{e}, \eFl_{> e}$).
  As a result, we get an isomorphism $\underline{\Upsilon}_{\F (\scrL)}^{\circ} \left( \uDelta_s^{H, \F (\scrL)} \right) \cong \uDelta_s^{\F (\scrL)}$ as desired.
\end{proof}

The functor $\tilde{\Upsilon}_{\F (\scrL)}^{\circ} : \widetilde{\scrD}_{\BS} (\fr{h}_{\F}, W_{\scrL}^\circ) \to \PME{\F (\scrL)}{\F (\scrL)}^{\BS, \circ} (\F)$ induces a functor
\[\underline{\tilde{\Upsilon}}_{\F (\scrL)}^{\circ} : \widetilde{D}^m (\fr{h}_{\F}, W_{\scrL}^{\circ}) \to \DE{\F (\scrL)}^m (\F).\]
By Lemma \ref{lem:par_Psi_on_stds}, we have isomorphisms
\begin{equation}\label{eq:par_ME_Psi_on_stds}
  \underline{\tilde{\Upsilon}}_{\F (\scrL)}^{\circ} (\uDelta_w^{H, \F (\scrL)}) \cong \uDelta_w^{\F (\scrL)} \qquad \text{and} \qquad \underline{\tilde{\Upsilon}}_{\F (\scrL)}^{\circ} (\unabla_w^{H, \F (\scrL)}) \cong \unabla_w^{\F (\scrL)}. 
\end{equation}
We now have the preliminaries need to prove Claim \ref{claim:par_field_case}.

\begin{midsecproof}{Claim \ref{claim:par_field_case}}
  It suffices to prove that $\underline{\tilde{\Upsilon}}_{\F (\scrL)}^{\circ}$ is fully faithful.
  By Lemma \ref{lem:hom_vanishing_for_ME}, we have that 
  \begin{equation}\label{eq:par_field_case_1}
    \Hom_{\widetilde{D}^m (\fr{h}_{\F}, W_{\scrL}^{\circ})} (\uDelta_x^{H, \F (\scrL)} , \unabla_y^{H, \F (\scrL)} \langle m \rangle [n]) \cong \begin{cases} \F & \text{if } x=y \text{ and } n=m=0, \\ 0 & \text{otherwise,}\end{cases}
  \end{equation}
  and
  \begin{equation}\label{eq:par_field_case_2}
    \Hom_{\DE{\F (\scrL)}^m (\F)} (\uDelta_x^{\F (\scrL)} , \unabla_y^{\F (\scrL)} \langle m \rangle [n]) \cong \begin{cases} \F & \text{if } x=y \text{ and } n=m=0,\\ 0 & \text{otherwise.}\end{cases}
  \end{equation}
   In view of this, (\ref{eq:par_ME_Psi_on_stds}), and Beĭlinson's lemma \cite[Lemma 3.9.3]{ABG}, it suffices to prove that for all $w \in W_{\scrL}^{\circ}$, there exists a nonzero morphism $f : \uDelta_w^{H, \F (\scrL)} \to \unabla_w^{H, \F (\scrL)}$ such that $\underline{\tilde{\Upsilon}}_{\F (\scrL)}^{\circ} (f) : \uDelta_w^{\F (\scrL)} \to \unabla_w^{\F (\scrL)}$ is nonzero. 

   Let $\uw$ be a reduced expression of $w$ in $W_{\scrL}^{\circ}$.
   By Proposition \ref{prop:mixed_conv_rules_par}, Corollary \ref{cor:endo_mixed_conv_rules}, and the chain complex description of $\uDelta_s^{H, \F(\scrL)}$ (resp. $\uDelta_{\us}^{\F (\scrL)}$) for $s \in S_{\scrL}^{\circ}$, the complex $\uDelta_w^{H, \F(\scrL)}$ (resp. $\uDelta_{w}^{\F(\scrL)}$) is isomorphic to a complex $\scrF_{\uw}^{H, \F(\scrL)}$ (resp. $\scrF_{\ux}^{\F(\scrL)}$)
   such that
   \[(\scrF_{\uw}^{H, \F(\scrL)})^n = \begin{cases} B_{\uw} & n = 0, \\ 0 & n < 0, \end{cases} \qquad \text{and}\qquad (\scrF_{\uw}^{\F(\scrL)})^n = \begin{cases} \scrE_{\uw}^{\F(\scrL)} & n =0, \\ 0 & n < 0.\end{cases}\]
   Moreover, for $n > 0$, $\left(j_{w}^{H*} \scrF_{\uw}^{H, \F(\scrL)} \right)^n =0$ and $\left( j_w^* \scrF_{\uw}^{\F(\scrL)} \right)^n = 0$. Similarly, $\unabla_w^{H, \F(\scrL)}$ (resp. $\unabla_{w}^{\F(\scrL)}$) is isomorphic to a complex $\scrG_{\uw}^{H, \F(\scrL)}$ (resp. $\scrG_{\uw}^{\F(\scrL)}$)
   such that
   \[(\scrG_{\uw}^{H, \F(\scrL)})^n = \begin{cases} B_{\uw} & n = 0, \\ 0 & n > 0, \end{cases} \qquad \text{and}\qquad (\scrG_{\uw}^{\F(\scrL)})^n = \begin{cases} \scrE_{\uw}^{\F(\scrL)} & n =0, \\ 0 & n > 0.\end{cases}\]
   Moreover, for $n < 0$, $\left( j_{w}^{H*} \scrG_{\uw}^{H, \F(\scrL)} \right)^n =0$ and $\left( j_w^* \scrG_{\uw}^{\F(\scrL)} \right)^n = 0$.

   Consider the map $f : \uDelta_w^{H, \F(\scrL)} \to \unabla_w^{H, \F(\scrL)}$ defined by the morphism of chain complexes
   \[\begin{tikzcd}[row sep=small]
    \vdots                                                & \vdots                                    \\
    {(\scrF_{\uw}^{H, \F(\scrL)})^1} \arrow[r, "0"] \arrow[u] & 0 \arrow[u]                               \\
    {B_{\uw}} \arrow[r, "\id"] \arrow[u]   & {B_{\uw}} \arrow[u]        \\
    0 \arrow[r, "0"] \arrow[u]                            & {(\scrG_{\uw}^{H, \F(\scrL)})^{-1}.} \arrow[u] \\
    \vdots \arrow[u]                                      & \vdots \arrow[u]                         
    \end{tikzcd}\]
   Under the isomorphisms $\uDelta_w^{\F(\scrL)} \cong \scrF_{\uw}^{\F(\scrL)}$ and $\unabla_w^{\F(\scrL)} \cong \scrG_{\uw}^{\F(\scrL)}$, the map $\scrF_{\uw}^{\F(\scrL)} \to \scrG_{\uw}^{\F(\scrL)}$ induced by $\underline{\tilde{\Upsilon}}_{\F(\scrL)}^{\circ} (f)$ corresponds to the chain map given by $(\scrF_{\uw}^{\F(\scrL)})^0 = \scrE_{\uw}^{\F(\scrL)} \stackrel{\id}{\to} \scrE_{\uw}^{\F(\scrL)} = (\scrG_{\uw}^{\F(\scrL)})^0$ and $(\scrF_{\uw}^{\F(\scrL)})^n \stackrel{0}{\to} (\scrG_{\uw}^{\F(\scrL)})^n$ for all $n \neq 0$.
   In particular, $\underline{\tilde{\Upsilon}}_{\F (\scrL)}^{\circ} (f)$ is nonzero.
\end{midsecproof}

\subsection{Whittaker Category and Rigidified Minimal IC Sheaves}\label{subsec:whit}

By Proposition \ref{prop:structure_of_min_IC}, there
are isomorphisms
\[c_{\gamma, \beta} : \IC_{w^\gamma}^{\scrL'} \star \IC_{w^\beta}^{\scrL} \stackrel{\sim}{\to} \IC_{w^{\gamma \beta}}^{\scrL}\]
for blocks $\beta \in \uW{\scrL'}{\scrL}$ and $\gamma \in \uW{\scrL''}{\scrL'}$.
However, these isomorphisms are not canonical, since the minimal IC sheaves themselves are only defined up to non-unique isomorphism. 
We would like to fix a particular representative of the isomorphism class of the minimal IC sheaves as well as isomorphisms between their convolution products in a coherent way.
In particular, for $\beta \in \uW{\scrL'}{\scrL}$, $\gamma \in \uW{\scrL''}{\scrL'}$, and $\delta \in \uW{\scrL'''}{\scrL''}$, we would like the isomorphisms to make the following diagram commute
\begin{equation}\label{eq:failure_of_assoc}
  \begin{tikzcd}
    \IC_{w^\delta}^{\scrL''} \star \IC_{w^\gamma}^{\scrL'} \star \IC_{w^\beta}^{\scrL} \arrow[r, "{c_{\delta, \gamma} \star \id}"] \arrow[d, "{\id \star c_{\gamma, \beta}}"] &
    \IC_{w^{\delta \gamma}}^{\scrL'} \star \IC_{w^\beta}^{\scrL} \arrow[d, "{c_{\delta \gamma, \beta}}"] \\
    \IC_{w^\delta}^{\scrL''} \star \IC_{w^{\gamma \beta}}^{\scrL} \arrow[r, "{c_{\delta, \gamma \beta}}"]                                                                 &
    \IC_{w^{\delta \gamma \beta}}^{\scrL}.
  \end{tikzcd}
\end{equation}
If one is too flippant about making choices, (\ref{eq:failure_of_assoc}) can indeed fail to commute (see \cite[Example 6.7]{LY}).

A solution given in \cite{LY} is to use the Whittaker model to rigidify the minimal IC's.
Unfortunately, the usual Whittaker category does not make sense when working with the analytic topology since no version of the Artin--Schreier sheaf exists.
A workaround is given in \cite{Gou} and \cite{BR} by temporarily changing the sheaf-theoretic setting from Betti sheaves to étale sheaves, and then comparing the resulting categories of sheaves.
We adopt a similar strategy to construct a monodromic Whittaker category that will allow us to rigidify the minimal IC sheaves.
Unfortunately, this comes at the cost: we will need to restrict our coefficient rings so that étale sheaves are well-behaved. 
Moreover, while the Whittaker category can be defined when $G$ is not of finite type, it is too large to rigidify the minimal IC sheaves. As a result, we will require that $G$ is finite-type for the remainder of the paper.

\subsubsection{Coefficient Requirements}\label{subsec:modular_triples}

Fix a finite extension $\O$ of $\Z_{\ell}$. Denote by $\K$ its field of fractions and by $\F$ its residue field.
For the time being, we will only allow $\k$ to be one of the rings $\K, \O$, or $\F$. The triple $(\K, \O, \F)$ is called an \emph{$\ell$-modular triple}.
Recall that the map $\O \twoheadrightarrow \F$ admits a multiplicative section $\F^{\times} \to \O^{\times}$ called the \emph{Teichmüller lift}.
Such a map is far from unique. Nonetheless, we will fix the choice of Teichmüller lift throughout.

Given a torsion local system $\scrL \in \Ch^{\mu} (T, \F)$, the Teichmüller lift allows us to define a sheaf $\scrL_{\O} \in \Ch^{\mu} (T, \O)$. 
Explicitly, $\scrL$ is defined by a morphism $\rho_{\scrL} : \bfY \to \F^{\times}$. We can then define $\scrL_{\O}$ via the composition $\rho_{\scrL_{\O}} : \bfY \to \F^{\times} \to \O^{\times}$.
Since the Teichmüller lift is a multiplicative section, we have that $\F (\scrL_{\O}) = \scrL$.

We can also consider the extension-of-scalars map
\begin{equation}\label{eq:O_to_K_loc_sys}
  \K (-) : \Ch^{\mu} (T, \O) \to \Ch^{\mu} (T, \K).
\end{equation}
Any torsion $\O$-local system on $T$ is defined by a morphism $\bfY \to \O^{\times}$ which takes values in the roots of unity in $\O$.
It is well-known that the roots of unity of both $\K$ and $\O$ are equal, as a result, the map (\ref{eq:O_to_K_loc_sys}) is a bijection.

Let $\scrL \in \Ch^{\mu} (T, \k)$. The discussion above allows us to uniquely define $\scrL_{\k'} \in \Ch^{\mu} (T, \k')$ where $\k' \in \{ \K, \O, \F \}$ which is either obtained by extension-of-scalars or such that $\scrL_{\k}$ can be recovered from extension-of-scalars.
In light of this, we will fix a $W$-orbit $\fr{o} = W \cdot \scrL$ and view it simultaneously as a subset of $\Ch^{\mu} (T, \k)$ for all $\k \in \{ \K, \O, \F\}$.
Likewise, we regard any $\scrL \in \fr{o}$ as a $\k$-local system for all $\k \in \{ \K, \O, \F\}$.

There are straightforward variations of the above for when $\k = \overline{\Q}_{\ell}$ or $\k = \overline{\F}_{\ell}$ as well; although, we reserve the notation of the $\ell$-modular triple $(\K, \O, \F)$ for the finite extension versions.
We note that if $\scrL \in \Ch^{\mu} (T, \overline{\F}_{\ell})$ (resp. for $\overline{\Q}_{\ell}$--coefficients), we can find some $\ell$-modular triple $(\K, \O, \F)$
such that $\scrL$ arises via extension-of-scalars from some $\scrL' \in \Ch^{\mu} (T, \O)$.

\subsubsection{Étale Local Systems on \texorpdfstring{$T$}{T}}

Recall that $T$ canonically arises via base change from a group scheme $T_{\Z}$ over $\Z$.
For a commutative ring $A$, we will write $T_A$ for the base change of $T_{\Z}$ over $A$.
For the time being, we will change our sheaf-theoretic setting to étale sheaves.
We will start by considering étale local systems of $\O$-modules on $T_{A}$.
Write $D_{\cons}^{\et} (T_A, \O)$ for the bounded derived category of étale constructible sheaves on $T_A$ with $\O$-coefficients.
We will consider the subcategory $\Loc^{\et} (T_A, \O)$ of $D_{\cons}^{\et} (T_A, \O)$ consisting of étale local systems.
Contained in  $\Loc^{\et} (T_A, \O)$ is the set $\Ch_{\et}^{\mu} (T_A, \O)$ of torsion rank 1 multiplicative local systems.

By \cite[ 6.1.2 (B'')]{BBD}, pullback along the morphism $T_{\C}^{\et} \to T_{\C}^{\an}$ from the étale to the analytic site on $T_{\C}$ induces an equivalence of categories
\[D_{\cons} (T_{\C}, \O) \stackrel{\sim}{\to} D_{\cons}^{\et} (T_{\C}, \O).\]
Moreover, the above equivalence restricts to that of abelian categories
\[\Loc (T_{\C}, \O) \stackrel{\sim}{\to} \Loc^{\et} (T_{\C}, \O).\]
This follows from the observation that the equivalences from \cite[ 6.1.2 (B'')]{BBD} commute with the six-functors.
Likewise, this equivalence gives a canonical $W$-equivariant bijection
\begin{equation}\label{eq:ch_C_etale_and_analytic}
  \Ch^{\mu} (T_{\C}, \O) \stackrel{\textnormal{1-1}}{\longleftrightarrow} \Ch_{\et}^{\mu} (T_{\C}, \O).
\end{equation}

Let $\scrL \in \Ch^{\mu} (T_{\C}, \O)$ and write $\fr{o} = W \cdot \scrL$ for its $W$-orbit.
We will abuse notation and regard $\scrL$ as an étale local system, and likewise, $\fr{o}$ as a subset of étale local systems by passing through the bijection (\ref{eq:ch_C_etale_and_analytic}).

Our next goal is to transport $\scrL$ to an étale local system on $T_{\kappa}$ where $\kappa$ is an algebraically closed field of positive characteristic.
The general technique follows from \cite[\S 6.1.9]{BBD}.
Let $n$ denote the order of $\scrL$. Recall that $n$ is also the order of any $\scrL' \in \fr{o}$.
We can then find a characteristic 0 commutative ring $A \subseteq \C$ which is finite type over $\Z$ such that $A$ contains a primitive $n$-th root of unity $\zeta$ and such that $n$ is invertible in $A$.
For example, we could take $A = \Z \left[ \frac{1}{n}, \zeta\right]$.
We can then find $A \subseteq \mathfrak{R} \subseteq \C$ which is a strictly henselian discrete valuation ring whose residue field $\kappa$ is algebraically closed and of positive characteristic.
The existence of such a ring is explained in \cite[Lemme 6.1.9]{BBD}.
Our choice of $A$ and $\mathfrak{R}$ ensures that the local systems in $\fr{o}$ descend to local systems on $T_{\mathfrak{R}}$.
In more detail, let $f : T_{\C} \to T_{\mathfrak{R}}$. For all $\scrL' \in \fr{o}$, we can find some $\scrL'' \in \Loc^{\et} (T_{\mathfrak{R}}, \O)$ such that $f^* \scrL'' \cong \scrL'$.
The fact that $\scrL''$ is actually a multiplicative sheaf is not entirely obvious. It follows from \cite[Lemma 11.4.1]{Gou} (see also \cite[Remark 11.4.2]{Gou}) that $\scrL'' \in \Ch_{\et}^{\mu} (T_{\mathfrak{R}, \O})$.
Moreover, this correspondence is $W$-equivariant. From here, we can then look at the residue map $q : T_{\mathfrak{R}} \to T_{\kappa}$ to construct a local system $q^* \scrL'' \in \Ch_{\et}^{\mu} (T_{\kappa}, \O)$.
Combining this construction with (\ref{eq:C_etale_vs_analytic}) produces a $W$-equivariant assignment
\[\fr{o} \to \Ch_{\et}^{\mu} (T_{\kappa}, \O).\] 
As a result, we can abuse notation, and identify our local systems in $\fr{o}$ with étale local systems on $T_{\kappa}$.
Note that $\kappa$ depends on $\fr{o}$, and more precisely, the order of $\scrL$. However, this will not be a problem since our $W$-orbit will always be fixed when comparing between the analytic and étale settings.

\begin{remark}\label{rem:isogeny_for_torsion}
  Our torsion assumption ensures that if $\scrL \in \Ch^{\mu} (T_{\C}, \O)$, we can find a finite étale central isogeny $\nu : \tilde{T}_{\C} \to T_{\C}$ and a character $\chi : \ker (\nu) \to \O^{\times}$ such that $\scrL$ appears as the direct summand of $\nu_* \underline{\O}_{\tilde{T}_{\C}}$ where $\ker (\nu)$ acts through $\chi$.
  Since the order of $\scrL$ is invertible in $\mathfrak{R}$ and $\kappa$, the same can be said when we view $\scrL$ as an étale sheaf on $T_{\mathfrak{R}}$ or $T_{\kappa}$.
\end{remark}

\subsubsection{From Betti Sheaves to Étale Sheaves}\label{subsec:etale_and_analytic}

Recall that we have an integral form $G_{\Z}$ for $G$ with Borel subgroup $B_{\Z}$.
Let $U_{\Z}$ denote the unipotent radical of $B_{\Z}$.
If $A$ is a commutative ring, we will write $G_A$, $B_A$, and $U_A$ for the base change of $G_{\Z}$, $B_{\Z}$, and $U_{\Z}$ over $A$ respectively. 
We can then consider the enhanced flag variety over $A$,
\[\eFl_A = U_{A} \backslash G_{A}.\]

We will take $A \in \{ \kappa, \mathfrak{R}, \C\}$.
Let $\scrL, \scrL' \in \fr{o}$. We can regard $\scrL$ as a multiplicative sheaf on $B_{A}$ via pullback along $B_{A} \to T_{A}$.
We can now define the bounded derived category of $(T_{A} \times B_{A}, \scrL' \boxtimes \scrL^{-1})$-equivariant constructible sheaves on $\eFl_{A}$,
\[\DEE{\scrL'}{\scrL}^{\et} (G_{A}, \O) \coloneq D_{\cons}^{\et} (T_{A} \backslash_{\scrL'} \eFl_{A} /_{\scrL} B_{A}, \O).\]
For a precise construction of this category see \cite{LY} or \cite{Gou}.
The same constructions from \S\S\ref{subsec:hecke}-\ref{sec:equiv_conv} can be made in the étale setting.
In particular, there is a convolution bifunctor
\[(-) \star (-) : \DEE{\scrL''}{\scrL'}^{\et} (G_{A}, \O) \times \DEE{\scrL'}{\scrL}^{\et} (G_{A}, \O) \to \DEE{\scrL''}{\scrL}^{\et} (G_{A}, \O)\]
which is suitably associative.

As explained in \cite[Lemma 11.5.3]{Gou} following ideas from \cite[\S6.1.2]{BBD}, pullback along the morphism $\eFl_{\C}^{\et} \to \eFl_{\C}^{\an}$ from the étale site to the analytic site on $\eFl_{\C}$
induces an equivalence of categories
\begin{equation}\label{eq:C_etale_vs_analytic}
  \DEE{\scrL'}{\scrL} (G_{\C}, \O) \stackrel{\sim}{\to} \DEE{\scrL'}{\scrL}^{\et} (G_{\C}, \O).
\end{equation}
Moreover, this equivalence commutes with convolution in the obvious way. Note that \cite{Gou} only uses finite rings, but the results of \cite{BBD} hold for $\O$ as well (cf., \cite[\S6.1.2 (B'')]{BBD}).
Similarly, while \cite[Lemma 11.5.3]{Gou} is only stated for the right equivariant Hecke category, the equivalence of (\ref{eq:C_etale_vs_analytic}) still follows from the same argument.
Since our construction of twisted-equivariant sheaves is different from \cite{Gou}, one must be slightly careful.
The key observation is that since $\fr{o}$ consists of torsion local systems, we can use Remark \ref{rem:isogeny_for_torsion} and Lemma \ref{lem:desc_as_equivariant_sheaves} to reconcile our definition of twisted-equivariant sheaves with Gouttard's.
The argument from \cite{Gou} then proceeds without issue.

Consider the diagram of schemes
\[\eFl_{\C} \to \eFl_{\mathfrak{R}} \leftarrow \eFl_{\kappa}.\]
By taking pullbacks along these morphisms, we obtain functors
\begin{equation}\label{eq:C_etale_vs_F_etale}
  \DEE{\scrL'}{\scrL}^{\et} (G_{\C}, \O) \leftarrow \DEE{\scrL'}{\scrL}^{\et} (G_{\mathfrak{R}}, \O) \to \DEE{\scrL'}{\scrL}^{\et} (G_{\kappa}, \O).
\end{equation}
By the same argument from \cite[(11.5.6)]{Gou} (see also \cite[\S 6.1.9]{BBD}), these functors are equivalences of categories and commute with convolution.

We combine the equivalences of (\ref{eq:C_etale_vs_analytic}) and (\ref{eq:C_etale_vs_F_etale}) to obtain equivalences
\begin{equation}\label{eq:F_etale_vs_analytic}
  \DEE{\scrL'}{\scrL} (G_{\C}, \O) \cong \DEE{\scrL'}{\scrL}^{\et} (G_{\kappa}, \O),
\end{equation}
which commute with convolution.

\subsubsection{The Monodromic Whittaker Category}

For this section, to ease the notation, we will have $G$ be defined over $\kappa$ instead of $\C$.
Likewise, all the relevant subgroups of $G$ and flag varieties will be over $\kappa$ as well.
At the end of the section, we will need to reintroduce the complex algebraic variants. This will be accomplished by adding a $\C$ subscript to the related objects.

Let $B^{-}$ denote the opposite Borel subgroup of $B$ containing $T$. We will write $U^{-}$ for the unipotent radical of $B^{-}$.
We fix an isomorphism $U_{\alpha} \cong \G_a$ for each $\alpha \in \Phi^{-}$. Consider the homomorphism
\[\psi : U^{-} \to U^{-} /[U^{-}, U^{-}] \cong \prod_{i \in I} U_{-\alpha_i} \stackrel{\sim}{\to} \prod_{i \in I} \G_a \stackrel{\textnormal{sum}}{\longrightarrow} \G_a.\]
Fix a $p$-th root of unity in $\k$ and an Artin--Schreier local system $\AS$ on $\G_a$. We define $\scrL_{\psi} \coloneq \psi^* \AS$.

Let $J \subseteq I$, and write $P_J = U^J L_J$ for the standard parabolic subgroup of $G$ containing $B$ generated by $U_{\alpha_j}$ for $j \in J$.
Let $\scrL \in \fr{o}$ such that $s_j (\scrL) \cong \scrL$ for all $j \in J$. We can then extend $\scrL$ to a multiplicative local system $\scrL^J$ on $L_J$ (see Lemma \ref{lem:extending_character_sheaves_to_levis}).
Moreover, we can regard $\scrL^J$ as a multiplicative local system on $P_J$ via pullback along the quotient map $P_J \to L_J$.
Define the \emph{monodromic Whittaker category},
\[\Whit{\scrL}^J (G, \O) \coloneq D_{\cons}^{\et} (U^{-} \backslash_{\scrL_{\psi}} G /_{\scrL^J} P_J, \O).\]
When $J = \emptyset$, we will omit its superscript from the notation.

The following lemma is well-known (see \cite[Lemma 12.6.7]{Gou} and \cite[Lemma 4.2.1]{BY}).

\begin{lemma}\label{lem:desc_of_whit}\quad
  \begin{enumerate}
    \item If $J \neq \emptyset$, then $\Whit{\scrL}^J (G, \O) = 0$.
    \item If $J = \emptyset$, then there is a $t$-exact equivalence of categories
    \begin{equation}\label{eq:desc_of_whit_1}
      \Whit{\scrL} (G, \O) \stackrel{\sim}{\to} D^b (\mod{\O}^{\fg}).
    \end{equation}
  \end{enumerate}
\end{lemma}

We will write $\scrF_{\psi}^{\scrL} \in \Whit{\scrL} (G, \O)$ for the object corresponding to $\O \in D^b (\mod{\O}^{\fg})$ under the equivalence (\ref{eq:desc_of_whit_1}).

The same convolution pattern from \S\ref{sec:equiv_conv} gives a functor
\[(-) \star (-) :  \Whit{\scrL'} (G, \O) \times \DEE{\scrL'}{\scrL}^{\et} (G, \O) \to \Whit{\scrL} (G, \O).\]
This functor is suitably associative with respect to convolution of monodromic Hecke categories. In particular, for $\scrL' = \scrL$, this gives a right action of $\DEE{\scrL}{\scrL}^{\et} (G, \O)$ on $\Whit{\scrL} (G, \O)$.
Under the equivalence of (\ref{eq:F_etale_vs_analytic}) we can replace the étale-version of the Hecke categories with the Betti-version.
In particular, there is a functor
\[(-) \star (-) : \Whit{\scrL'} (G_{\kappa}, \O) \times \DEE{\scrL'}{\scrL} (G_{\C}, \O) \to \Whit{\scrL} (G_{\kappa}, \O).\]

Note that $\O$ is a complete local ring, so we can make sense of parity sheaves in $\DEE{\scrL'}{\scrL} (G_{\C}, \O)$.

\begin{lemma}\label{lem:conv_in_whit}
  Let $w \in \W{\scrL'}{\scrL}$. Then
  \[\scrF_{\psi}^{\scrL'} \star \scrE_w^{\scrL} \cong \begin{cases} \scrF_{\psi}^{\scrL} & w = w^{\beta} \text{ for some block } \beta \in \uW{\scrL'}{\scrL}; \\ 0 & \text{otherwise.}\end{cases}\]
\end{lemma}
\begin{proof}
  The following argument is adapted from \cite[Lemma 5.10]{LY}.
  By induction on $\ell(w)$ and Theorem \ref{thm:conv_preserves_parity}, it suffices to just consider the case when $w = s$ is a simple reflection.
  In this case, we have that $\scrE_s^{\scrL} \cong \IC_s^{\scrL}$.

  If $s \notin W_{\scrL}^{\circ}$, then since $\IC_s^{s \scrL} \star \IC_s^{\scrL} \cong \IC_e^{\scrL}$, we have that 
  \[(-) \star \IC_s^{\scrL} : \Whit{ s \scrL} (G_{\kappa}, \O) \to \Whit{\scrL} (G_{\kappa}, \O)\]
  is an equivalence of categories. It is a routine exercise to check that convolution by standards (resp. costandards) is left (resp. right) perverse $t$-exact.
  Since $\nabla_s^{\scrL} \cong \IC_s^{\scrL} \cong \Delta_s^{\scrL}$, it then follows that $(-) \star \IC_s^{\scrL}$ is perverse $t$-exact.
  As a result, $\scrF_{\psi}^{s \scrL } \star \scrE_s^{\scrL}$ must be a simple perverse sheaf in $\Whit{\scrL} (G_{\kappa}, \O)$. 
  By Lemma \ref{lem:desc_of_whit}, there is a unique simple perverse Whittaker sheaf, and so we conclude that $\scrF_{\psi}^{s \scrL } \star \scrE_s^{\scrL} \cong \scrF_{\psi}^{\scrL}$.

  Now suppose $s \in W_{\scrL}^{\circ}$. Let $i \in I$ such that $s_i = s$. Consider the pair of adjoint functors
  \[\pi_s^* : \Whit{\scrL}^{\{i\}} (G_{\kappa}, \O) \rightleftarrows  \Whit{\scrL} (G_{\kappa}, \O)  : \pi_{s*}.\]
  An easy analogue of Lemma \ref{lem:conv_with_ICs_rewritting} gives a natural isomorphism
  \[\scrF_{\psi}^{\scrL} \star \IC_s^{\scrL} \cong \pi_s^* \pi_{s*} \scrF_{\psi}^{\scrL} [1].\]
  However, $\Whit{\scrL}^{\{i\}} (G_{\kappa}, \O) = 0$ by Lemma \ref{lem:desc_of_whit}, and hence $\scrF_{\psi}^{\scrL} \star \IC_s^{\scrL} = 0$.
\end{proof}

\subsubsection{Rigidified Minimal IC Sheaves}

\begin{definition}\label{def:rigid_min_IC}
  A \emph{rigidified minimal IC sheaf} consists of a pair $(\scrG_{\beta}^{\scrL}, \epsilon)$ where $\scrG_{\beta}^{\scrL}$ is a representative of the isomorphism class $\IC_{w^{\beta}}^{\scrL} \in \DEE{w^{\beta} \scrL}{\scrL} (G_{\C}, \O)$ and $\epsilon$ is an isomorphism in
  $\Whit{\scrL} (G_{\kappa}, \O)$,
  \[\epsilon : \scrF_{\psi}^{w^\beta \scrL} \star \scrG_{\beta}^{\scrL} \stackrel{\sim}{\to} \scrF_{\psi}^{\scrL}.\]
\end{definition}

By Lemma \ref{lem:conv_in_whit}, rigidified minimal IC sheaves exist for every block. 
Since $\End (\IC_{w^\beta}^{\scrL}) \cong \O$ and $\End (\scrF_{\psi}^{\scrL}) \cong \O$ (see Lemma \ref{lem:desc_of_whit}), for any two rigidified
minimal IC sheaves $(\scrG_{\beta}^{\scrL}, \epsilon)$, $(\scrH_{\beta}^{\scrL}, \epsilon')$, there exists a unique isomorphism $\alpha : \scrG_{\beta}^{\scrL} \to \scrH_{\beta}^{\scrL}$ such
that $\epsilon' \circ (\id \star \alpha) = \epsilon$.
For each block $\beta$, we will fix a rigidified minimal IC sheaf. We abuse notation and write $\IC_{w^\beta}^{\scrL}$ for the underlying object and $\epsilon_\beta$ for the map making ($\IC_{w^\beta}^{\scrL}$, $\epsilon_{\beta}$) into a rigidified minimal IC sheaf.

Our definition of rigidified minimal IC sheaf requires working with $\O$-coefficients. 
There are versions of the Whittaker category that work with $\K$- and $\F$-coefficients, but this slightly complicates the procedure from \S\ref{subsec:etale_and_analytic}.
Namely, the $\F$-coefficient version can be carried out exactly the same using results from \cite{BBD}; however, for $\K$-coefficients, the complex étale and complex analytic categories of sheaves are not equivalent (see \cite[\S6.1.2]{BBD}). 
Fortunately, having already picked out representatives of the minimal IC sheaves for $\O$-coefficients, we can produce fixed representatives for the other coefficient rings by extension-of-scalars.
Namely, we can define $\IC_{w^{\beta}}^{\k (\scrL)} \coloneq \k \left( \IC_{w^{\beta}}^{\scrL} \right)$ for all $\scrL \in \fr{o}_{\O}$ and $\k \in \{\K, \F, \overline{\Q}_{\ell}, \overline{\F}_{\ell}\}$.
By the discussion from \S\ref{subsec:modular_triples}, all such isomorphism classes of minimal IC sheaves arise in this way.

\begin{lemma}\label{lem:assoc_of_rigid_min_IC}
  Let $\k \in \{\K, \O, \F, \overline{\Q}_{\ell}, \overline{\F}_{\ell}\}$.
  For $\beta \in \uW{\scrL'}{\scrL}$ and $\gamma \in \uW{\scrL''}{\scrL'}$, there exists a canonical isomorphism,
  \[b_{\gamma, \beta} : \IC_{w^\gamma}^{\scrL'} \star \IC_{w^\beta}^{\scrL} \stackrel{\sim}{\to} \IC_{w^{\gamma \beta}}^{\scrL}.\]
  Moreover, if $\delta \in \uW{\scrL'''}{\scrL''}$, there is an equality of morphisms
  \[ b_{\delta \gamma, \beta } (b_{\delta, \gamma} \times \id) = b_{\delta, \gamma \beta} (\id \times b_{\gamma, \beta}).\]
\end{lemma}
\begin{proof}
  The lemma follows for $\O$-coefficients from the same argument given in the proof of \cite[Lemma 5.12]{LY}. We deduce the lemma for $\k$-coefficients with $\k \in \{\K, \F, \overline{\Q}_{\ell}, \overline{\F}_{\ell}\}$ by extension-of-scalars.
\end{proof}

\begin{corollary}\label{cor:trivialization_of_bicats}
   Let $\k \in \{\K, \O, \F, \overline{\Q}_{\ell}, \overline{\F}_{\ell}\}$.
  The assignment $\beta \mapsto \IC_{w^\beta}^{\scrL}$ extends to a 2-functor of $\IC_{\min} : \Xi \to \Parity{}{} (G_{\C}, \k)$.
\end{corollary}

\begin{remark}
  If $G$ is instead of affine type, then the Whittaker category decomposes into a direct sum of categories which look like the Whittaker categories of reductive groups. 
  A crucial step towards developing a version of Corollary \ref{cor:trivialization_of_bicats} for Kac--Moody groups of affine type is to find a means of identifying these categories.
  Recent work of Dhillon--Li--Yun--Zhu \cite{DLYZ} suggests that this should be doable, but we will not investigate this any further in this paper.
\end{remark}

\begin{remark}
  The requirement that $\k \in \{\K, \O, \F, \overline{\Q}_{\ell}, \overline{\F}_{\ell}\}$ is fairly strong, but probably unnecessary.
  The author expects that no new requirement on $\k$ should be needed, provided one restricts to torsion local systems.
  Currently, the main obstacle is defining a Whittaker category using Betti sheaves rather than étale sheaves.
  The Kirillov model of Gaitsgory and Lysenko \cite{GL, Gai} seems well suited for this task; however, there is certainly some complicated subtleties.
  Namely, all $\G_m$-actions on $U^{-}$ which allow for $\psi : U^{-} \to \G_a$ to be $\G_m$-equivariant require that $G$ is of adjoint type (e.g., $\G_m$ could act by the half-sum of simple coroots).
  In this setting, there is a unique block between any $\scrL, \scrL' \in \fr{o}$ which is rather uninteresting.
  One could instead allow $\G_m$ to act on $\G_a$ by dilation to some power, but this makes the algebraic description from Lemma \ref{lem:desc_of_whit} more complicated.
\end{remark}

\subsection{The Endoscopic Groupoid}

In order to extend Theorem \ref{thm:endoscopy_neutral_block_Kac_moody} to all blocks, we need to describe a 2-categorical enhancement of parity sheaves on the endoscopic group.
To accomplish this, we will construct a groupoid $\fr{H}$ in schemes. The groupoid $\fr{H}$ carries a morphism of groupoids $\fr{H} \to \Xi$ whose fibers over $\scrL \in \Xi$ are given by the endoscopic groups $H_{\scrL}^{\circ}$.
The geometric aspects of this section are covered in \cite[\S10]{LY}. We continue the assumption that $G$ is of finite type. 
This assumption is not strictly necessary, and our constructions also hold for general $G$, in which case, $\fr{H}$ will be a groupoid in ind-schemes.

\subsubsection{Groupoid of Lifts}

Let $\beta \in \uW{\scrL'}{\scrL}$. Consider the scheme ${}_{\scrL'} \tilde{\Xi}_{\scrL}^{\beta}$ given by the moduli space of lifts of $w^{\beta}$,
\[{}_{\scrL'} \tilde{\Xi}_{\scrL}^{\beta} \coloneq \{ \dot{w} \in N_G (T) \mid \dot{w} T = w^{\beta} \}. \]
We can then construct a groupoid $\tilde{\Xi}$ in schemes whose object set is $\fr{o}$, and the morphism set is given by
\[ {}_{\scrL'} \tilde{\Xi}_{\scrL} \coloneq \bigsqcup_{\beta \in \uW{\scrL'}{\scrL}} {}_{\scrL'} \tilde{\Xi}_{\scrL}^{\beta}.\]
Composition in $\tilde{\Xi}$ is given by multiplication in $N_G (T)$.

\begin{definition}
  A \emph{relative pinning} for the endoscopic group $H_{\scrL}^{\circ}$ is a collection of isomorphisms $H_{\scrL, \alpha}^{\circ} \to G_{\alpha}$ for each simple root $\alpha \in \Phi_{\scrL}^+$ where $H_{\scrL, \alpha}^{\circ}$ (resp. $G_{\alpha}$) is the root subgroup for $\alpha$ of $H_{\scrL}^{\circ}$ (resp. $G$).
\end{definition}

Relative pinnings for endoscopic groups exist and are unique up to a unique isomorphism. For each $\scrL \in \Ch^{\mu} (T, \k)$, we fix a relative pinning of $H_{\scrL}^{\circ}$.
For each $\dot{w} \in {}_{\scrL'} \tilde{\Xi}_{\scrL}^{\beta}$, there is a unique isomorphism
\[\sigma (\dot{w}) : H_{\scrL}^{\circ} \stackrel{\sim}{\to} H_{\scrL'}^{\circ},\]
which preserves the fixed relative pinning. In more detail, $\sigma\vert_T$ is given by the action of $w$. Moreover, it induces an isomorphism of Kac--Moody root data between $(S_{\scrL}^{\circ}, \bfX, \{ \alpha_s\}_{s \in S_{\scrL}^{\circ}}, \{\alpha_s^{\vee}\}_{s \in S_{\scrL}^{\circ}})$ and $(S_{\scrL'}^{\circ}, \bfX, \{ \alpha_s\}_{s \in S_{\scrL'}^{\circ}}, \{\alpha_s^{\vee}\}_{s \in S_{\scrL'}^{\circ}})$.
For each simple real root $\alpha \in \Phi_{\re, \scrL}$, the morphism $\sigma (\dot{w})$ restricts to an isomorphism of root subgroups $H_{\scrL, \alpha}^{\circ} \stackrel{\sim}{\to} H_{\scrL', w \alpha}^{\circ}$ which makes the following diagram commute
\[ \begin{tikzcd}
  {H_{\scrL, \alpha}^{\circ}} \arrow[d, "\sigma (\dot{w})"] \arrow[r] & G_{\alpha} \arrow[d, "\textnormal{Ad} (\dot{w})"] \\
  {H_{\scrL', w \alpha}^{\circ}} \arrow[r]                            & G_{w \alpha}.                                     
  \end{tikzcd}
\]

\subsubsection{Translations of Endoscopic Groups}

Let $\beta \in \uW{\scrL'}{\scrL}$ and define a scheme
\[{}_{\scrL'} H_{\scrL}^{\beta} \coloneq {}_{\scrL'} \tilde{\Xi}_{\scrL}^{\beta} \times^T H_{\scrL}^{\circ},\]
where the action of $T$ on $H_{\scrL}^{\circ}$ is by left translation, and its action on $ {}_{\scrL'} \tilde{\Xi}_{\scrL}^{\beta}$ is by right translation.
There is a canonical isomorphism
\[{}_{\scrL'} H_{\scrL}^{\beta} \coloneq {}_{\scrL'} \tilde{\Xi}_{\scrL}^{\beta} \times^T H_{\scrL}^{\circ} \cong H_{\scrL'}^{\circ} \times^T  {}_{\scrL'} \tilde{\Xi}_{\scrL}^{\beta}\]
given by $(\dot{w}, h)\mapsto (\sigma (\dot{w})(h), \dot{w})$. 
As a result, for blocks $\beta \in \uW{\scrL'}{\scrL}$ and $\gamma \in \uW{\scrL''}{\scrL'}$, we can define a multiplication map
\[{}_{\scrL''} H_{\scrL'}^{\gamma} \times {}_{\scrL'} H_{\scrL}^{\beta} \to {}_{\scrL''} H_{\scrL}^{\gamma \beta},\]
\[ (\dot{y}, h', \dot{x}, h) \mapsto (\dot{y} \dot{x}, \sigma (\dot{x}^{-1}) (h') h). \]
The composition map is associative over the set of blocks. As a result, we can define a groupoid $\fr{H}$ in schemes whose object set is $\fr{o}$ and morphisms are given by $\bigsqcup_{\beta \in \uW{\scrL}{\scrL'}} {}_{\scrL'} H_{\scrL}^{\beta}$.

\subsubsection{Parity Sheaves on \texorpdfstring{$\fr{H}$}{the Endoscopic Groupoid}}

We define the flag variety ${}_{\scrL'} \Fl_{\scrL}^{\beta}$ associated to ${}_{\scrL'} H_{\scrL}^{\beta}$ as the scheme
\[{}_{\scrL'} \Fl_{\scrL}^{\beta} \coloneq B_{\scrL'} \backslash {}_{\scrL'} H_{\scrL}^{\beta}.\]
It is naturally acted on the right by the Borel $B_{\scrL}$ of $H_{\scrL}^{\circ}$.
This gives rise to a Bruhat decomposition of $B_{\scrL}$-orbits.
\[{}_{\scrL'} \Fl_{\scrL}^{\beta} = \bigsqcup_{w \in \beta} {}_{\scrL'} \Fl_{\scrL, w}^{\beta},\]
where each ${}_{\scrL'} \Fl_{\scrL, w}^{\beta}$ is isomorphic to an affine space of dimension $\ell_{\beta} (w)$. We will write $j_w^H : {}_{\scrL'} \Fl_{\scrL, w}^{\beta} \to {}_{\scrL'} \Fl_{\scrL}^{\beta}$ for the obvious inclusion map.
Moreover, there is a $B_{\scrL}$-invariant isomorphism of schemes
\begin{equation}\label{eq:shifted_flag_vars_are_shifted}
  {}_{\scrL'} \Fl_{\scrL}^{\beta} \stackrel{\sim}{\to} \Fl_{H_{\scrL}^{\circ}}
\end{equation}
induced from the non-canonical isomorphism ${}_{\scrL'} \tilde{\Xi}_{\scrL}^{\beta} \cong T$. The isomorphism (\ref{eq:shifted_flag_vars_are_shifted}) induces isomorphisms of Bruhat strata
\[{}_{\scrL'} \Fl_{\scrL, w}^{\beta} \stackrel{\sim}{\to} \Fl_{H_{\scrL}^{\circ}, w^{\beta, -1} w}.\]

We can consider the derived category of $B_{\scrL}$-equivariant constructible sheaves on ${}_{\scrL'} \Fl_{\scrL}^{\beta}$, denoted by $D_{\cons} ( {}_{\scrL'} \Fl_{\scrL}^{\beta} / B_{\scrL}, \k)$.
The groupoid structure on $\fr{H}$ gives rise to convolution bifunctors
\[(-) \star (-) : D_{\cons} ( {}_{\scrL''} \Fl_{\scrL'}^{\gamma} / B_{\scrL'}, \k) \times D_{\cons} ( {}_{\scrL'} \Fl_{\scrL}^{\beta} / B_{\scrL}, \k) \to D_{\cons} ( {}_{\scrL''} \Fl_{\scrL}^{\gamma \beta} / B_{\scrL}, \k),\]
which carries an associativity natural transformation and satisfies a variant of the pentagon axiom.

For each $w \in \beta$, we can define the standard and costandard sheaves, respectively,
\[\Delta_w^{H, \scrL} \coloneq j_{w!}^H \uk_{{}_{\scrL'} \Fl_{\scrL, w}^{\beta}} [\ell_{\beta} (w)] \hspace{0.5cm}\text{and}\hspace{0.5cm} \nabla_w^{H, \scrL} \coloneq j_{w*}^H \uk_{{}_{\scrL'} \Fl_{\scrL, w}^{\beta}} [\ell_{\beta} (w)].\]
When $w = w^{\beta}$, the stratum ${}_{\scrL'} \Fl_{\scrL, w^{\beta}}^{\beta}$ is closed, and we get a canonical isomorphism $\Delta_{w^{\beta}}^{H, \scrL} \cong \nabla_{w^{\beta}}^{H, \scrL}$.
The following lemma is an endoscopic variation of Proposition \ref{prop:conv_rules}. 
The proof follows the standard argument given for the usual Hecke category (cf., \cite[\S2.2]{BBM}).

\begin{lemma}
  Let $x \in \gamma$ and $y \in \beta$ where $\gamma \in \uW{\scrL''}{\scrL'}$ and $\beta \in \uW{\scrL'}{\scrL}$. 
  We have natural isomorphisms
  \begin{enumerate}
    \item $\Delta_{xy}^{H, \scrL} \cong \Delta_{x}^{H, \scrL'} \star \Delta_y^{H, \scrL}$ if $\ell_{\gamma \beta} (xy) = \ell_{\gamma} (x) + \ell_{\beta} (y)$;
    \item $\nabla_{xy}^{H, \scrL} \cong \nabla_x^{H, \scrL'} \star \nabla_y^{H, \scrL}$ if $\ell_{\gamma \beta} (xy) = \ell_{\gamma} (x) + \ell_{\beta} (y)$;
    \item $\nabla_{y^{-1}}^{H, \scrL'} \star \Delta_y^{H, \scrL} \cong \Delta_e^{H, \scrL} \cong \Delta_{y^{-1}}^{H, \scrL'} \star \nabla_y^{H, \scrL}$.
  \end{enumerate}
\end{lemma}

\begin{corollary}\label{cor:block_translation_for_endo_cat}
  Let $\beta \in \uW{\scrL'}{\scrL}$. 
  There are equivalences of categories
  \[\Delta_{w^{\beta, -1}}^{H, \scrL'} \star (-) : D_{\cons} ( {}_{\scrL'} \Fl_{\scrL}^{\beta} / B_{\scrL}, \k) \to D_{\cons} (\Fl_{H_{\scrL}^{\circ}} /B_{\scrL}, \k),\]
  and
  \[(-) \star \Delta_{w^{\beta, -1}}^{H, \scrL'} : D_{\cons} ({}_{\scrL'} \Fl_{\scrL}^{\beta} / B_{\scrL}, \k) \to D_{\cons} (\Fl_{H_{\scrL'}^{\circ}} / B_{\scrL'}, \k).\]
\end{corollary}

The isomorphism  (\ref{eq:shifted_flag_vars_are_shifted}) ensures that $D_{\cons} ({}_{\scrL'} \Fl_{\scrL}^{\beta} / B_{\scrL}, \k)$ satisfies the parity conditions.
As a result, when $\k$ is a field or a complete local ring, we can consider the subcategory of parity sheaves $\Par ({}_{\scrL'} \Fl_{\scrL}^{\beta} / B_{\scrL}, \k)$.
Moreover, parity extensions always exist. 
More precisely, for each $w \in \beta$, there is an indecomposable parity extensions $\scrE_w^{H, \scrL}$ of $\uk_{{}_{\scrL'} \Fl_{\scrL, w}^{\beta}} [\ell_{\beta} (w)]$.

\begin{lemma}\label{lem:conv_preserves_par_sh_for_endo_cat}
  Assume that $\k$ is a field or a complete local ring.
  Let $\gamma \in \uW{\scrL''}{\scrL'}$ and $\beta \in \uW{\scrL'}{\scrL}$. 
  The convolution bifunctor
  \[(-) \star (-) : D_{\cons} ({}_{\scrL''} \Fl_{\scrL'}^{\gamma} / B_{\scrL'}, \k) \times D_{\cons} ( {}_{\scrL'} \Fl_{\scrL}^{\beta} / B_{\scrL}, \k) \to D_{\cons} ( {}_{\scrL''} \Fl_{\scrL}^{\gamma \beta} / B_{\scrL}, \k)\]
  takes parity sheaves to parity sheaves.
\end{lemma}
\begin{proof}
  Let $x \in \gamma$ and $y \in \beta$.
  By Corollary \ref{cor:block_translation_for_endo_cat}, it can easily be checked that there are isomorphisms 
  \[\Delta_{w^{\gamma, -1}}^{H, \scrL'} \star \scrE_{x}^{H, \scrL'} \cong \scrE_{w^{\gamma, -1} x}^{H, \scrL'} \hspace{0.5cm} \text{and} \hspace{0.5cm} \scrE_{y}^{H, \scrL} \star \Delta_{w^{\beta, -1}}^{H, \scrL'} \cong \scrE_{y w^{\beta, -1}}^{H, \scrL'}.\]
  The result then follows from associativity of convolution and the fact that convolution takes parity sheaves to parity sheaves in the unipotent Hecke category $D_{\cons} (\Fl_{H_{\scrL'}^{\circ}} / B_{\scrL'}, \k)$. This is proved in \cite[\S 4.1]{JMW} or can be derived from the unipotent case of Theorem \ref{thm:conv_preserves_parity}.
\end{proof}

As a consequence of Lemma \ref{lem:conv_preserves_par_sh_for_endo_cat}, when $\k$ is a field or a complete local ring, we can now construct a 2-category $\ParH (\k)$ over $\Xi$ as follows.
The object set of $\ParH (\k)$ is $\fr{o}$. The morphism categories are given by 
\[ \Hom_{\ParH (\k)} (\scrL, \scrL') \coloneq {}_{\scrL'} \ParH_{\scrL} (\k) \coloneq \bigsqcup_{\beta \in \uW{\scrL'}{\scrL}} \Par ({}_{\scrL'} \Fl_{\scrL}^{\beta} / B_{\scrL}, \k),\] 
where composition is given by convolution.
We call $\ParH (\k)$ the \emph{endoscopic Hecke 2-category of parity sheaves}.

\subsection{Block Conjugation Equivariance}

We will assume that $G$ is of finite type, $\k \in \{\K, \O, \F, \overline{\Q}_{\ell}, \overline{\F}_{\ell} \}$, and that $\fr{o} \subset \Ch^{\mu} (T, \k)$ is a $W$-orbit. 
This ensures that the Whittaker model of \S \ref{subsec:whit} can be used.
Note that all such $\k$ are either complete local rings or fields.

\subsubsection{Groupoid Actions on Bicategories}

Let $\Gamma$ be a small groupoid.

\begin{definition}\quad
  \begin{enumerate}
    \item Let $\fr{C}^{\circ} = \{ \scrC_x^{\circ} \}_{x \in \Ob (\Gamma)}$ denote a collection of monoidal categories. Define a 2-category $\textnormal{Free} (\fr{C}^{\circ})$ over $\Gamma$ whose morphism categories are given by 
    $\Hom_{\textnormal{Free} (\fr{C}^{\circ})} (x,y) \coloneq \Fun^{\simeq, \textnormal{mon}} (\scrC_x^{\circ}, \scrC_y^{\circ})$ where $\Fun^{\simeq, \textnormal{mon}} (\scrC_x^{\circ}, \scrC_y^{\circ})$ denotes the category of monoidal equivalences $\scrC_x^{\circ} \stackrel{\sim}{\to} \scrC_y^{\circ}$.
    \item Let $\fr{C}^{\circ} = \{ \scrD_x^{\circ} \}_{x \in \Ob (\Gamma)}$ denote another collection of monoidal categories. For each $x \in \Ob (\Gamma)$, let $F_x : \scrC_x^{\circ} \to \scrD_x^{\circ}$ be a monoidal equivalence along with inverse $F_x^{-1} : \scrD_x^{\circ} \to \scrC_x^{\circ}$.
        We can then define a 2-functor $F : \textnormal{Free} (\fr{C}^{\circ}) \to \textnormal{Free} (\fr{D}^{\circ})$ which is the identity on objects and on morphism categories is given by
        \[\Fun^{\simeq, \textnormal{mon}} (\scrC_x^{\circ}, \scrC_y^{\circ}) \to \Fun^{\simeq, \textnormal{mon}} (\scrD_x^{\circ}, \scrD_y^{\circ}) \hspace{1cm} G \mapsto F_y \circ G \circ F_x^{-1}.\]
  \end{enumerate}
\end{definition}

\begin{definition}\quad
  \begin{enumerate}
    \item Let $\fr{C}$ be a 2-category over $\Gamma$. An \emph{action} of $\Gamma$ on $\fr{C}$ is a 2-functor $\Gamma \to \fr{C}.$ If $\fr{D}$ is a 2-category over $\Gamma$ with a $\Gamma$-action $\Gamma \to \fr{D}$, then a \emph{$\Gamma$-functor} from $\fr{C}$ to $\fr{D}$ is a 2-functor $\fr{C} \to \fr{D}$ such that the following diagram commutes up to natural isomorphism.
    \[\begin{tikzcd}
      & \Gamma \arrow[ld] \arrow[rd] &        \\
\fr{C} \arrow[rr] &                              & \fr{D}.
\end{tikzcd}\]
    \item Let $\fr{C}^{\circ} \coloneq \{ \scrC_x^{\circ} \}_{x \in \Ob (\Gamma)}$ denote a collection of monoidal categories.
    We define an \emph{action} of $\Gamma$ on $\fr{C}^{\circ}$ to be $\Gamma$-action on $\textnormal{Free} (\fr{C}^{\circ})$. Let $\fr{D}^{\circ} \coloneq \{ \scrD_x^{\circ} \}_{x \in \Ob (\Gamma)}$ denote another collection of monoidal categories.
    For each $x \in \Ob (\Gamma)$, let $F_x : \scrC_x^{\circ} \to \scrD_x^{\circ}$ be a monoidal equivalence along with inverse $F_x^{-1} : \scrD_x^{\circ} \to \scrC_x^{\circ}$.
    We say that the collection of functors $\{F_x, F_x^{-1}\}_{x\in \Ob(\Gamma)}$ is \emph{$\Gamma$-compatible} if the induced 2-functor  $F : \textnormal{Free} (\fr{C}^{\circ}) \to \textnormal{Free} (\fr{D}^{\circ})$ is a $\Gamma$-functor.
  \end{enumerate}
\end{definition}

Let $\fr{C}^{\circ} \coloneq \{ \scrC_x^{\circ} \}_{x \in \Gamma}$ be a collection of monoidal categories. Let $A^{\circ} : \Gamma \to \textnormal{Free} (\fr{C}^{\circ})$ be a $\Gamma$-action on $\fr{C}^{\circ}$.
We can then define the semidirect product $\fr{C}^{\circ} \rtimes \Gamma$ as follows.
\begin{enumerate}
  \item $\Ob (\fr{C}^{\circ} \rtimes \Gamma) = \Ob (\Gamma)$.
  \item For $x,y \in \Ob (\Gamma)$, we define the morphism categories
    \[\Hom_{\fr{C}^{\circ} \rtimes \Gamma} (x,y) \coloneq \bigsqcup_{\xi \in {}_y \Gamma_x} {}_y (\fr{C}^{\circ} \rtimes \Gamma)_x^{\xi},\]
    where each ${}_y (\fr{C}^{\circ} \rtimes \Gamma)_x^{\xi}$ is defined to be a copy of $\scrC_y^{\circ}$.
  \item For $\xi \in {}_y \Gamma_x$ and $\eta \in {}_z \Gamma_y$, the composition of 1-morphisms is given by
      \[{}_z (\fr{C}^{\circ} \rtimes \Gamma)_y^{\eta} \times {}_y (\fr{C}^{\circ} \rtimes \Gamma)_x^{\xi} \to {}_z (\fr{C}^{\circ} \rtimes \Gamma)_x^{\eta \xi},\]
      \[(\scrG, \scrF) \mapsto \scrG \circ (A^{\circ} (\eta)) (\scrF).\]
\end{enumerate}
The semidirect product comes equipped with a canonical $\Gamma$-action, $\tilde{A} : \Gamma \to \fr{C}^{\circ} \rtimes \Gamma$.
Explicitly, $\tilde{A}$ is the identity on objects, and for $\xi \in {}_y \Gamma_x$, we define $\tilde{A} (\xi) = \mathds{1}_y \in \scrC_y^{\circ} \eqcolon {}_y (\fr{C}^{\circ} \rtimes \Gamma)_x^{\xi}$.

Let $\fr{C}$ be a 2-category over $\Gamma$ with $\Gamma$-action $A : \Gamma \to \fr{C}$.
Consider the collection of monoidal categories $\fr{C}^{\circ} = \{ {}_x \fr{C}_x^{\id_x} \}_{x \in \Ob (\Gamma)}$.
For each $\xi \in {}_y \Gamma_x$, we can define a monoidal equivalence
\[A^{\circ} (\xi) : {}_x \fr{C}_x^{\id_x} \to {}_y \fr{C}_y^{\id_y}.\]
These monoidal equivalences can be assembled into a 2-functor $A^{\circ} : \Gamma \to \textnormal{Free} (\fr{C}^{\circ})$. In other words, $A : \Gamma \to \fr{C}$ induces a $\Gamma$-action on $\fr{C}^{\circ}$. 

The following lemmas follow easily from the definitions.

\begin{lemma}\label{lem:trivializing_bicats_over_gpds}
  Let $\fr{C}$ be a 2-category over $\Gamma$ with $\Gamma$-action. Consider the collection of monoidal categories $\fr{C}^{\circ} = \{ {}_x \fr{C}_x^{\id_x} \}_{x \in \Ob (\Gamma)}$.
  Then there is a canonical $\Gamma$-functor
  \[\fr{C}^{\circ} \rtimes \Gamma \to \fr{C}.\]
\end{lemma}

\begin{lemma}\label{lem:mon_cat_compat_Gamma}
  Let $\fr{C}^{\circ} = \{\scrC_x^{\circ} \}_{x \in \Ob (\Gamma)}$ and $\fr{D}^{\circ} = \{\scrD_x^{\circ} \}_{x \in \Ob (\Gamma)}$ be collections of monoidal categories.
  For each $x \in \Ob (\Gamma)$, let $F_x : \scrC_x^{\circ} \to \scrD_x^{\circ}$ be a monoidal equivalence along with inverse $F_x^{-1} : \scrD_x^{\circ} \to \scrC_x^{\circ}$ such that the collection $\{F_x, F_x^{-1}\}_{x\in \Ob(\Gamma)}$ is $\Gamma$-compatible.
  Then there is a canonical $\Gamma$-functor
  \[\fr{C}^{\circ} \rtimes \Gamma \to \fr{D}^{\circ} \rtimes \Gamma,\]
  which is an equivalence of 2-categories.
\end{lemma}

\subsubsection{Compatibilities with Block Conjugation}

Let $\beta \in \uW{\scrL'}{\scrL}$.
We have an isomorphism of Coxeter groups $(W_{\scrL}^{\circ}, S_{\scrL}^{\circ}) \to (W_{\scrL'}^{\circ}, S_{\scrL'}^{\circ})$ which induces monoidal equivalences of categories
\[{}^{\beta} (-) : \scrD_{\BS} (\fr{h}_{\k}, W_{\scrL}^\circ) \to \scrD_{\BS} (\fr{h}_{\k}, W_{\scrL'}^\circ) \hspace{0.5cm}\text{and}\hspace{0.5cm} {}^{\beta} (-) : \scrD (\fr{h}_{\k}, W_{\scrL}^\circ) \to \scrD (\fr{h}_{\k}, W_{\scrL'}^\circ).\]
Let $\ux$ be an endo-reduced expression for $w^{\beta}$. Let $\overline{\ux}$ denote the reversed expression for $\ux$. The functor
\[{}^{\ux} (-) : \Parity{\scrL}{\scrL}^{\BS, \circ} (\k) \to \Parity{\scrL'}{\scrL'}^{\BS, \circ} (\k) \hspace{1cm} \scrF \mapsto \scrE_{\ux}^{\scrL} \star \scrF \star \scrE_{\overline{\ux}}^{\scrL'}\]
is an equivalence of monoidal categories.

\begin{lemma}\label{lem:beta_conj_compat_with_Phi_BS}
  Let $\beta \in \uW{\scrL'}{\scrL}$. Let $\ux$ be an endo-reduced expression for $w^{\beta}$.
  There is a natural isomorphism of functors
  \[ \begin{tikzcd}[column sep=huge]
    \scrD_{\BS} (\fr{h}_{\k}, W_{\scrL}^\circ)
      \arrow[bend left=20]{r}[name=U,label=above:${}^{\ux} \Upsilon_{\scrL}^{\circ} (-)$]{}
      \arrow[bend right=20]{r}[name=D,label=below:$\Upsilon_{\scrL'}^{\circ} {}^\beta (-)$]{} &
      \Parity{\scrL'}{\scrL'}^{\BS, \circ} (\k).
      \arrow[shorten <=5pt,shorten >=0pt,Rightarrow,to path={(U) -- node[label=right:$C_{\ux}^{\scrL}$] {} (D)}]{}
    \end{tikzcd}\]
    Moreover, for all $\gamma \in \uW{\scrL''}{\scrL'}$ and endo-reduced expressions $\uy$ of $w^{\gamma}$, we have that $C_{\uy}^{\scrL'} \circ {}^{\uy} C_{\ux}^{\scrL} = C_{\uy \ux}^{\scrL}$.
\end{lemma}
\begin{proof}
  Since ${}^{\ux} \Upsilon_{\scrL}^{\circ} (-)$ and $\Upsilon_{\scrL'}^{\circ} {}^\beta (-)$ are monoidal functors, it suffices to define $C_{\ux}^{\scrL}$ just on $B_s$ where $s \in S_{\scrL}^{\circ}$.

  The sheaf $\scrE_{\ux}^{\scrL} \star \IC_s^{\scrL} \star \scrE_{\overline{\ux}}^{\scrL'}$ has a Frobenius algebra structure induced from that of $\IC_s^{\scrL}$. 
  Explicitly, the structure maps are given by
  \[\eta : \IC_e^{\scrL'} [-1] \stackrel{\sim}{\to} \scrE_{\ux}^{\scrL} \star \IC_e^{\scrL} \star \scrE_{\overline{\ux}}^{\scrL'} [-1] \stackrel{\id \star \eta_s \star \id}{\longrightarrow} \scrE_{\ux}^{\scrL} \star \IC_s^{\scrL} \star \scrE_{\overline{\ux}}^{\scrL'}, \]
  \[\epsilon : \scrE_{\ux}^{\scrL} \star \IC_s^{\scrL} \star \scrE_{\overline{\ux}}^{\scrL'} \stackrel{\id \star \epsilon_s \star \id}{\longrightarrow} \scrE_{\ux}^{\scrL} \star \IC_e^{\scrL} \star \scrE_{\overline{\ux}}^{\scrL'} [1] \stackrel{\sim}{\to} \IC_e^{\scrL'},\]
  \[\mu : \scrE_{\ux}^{\scrL} \star \IC_s^{\scrL} \star \scrE_{\overline{\ux}}^{\scrL'} \star \scrE_{\ux}^{\scrL} \star \IC_s^{\scrL} \star \scrE_{\overline{\ux}}^{\scrL'} \stackrel{\sim}{\to} \scrE_{\ux}^{\scrL} \star \IC_s^{\scrL} \star \IC_s^{\scrL} \star \scrE_{\overline{\ux}}^{\scrL'} \stackrel{\id \star m_s \star \id}{\longrightarrow} \scrE_{\ux}^{\scrL} \star \IC_s^{\scrL} \star \scrE_{\overline{\ux}}^{\scrL'},\]
  \[ \nu : \scrE_{\ux}^{\scrL} \star \IC_s^{\scrL} \star \scrE_{\overline{\ux}}^{\scrL'} \stackrel{\id \star \nu_s \star \id}{\longrightarrow}\scrE_{\ux}^{\scrL} \star \IC_s^{\scrL} \star \IC_s^{\scrL} \star \scrE_{\overline{\ux}}^{\scrL'} \stackrel{\sim}{\to} \scrE_{\ux}^{\scrL} \star \IC_s^{\scrL} \star \scrE_{\overline{\ux}}^{\scrL'} \star \scrE_{\ux}^{\scrL} \star \IC_s^{\scrL} \star \scrE_{\overline{\ux}}^{\scrL'}.\]
  Moreover, from \ref{subsec:ver_relns}, we see that $\epsilon \circ \eta = w^{\beta} (\alpha_s) \cdot \id_{\IC_e^{\scrL'}}$.

  Let $u = w^{\beta} s w^{\beta, -1}$ which is an endosimple reflection for $W_{\scrL'}^{\circ}$.
  There is a unique isomorphism $C_{\ux}^{\scrL} (B_s) : \scrE_{\ux}^{\scrL} \star \IC_s^{\scrL} \star \scrE_{\overline{\ux}}^{\scrL'} \to \scrE_{u}^{\scrL'}$ making the following diagram commute
  \begin{equation}\label{eq:beta_conj_compat_with_Phi_BS_0}
  \begin{tikzcd}
\scrE_{\ux}^{\scrL} \star \IC_s^{\scrL} \star \scrE_{\overline{\ux}}^{\scrL'} \arrow[r, "C_{\ux}^{\scrL} (B_s)"] \arrow[d, "\id \star \epsilon_s \star \id"'] & \scrE_{u}^{\scrL'} \arrow[d, "\epsilon_u"] \\
{\scrE_{\ux}^{\scrL} \star \IC_e^{\scrL} \star \scrE_{\overline{\ux}}^{\scrL'} [1]} \arrow[r, "\sim"]                                                         & \IC_e^{\scrL'} [1].                             
\end{tikzcd}
\end{equation}
Since $\epsilon_u \circ \eta_u = \alpha_u \cdot \id_{\IC_e^{\scrL'}} = \epsilon \circ \eta$ and $\epsilon_u \circ C_{\ux}^{\scrL} (B_s) = \epsilon$, it can be checked that $C_{\ux}^{\scrL} (B_s)$ must necessarily be an isomorphism of Frobenius algebras.
  
  It suffices to check that $C_{\ux}^{\scrL}$ defines a natural transformation just on generators for the morphisms in $\scrD_{\BS} (\fr{h}_{\k}, W_{\scrL}^\circ)$.
  The compatibility conditions for one-color morphisms follows from $C_{\ux}^{\scrL} (B_s)$ being an isomorphism of Frobenius algebras.
  As a result, we only need to check the compatibility condition for the $2m$-valent vertex.
  Let $t \in S_{\scrL}^{\circ}$ and $v = w^{\beta} t w^{\beta, -1}$. Note that $m_{s,t} = m_{u,v}$. We use the notation introduced in \S\ref{subsec:constr_of_Phi}.
  We must show that the following diagram commutes
  \begin{equation}\label{eq:beta_conj_compat_with_Phi_BS_1}
    \begin{tikzcd}
      {}^{\ux} (\scrF_s^{\scrL}) \arrow[d, "{{}^{\ux} f_{s,t}}"'] \arrow[r, "C_{\ux}^{\scrL}"] & \scrF_u^{\scrL'} \arrow[d, "{f_{u,v}}"] \\
      {}^{\ux} (\scrF_t^{\scrL}) \arrow[r, "C_{\ux}^{\scrL}"]                                  & \scrF_v^{\scrL'}.                       
      \end{tikzcd}
  \end{equation}
  Consider the following ``all-dot'' maps:
  \[{}^{\ux} \eta_s^{\scrL} \coloneq \id_{\scrE_{\ux}^{\scrL}} \star \underbrace{\epsilon_{s} \star \epsilon_{t} \star \ldots}_{m_{s,t} \text{ terms}} \star \id_{\scrE_{\overline{\ux}}^{\scrL'}} : \scrE_{\emptyset}^{\scrL'} [-m_{s,t}] \to {}^{\ux} (\scrF_s^{\scrL}),\]
  \[\epsilon_v^{\scrL'} \coloneq \underbrace{\epsilon_{v} \star \epsilon_{u} \star \ldots}_{m_{s,t} \text{ terms}}  : \scrF_v^{\scrL'} \to  \scrE_{\emptyset}^{\scrL'} [m_{s,t}].\]
  Since $C_{\ux}^{\scrL}$ is a morphism of Frobenius algebras on simple reflections and Theorem \ref{thm:endoscopy_neutral_block_Kac_moody}, under the isomorphism $\Hom (\scrE_{\emptyset}^{\scrL'} [-m_{s,t}], \scrE_{\emptyset}^{\scrL'} [m_{s,t}]) \cong H_T^{2m_{s,t}} (\pt; \k)$, we have that 
  \[\epsilon_v^{\scrL'} \circ (C_{\ux}^{\scrL} \circ {{}^{\ux} f_{s,t}}) \circ {}^{\ux} \eta_s^{\scrL} = \mu_{s,t} = \epsilon_v^{\scrL'} \circ (f_{u,v} \circ C_{\ux}^{\scrL}) \circ {}^{\ux} \eta_s^{\scrL},\]
  where $\mu_{s,t}$ is the product of all positive roots in the root subsystem of $\Phi_{\scrL}$ corresponding to $s,t$.
  Since $\Hom ( {}^{\ux} (\scrF_s^{\scrL}), \scrF_v^{\scrL'}) \cong \k$, we can conclude that $C_{\ux}^{\scrL} \circ {{}^{\ux} f_{s,t}} = f_{u,v} \circ C_{\ux}^{\scrL}$.

  Finally, if $\gamma \in \uW{\scrL''}{\scrL'}$ and $\uy$ is an endo-reduced expression of $w^{\gamma}$, then $C_{\uy}^{\scrL'} \circ {}^{\uy} C_{\ux}^{\scrL} = C_{\uy \ux}^{\scrL}$ by the uniqueness constraint of (\ref{eq:beta_conj_compat_with_Phi_BS_0}).
\end{proof}

Let $\beta \in \uW{\scrL'}{\scrL}$. We define a monoidal equivalence of categories
\[{}^{\beta} (-) : \Parity{\scrL}{\scrL}^{\circ} (\k) \to \Parity{\scrL'}{\scrL'}^{\circ} (\k), \qquad \scrF \mapsto \IC_{w^{\beta}}^{\scrL} \star \scrF \star \IC_{w^{\beta, -1}}^{\scrL'}.\]
Note that for $\scrF \in \Parity{\scrL}{\scrL}^{\circ} (\k)$ and $\gamma \in \uW{\scrL''}{\scrL'}$, there is a natural isomorphism ${}^{\gamma} ({}^{\beta} (\scrF)) \cong {}^{\gamma \beta} (\scrF)$ afforded by Lemma \ref{lem:assoc_of_rigid_min_IC}.
Moreover, this natural isomorphism is associative with respect to block composition.

\begin{lemma}\label{lem:beta_conj_compat_with_Phi} 
  There is a natural isomorphism making the following diagram commutative
  \[ \begin{tikzcd}
    {\scrD (\fr{h}_{\k}, W_{\scrL}^\circ)} \arrow[d, "{}^{\beta} (-)"'] \arrow[r, "\Upsilon_{\scrL}^{\circ}"] & \Parity{\scrL}{\scrL}^{\circ} (\k) \arrow[d, "{}^{\beta} (-)"] \\
    {\scrD (\fr{h}_{\k}, W_{\scrL'}^\circ)} \arrow[r, "\Upsilon_{\scrL'}^{\circ}"]                             & \Parity{\scrL'}{\scrL'}^{\circ} (\k).                          
    \end{tikzcd}
  \]
  Moreover, these isomorphisms are compatible with block composition.
\end{lemma}
\begin{proof}
  The statement follows from composing the natural isomorphism given in Lemma \ref{lem:beta_conj_compat_with_Phi_BS} (after taking the additive hull and idempotent completion) with the natural isomorphisms given in Lemma \ref{lem:assoc_of_rigid_min_IC}.
\end{proof}

Let $\beta \in \uW{\scrL'}{\scrL}$. Let $\dot{w}$ be a lift of $w^{\beta}$ in $N_G (T)$.
The isomorphism $\sigma (\dot{w}) : H_{\scrL}^{\circ} \to H_{\scrL}^{\circ}$ induces an equivalence of categories
\[ {}^{\beta} (-) : \Par (\Fl_{H_{\scrL}^{\circ}} / B_{\scrL}, \k) \to \Par (\Fl_{H_{\scrL'}^{\circ}} / B_{\scrL'}, \k).\]
Explicitly, for $\scrF \in  \Par (\Fl_{H_{\scrL}^{\circ}} / B_{\scrL}, \k)$, there is a canonical isomorphism
\[{}^{\beta} \scrF \cong \Delta_{w^{\beta}}^{H, \scrL} \star \scrF \star \Delta_{w^{\beta, -1}}^{H, \scrL'}.\]
In particular, $ {}^{\beta} (-)$ is independent of the choice of lift $\dot{w}$.

The following lemma is then immediate from Lemma \ref{lem:beta_conj_compat_with_Phi} and the uniqueness of relative pinnings. 
\begin{lemma}\label{lem:beta_conj_compat_with_Psi}
  There is a natural isomorphism making the following diagram commutative
  \[ \begin{tikzcd}
    {\Par (\Fl_{H_{\scrL}^{\circ}} / B_{\scrL}, \k)} \arrow[d, "{}^{\beta} (-)"'] \arrow[r, "\Psi_{\scrL}^{\circ}"] & {\PEE{\scrL}{\scrL}^{\circ} (\k) } \arrow[d, "{}^{\beta} (-)"] \\
    {\Par (\Fl_{H_{\scrL'}^{\circ}} / B_{\scrL'}, \k)} \arrow[r, "\Psi_{\scrL'}^{\circ}"]                           & {\PEE{\scrL'}{\scrL'}^{\circ} (\k).}                          
    \end{tikzcd}
  \]
  Moreover, these isomorphisms are compatible with block composition.
\end{lemma}

\subsection{Monodromic-Endoscopic Equivalence for All Blocks}\label{subsec:all_block_endo}

We can now state the extension of Theorem \ref{thm:endoscopy_neutral_block_Kac_moody} to the non-neutral blocks.
\begin{theorem}[Monodromic-Endoscopic equivalence for all blocks]\label{thm:endoscopy_Kac_moody}
  Let $\k \in \{\K, \O, \F, \overline{\Q}_{\ell}, \overline{\Z}_{\ell}\}$. 
  Assume that $G$ is of finite type and that $\fr{o} \subseteq \Ch^{\mu} (T, \k)$.
  There is an equivalence of 2-categories over $\Xi$,
  \[\Psi : \ParH (\k) \to \Parity{}{} (\k).\]
  For all blocks $\beta \in \uW{\scrL'}{\scrL}$, $\Psi$ restricts to an equivalence of categories
  \[{}_{\scrL'} \Psi_{\scrL}^{\beta} : \Par ( {}_{\scrL'} \Fl_{\scrL}^{\beta} / B_{\scrL}, \k) \to \Parity{\scrL'}{\scrL}^{\beta} (\k) \]
  such that for all $w \in \beta$, ${}_{\scrL'} \Psi_{\scrL}^{\beta}$ takes $\scrE_w^{H, \scrL}$ to $\scrE_w^{\scrL}$.
  Moreover, when $\beta$ is the neutral block, ${}_{\scrL} \Psi_{\scrL}^{\circ}$ is the equivalence $\Psi_{\scrL}^{\circ}$ from Theorem \ref{thm:endoscopy_neutral_block_Kac_moody}.
\end{theorem}
\begin{proof}
  By Corollary \ref{cor:trivialization_of_bicats}, there is a $\Xi$-action  on $\Parity{}{} (\k)$.
  By Lemma \ref{lem:trivializing_bicats_over_gpds}, this action induces a $\Xi$-functor
  \begin{equation}\label{eq:endoscopy_Kac_Moody_1}
    \Parity{}{}^{\circ} (\k) \rtimes \Xi \to \Parity{}{} (\k).
  \end{equation}
  Proposition \ref{prop:structure_of_min_IC} tells us that (\ref{eq:endoscopy_Kac_Moody_1}) is an equivalence of 2-categories over $\Xi$.
  
  Similarly, one can define a 2-functor $\Xi \to  \ParH (\k)$ by $\beta \mapsto \Delta_{w^{\beta}}^{H, \scrL}$.
  This induces a $\Xi$-functor
  \begin{equation}\label{eq:endoscopy_Kac_Moody_2}
    \Par^{H, \circ} (\k) \rtimes \Xi \to \ParH (\k).
  \end{equation}
  By Corollary \ref{cor:block_translation_for_endo_cat}, the 2-functor (\ref{eq:endoscopy_Kac_Moody_2}) is also an equivalence of 2-categories over $\Xi$.

  By Theorem \ref{thm:endoscopy_neutral_block_Kac_moody}, for each $\scrL \in \fr{o}$, there is an equivalence of monoidal categories
  \[ \Psi_{\scrL}^{\circ} :  \Par (\Fl_{H_{\scrL}^{\circ}} / B_{\scrL}, \k) \stackrel{\sim}{\to} \Parity{\scrL}{\scrL}^{\circ} (\k).\]
  From Lemma \ref{lem:beta_conj_compat_with_Psi}, the collection $\{\Psi_{\scrL}^{\circ} \}_{\scrL \in \fr{o}}$ is a $\Xi$-compatible collection.
  By Lemma \ref{lem:mon_cat_compat_Gamma}, we get an equivalence of 2-categories over $\Xi$,
  \begin{equation}\label{eq:endoscopy_Kac_Moody_3}
    \Par^{H, \circ} (\k) \rtimes \Xi \stackrel{\sim}{\to} \Parity{}{}^{\circ} (\k) \rtimes \Xi.
  \end{equation}
  The proof then follows from combining the equivalences (\ref{eq:endoscopy_Kac_Moody_1}), (\ref{eq:endoscopy_Kac_Moody_2}), and (\ref{eq:endoscopy_Kac_Moody_3}).
\end{proof}

\begin{remark}
  The only place where finite type is essentially used is in the construction of the 2-functor $\IC_{\min} : \Xi \to \Parity{}{} (\k)$ from Corollary \ref{cor:trivialization_of_bicats}.
  We conjecture that this functor should exist even when $G$ is of affine type. Whenever $\IC_{\min}$ can be constructed, the argument for Theorem \ref{thm:endoscopy_Kac_moody} we have provided will generalize.
\end{remark}

        \appendix
        \section{Six-Functor Formalisms}\label{apdx:A}

Six-functor formalisms for étale constructible sheaves has been well-developed (cf., \cite{LZ, Mann, Sch}). 
However, a six-functor formalism
for the complex analytic topology is not detailed in the literature.
We will develop
such a six-functor formalism for three notable settings:
\begin{enumerate}
    \item constructible sheaves on schemes over $\C$;
    \item constructible sheaves on algebraic stacks over $\C$;
    \item twisted equivariant constructible sheaves on stacks over $\C$.
\end{enumerate}
The scheme and algebraic stack settings require essentially no modifications from work covered in existing literature. 
For these, we will only recall the basic constructions and some key results.

The twisted equivariant setting is not substantively covered in existing literature. 
The main goal of this appendix is then to provide a thorough treatment of the twisted equivariant setting which allows for more general coefficient rings and monodromy parameters than covered by \cite{Gou} and \cite{LY}.

\subsubsection{Preliminaries on \texorpdfstring{$(\infty, 1)$}{inf}-categories}

Unlike the rest of the paper, this appendix will work with the formalism of $\infty$-categories, and more precisely the model of quasicategories presented in \cite{HTT} and \cite{HA}.

We denote by $\infCat$ (resp. $\infCat_{\k}$) the $\infty$-category of small stable $\infty$-categories (resp. the $\infty$-category of $\k$-linear small $\infty$-categories).
Similarly, we denote by $\Pr$ (resp. $\Pr_{\k}$) the $\infty$-category of presentable stable $\infty$-categories (resp. the $\infty$-category of $\k$-linear presentable $\infty$-categories).
We will write $\Pr^L_{\k}$ (resp. $\Pr^R_{\k}$ or $\Pr^{LR}_{\k}$) for the 1-full subcategory of $\Pr_{\k}$ where we restrict morphisms to continuous (resp. cocontinuous or both continuous and cocontinuous) functors. 

The Lurie tensor product $\otimes$ equips $\Pr^L_{\k}$ with the structure of a symmetric monoidal category.
If $\scrA, \scrB, \scrC \in \Alg (\Pr_L^{\k})$ are algebras, then we can consider the category of $(\scrA, \scrB)$-bimodules ${}_{\scrA} \textnormal{Bimod}_{\scrB} (\Pr_{\k}^L)$.
The relative Lurie tensor product defines a functor 
\[(-) \otimes_{\scrB} (-) : {}_{\scrA} \textnormal{Bimod}_{\scrB} (\Pr_{\k}^L) \times {}_{\scrB} \textnormal{Bimod}_{\scrC} (\Pr_{\k}^L) \to {}_{\scrA} \textnormal{Bimod}_{\scrC} (\Pr_{\k}^L).\]
Similarly, if we just consider left $\scrB$-modules, $\Mod_{\scrB} (\Pr_{\k}^L)$, the relative tensor product gives rise to a functor
\[(-) \otimes_{\scrB} (-) : {}_{\scrA} \textnormal{Bimod}_{\scrB} (\Pr_{\k}^L) \times \Mod_{\scrB} (\Pr_{\k}^L) \to \Mod_{\scrA} (\Pr_{\k}^L).\]

There is a functor
\[\ind : \infCat \to \Pr^L,\]
called the \emph{ind-completion} which freely adds filtered colimits to a small $\infty$-category \cite[\S5.5]{HA}.
Similarly, there is a functor
\[(-)^{c} : \Pr^L \to \infCat,\]
which takes a presentable $\infty$-category to its small subcategory of compact objects.

For a (dg-)algebra $A$ over $\k$, we will write $D(A)$ for the $\infty$-derived category of $A$.

\subsection{Six-Functor Formalisms}

We first recall some preliminaries on 3-functor and 6-functor formalisms. For a comprehensive treatise, see works of Scholze \cite{Sch} and Mann \cite{Mann}.

Let $\scrC$ be a category admitting all finite limits. Let $E$ be a class of morphisms of $\scrC$ which is stable under pullback, composition, and contains all isomorphisms.
The pair $(\scrC, E)$ is called a \emph{geometric setup}.
Associated to a geometric setup $(\scrC, E)$ is the symmetric monoidal $\infty$-category $\Corr (\scrC, E)$ of correspondences.
If $E$ consists of all morphisms in $\scrC$, we will shorten $\Corr (\scrC, E)$ to $\Corr (\scrC)$.
The objects of  $\Corr (\scrC, E)$ are the same as the objects of $\scrC$. The 1-morphisms $X \to Y$ in $\Corr (\scrC, E)$ are given by correspondence diagrams
\[X \leftarrow Z \to Y,\]
where the maps $Z \to Y$ belong to $E$. The monoidal structure on $\Corr (\scrC, E)$ is via the Cartesian product in $\scrC$.

\begin{definition}\label{def:3functor}
    A \emph{3-functor formalism} is a lax-symmetric monoidal functor $D : \Corr (\scrC, E) \to \infCat$.
\end{definition}

Let $D : \Corr (\scrC, E) \to \infCat$ be a 3-functor formalism. Let $f : X \to Y$.
\begin{enumerate}
    \item We can form a correspondence $Y \stackrel{f}{\leftarrow} X \stackrel{=}{\to} X$. The \emph{pullback} functor along $f$, denoted by $f^* : D( Y) \to D(X)$, is the functor obtained by applying $D$ to the aforementioned correspondence.
    \item If $f \in E$, we can form a correspondence $X \stackrel{=}{\leftarrow} X \stackrel{f}{\to} X$. The \emph{proper pushforward} functor along $f$, denoted by $f_! : D(X) \to D(Y)$, is the functor obtained by applying $D$ to the aforementioned correspondence diagram.
    \item Since $D$ is a lax-symmetric monoidal functor, we can equip $D(X)$ with the structure of a symmetric monoidal category given by the composition of functors
        \[D(X) \times D(X) \stackrel{\boxtimes}{\to} D(X \times X) \stackrel{\Delta^*}{\to} D(X),\]
        where the first map, called the \emph{external tensor product}, comes from the monoidality constraint and the second map is pullback along the diagonal map $\Delta : X \to X \times X$.
        The resulting bifunctor on $D(X)$, denoted $\otimes^L$, is called the \emph{sheaf tensor product}.
\end{enumerate}

\begin{definition}
    A 3-functor formalism $D : \Corr (\scrC, E) \to \infCat$ is said to be a \emph{6-functor formalism} if all the functors $f^*$, $f_!$ and $A \otimes^L -$ admit right adjoints.
\end{definition}

Let $D : \Corr (\scrC, E) \to \infCat$ be a 3-functor formalism. Let $f : X \to Y$.
\begin{enumerate}
    \item The \emph{pushforward} functor along $f$, denoted by $f_* : D(X) \to D(Y)$, is the right adjoint of $f^*$.
    \item Assume that $f \in E$. The \emph{exceptional pullback} functor along $f$, denoted by $f^! : D(Y) \to D(X)$, is the right adjoint of $f_!$.
    \item The \emph{sheaf Hom} functor, denoted by $\RHom (-, -) : D(X)^{\op} \times D(X) \to D(X)$, is the right adjoint of $\otimes^L$.
\end{enumerate}

\begin{remark}
    It is often useful to consider 3-functor/6-functor formalisms which take values in $\Pr$, $\Pr_{\k}$, or $\infCat_{\k}$ instead of $\infCat$.
    All the definitions given in this section make sense with these replacements.
\end{remark}

\subsection{Constructible Sheaves on Schemes}

Let $\Sch_{\C}$ be the category of separated schemes of finite type over $\C$.
Given a complex scheme $X$, we can equip its $\C$-points, $X (\C)$, with its analytic topology. There exists a six-functor formalism
\[D (-, \k) : \Corr (\Sch_{\C}) \to \Pr_{\k}^L, \hspace{1cm} X \mapsto D (X, \k),\]
where $D (X, \k)$ is the unbounded $\infty$-derived category of $\k$-valued sheaves on $X(\C)$ (see \cite{HTT}). 
The construction of this six-functor formalism is well-known. For example, such a six-functor formalism is carefully detailed in \cite{Sch} where $\Sch_{\C}$ is replaced by locally compact Hausdorff topological spaces.
The only substantive difference when working instead with $\Sch_{\C}$ is that the Stone--Čech compactification should be replaced by the Nagata compactification. 

There are a few features of sheaves which are not encoded in abstract six-functor formalisms. Since they will be significant for our purposes, we will detail them here.
\begin{enumerate}
    \item If $\k \to \k'$ is a ring homomorphism, there is an \emph{extension of scalars} functor
            \[\k' \otimes_{\k} (-) : D (X, \k) \to D(X, \k').\]
            Extension of scalars commutes with $\boxtimes$ and $f^*$.
    \item Let $\Sh (X, \k)$ denote the ordinary abelian category of sheaves on $X (\C)$ with $\k$-coefficients. It is the heart of the natural $t$-structure on $D(X, \k)$.
            In particular, for all $n \in \Z$, there are cohomology functors
            \[H^n : D(X, \k) \to \Sh (X, \k),\]
            which are intertwined by the suspension functor, $H^0 ((-) [1]) = H^1 (-)$.
\end{enumerate}

We can now begin the process of restricting the six-functor formalism to the so-called \emph{constructible sheaves}.
Let $\Loc (X, \k)$ denote the full (abelian) subcategory of $\Sh (X, \k) \subset D (X, \k)$ consisting of local systems on $X$.
We will write $\Loc_{\textnormal{f}} (X, \k)$ for the full subcategory of $\Loc (X, \k)$ consisting of finite rank local systems.

\begin{definition}\label{def:lisse}
    A bounded complex of sheaves $\scrF \in D(X, \k)$ is said to be \emph{lisse} if $H^n (\scrF) \in \Loc_{\textnormal{f}} (X, \k)$ for all $n \in \Z$.
    Let $D_{\lisse} (X, \k)$ denote the full subcategory of $D(X, \k)$ consisting of lisse sheaves.

    Define the category of \emph{ind-lisse} sheaves, denoted $D_{\indlisse} (X, \k)$, as the full subcategory of $D(X, \k)$ consisting of sheaves which are equivalent to filtered colimits of lisse sheaves.
\end{definition}

\begin{definition}\label{def:constructible}
    A bounded complex of sheaves $\scrF \in D(X, \k)$ is said to be \emph{constructible} if there exists a stratification $\{X_\lambda\}_{\lambda \in \Lambda}$ such that $\scrF\vert_{X_\lambda}$ is
    lisse for all $\lambda \in \Lambda$.
    Let $D_{\cons} (X, \k)$ denote the full subcategory of $D(X, \k)$ consisting of constructible sheaves.
  
    Define the category of \emph{ind-constructible} sheaves, denoted $D_{\ic} (X, \k)$ as the full subcategory of $D(X, \k)$ consisting of sheaves which are equivalent
    to filtered colimits of constructible sheaves.
  \end{definition}
  
  \begin{lemma}[{\cite[Proposition 8.2]{HRS}}]\label{lem:cpts_in_Dic}
    A sheaf $\scrF \in D_{\ic} (X, \k)$ is compact if and only if $\scrF$ is constructible. Likewise, $\scrF \in D_{\indlisse} (X, \k)$ is compact if and only if $\scrF$ is lisse.
  \end{lemma}

  \begin{corollary}\label{cor:Dic_is_presentable}
    The inclusions $D_{\lisse} (X, \k)$ and $D_{\cons} (X, \k) \subset D_{\ic} (X, \k)$ extend to a colimit-preserving equivalences
    \[ \ind (D_{\lisse} (X, \k)) \cong D_{\indlisse} (X, \k) \qquad\text{and}\qquad \ind (D_{\cons} (X,\k)) \cong D_{\ic} (X,\k).\]
    As a result, $D_{\ic} (X, \k)$ is presentable.
  \end{corollary}
  
  \begin{lemma}\label{lem:sf_for_constructible}
    The six sheaf functors take constructible sheaves to constructible sheaves. As a result, there are six-functor formalisms
    \begin{align*}
      D_{\cons} (-, \k) : \Corr (\Sch_{\C}) &\to \infCat_{\k}, & D_{\ic} (- , \k): \Corr (\Sch_{\C}) &\to \Pr_{\k}^L, \\
      X &\mapsto D_{\cons} (X, \k), & X &\mapsto D_{\ic} (X, \k).
    \end{align*}
  \end{lemma}
  \begin{proof}
    The six-functor formalism for constructible sheaves follows from the classical fact that the six functors preserve constructible sheaves.
    The extension to ind-constructible follows from Corollary \ref{cor:Dic_is_presentable}.
  \end{proof}

  \begin{remark}
    Let $\k \to \k'$ be a ring homomorphism. Extension of scalars preserves constructibility, so there is a functor
    \[\k' \otimes_{\k} (-) : D_{\cons} (X, \k) \to D_{\cons} (X, \k').\]
     Similarly, there is an extension of scalars functor for ind-constructible sheaves.
  \end{remark}

  \begin{remark}\label{rem:homotopy_cat_of_constr_sh}
    The homotopy category of $D_{\cons} (X, \k)$ recovers the usual bounded derived category of constructible sheaves on $X$ with coefficients in $\k$.
  \end{remark}

\subsection{Constructible Sheaves on Stacks}

Let $\Stk_{\C}$ denote the category of algebraic stacks locally of finite type over $\C$.\footnote{On occasion, we write stack to refer to a (2,1)-sheaf with the fppf topology. We reserve algebraic stack to mean a stack which has representable diagonal and a smooth atlas.}
For an algebraic stack $X$, we can define the over slice category $\Sch_{/X}$ as the category of representable morphisms $S \to X$ where $S$ is a scheme. 

\begin{definition}\label{def:cons_for_stks}
    Let $X \in \Stk_{\C}$. Define a category $D_{\ic} (X, \k)$ given by,
    \[D_{\ic} (X, \k) = \lim_{\stackrel{\longleftarrow}{S \in \Stk_{/X}}} D_{\ic} (S, \k),\]
    where the transition maps are given by $*$-pullbacks.
    
    We define $D_{\cons} (X, \k) \subset D_{\ic} (X, \k)$ as the full subcategory of $D_{\ic} (X, \k)$ of objects $\scrF$ such that for all $f : S
    \to X$ with $S \in \Sch_{\C}$, one has that $f^* \scrF \in D_{\cons} (S, \k)$.

    We call the objects of $D_{\cons} (X, \k)$ (resp. $D_{\ic} (X, \k)$) are called \emph{constructible} (resp. \emph{ind-constructible}) sheaves on $X$.
\end{definition}  

\begin{proposition}[{\cite[Proposition A.5.16]{Mann}}]\label{prop:six_functor_for_stks}
    There is an extension of the 6-functor formalism for schemes to a 6-functor formalism
    \[D_{\ic} (-, \k) : \Corr (\Stk_{\C}) \to \Pr^L \hspace{1cm} X \mapsto D_{\ic} (X, \k).\]
\end{proposition}

\subsubsection{Smooth Descent}

In the definition of ind-constructible sheaves on algebraic stacks, we take a limit over a large class of objects. 
This is beneficial since it allows us to build the six-formalism via Kan extensions. However, in practice
this also makes constructing sheaves on an algebraic stack rather difficult due to the extensive
number of compatibilities one must check. A convenient way to simplify this complexity is via smooth $*$-descent.

\begin{definition}\label{def:descent}
    Let $D$ be a 6-functor formalism on $\Sch_{\C}$. Let $f : X \to Y$ be a smooth surjective morphism of schemes.
    Let $X^{n/Y}$ denote the $n$-fold product of $X$ over $Y$. This comes equipped with projection maps $f_n : X^{n/Y} \to Y$.
    We say that $D$ satisfies \emph{smooth $*$-descent along $f$} if
    \[(f_\Delta^*) : D (Y) \to \lim_{\stackrel{\longleftarrow}{\Delta}} D (X^{n/Y})\]
    is an isomorphism where the transition maps of the limit are given by $*$-pullback.
  \end{definition}
  
  \begin{proposition}\label{prop:smooth_star_descent}
    The 6-functor formalism $D_{\ic}$ on $\Sch_{\C}$ satisfies smooth $*$-descent.
  \end{proposition}
  \begin{proof}
    By \cite[Proposition 6.18]{Sch}, it suffices to show that for any smooth surjective morphism $f : X \to Y$, the functor $f^* : D_{\ic} (Y, \k) \to D_{\ic} (X, \k)$
    is conservative. This can be shown via a simple stalk computation.
  \end{proof}

Let $\pi : U \to X$ be a smooth atlas of $X$. Consider the Čech nerve $C_\bullet (\pi)$. It is the simplicial object in $\Sch_{\C}$ whose degree $n$ term is $U^{n/X}$.
Note that $U^{n/X} \in \Sch_{/X}$.

\begin{lemma}\label{lem:cech_complex_suffices}
  The induced functor
  \[ \lim_{\stackrel{\longleftarrow}{S \in C_\bullet (\pi)}} D_{\ic} (S, \k) \to \lim_{\stackrel{\longleftarrow}{S \in \Sch_{/ X}}} D_{\ic} (S, \k) \]
  is an equivalence of categories.
\end{lemma}
\begin{proof}
  The lemma amounts to showing that for every $\scrF \in \lim_{S \in C_\bullet (\pi)} D_{\ic} (S, \k)$, we can uniquely define $\scrF (T)$ for all $T \in \Sch_{/X}$ such
  that for any $f : T \to T'$ in $\Sch_{/X}$, there is an isomorphism $f^* \scrF (T') \cong \scrF (T)$.
  Moreover, these isomorphisms should extend those given in $\scrF (U^{n/X})$.

  Let $Z = T \times_X U$ which is a scheme equipped with maps $p_T : Z \to T$ and $p_U^0 : Z \to U$. Consider the commutative cube
\[
    \begin{tikzcd}
      Z^{1/T} \arrow[rr] \arrow[dd] \arrow[rd, "p_U^1"] &                               & Z \arrow[rd] \arrow[dd] &              \\
      & U^{1/X} \arrow[rr] \arrow[dd] &                         & U \arrow[dd] \\
      Z \arrow[rd, "p_U^0"] \arrow[rr]                  &                               & T \arrow[rd]            &              \\
      & U \arrow[rr]                  &                         & X,
    \end{tikzcd}
\]
  where $p_U^1 : Z^{T/1} \to U^{X/1}$ follows from the universal property of pullbacks. We can repeat this process to obtain morphisms $p_U^n : Z^{n/T} \to U^{n/X}$
  which fit in the commutative diagram
\[
    \begin{tikzcd}
      \ldots \arrow[r] \arrow[r, shift right] \arrow[r, shift left=2] \arrow[r, shift left] & U^{2/X} \arrow[r, shift left] \arrow[r] \arrow[r, shift right]
      & U^{1/X} \arrow[r] \arrow[r, shift left]                     & U \arrow[r]                     & X.           \\
      \ldots \arrow[r] \arrow[r, shift left] \arrow[r, shift right] \arrow[r, shift left=2] & Z^{2/T} \arrow[r, shift left] \arrow[u, "p_U^2"'] \arrow[r] \arrow[r, shift
      right] & Z^{1/T} \arrow[r] \arrow[u, "p_U^1"'] \arrow[r, shift left] & Z \arrow[r] \arrow[u, "p_U^0"'] & T \arrow[u]
    \end{tikzcd}
\]
  Note that $C_\bullet (\pi) = U^{\bullet/X}$ and the above diagram gives a morphism of simplicial schemes $p_U : Z^{\bullet/T} \to U^{\bullet/X}$. Therefore, we can
  define $\scrG \in \lim D_{\ic} (Z^{n/T}, \k)$ by $p_u^* \scrG$.
  By Proposition \ref{prop:smooth_star_descent}, we have that there is a unique object $\scrF (T) \in D_{\ic} (T)$ such that its $*$-pullback to $Z^{n/T}$ is $\scrG$.
  It is clear that this object satisfies the pullback compatibilities extending $\scrF$ uniquely to an object in $D_{\ic} (X, \k)$.
\end{proof}

\subsubsection{Equivariant Sheaves}

Let $H$ be an algebraic group over $\C$. By algebraic group, we will always mean a group scheme of finite type which is not necessarily reduced. Suppose $H$ acts on a scheme $S \in \Sch_{\C}$. We write $H \backslash S$ for the quotient stack afforded by this action.
The $H$-equivariant constructible sheaves on $S$ should be thought of as the category of constructible sheaves on $H\backslash S$.

In the remainder of the section, we will describe $D_{\ic} (H\backslash S, \k)$ in the language of categorical representation theory. This perspective
is particularly nice as it allows one to construct $D_{\ic} (H \backslash S, \k)$ from $D_{\ic} (S, \k)$ via homological algebra.
This perspective will be invaluable when defining twisted equivariant sheaves.

We can equip $D_{\ic} (H, \k)$ with a monoidal structure. Let $m : H \times H \to H$ denote the multiplication map. The convolution is defined as
\[\scrF \star \scrG \coloneq m_! (\scrF \boxtimes \scrG ),\]
where $\scrF \in D_{\ic} (H, \k)$ and $\scrG \in D_{\ic} (X, \k)$. 

One can also define a $D_{\ic} (H, \k)$-module structure on $D_{\ic} (S, \k)$ through similar means. Namely, let $a : H \times S \to S$ denote the action map. We can then define
\[\scrF \star \scrG \coloneq a_! (\scrF \boxtimes \scrG),\]
where $\scrF \in D_{\cons} (H, \k)$ and $\scrG \in D_{\cons} (S, \k)$. 

\begin{definition}\label{def:inv_and_coinv}
  Let $\scrC \in \Mod_{D_{\ic} (H, \k)} (\Pr^L)$. We define 
  \begin{enumerate}
    \item the category of \emph{$H$-invariants} $\scrC^H \coloneq \Hom_{D_{\ic} (H, \k)} (D(\k), \scrC)$,
    \item the category of \emph{$H$-coinvariants} $\scrC_H \coloneq \scrC \otimes_{D_{\ic} (H, \k)} D(\k)$.
  \end{enumerate}
\end{definition}

\begin{theorem}[{\cite[Theorem B.1.2]{Gai}}]\label{thm:invs_vs_coinvs}
    Let $\scrC \in \Mod_{D_{\ic} (H, \k)} (\Pr^L)$. 
    The natural forgetful functor $\For_H : \scrC^H \to \scrC$ admits a left adjoint $\Av_{H!} : \scrC \to \scrC^H$ which factors through an equivalence of categories
    \[\scrC_H \to \scrC^H.\]
\end{theorem}
\begin{proof}
    The argument given in \cite{Gai} is for $\overline{\Q}_{\ell}$ sheaves; however, the only characteristic dependent part of the argument is the statement that $R\Gamma (\uk_H) \in D(\k)$ is compact.
    Since $\k$ has finite global dimension and since $H$ is finite dimensional, $R\Gamma (\uk_H)$ is a perfect $\k$-module.
\end{proof}

\begin{proposition}\label{prop:two_views_of_equivariant_sheaves}
  There is an equivalence of categories,
  \[D_{\ic} (H \backslash S, \k) \cong D_{\ic} (S, \k)^H.\]
\end{proposition}
\begin{proof}
  Let $\pi : S \to H \backslash S$ be the quotient map. Note that $\pi$ is a smooth atlas for $H \backslash S$.
  From the same argument given in Theorem \ref{thm:invs_vs_coinvs}, $D_{\ic} (H, \k)$ is compactly generated and in fact self-dual via the Verdier duality functor.
  We can then compute 
  \begin{align*}
    \Hom_{\Pr^L} (D_{\ic} (H,\k)^{\otimes n} \otimes \Vect, D_{\ic} (S, \k)) &\cong \Hom_{\Pr^L} (D_{\ic} (H,\k)^{\otimes n} , D_{\ic} (S, \k)) \\
    &\cong (D_{\ic} (H,\k)^{\otimes n})^\vee \otimes D_{\ic} (S, \k) \\
    &\cong D_{\ic} (H, \k)^{\otimes n} \otimes D_{\ic} (S, \k)
  \end{align*}
  One can then compute $D_{\ic} (S, \k)^H$ as the totalization of the following diagram
\[
    \begin{tikzcd}
      {D_{\ic} (S, \k)} \arrow[r, shift left] \arrow[r] & {D_{\ic} (H, \k) \otimes D_{\ic} (S, \k)} \arrow[r, shift left] \arrow[r] \arrow[r, shift right] & {D_{\ic} (H,
      \k)^{\otimes 2} \otimes D_{\ic} (S, \k)} \arrow[r, shift left=2] \arrow[r, shift left] \arrow[r] \arrow[r, shift right] & \ldots. {}
    \end{tikzcd}
  \]
  By \cite[1.4.6]{Gai}, this can be rewritten as the following the totalization of the following diagram
\[
    \begin{tikzcd}
      {D_{\ic} (S, \k)} \arrow[r, shift left] \arrow[r] & {D_{\ic} (H \times S, \k)} \arrow[r, shift left] \arrow[r] \arrow[r, shift right] & {D_{\ic} (H^2 \times S,
      \k)} \arrow[r, shift left=2] \arrow[r, shift left] \arrow[r] \arrow[r, shift right] & \ldots, {}
    \end{tikzcd}
  \]
  whose transition maps are given by pullbacks along partial multiplication maps.

  Likewise, the Čech complex $C_\bullet (\pi)$ is given by $S^{n/(H \backslash S)} \cong H^n \times S$ where the projection maps $S^{n/(H \backslash S)} \to S^{(n-1)/(H
  \backslash S)}$ become the partial multiplication maps $H^n \times S \to H^{n-1} \times S$.
  As a result, the proposition then follows from Lemma \ref{lem:cech_complex_suffices}.
\end{proof}

We will finish this section by comparing the homotopy category of $D_{\cons} (H \backslash S, \k)$ with the Bernstein--Lunts triangulated derived category of $H$-equivariant sheaves on $S$.

\begin{proposition}\label{prop:homotopy_cat_of_equiv_sh}
  There is an equivalence of triangulated categories,
  \[\Ho (D_c (H \backslash S, \k)) \cong hD_H (S, \k),\]
  where $hD_H (S, \k)$ is the (triangulated) bounded derived category of $H$-equivariant constructible sheaves on $S$.
\end{proposition}
\begin{proof}
  Let $\Sch_{/(H \backslash S)}^{\textnormal{resn}}$ be the category of $H$-resolutions over $S$.
  Explicitly, the objects in $\Sch_{/(H \backslash S)}^{\textnormal{resn}}$ consist of pairs $(H \backslash P, \pi)$ where $P$ is a principle $H$-bundle equipped with a smooth $H$-equivariant map $\pi : P \to S$. We will also write $\pi : H
  \backslash P \to H \backslash S$ for the induced map on quotient stacks.
  For a fixed $H$-resolution $P$, $P \times_S P$ is also an $H$-resolution, and $\pi : H \backslash P \to H \backslash S$ is a smooth atlas.
  By Lemma \ref{lem:cech_complex_suffices}, one has that
  \[D_{\ic} (H \backslash S, \k) \cong \lim_{\stackrel{\longleftarrow}{H \backslash P \in \Sch_{/(H \backslash S)}^{\textnormal{resn}}}} D_{\ic} (H \backslash P, \k).\]
  We can then construct the functor,
  \[\Phi : \Ho (D_{\cons} (H \backslash S, \k)) \to hD_H (S, \k)\]
  by $\Phi (\scrF) = (P \mapsto \scrF (H \backslash P))$.
  It then follows from Remark \ref{rem:homotopy_cat_of_constr_sh} that this will give an equivalence of triangulated categories.
\end{proof}

\subsection{Twisted Equivariant Sheaves}\label{subsec:twisted_eq_sh}

Let $H$ be a complex algebraic group of finite type and let $X$ be an algebraic $H$-stack over $\C$ of finite type.
Let $\scrF \in D_{\ic} (H \backslash X, \k)$ and let $\pi : X \to H \backslash X$ be the quotient map.
We can then consider the action map $a : H \times X \to X$.
A basic property of equivariant sheaves states that there is an isomorphism
\[a^* \pi^* \scrF \cong \underline{\k}_H \boxtimes \pi^* \scrF.\]
One can think of this condition as a weak equivariance condition. Informally, equivariant sheaves are eigen-objects for $a^*$ with eigenvalue $\underline{\k}_H$.
The basic premise with twisted equivariant sheaves is that one could replace the eigenvalue with any multiplicative sheaf on $H$.

The theory of twisted equivariant sheaves is well-developed in the triangulated category setting in \cite{LY} and \cite{Gou}. In the $\infty$-category setting for
$D$-modules this has been detailed in \cite{Gai}, and in the $\infty$-category setting for étale constructible sheaves the theory is contained in \cite{Et}.
We follow Gaitsgory and Eteve's formulation to define twisted equivariant sheaves in our setting. The main benefit in their approach is that it allows for additional flexibility in the types of multiplicative local systems allowed.
Unlike constructible sheaves on schemes and algebraic stacks, twisted equivariant
sheaves will not satisfy a 6-functor formalism in the usual sense.
The main obstacle encountered is that our category of correspondences in this setting lacks a terminal object.

The motivation of the definition comes from our reformulation in Proposition \ref{prop:two_views_of_equivariant_sheaves}.
Recall that for an algebraic $H$-stack $X \in \Stk_{\C}$, we had a $D_{\ic} (H, \k)$-action on $D_{\ic} (X, \k)$ denoted by $\star$ which made $D_{\ic} (X, \k)$ into a $D_{\ic} (H, \k)$-module.

\begin{definition}\label{def:mult_loc_sys}
  A rank one local system $\scrL \in D_{c} (H, \k)$ is \emph{multiplicative} if it is equipped with the following data:
  \begin{enumerate}
    \item A trivialization at $1 \in H$, i.e., an isomorphism $1^* \scrL \cong \k$,
    \item An isomorphism $m^* \scrL \cong \scrL \boxtimes \scrL$ where $m : H \times H \to H$ is the multiplication map such that the restriction at $(1,1) \in H \times
      H$ gives a compatible isomorphism with (1).
  \end{enumerate}
  We denote the set of all multiplicative local systems by $\Ch (H, \k)$.
\end{definition}

\begin{lemma}\label{lem:ch_closed_under_extension_of_scalars}
  Let $\scrL \in \Ch (H, \k)$. Let $\k' \to \k$ be a ring homomorphism. Then $\k' (\scrL) \in \Ch (H, \k')$.
\end{lemma}
\begin{proof}
  Note that the trivialization condition ensures that $\scrL$ is locally free. As a result, $\k' (\scrL)$ is a local system.
  The result then follows from extension of scalars commuting with sheaf functors.
\end{proof}

\subsubsection{Six-Functors}

Fix a multiplicative local system $\scrL \in D_{\cons} (H, \k)$.
We define a new action of $D_{\ic} (H, \k)$-action on $D_{\ic} (X, \k)$ denoted by $\star^{\scrL}$ obtained by twisting by $\scrL$. Precisely, we define
\[\scrF \star^{\scrL} \scrG = (\scrF \otimes^L \scrL^{-1}) \star \scrG,\]
where $\scrF \in D_{\ic} (H, \k)$ and $\scrG \in D_{\ic} (X, \k)$. This action makes $D_{\ic} (X, \k)$ into a $D_{\ic} (H, \k)$-module.
We will write $D_{\ic} (H, \k) \stackrel{\scrL}{\curvearrowright} D_{\ic} (X, \k)$ to indicate this action.

\begin{definition}\label{def:twisted_equiv_sh}
  The category of \emph{ind-constructible $(H, \scrL)$-equivariant sheaves on $X$} is defined to be
  \[D_{\ic} (H \backslash_{\scrL} X, \k) \coloneq D_{\ic} (X, \k)^{(H, \scrL)},\]
  i.e., the $H$-invariants of sheaves on $X$ with respect to the $\scrL$-twisted action.
\end{definition}

As in the equivariant setting, we get a pair of adjoint functors,
\[\Av_{H, \scrL!} : D_{\ic} ( X, \k) \leftrightarrows D_{\ic} (H \backslash_{\scrL} X, \k) : \For_{H, \scrL}.\]
We define the subcategory of \emph{constructible $(H, \scrL)$-equivariant sheaves on $X$}, denoted $D_{\cons} (H \backslash_{\scrL} X, \k)$, as the full subcategory of $D_{\ic} (H \backslash_{\scrL} X, \k)$ consisting of sheaves $\scrF$ such that $\For_{H, \scrL} (\scrF) \in D_{\cons} (X, \k)$.

\begin{definition}\label{def:twisted_stacks}
    A \emph{twisted stack} is a triple $(X, H, \scrL)$ where $H$ is an algebraic group, $X \in \Stk_{\C}$ is an $H$-stack, and $\scrL$ is a multiplicative rank one local system on $H$.
    A morphism of twisted stacks $\tilde{f} : (X', H', \scrL') \to (X, H, \scrL)$ consists of a triple $(f, \varphi, \theta)$ where $f : X' \to X$ is a representable morphism of
    stacks, $\varphi : H' \to H$ is a morphism of algebraic groups, and $\theta : \varphi^* \scrL \to \scrL'$ is a morphism of local systems such that
    \begin{enumerate}
      \item $f$ is $\varphi$-equivariant, i.e., the following diagram commutes
        \[
          \begin{tikzcd}
            H' \times X' \arrow[r, "\varphi \times f"] \arrow[d, "a'"] & H \times X \arrow[d, "a"] \\
            X' \arrow[r, "f"]                                          & X,
          \end{tikzcd}
        \]
        where $a : H \times X \to X$ and $a' : H' \times X' \to X'$ are the action maps.
      \item $\theta : \varphi^* \scrL \stackrel{\sim}{\to} \scrL'$ is an isomorphism of local systems.
    \end{enumerate}
    We will frequently denote a twisted stack by $H \backslash_{\scrL} X$. If $\tilde{f} = (f,\varphi, \theta)$ is a morphism of twisted stacks, we will often abuse notation and
    write $f$ to refer to the whole triple.
    We denote the category of twisted stacks by $\TwStk_{\C}$.
  \end{definition}
  
  \begin{lemma}\label{lem:twstk_and_pullbacks}
    The category of twisted stacks $\TwStk_{\C}$ admits all pullbacks.
  \end{lemma}
  \begin{proof}
    Consider the diagram
    \[
      \begin{tikzcd}
        & H_2 \backslash_{\scrL_2} X_2 \arrow[d, "\tilde{f}_2"] \\
        H_1 \backslash_{\scrL_1} X_1 \arrow[r, "\tilde{f}_1"] & H \backslash_{\scrL} X.
      \end{tikzcd}
    \]
    Our goal is to construct the pullback of $\tilde{f}_1$ and $\tilde{f}_2$.
    Consider the cartesian squares
    \[
      \begin{tikzcd}
        H_1 \times_H H_2 \arrow[r] \arrow[d] & H_2 \arrow[d] &  & X_1 \times_X X_2 \arrow[r] \arrow[d] & X_2 \arrow[d] \\
        H_1 \arrow[r]                        & H,             &  & X_1 \arrow[r]                        & X.
      \end{tikzcd}
    \]
    Note that $H_1 \times_H H_2$ is an algebraic group. We then can construct an action of $H_1 \times_H H_2$ on $X_1 \times_X X_2$ via the commutative diagram
    \[
      \begin{tikzcd}[column sep=small, row sep=small]
        H_1 \times_H H_2 \times X_1 \times_X X_2 \arrow[dd] \arrow[rr] \arrow[rd, "a'", dashed] &                                        & H_2 \times X_2 \arrow[rd, "a_2"]
        \arrow[dd] &                \\
        & X_1 \times_X X_2 \arrow[dd] \arrow[rr] &                                             & X_2 \arrow[dd] \\
        H_1 \times X_1 \arrow[rd, "a_1"] \arrow[rr]                                             &                                        & H \times X \arrow[rd, "a"]
        &                \\
        & X_1 \arrow[rr]                         &                                             & X,
      \end{tikzcd}
    \]
    where $a_1 : H_1 \times X_1 \to X_1$, $a_2 : H_2 \times X_2 \to X_2$, and $a : H \times X \to X$ denote the action maps. The existence of $a'$ follows from the
    universal property for pullbacks.
    It is easy to check that $a' :  H_1 \times_H H_2 \times X_1 \times_X X_2 \to X_1 \times_X X_2$ defines an $H_1 \times_H H_2$-action on $X_1 \times_X X_2$.
    Let $\scrL = \pr_1^* \scrL_1 \cong \pr_2^* \scrL_2$ where $\pr_i : H_1 \times_H H_2 \to H_i$ are the projection maps for $i=1,2$.
    Therefore, $(X_1 \times_X X_2, H_1 \times_H H_2, \scrL)$ is a twisted stack. Since $X_1 \times_X X_2$ and $H_1 \times_H H_2$ are pullbacks, it is easy to check that
    $(X_1 \times_X X_2, H_1 \times_H H_2, \scrL)$ is a pullback of $\tilde{f}_1$ and $\tilde{f}_2$.
  \end{proof}
  
    Let $E$ be the class of morphisms in $\TwStk_{\C}$ consisting of morphisms $\tilde{f} = (f, \varphi, \theta) : H' \backslash_{\scrL'} X' \to H \backslash_{\scrL} X$ such that $\varphi$ is an isomorphism of algebraic groups.
    Note that the category of twisted stacks $\TwStk_{\C}$ is not finitely complete. 
    In particular, it does not have a terminal object. 
    Nonetheless, the category of correspondences $\Corr (\TwStk_{\C}, E)$ is still well-defined. Moreover, the Cartesian product equips it with the structure of a monoidal category (without unit).
  
  \begin{lemma}\label{lem:equivariance_of_sh_functors}
    Let $f : H' \backslash_{\scrL'} X' \to H \backslash_{\scrL} X$ be a morphism of twisted stacks. 
    The functor $\varphi^* : D_{\ic} (H, \k) \to D_{\ic} (H', \k)$ gives rise to a $D_{\ic} (H, \k)$ action on $D_{\ic} (X', \k)$. 
    The functor $f^* : D_{\ic} (X, \k) \to D_{\ic} (X', \k)$ is $D_{\ic} (H, \k)$-equivariant with respect to the $\scrL$-twisted $D_{\ic} (H, \k)$-action.
  \end{lemma}

  Consider the category $\LMod_{\BiAlg} (\Pr_{\k}^L)$ of left modules over bialgebras in $\Pr_{\k}^L$. 
  The objects of $\LMod_{\BiAlg} (\Pr_{\k}^L)$ consists of pairs $(\scrA, \scrC)$ where $\scrA \in \BiAlg (\Pr_{\k}^L)$ is a bialgebra and $\scrC \in \LMod_{\scrA} (\Pr_{\k}^{L})$ is a left module over $\scrA$ (with respect to just the algebra structure). 
  The external tensor product of modules gives rise to functors
  \begin{equation}\label{eq:LMod_mon_str}
    \LMod_{\scrA} (\Pr_{\k}^L) \otimes \LMod_{\scrB} (\Pr_{\k}^L) \to \LMod_{\scrA \otimes \scrB} (\Pr_{\k}^L),
  \end{equation}
  which makes $\LMod_{\BiAlg} (\Pr_{\k}^L)$ into a monoidal category.

  Let $\scrA \in \BiAlg (\Pr_{\k}^L)$. The counit for $\scrA$ gives an action of $\scrA$ on $D(\k)$.
  There is a \emph{coinvariants} functor
  \[\Coinv_{\scrA} : \LMod_{\scrA} (\Pr_{\k}^L) \to \Pr_{\k}^L,\]
  informally given on objects by $\Coinv_{\scrA} \scrC = D(\k) \otimes_{\scrA} \scrC$.
  If $\Phi : \scrA \to \scrB$ is a bialgebra morphism, then there is an induced functor
  \[\For^{\scrA}_{\scrB} : \Coinv_{\scrA} (\scrC) \to \Coinv_{\scrB} (\scrC). \]
  This allows us to define a functor
  \[\Coinv : \LMod_{\BiAlg} (\Pr_{\k}^L) \to \Pr_{\k}^L\]
  defined by $\Coinv (\scrA, \scrC) = \Coinv_{\scrA} (\scrC)$ and on morphisms by
  \[\Coinv (\scrC \stackrel{F}{\to} \scrD, \scrA \stackrel{\Phi}{\to} \scrB) = \For_{\scrA}^{\scrB} \Coinv_{\scrA} (F). \]
  In fact, since (\ref{eq:LMod_mon_str}) is an equivalence of categories \cite[Proposition 8.5.4]{GR}, the coinvariants functor is monoidal. 

  \begin{proposition}\label{prop:actstk_six_functor}
    There exists a six-functor formalism
    \[D_{\ic} (-, \k): \Corr (\TwStk_{\C}, E) \to \Pr_{\k}^L, \]
    satisfying
    \[D_{\ic} (-, \k) (H \backslash_{\scrL} X) = D_{\ic} (H \backslash_{\scrL} X, \k).\]
  \end{proposition}
  \begin{proof}
    We can equip $D_{\ic} (H, \k)$ with the structure of a coalgebra by the pullback maps
    \[e^* : D_{\ic} (H, \k) \to D_{\ic} (\pt, \k) \hspace{0.5cm} \text{and} \hspace{0.5cm} m^* : D_{\ic} (H, \k) \to D_{\ic} (H \times H, \k),\]
    where $e$ is the unit map and $m$ is the multiplication map for $H$. This coalgebra structure along with $!$-convolution make $D_{\ic} (H, \k)$ into a bialgebra.

    By Lemma \ref{lem:equivariance_of_sh_functors}, there is a lax-monoidal\footnote{This should be interpreted as a lax-monoidal functor without any unit constraint.} functor
    \[\Corr (\TwStk_{\C}, E) \to \LMod_{\BiAlg} (\Pr_{\k}^L) \hspace{1cm} (H, X) \mapsto (D_{\ic} (H, \k) \stackrel{\scrL}{\curvearrowright}  D_{\ic} (X, \k)),\]
    making the following diagram commute
    \[ \begin{tikzcd}
        \Corr (\TwStk_{\C}, E) \arrow[r] \arrow[d] & \LMod_{\BiAlg} (\Pr_{\k}^L) \arrow[d] \\
        \Corr (\Stk) \arrow[r]              & \Pr_{\k}^L,                
        \end{tikzcd}
    \]
    where the vertical maps are the obvious projections. We can then define $D_{\ic} (-, \k): \Corr (\TwStk_{\C}, E) \to \Pr_{\k}^L$ as the composition
    \[\Corr (\TwStk_{\C}, E) \to \LMod_{\BiAlg} (\Pr_{\k}^L) \stackrel{\Coinv}{\to} \Pr_{\k}^L.\]
  \end{proof}

  One aspect lacking from Proposition \ref{prop:actstk_six_functor} is the proper pushforward of morphisms of twisted stacks of the form $f : H' \backslash_{\scrL'} X \to H \backslash_{\scrL} X$ when $H \not\cong H'$. 
  Nonetheless, it is often the case that the pullback functor $f^*$ admits a left adjoint. For example, if the morphism $H' \to H$ is smooth, then $f^*$ admits a left adjoint.
  We will often abuse notation and write $f_!$ for this left adjoint. 

  The categories $D_{\ic} (H \backslash_{\scrL} X, \k)$ are not monoidal unless $\scrL = \uk_H$-- a facet of the category of twisted stacks not admitting a terminal object.
  Nonetheless, we can define a tensor product of sheaves.
  There are morphisms of twisted stacks
  \[\Delta_{\scrL, \scrL'} : H \backslash_{\scrL \otimes^L \scrL'} X \to (H \times H) \backslash_{\scrL \boxtimes \scrL'} X \times X.\]
  We then define the \emph{tensor product} of twisted equivariant sheaves as the bifunctor given by the composition
  \begin{align*}
    \otimes^L : D_{\ic} (H \backslash_{\scrL} X, \k) \otimes D_{\ic} (H \backslash_{\scrL'} X, \k) &\stackrel{\sim}{\to} D_{\ic} ( (H \times H) \backslash_{\scrL \boxtimes \scrL'} X \times X, \k) \\
    &\stackrel{\Delta_{\scrL,\scrL'}^*}{\to} D_{\ic} (H \backslash_{\scrL \otimes \scrL'} X, \k).
  \end{align*}
  These bifunctors satisfy variants of the associativity and unit axioms. 

  The tensor product of sheaves admits a right adjoint,
  \[\RHom : D_{\ic} (H \backslash_{\scrL} X, \k)^{\op} \otimes D_{\ic} (H \backslash_{\scrL'} X, \k) \to  D_{\ic} (H \backslash_{\scrL^{-1} \otimes \scrL'} X, \k),\]
  called the \emph{sheaf-Hom functor} for twisted equivariant sheaves. Let $\omega_{H \backslash X}$ denote the dualizing complex on $H \backslash X$.
  We can then define a \emph{Verdier duality} functor
  \[\DD : D_{\ic} (H \backslash_{\scrL} X, \k)^{\op} \to D_{\ic} (H \backslash_{\scrL^{-1}} X, \k)\]
  by $\DD = \RHom (-, \omega_{H \backslash X})$.

  Let $\k \to \k'$ be a ring homomorphism. 
  The extension of scalars functor $\k' (-) : D_{\ic} (H, \k) \to D_{\ic} (H, \k')$ is an algebra morphism. As a result, after taking coinvariants there are induced functors
  \[\k' (-) : D_{\ic} (H \backslash_{\scrL} X, \k) \to D_{\ic} (H \backslash_{\k' (\scrL)} X, \k'),\]
  which we also call \emph{extension of scalars}. 
  Since extension of scalars commutes with the six-functor formalism of algebraic stacks, it is easy to deduce that it will commute with all six sheaf operations and Verdier duality for twisted equivariant sheaves as well.

  \begin{definition}
    Let $H \backslash_{\scrL} X$ be a twisted stack.
    A sheaf $\scrF \in D_{\ic} (H \backslash_{\scrL} X, \k)$ is said to be \emph{constructible} if $\For_{H, \scrL} \scrF \in D_{\cons} (X, \k)$.
    We denote by $D_{\cons} (H \backslash_{\scrL} X, \k)$ the full subcategory of $D_{\ic} (H \backslash_{\scrL} X, \k)$ consisting of constructible sheaves.
  \end{definition}

  \begin{remark}\label{rem:for_and_sheaf_ops}\quad
    \begin{enumerate}
      \item  There are forgetful functors $\For_{H, \scrL} : D_{\ic} (H \backslash_{\scrL} X, \k) \to D_{\ic} (X, \k)$ given by evaluation on 0-simplices.
      By construction, these forgetful functors commute with the six sheaf operations and extension of scalars.
      \item By (1), there is a six-functor formalism
        \[\Corr (\TwStk_{\C}, E) \to \infCat_{\k} \hspace{1cm} H \backslash_{\scrL}X \mapsto D_{\cons} (H \backslash_{\scrL} X, \k)\]
        given by restricting the six-functor formalism of Proposition \ref{prop:actstk_six_functor} to constructible sheaves. 
        Moreover, extension of scalars preserves constructibility.
    \end{enumerate}
  \end{remark}

  \subsubsection{Averaging and Forgetful Functors}

  Let $\varphi : H' \to H$ be a morphism of algebraic groups. Let $X$ be an algebraic $H$-stack of finite type and $\scrL \in \Ch (H, \k)$.
  The map $\overline{\varphi} : H' \backslash_{\varphi^* \scrL} X \to H \backslash_{\scrL} X$ given by $X \stackrel{\id}{\to} X$ and $\varphi : H' \to H$ gives rise to sheaf functors  
  \[\overline{\varphi}^* : D_{\ic} (H \backslash_{\scrL} X, \k) \leftrightarrows D_{\ic} (H' \backslash_{\varphi^* \scrL} X, \k) : \overline{\varphi}_*.\]
  We think of these functors as variants of equivariance forgetful and averaging functors.
  It will be useful to give an explicit description of these functors and applying $\For_{H, \scrL}$.

  Let $\scrF \in D_{\ic} (H \backslash_{\scrL} X, \k)$.
  It can be easily checked from the definitions that $\For_{H, \scrL} \scrF$ and $\For_{H', \scrL'} \overline{\varphi}^* \scrF$ are naturally isomorphic.

  Let $\scrG \in D_{\ic} (H' \backslash_{\scrL'} X, \k)$.
  The sheaf $\scrL \boxtimes \scrG$ can be regarded as an $(H \times H \times H', \scrL \boxtimes \scrL^{-1} \boxtimes \scrL')$-equivariant sheaf on $H \times X$ with respect to the action of $H \times H \times H'$ on $H \times X$ via $(h_1, h_2, h_3) \cdot (g,x) = (h_1 g h_2^{-1}, h_3 x)$.
  Let $\tau : H \times H' \to H \times H \times H'$ be given by $\tau (h_1, h_2) = (h_1, \varphi(h_2), h_2)$. 
  This gives rise to a morphism of twisted algebraic stacks 
  \[ \overline{\tau} : (H \times H') \backslash_{\scrL \times \uk_{H'}} H \times X \to (H \times H \times H') \backslash_{\scrL \boxtimes \scrL^{-1} \boxtimes \scrL'} H \times X,\]
  where $H \times H'$ acts on $H \times X$ via $(h_1, h_2) \cdot (g,x) = (h_1 g h_2^{-1} , h_2 x)$.
  Via the quotient equivalence, the sheaf $\overline{\tau}^* (\scrL \boxtimes \scrG) \in D_{\ic} \left( (H \times H') \backslash_{\scrL \times \uk_{H'}} H \times X , \k \right)$ can be regarded as a sheaf, denoted $\scrL \tilde{\boxtimes} \scrG$, in $D_{\ic} (H \backslash_{\scrL} H \times^{H'} X, \k)$. 
  Let $a : H \times^{H'} X \to X$ denote the action map.
  It follows from the Čech complex description of categorical coinvariants that $\overline{\varphi}_* \scrG \cong a_* (\scrL \tilde{\boxtimes} \scrG)$.
  In particular, the underlying sheaf of $\For_{H, \scrL} (\overline{\varphi}_* \scrG)$ is given by $a_* \For_{H, \scrL} (\scrL \tilde{\boxtimes} \scrG)$.

  When $\overline{\varphi} : H' \hookrightarrow H$ is an embedding, we will occasionally write $\Av_{(H', \scrL') *}^{(H, \scrL)} \coloneq \overline{\varphi}_*$ and $\For_{(H', \scrL')}^{(H, \scrL)} \coloneq \overline{\varphi}^*$. 

  \subsubsection{Comparison with Equivariant Sheaves}

  Let $\nu : \tilde{H} \to H$ be a finite central isogeny with kernel $K$. Let $\Ch_{\nu} (H, \k)$ denote the collection of simple local systems on $H$ (with multiplicity) which are summands of $\nu_* \underline{\k}_{\tilde{H}}$.
  When $\scrL \in \Ch_{\nu} (H, \k)$ it is often useful to view $(H, \scrL)$-equivariant sheaves as a subcategory of $\tilde{H}$-equivariant sheaves.
  We warn that in general such a $\nu$ need not exist.

  \begin{lemma}[{\cite[1.5.4]{Gai}}]\label{lem:monad_desc}
    The endofunctor $\For_{H, \scrL} \Av_{H, \scrL!}$ is given by $\scrL \star (-)$.
  \end{lemma}

  \begin{lemma}\label{lem:desc_as_equivariant_sheaves}
    Let $\nu : \tilde{H} \to H$ be a finite central isogeny.
    The category $D_{\ic} (\tilde{H} \backslash X, \k)$ decomposes as a direct sum of categories,
    \[D_{\ic} (\tilde{H} \backslash X, \k) = \bigoplus_{\scrL \in \Ch_{\nu} (H, \k)} D_{\ic} (\tilde{H} \backslash X, \k)_{\scrL}.\]
    Moreover, there is an equivalence of categories on the component pieces with twisted equivariant sheaves,
    \[D_{\ic} (\tilde{H} \backslash X, \k)_{\scrL} \cong D_{\ic} (H \backslash_{\scrL} X, \k).\]
  \end{lemma}
  \begin{proof}
    Write $A_{\scrL}$ (resp. $A_{\tilde{H}}$) for the monad given by $\scrL \star (-)$ (resp. $\uk_{\tilde{H}} \star (-)$) for the usual action of $D_{\ic} (H, \k)$ (resp. $D_{\ic} (\tilde{H}, \k)$ on $D_{\ic} (X, \k)$).
    By Lemma \ref{lem:monad_desc} and \cite[1.4.5]{Gai}, we have that $D_{\ic} (H \backslash_{\scrL} X, \k)$ can be identified with modules over the monad $A_{\scrL}$.
    Likewise, $D_{\ic} (\tilde{H} \backslash X, \k)$ can be identified with modules over the monad $A_{\tilde{H}}$.
    A simple sheaf computation gives that $A_{\tilde{H}} \cong \nu_! \uk_{\tilde{H}} \star (-)$ where the convolution here is instead in terms of the $D_{\ic} (H, \k)$-action.
    By definition of $\Ch_{\nu} (H, \k)$, we then have an isomorphism of comonads $A_{\tilde{H}} \cong \bigoplus_{\scrL \in \Ch_{\nu} (H, \k)} A_{\scrL}$.
    In particular, we obtain an equivalence of comodules over these comonads,
    \[D_{\ic} (\tilde{H} \backslash X, \k) = \bigoplus_{\scrL \in \Ch_{\nu} (H, \k)} D_{\ic} (H \backslash_{\scrL} X, \k)\]
    as desired.
  \end{proof}

  The main utility of Lemma \ref{lem:desc_as_equivariant_sheaves} is that it matches our definition of twisted-equivariant sheaves with other known definitions such as \cite{LY} and \cite{Gou}.

  \subsection{Twisted Equivariant Sheaves on Ind-Algebraic Stacks}

  We extend the theory of twisted equivariant sheaves to ind-algebraic stacks of ind-finite type with an action of an algebraic group of pro-finite type.
  
  \subsubsection{Ind-Algebraic Stacks}
  
  We introduce the theory of ind-algebraic stacks. The geometry reviewed in this section is essentially well-known (cf., \cite{Emerton}).
  
  A stack $X$ is called a \emph{(strict) ind-algebraic stack} if there exists a directed system $\{X_i\}_{i \in I}$ of algebraic stacks and an isomorphism of stacks $X \cong \underrightarrow{\lim}_{\stackrel{}{i}} X_i$ where the transition maps $X_i \to X_{i'}$ for $i \leq i'$ are closed embeddings.
  We further say that $X$ is of \emph{ind-finite type} if the $X_i$'s can be chosen such that $X_i$ is of finite type for all $i \in I$. 
  
  A morphism $f : X \cong \underrightarrow{\lim}_{\stackrel{}{i\in I}} X_i \to \underrightarrow{\lim}_{\stackrel{}{j \in J}} Y_j \cong Y$ is said to be \emph{bounded} if for all $j \in J$, the pre-image $f^{-1} (Y_j)$ is contained in $X_i$ for some $i \in I$, and the restriction of $f$ to $Y_i$ is representable of finite type.
  We will abuse notation and write $\IndStk_{\C}$ for the subcategory of stacks consisting of ind-algebraic stacks of ind-finite type with bounded morphisms.
  
  Given $\underrightarrow{\lim}_{\stackrel{}{i\in I}} X_i \cong X \in \IndStk_{\C}$ and $i \leq i'$, write $\kappa_{i, i'} : X_i \to X_{i'}$ for the transition map.
  We can then define the derived category of constructible sheaves on $X$ as the colimit of categories
  \[D_{\cons} (X, \k) \coloneq \lim_{\stackrel{\longrightarrow}{i\in I}} D_{\cons} (X_i, \k),\]
  where the transition maps are given by pushforward maps $\kappa_{i,i' !}$ for $i \leq i'$.
  Likewise, the category of ind-constructible sheaves on $X$, denoted $D_{\ic} (X, \k)$, can be obtained by replacing $D_{\cons} (X_i, \k)$ with $D_{\ic} (X_i, \k)$ in the above definition.
  It follows from a standard argument that $D_{\cons} (X, \k)$ does not depend on the choice of directed system $\{X_i\}_{i \in I}$.
  
  Let $f : X \cong \underrightarrow{\lim}_{\stackrel{}{i\in I}} X_i \to \underrightarrow{\lim}_{\stackrel{}{j \in J}} Y_j \cong Y$ be a bounded morphism of ind-algebraic stacks of ind-finite type.
  We can define sheaf functors
  \[f_! : D_{\cons} (X, \k) \to D_{\cons} (Y, \k) \hspace{0.5cm} \text{and}\hspace{0.5cm} f^* : D_{\cons} (Y, \k) \to D_{\cons} (X, \k).\]
  The boundedness constraint ensures that there are representable morphisms $f_{i,j} : X_i \to Y_j$ for all $j \in J$ and $i \in I$ satisfying $X_i \subseteq f^{-1} (Y_j)$ given by the restriction to $f$.
  We can then define $f_!$ and $f^*$ component-wise. Explicitly, if $\scrF \in D_{\cons} (X_i, \k)$, we can find some $j \in J$ such that $f (X_i) \subset Y_j$. Then $f_! \scrF \in D_{\cons} (Y_j, \k)$ can be viewed as an object in $D_{\cons} (Y, \k)$.
  Similarly, if $\scrG \in D_{\cons} (Y_j, \k)$, the $i$-th component of $f^* \scrG$ is given by $f_{i,j}^* \scrG$ for all $i \in I$ satisfying  $X_i \subseteq f^{-1} (Y_j)$.
  
  \subsubsection{Algebraic Groups of Pro-Finite Type}
  
  An algebraic group $H$ is said to be of \emph{pro-finite type} if there exists a projective system $\{H_j\}_{j \in J}$ of algebraic groups of finite type whose transition maps $H_j \to H_{j'}$ for $j \geq j'$ are smooth surjections, along with an isomorphism of algebraic groups $H \cong \underleftarrow{\lim}_{\stackrel{}{j}} H_j$.
  We further say that $H$ is \emph{predominantly pro-unipotent} if it contains a subgroup of finite codimension which is pro-unipotent, i.e., for sufficiently large $j$, $\ker (H \to H_j)$ is an inverse limit of unipotent algebraic groups of finite type.
  
  Let $H \cong \underleftarrow{\lim}_{\stackrel{}{j \in J}} H_j$ be an algebraic group of pro-finite type.
  We can define the derived category of ${}^*$-constructible sheaves on $H$ as the inductive limit,
  \[D_{\cons}^* (H, \k) = \lim_{\stackrel{\longrightarrow}{j \in J}} D_{\cons} (H_j, \k),\]
  where the transitions maps are given by pullbacks along the transition maps $\pi_{j,j'} : H_j \to H_{j'}$ where $j \geq j'$.
  
  A \emph{multiplicative local system} $\scrL \in D_{\cons}^* (H, \k)$ consists of a collection of multiplicative local systems $\scrL_j \in \Loc (H_j, \k)$ such that for $j'$ sufficiently large and $j \geq j'$, there are compatible isomorphisms $\pi_{j,j'}^* \scrL_{j'} \cong \scrL_j$.
  We denote the set of multiplicative local systems on $H$ by $\Ch (H, \k)$.

  Let $f : H  = \underleftarrow{\lim}_{j \in J} H_{j} \to G = \underleftarrow{\lim}_{i \in I} G_i$ be a morphism of algebraic groups of pro-finite type.
  For all $i \in I$, there exists some $j \in J$, such that $f$ factors through a group homomorphism $f_{i,j} : H_j \to G_i$. 
  We can then define the pullback functor
  \[ f^* : D_{\cons}^* (G, \k) \to D_{\cons}^* (H, \k) \]
  by $f^* \scrF_i \coloneq f_{i,j}^* \scrF_j$. 
  Similarly, given ${}^*$-sheaves $\scrF$ and $\scrG$ on $H$, we can define a ${}^*$-sheaf $\scrF \otimes^L \scrG$ by $(\scrF \otimes^L \scrG)_i \coloneq \scrF_i \otimes^L \scrG_i$.
  This gives rise to functor $\otimes^L : D_{\cons}^* (H, \k) \otimes D_{\cons}^* (H, \k) \to D_{\cons}^* (H, \k)$.
  Similarly, we can define the external tensor product, 
  \[ \boxtimes : D_{\cons}^* (H, \k) \times D_{\cons}^* (G, \k) \to D_{\cons}^* (H \times G, \k) \hspace{1cm} \scrF \boxtimes \scrG \coloneq \pr_H^* \scrF \otimes^L \pr_G^* \scrG,\]
  where $\pr_H : H \times G \to H$ and $\pr_G : H \times G \to G$ are the obvious projection maps.

  It is easy to check from definitions that if $m : H \times H \to H$ is the multiplication map for an algebraic group $H$ of pro-finite type and $\scrL \in \Ch (H, \k)$, then $m^* \scrL \cong \scrL \boxtimes \scrL$. 
  
  \subsubsection{Twisted Equivariant Sheaves}\label{subsec:twisted_eq_sh_on_ind_stacks}
  
  If $H$ is a predominantly pro-unipotent algebraic group, then we can write $H \cong \underleftarrow{\lim}_{j \in J} H_j$ where $H_j = H/K_j$ with $K_j \trianglelefteq K_{j'}$ for $j \leq j'$, $K_j$ unipotent for large $j$, and such that $K_j$ has finite codimension in $H$.
  If $H$ acts on an ind-algebraic stack $X \cong \underrightarrow{\lim}_{\stackrel{}{i\in I}} X_i$, then we say that this action is \emph{compatible} if for all $i$, there exists some $i' \geq i$ and $j \in J$ such that the algebraic stack $X_{i'}$ is $H$-stable and such that $K_j$ acts on $X_{i'}$ trivially.
  If $\scrL \in \Ch (H, \k)$, we will call the triple $H \backslash_{\scrL} X \coloneq (X, H, \scrL)$ a \emph{twisted ind-algebraic stack}.

  Fix $i \in I$. We can then find some $j \in J$ such that $K_j$ acts trivially on $X_i$. We define
  \[D_{\cons} (H \backslash_{\scrL} X_i, \k) \coloneq D_{\cons} (H_j \backslash_{\scrL_j} X_i, \k).\]
  If $j' \geq j$, then the inflation functor $D_{\cons} (H_j \backslash_{\scrL_j} X_i, \k) \to D_{\cons} (H_{j'} \backslash_{\scrL_{j'}} X_i, \k)$ is an equivalence of categories. 
  As a result, $D_{\cons} (H \backslash_{\scrL} X_i, \k)$ is independent of the choice of $j \in J$.

  We can now define the derived category of $(H, \scrL)$-equivariant sheaves on $X$ as the direct limit,
  \[D_{\cons} (H \backslash_{\scrL} X, \k) = \lim_{\stackrel{\longrightarrow}{i \in I}} D_{\cons} (H \backslash_{\scrL} X_i, \k),\]
  where the transitions maps are given by pushforwards along the maps $H_j \backslash_{\scrL_j} X_i \to H_j \backslash_{\scrL_j} X_{i'}$ where $i \leq i'$ and $j \in J$ is taken sufficiently large.
  
  There are variations of the sheaf functors for sheaves on twisted ind-algebraic stacks. Namely, if $f : H \backslash_{\scrL} X \to H' \backslash_{\scrL'} X'$ is a bounded morphism of twisted ind-algebraic stacks, there are sheaf functors $f_*$ and $f^*$.
  Moreover, when $H=H'$ and the component map $\varphi : H \to H'$ is the identity map, we can also have sheaf functors $f_!$ and $f^!$.
  We will not explicitly detail the construction of the six-functors nor assemble them into an abstract six-functor formalism.
  The key observation is that just as in the untwisted setting, sheaves on ind-algebraic stacks are entirely determined by their behavior on twisted algebraic stacks of finite type.
  As a result, all the properties of constructible sheaves on twisted algebraic stacks transfer without issue.

  Let $X \cong \underrightarrow{\lim}_{i \in I} X_i$ be an ind-algebraic stack of ind-finite type with a compatible action of an algebraic group of pro-finite type $H = \underleftarrow{\lim}_{j \in J} H_j$.
  Let $H' = \underleftarrow{\lim}_{k \in K} H_k'$ be another algebraic group of pro-finite type along with a morphism $\varphi : H' \to H$ of algebraic groups.
  We can define the \emph{induction space} as the ind-algebraic stack $H' \times^H X \coloneq  \underrightarrow{\lim} H_k' \times^{H_j} X_i$ where $i \in I, j \in K$, and $k \in K$ are taken to satisfy that $X_i$ is $H_j$-stable and $\varphi (H_k) \subseteq H_i$.
  Suppose that $H$ and $H'$ are predominantly unipotent. Let $\scrL \in \Ch (H,\k)$ and  $\scrL' = \varphi^* \scrL$. For each $\scrF \in D_{\cons} (H \backslash X, \k)$, 
  the twisted external product construction of \S \ref{subsec:twisted_eq_sh} gives rise to a sheaf $\scrL' \tilde{\boxtimes} \scrF \in D_{\cons} (H' \backslash_{\scrL'} H' \times^{H} X, \k)$.
  Via $*$-pushforward along the action map $a : H' \times^H X \to X$, we recover the classical description of the $*$-averaging functor $D_{\cons} (H \backslash_{\scrL} X, \k) \to D_{\cons} (H'  \backslash_{\scrL'} X, \k)$. 

\subsection{Perverse Sheaves}

In this section, we will construct various abelian categories of perverse sheaves for constructible on the various geometric settings we have considered thus far.

Let $S \in \Sch_{\C}$. We can define the \emph{perverse $t$-structure} on $D_{\cons} (S, \k)$ by
\[{}^p D_{\cons} (S, \k)^{\leq 0} = \{ \scrF \in D_{c} (S, \k) \mid \text{for all } i, \text{ we have } \dim \supp H^i (\scrF) \leq -i \}, \]
The nonnegative portion of the perverse $t$-structure ${}^p D_{\cons} (S, \k)^{\geq 0}$ is characterized by $\Hom (\scrF, \scrG [-1]) = 0$ for all $\scrF \in {}^p D_{\cons} (S, \k)^{\leq 0}$ and $\scrG \in {}^p D_{\cons} (S, \k)^{\geq 0}$.
The fact that this defines a $t$-structure is a standard argument (see \cite[\S3.1]{A1}).
We can then define the \emph{perverse sheaves} on $S$, denoted $\Perv (S, \k)$, as the heart of this $t$-structure.

Now let $X \in \Stk_{\C}$ and let $\pi : U \to X$ be a smooth atlas in $\Sch_{/X}$ of relative dimension $d$.
The \emph{perverse $t$-structure} on $D_{\cons} (X, \k)$ is given by
\[{}^p D_{\cons} (X, \k)^{\leq 0} = \left \{ \scrF \in D_{\cons} (X, \k) \mid \pi^* \scrF [d] \in {}^p D_{\cons} (S, \k)^{\leq 0} \right \}, \]
\[{}^p D_{\cons} (X, \k)^{\geq 0} = \left \{ \scrF \in D_{\cons} (X, \k) \mid \pi^* \scrF [d] \in {}^p D_{\cons} (S, \k)^{\geq 0} \right \}, \]
It is routine to then check that this is a $t$-structure and the resulting category of perverse sheaves on $X$, denoted $\Perv (X, \k)$, has the following description,
\[\Perv (X, \k) = \left \{ \scrF \in D_{\cons} (X, \k) \mid \pi^* \scrF [d] \in \Perv (S, \k) \right \}.\]
Moreover, a routine argument using functorality of pullbacks shows that the perverse $t$-structure on $D_{\cons} (X, \k)$ does not depend on the choice of smooth atlas.
If $X \cong \underleftarrow{\lim}_{i \in I} X_i$ is instead an ind-algebraic stack, we define the perverse $t$-structure on $D_{\cons} (X, \k)$ as the inductive limit of $t$-structures on the $D_{\cons} (X_i, \k)$'s.

Let $H \backslash_{\scrL} X$ be a twisted stack.
Recall, that there is a forgetful functor $\For_{H, \scrL} : D_{\cons} (H \backslash_{\scrL} X, \k) \to D_{\cons} (X, \k)$.
The \emph{perverse $t$-structure} on $D_{\cons} (H \backslash_{\scrL} X, \k)$ is given by
\[{}^p D_{\cons} (H \backslash_{\scrL} X, \k)^{\leq 0} = \left \{ \scrF \in D_{\cons} (H \backslash_{\scrL} X, \k) \mid \For_{H, \scrL} (\scrF) [\dim H] \in {}^p D_{\cons} (X, \k)^{\leq 0} \right \}, \]
\[{}^p D_{\cons} (H \backslash_{\scrL} X, \k)^{\geq 0} = \left \{ \scrF \in D_{\cons} (H \backslash_{\scrL} X, \k) \mid \For_{H, \scrL} (\scrF) [\dim H] \in {}^p D_{\cons} (X, \k)^{\geq 0} \right \}, \]
As before, we obtain an abelian category $(H, \scrL)$-equivariant perverse sheaves on $X$, denoted $\Perv (H \backslash_{\scrL} X, \k)$, which has the following description,
\[\Perv (H \backslash_{\scrL} X, \k) = \left \{ \scrF \in D_{\cons} (H \backslash_{\scrL} X, \k) \mid \For_{H, \scrL} (\scrF) [\dim H] \in \Perv (X, \k) \right \}.\]
If $H \backslash_{\scrL} X$ is instead a twisted ind-algebraic stack, we define the perverse $t$-structure on $D_{\cons} (H \backslash_{\scrL} X, \k)$ as the inductive limit of $t$-structures on the $D_{\cons} (H_j \backslash_{\scrL_j} X_i, \k)$'s.

It is easy to check that when $\k$ is a field, the Verdier duality functor $\DD$ restricts to a functor
\[\DD : \Perv (H \backslash_{\scrL} X, \k)^{\op} \to \Perv (H \backslash_{\scrL^{-1}} X, \k).\]

\subsection{Fixed Stratifications}

It will frequently be helpful to consider constructible sheaves with respect to a fixed stratification. We will always assume that our stratifications of ind-algebraic stack are by algebraic stacks of finite type. 

\begin{definition}\label{def:fixed_strat}
  Let $X = \bigsqcup_{\lambda \in \Lambda} X_{\lambda}$ be a stratification of an (ind)-algebraic stack $X$.
  We denote by $D_{\Lambda} (X, \k)$ the full subcategory of $D_{\cons} (X, \k)$ consisting of sheaves which are constructible with respect to $\{X_\lambda\}_{\lambda \in
  \Lambda}$.

  Further, suppose that $X$ is an (ind-)algebraic $H$-stack and that $\{X_\lambda\}_{\lambda \in \Lambda}$ is an $H$-stable stratification.
  Let $\scrL \in \Ch (H, \k)$. We denote by $D_{\Lambda} (H \backslash_{\scrL} X, \k)$ the full subcategory of $D_{\cons} (H \backslash_{\scrL} X, \k)$ consisting of sheaves $\scrF$ such that $\For_{H, \scrL} \scrF \in D_{\Lambda} (X, \k)$.
\end{definition}

\begin{definition}\label{def:stratified_mor}
  Let $f : H \backslash_{\scrL} X \to H' \backslash_{\scrL'} X'$ be a morphism of twisted stacks.
Let $X = \bigsqcup_{\lambda \in \Lambda } X_\lambda$ be an $H$-stable stratification of $X$. Let $X' = \bigsqcup_{\mu \in \Lambda' } X_\mu$ be an $H'$-stable stratification of $X'$. 
Suppose that $f$ further satisfies the following two conditions:
\begin{enumerate}
  \item for all $\lambda \in \Lambda$, $f^{-1} (X_\lambda) = \bigcup_{\mu \in \Lambda'_{\lambda} } X_\mu'$ where $\Lambda_{\lambda}' \subset \Lambda'$;
  \item for each $\mu \in \Lambda_{\lambda}'$, $f : X_\mu' \to X_\lambda$ is smooth.
\end{enumerate}
In this case, we will call $f$ is \emph{stratified}.
\end{definition}

\begin{definition}
  Let $H \backslash_{\scrL} X$ be a twisted (ind-)algebraic stack.
  We define a full subcategory $\Loc_{\textnormal{f}} (H \backslash_{\scrL} X, \k)$ of $D_{\cons} (H \backslash_{\scrL} X,\k)$ consisting of sheaves $\scrF$ such that $\For_{H, \scrL} \scrF \in \Loc_{\textnormal{f}} (X, \k)$.
  The objects in $\Loc_{\textnormal{f}} (H \backslash_{\scrL} X, \k)$ are called the \emph{$(H, \scrL)$-equivariant local systems on $X$}.
\end{definition}

Suppose $X$ is of finite type.
The category $D_{(X)} (H \backslash_{\scrL} X, \k)$ of $(H, \scrL)$-equivariant constructible sheaves with respect to the trivial stratification consists of sheaves $\scrF \in D_{\cons} (H \backslash_{\scrL} X, \k)$ such that ${}^p H^i (\scrF) [-\dim X_\lambda + \dim H] \in \Loc_{\textnormal{f}} (H \backslash_{\scrL} X, \k)$ for all $i \in \Z$.
It is convenient to define a functor
\[{}^p H^i : D_{(X)} (H \backslash_{\scrL} X, \k) \to \Loc (H \backslash_{\scrL} X, \k),\]
\[\scrF \mapsto {}^p H^i (\scrF) [-\dim X_\lambda + \dim H].\]

The following lemma follows from recollement.
\begin{lemma}
  Let $X$ be an algebraic $H$-stack with finitely many $H$-orbits.  
  Let $X = \bigsqcup_{\lambda \in  \Lambda} X_{\lambda}$ be the stratification of $X$ by the $H$-orbits. We denote by $j_{\lambda} : X_\lambda \to X$ the obvious inclusion maps.
  The category $D_{\Lambda} (H \backslash_{\scrL} X, \k)$ is the smallest stable subcategory of $D_{\cons} (H \backslash_{\scrL} X, \k)$ generated by $j_{\lambda!} \scrK$ (alternatively,  $j_{\lambda*} \scrK$) for all $\scrK \in \Loc_{\textnormal{f}} (H \backslash_{\scrL} X_\lambda, \k)$ and $\lambda \in \Lambda$.
\end{lemma}

\begin{remark}
  Let $f : H \backslash_{\scrL} X \to H' \backslash_{\scrL'} X'$ be a morphism of twisted stacks.
  Assume that there are finitely many $H$-orbits on $X$ and $H'$-orbits on $X'$. We take stratifications for $X$ and $X'$ with respect to the orbits.
  In this case, if $f$ is stratified, then $f_!$, $f_*$, $f^!$, and $f^*$ preserve the fixed stratifications.
\end{remark}

\end{document}